\documentclass[11pt]{article}

\usepackage[dvips]{epsfig}
\usepackage{graphicx}
\usepackage{amssymb,graphics,amsmath,amsthm,amsopn,amstext,amsfonts,color}
\usepackage{algorithm,algorithmic}

\usepackage{subfigure}
\usepackage{multirow}

\newcommand{\ubld}{{\mbox{\bf u}}}

\newcommand{\gap}{\vspace{0.1in}}

\newcommand{\abld}{{\mbox{\bf a}}}

\newcommand{\zbld}{{\mbox{\bf z}}}
\newcommand{\Hbld}{{\mbox{\bf H}}}

\newcommand{\dbld}{{\mbox{\bf d}}}

\newcommand{\wbld}{{\mbox{\bf w}}}
\newcommand{\wt}{\widetilde}

\newcommand{\wh}{\widehat}

\newcommand{\ol}{\overline}

\newcommand{\W}{\mathcal W}

\newcommand{\Wcal}{\mathcal W}
\newcommand{\Ical}{\mathcal I}
\newcommand{\Scal}{\mathcal S}
\newcommand{\Ecal}{\mathcal E}

\newcommand{\Ucal}{\mathcal U}

\newcommand{\Gcal}{\mathcal{G}}
\newcommand{\Pcal}{\mathcal{P}}

\newcommand{\Vcal}{\mathcal V}

\newcommand{\alphabf}{\boldsymbol\alpha}
\newcommand{\betabf}{\boldsymbol\beta}
\newcommand{\zetabf}{\boldsymbol\zeta}

\newcommand{\varphibf}{\boldsymbol\varphi}
\newcommand{\deltabf}{\boldsymbol\delta}

\newcommand{\mycut}[1]{{}}

\newcommand{\argmin}{\operatornamewithlimits{\arg\min}}

\newcommand{\setnewlength}[2]{\newlength{#1}\setlength{#1}{#2}}
\newlength{\defaultfboxruletodo}
\newlength{\defaultfboxseptodo}
\newlength{\boxwidthtodo}
\setnewlength{\fboxruletodo}{0.5pt} \setnewlength{\fboxseptodo}{3mm}

\newtheorem{theorem}{Theorem}[section] 
\newtheorem{lemma}{Lemma}[section] 
\newtheorem{corollary}{Corollary}[section] 
\newtheorem{proposition}{Proposition}[section] 
\newtheorem{definition}{Definition}[section] 
\newtheorem{remark}{Remark}[section]

\begin{document}

\title{Nonconvex, Fully Distributed  Optimization  based CAV Platooning Control under Nonlinear Vehicle Dynamics}


\author{
 Jinglai Shen\thanks{Jinglai Shen is with Department of Mathematics and
    Statistics, University of Maryland Baltimore County, MD 21250, USA. Email:
    {\tt\small shenj@umbc.edu}.},  \ \ \ \
 Eswar Kumar H. Kammara\thanks{Eswar Kumar H. K. is with Department of Mathematics and
    Statistics, University of Maryland Baltimore County, MD 21250, USA. Email:
    {\tt\small eswar1@umbc.edu}.}, \ \ \ \
   Lili Du\thanks{Lili Du is with Department of Civil and Coastal Engineering,
University of Florida, Gainesville, FL 32608, USA. Email: {\tt\small lilidu@ufl.edu}.}
     }

  \maketitle

%
\begin{abstract}
CAV platooning technology has received considerable attention in the past few years, driven by the next generation smart transportation systems.
%
%
Unlike most of the existing platooning methods that focus on linear vehicle dynamics of CAVs, this paper considers nonlinear vehicle dynamics and develops fully distributed optimization based CAV platooning control schemes via the model predictive control (MPC) approach for a possibly heterogeneous CAV platoon.
The nonlinear vehicle dynamics leads to several major difficulties in distributed algorithm development and control analysis and design. Specifically, the underlying MPC optimization problem is  nonconvex and densely coupled. Further, the closed loop dynamics becomes a time-varying nonlinear system subject to external perturbations, making closed loop stability analysis rather complicated.
%
%
To overcome these difficulties, we formulate the underlying MPC optimization problem as a locally coupled, albeit nonconvex, optimization problem and develop a sequential convex programming based fully distributed scheme for a general MPC horizon. Such a scheme can be effectively implemented for real-time computing using operator splitting methods.
To analyze the closed loop stability,  we apply various tools from global implicit function theorems, stability of linear time-varying systems, and Lyapunov theory for input-to-state  stability to show that the closed loop system is locally input-to-state stable uniformly in all small coefficients pertaining to the nonlinear dynamics.
%
%
Numerical tests on homogeneous and heterogeneous CAV platoons demonstrate the effectiveness of the proposed fully distributed schemes and CAV platooning control.
\end{abstract}


\mycut{
There is a surging interest in CAV platooning, supported by advanced sensing, vehicle communication, and portable computing technologies.

Various distributed optimization or control schemes have been developed for CAV platooning. However, the existing distributed schemes require either centralized data processing or centralized computation in at least one step of their schemes, referred to as partially distributed schemes. In this paper, we develop fully distributed optimization based CAV platooning control under the linear vehicle dynamics via the model predictive control approach with a general prediction horizon. The distributed schemes developed in this paper do not require centralized data processing or centralized computation through the entire schemes. To develop such the schemes, we propose a new formulation of the objective function and a decomposition method that decomposes the densely coupled central objective function into the sum of several locally coupled functions whose coupling satisfies the network topology constraint. We then exploit the formulation of locally coupled optimization and operator splitting methods to develop fully distributed schemes. Control design and stability analysis is carried out to achieve desired traffic transient performance and asymptotic stability. Numerical tests demonstrate the effectiveness of the proposed fully distributed schemes and CAV platooning control.
}

%
\section{Introduction}

Inspired by the next generation smart transportation systems, connected and autonomous vehicle (CAV) technologies emerge and offer tremendous opportunities to reduce traffic congestion and improve road safety and traffic efficiency in all aspects, through innovative traffic flow control and operations.
%
%
 Among a variety of  CAV technologies, vehicle platooning technology links a group of CAVs through cooperative acceleration or speed control. Different from many other CAV technologies that mainly focus on neighborhood traffic efficiency and individual vehicle's safety, the vehicle platooning technology focuses on  system efficiency and safety. Specifically, by using the vehicle platooning technology, adjacent group members of a CAV platoon can travel safely at a higher speed with smaller spacing. This will increase lane capacity, improve traffic flow efficiency, and reduce congestion, emission, and fuel consumption \cite{bergenhem2012overview, kavathekar2011vehicle}.

%
%

Extensive research on CAV platooning control has been conducted, and many approaches have been proposed, e.g., adaptive cruise control (ACC) \cite{kesting2008adaptive, li2011model, marsden2001towards, vander2002effects, zhou2017rolling}, cooperative adaptive cruise control (CACC)  \cite{shladover2015cooperative, Shladover2012, VanA2006, zhao2020distributionally}, and platoon centered vehicle platooning control \cite{GongDu_TRB18, GShenDu_TRB16, wang2019real, wang2014rolling}. The ACC and CACC approaches aim to improve an individual vehicle's safety and mobility as well as string stability instead of system performance of the entire platoon, although simulations and field experiments demonstrate that they do enhance system performance to some extent.
%
%
On the other hand, the recently developed platoon centered approach seeks to optimize the platoon's transient traffic dynamics for a smooth traffic flow and to achieve stability and other desired long-time dynamical behaviors. This approach can significantly improve system performance and efficiency of the entire platoon \cite{GShenDu_TRB16, wang2019real}. Despite this advantage, the platoon centered CAV platooning approach often encounters large-scale optimization or optimal control problems that require efficient numerical solvers for real-time computation \cite{wang2019real}.
 Distributed optimization techniques provide a favorable solution for the platoon centered approach. Supported by  portable computing capability of each vehicle and vehicle-to-vehicle (V2V) communication \cite{wang2016cooperative}, distributed computation can handle high computation load efficiently,
 is more flexible to communication network topologies,  and is more robust to communication delays or network malfunctions \cite{mesbahi2010graph, wang2016cooperative}. In this paper, we focus on the platoon centered CAV platooning via distributed optimization.

Various distributed control or optimization schemes have been proposed for CAV platooning \cite{wang2019real, wang2016cooperative, zhao2020distributionally, zhou2019distributed}. These schemes can be classified into two types: partially distributed schemes, and fully distributed schemes. Partially distributed schemes are referred  to as those schemes  that either require all vehicles to exchange information with a central component for centralized data processing or perform centralized computation in at least one step \cite{KNS_SIOPT11}, whereas fully distributed schemes do not require centralized data processing or carry out centralized computation through the entire schemes \cite{ShenEswarDu_Arx20}.
%
%
%
The former type includes \cite{GongDu_TRB18, GShenDu_TRB16}. In particular,
model predictive control (MPC) based CAV platooning is developed in \cite{GShenDu_TRB16} and implemented by  partially distributed schemes.
%
%
%
The paper \cite{GongDu_TRB18} extends these distributed schemes to a mixed traffic flow consisting of both CAVs and human-driven vehicles.
The second type includes the recent paper \cite{ShenEswarDu_Arx20}, which develops fully distributed schemes for CAV platooning under the linear vehicle dynamics. Compared with partially distributed schemes, fully distributed schemes do not need data synchronization or a central computing equipment, and they impose less restriction on vehicle communication networks and can be easily implemented on a wide range of vehicle networks; see \cite{ShenEswarDu_Arx20} for more details.

In spite of the abovementioned progress, most of the existing research considers the linear vehicle dynamics \cite{GShenDu_TRB16, GongDu_TRB18, ShenEswarDu_Arx20, wang2019real}. Although the linear vehicle dynamics is suitable for smaller passenger vehicles, nonlinear dynamic effects, e.g, aerodynamic drag, friction, and rolling resistance, play a non-negligible role in trucks, heavy duty vehicles, and other types of CAVs. %
%
Motivated by the lack of research for nonlinear vehicle dynamics, this paper aims to develop fully distributed optimization based and platoon centered CAV platooning under nonlinear vehicle dynamics over a general vehicle communication network.
To achieve this goal, we propose a general $p$-horizon MPC model subject to the nonlinear vehicle dynamics of the CAVs and various physical or safety constraints. Several new challenges arise for the MPC horizon $p\ge 2$ when the nonlinear vehicle dynamics is considered. First, the underlying MPC optimization problem gives rise to a  densely coupled, {\em nonconvex} optimization problem, where both the objective function and constraints are nonconvex. This is very different from the linear vehicle dynamics treated in \cite{ShenEswarDu_Arx20}, for which a convex MPC model is obtained so that various convex distributed optimization schemes can be used.
Second, a local optimal solution to the MPC is characterized by a highly sophisticated nonlinear equation and does not attain a closed-form expression. Hence, the closed loop system is defined by a time-varying nonlinear dynamical system, whose right-hand side has no closed form expression, subject to non-vanishing external disturbances. These pose a difficulty in closed loop stability analysis and control design. To address these challenges, we exploit new techniques for distributed algorithm development and control analysis and design, which constitute main contributions of this paper.

The major contributions of this paper are summarized as follows:
\begin{itemize}
  \item [(1)] To develop fully distributed schemes for the  nonconvex  MPC optimization problem when $p\ge 2$, we first formulate the underlying densely coupled MPC optimization problem as a locally coupled, albeit nonconvex, optimization problem using a decomposition method recently developed  for the linear CAV dynamics \cite{ShenEswarDu_Arx20}. Furthermore, we propose a sequential convex programming (SCP) \cite{ZSLu_Technote2013} based distributed scheme to solve the locally coupled optimization problem. This SCP based scheme solves a sequence of convex,  quadratically constrained quadratic programs (QCQPs) that approximate the original nonconvex program at each iteration; such a convex QCQP can be efficiently solved using (generalized) Douglas-Rachford method or other operator splitting methods \cite{DavisYin_SVA17} in the fully distributed manner. The  SCP based distributed scheme converges to a stationary point, which often coincides or is close to an optimal solution, under mild assumptions.

  \item [(2)] To analyze the closed loop dynamics, we first formulate the closed loop system as a tracking system defined by a time-varying, nonlinear dynamical system subject to non-vanishing external disturbances. The right-hand side of this nonlinear dynamical system depends on a local optimal solution to the underlying MPC optimization problem, which does not attain a closed-form expression. By exploiting global implicit function theorems, we show that this (local) optimal solution is an implicit smooth function of state variables for all sufficiently small parameters pertaining to the nonlinear dynamic effects. We then apply stability theory of linear time-varying systems and Lyapunov theory for input-to-state stability to show that for all sufficiently small parameters pertaining to the nonlinear dynamics terms, the closed loop system is locally input-to-state stable provided that the corresponding linear closed loop dynamics under the linear vehicle dynamics (or equivalently when the abovementioned parameters are zero) is Schur stable.

  %
  %

  \item [(3)] For real-time implementation of the proposed fully distributed schemes, initial guess warm-up techniques are developed. Besides, a further stability analysis shows that steady state errors  of spacing exist in  the close loop dynamics but can be made small by choosing suitable weights in the MPC model while ensuring the input-to-state stability and satisfactory performance of transient dynamics. Extensive numerical tests have been carried out for three types of CAV platoons in different scenarios: a homogeneous small-size CAV platoon, a heterogeneous medium-size CAV platoon, and a homogeneous large-size CAV platoon. The numerical results illustrate the effectiveness of the proposed distributed scheme and CAV platooning control under the nonlinear vehicle dynamics.
 %
%
\end{itemize}

The paper is organized as follows. Section~\ref{sect:dynamics_constraint_topology} introduces the nonlinear vehicle dynamics, state and control constraints, and vehicle communication networks. Sequential feasibility and properties of the constraint sets are established in Section~\ref{sect:seq_feasibility}; these properties lay a ground for distributed optimization.
 A MPC model with a general prediction horizon $p$ is proposed in Section~\ref{sect:MPC_formulation} and is formulated as a nonconvex constrained  optimization problem.
 Section~\ref{sect:dist_scheme} develops sequentially convex programming based fully distributed schemes for the densely coupled nonconvex  MPC optimization problem. Control design and stability analysis for the closed loop dynamics is given in Section~\ref{sect:control_analysis}, and numerical tests and their results are presented in Section~\ref{sect:numerical_results}. Finally, conclusions are made in Section~\ref{sect:conclusion}.

%
\section{Vehicle Dynamics, Constraints, and Communication Topology} \label{sect:dynamics_constraint_topology}

We consider a platoon consisting of heterogeneous vehicles (e.g., cars and trucks) on a roadway, where the (uncontrolled) leading vehicle is labeled by the index 0 and its $n$  following CAVs are labeled by the indices $i=1, \ldots, n$, respectively. Let $x_i, v_i$ denote the longitudinal position and speed of the $i$th vehicle, respectively. Let $\tau>0$ be the sampling time, and each time interval is given by $[k\tau, (k+1)\tau)$ for $k \in \mathbb Z_+:=\{0, 1, 2, \ldots\}$.  We introduce vehicle dynamical models as follows. 

We first introduce the following nonlinear vehicle dynamical model which captures aerodynamic drag, friction, and rolling resistance \cite{ZhengLiBorrelliH_TCS16}:
%
\begin{subequations} \label{eqn:d_model_longit}
\begin{eqnarray} 
  x_i(k+1) & = & x_i(k) + \tau v_i(k)+ \frac{\tau^2}{2} \big( u_i(k) - c_{2, i} \cdot v^2_i(k) - c_{3, i} \cdot g\big), \\ 
  v_i(k+1)  & = & v_i(k) + \tau\big( u_i(k) - c_{2, i} \cdot v^2_i(k) - c_{3, i} \cdot g\big),
\end{eqnarray}
\end{subequations}
where $u_i(k)$ denotes the desired driving/braking acceleration treated as the control input.
$c_{2, i} \cdot v^2_i(k)$ characterizes the deceleration due to aerodynamic drag with the coefficient $c_{2,i}>0$, and $c_{3, i} \cdot g$ characterizes friction and rolling resistance with $g=9.8 m/s^2$ being the gravity constant and $c_{3, i}>0$ being the rolling friction coefficient.  For different vehicles, the coefficients $ c_{2, i}, c_{3, i}$ can be different.
%
%

The coefficients $c_{2,i}$ and $c_{3, i}$ in model (\ref{eqn:d_model_longit}) are usually small for many different types of cars or road conditions. For example, $c_{2, i}$ typically ranges from $2.5\times 10^{-4}/m$ to $4.5\times 10^{-4}/m$, and  $c_{3, i}$ typically ranges from 0.006 to 0.015 \cite{ZhengLiBorrelliH_TCS16}. Since these coefficients are small, the nonlinear terms in (\ref{eqn:d_model_longit}) are often neglected in system-level studies.
This yields the following widely adopted double-integrator linear  model:
%
%
\begin{eqnarray} \label{eqn:model_longit_double_integrator}
  x_i(k+1) \, = \, x_i(k) + \tau v_i(k) + \frac{\tau^2}{2} u_i(k), \qquad
  v_i(k+1)  \, = \, v_i(k) + \tau u_i(k).
\end{eqnarray}
The model (\ref{eqn:model_longit_double_integrator}) is suitable for small-size passenger cars, while model (\ref{eqn:d_model_longit}) can be used for medium-size or large-size vehicles, e.g., trucks and heavy-duty vehicles.  These models are all well studied and widely accepted in the literature.
%
%

\gap

\noindent {\bf State and control constraints.} \ Each vehicle in a platoon is subject to several important state and control constraints. For each $i=1, \ldots, n$,
\begin{itemize}
  \item [(i)] \underline{Control constraint}: 
$ a_{i, \min} \le u_i \le a_{i, \max}$, where
$a_{i, \min}<0$ and $a_{i, \max}>0$ are pre-specified  acceleration or deceleration bounds for the $i$th vehicle;

 \item [(ii)] \underline{Speed constraint}: $v_{\min} \le v_i \le v_{\max}$, where $0\le v_{\min}<v_{\max}$ are pre-specified bounds on longitudinal speed for the $i$th vehicle;

 \item [(iii)] \underline{Safety distance constraint}: this constraint guarantees sufficient spacing between neighboring vehicles to  avoid collision even if the leading vehicle comes to a sudden stop.
This gives rise to the safety distance constraint of the following form:
 \begin{equation}
     x_{i-1} - x_{i}  \, \ge \, L_i + r_i \cdot v_i - \frac{(v_i-v_{\min})^2}{2a_{i,\min}}, \label{eqn:safety_constraint}
\end{equation}
where $L_i>0$ is a constant depending on vehicle length, and $r_i>0$ is the reaction time of 
vehicle $i$.
\end{itemize}
In the above constraints, the acceleration/decelerations bounds as well as the vehicle length $L_i$ and the reaction time $r_i$ can be different for different types of vehicles. Further,  constraints (i) and (ii) are decoupled across vehicles, whereas the safety distance constraint (iii) is state-control coupled since such a constraint involves  control inputs of two vehicles. This yields challenges to distributed computation.

\gap

\noindent {\bf Communication network topology.}
In this paper, we consider a general communication network whose topology is modeled by a  graph $\Gcal(\mathcal V, \mathcal E)$,
 where $\mathcal V=\{1, 2, \ldots, n\}$ is the set of nodes where the $i$th node corresponds to the $i$th CAV, and $\mathcal E$ is the set of edges connecting two nodes in $\mathcal V$.
 Let $\mathcal N_i$ denote the set of neighbors of node $i$, i.e., $\mathcal N_i =\{ j \, | \, (i, j) \in \mathcal E\}$.
The following assumption on the communication network topology is made throughout the paper:
\begin{itemize}
  \item [$\bf A.1$] The graph $\mathcal G(\mathcal V, \mathcal E)$ is undirected and connected. Further, two neighboring vehicles form a  bidirectional edge of the graph, i.e., $(1, 2), (2, 3), \ldots, (n-1, n) \in \mathcal E$.
\end{itemize}
Since the graph is undirected,  for any $i, j\in \mathcal V$ with $i \ne j$, $(i, j) \in \mathcal E$ means that there exists an edge between node $i$ and node $j$. In other words, vehicle $i$ can receive information from  vehicle $j$ and send information to vehicle $j$, and so does vehicle $j$.
The above setting given by $\bf A.1$ includes many widely used communication networks of CAV platoons, e.g., immediate-preceding,  multiple-preceding, and  preceding-and-following \cite{ZhengLiBorrelliH_TCS16}.
%
%
We also assume that the first vehicle can receive $x_0$, $v_0$ and $u_0$ from the leading vehicle.

\mycut{
In this paper, we consider a general communication network topology modeled by a {\bf connected and undirected graph} $\Gcal(\mathcal V, \mathcal E)$,
 where $\mathcal V=\{1, 2, \ldots, n\}$ is the set of nodes corresponding to the $i$th vehicle, and $\mathcal E$ is the set of bidirectional edges connecting two nodes in $\mathcal V$, i.e., if two nodes are connected by an edge in $\mathcal E$, then two vehicles represented by the two nodes can exchange information.
{\bf We always assume that
two neighboring vehicles form a  bidirectional edge of the graph, i.e., $(1, 2), (2, 3), \ldots, (n-1, n) \in \mathcal E$.
}
We also assume that the first vehicle can obtain $x_0$, $v_0$ and $u_0$ from the leading vehicle.
This setting includes many typical communication network topologies  of a CAV platoon, e.g., immediate-preceding,  multiple-preceding, and  preceding-and-following \cite{ZhengLiBorrelliH_TCS16}.
 Particularly, for $i, j\in \mathcal V$ with $i \ne j$, $(i, j) \in \mathcal E$ means that there exists a bidirectional edge between node $i$ and node $j$ (in other words, nodes $i$ can receive information from  node $j$ and send information to node $j$).
 Let $\mathcal N_i$ denote the set of neighbors of node $i$.
}

%
\section{Sequential Feasibility and Properties of Constraint Sets} \label{sect:seq_feasibility}

As indicated in \cite{GShenDu_TRB16}, the constraint set of the underlying MPC optimization problem  at time $k$ (cf. Section~\ref{sect:MPC_formulation}) depends on the position and speed of the vehicles at times $0,1, \ldots, k-1$. A fundamental question is whether the constraint set is nonempty at each time along a system trajectory for an arbitrary feasible initial condition at $k=0$,
provided that
$(u_0 (k ), v_0(k ))$ of the leading vehicle satisfies the acceleration and speed constraints for all $k \in \mathbb Z_+$. If the answer is affirmative,  the system is {\em sequentially feasible} \cite{GShenDu_TRB16}. The sequential feasibility has been shown for a CAV platoon under the linear vehicle dynamics \cite{GShenDu_TRB16}.
In what follows, we establish the sequential feasibility under the nonlinear vehicle dynamics (\ref{eqn:d_model_longit}).
%
%

%
\subsection{Sequential Feasibility}
%

%
%

For notational convenience, define $a_i(k, u_i(k)):= u_i(k)  - c_{2, i} v^2_i(k) - c_{3, i} g$ for given $(v_i(k), u_i(k))$'s.
It is noted that  $u_0(k)$ is the actual acceleration of the leading vehicle at time $k$ instead of its control. Hence, we set $c_{2, 0}=c_{3, 0}=0$ for notational convenience.
 When $k$ is fixed, we often write  it as $a_i(u_i)$ to emphasize the dependence of $a_i$ on $u_i$.
 Then the nonlinear vehicle dynamics given by (\ref{eqn:d_model_longit}) can be written as
\[
  x_i(k+1) \, = \, x_i(k) + \tau v_i(k) + \frac{\tau^2}{2} a_i(k, u_i(k)), \qquad
  v_i(k+1)  \, = \, v_i(k) + \tau a_i(k, u_i(k)).
\]

Given $(x_i, v_i)^n_{i=0}$ and  $u_0$, we introduce the following constraint set on the control $u$ subject to the nonlinear vehicle dynamics and the state and control constraints:
\[
  \Wcal((x_i, v_i)^n_{i=0}, u_0) \, : =  \Big\{ u \in \mathbb R^n \, | \, a_{i, \min} \le u_i \le a_{i, \max}, \, v_{\min} \le v_i + \tau  a_i(u_i) \le v_{\max},  \, h_{i}(u) \le 0, \, i=1, \ldots, n \Big\},
\]
where the function $h_i$'s are given by
\begin{align}
h_{i}(u) & \, := \, L_i + r_i ( v_i + \tau a_i(u_i) ) - \frac{(v_i + \tau a_i(u_i) - v_{\min})^2}{2 a_{i, \min}} +
 (x_i - x_{i-1}) + \tau ( v_i - v_{i-1}) \notag \\
  & \qquad + \frac{\tau^2}{2}[ a_{i}(u_i) - a_{i-1}(u_{i-1})], \label{eqn:h_i_func}
\end{align}
and by abusing notation, $a_i(u_i) :=  u_i  - c_{2, i} v^2_i - c_{3, i} g$ for each $i=0,1, \ldots, n$.
Further,  we define the following functions that describe the safety distances between two adjacent vehicles on their current position and speed:
\begin{eqnarray*}
%
p_{j}((x_i, v_i)^n_{i=0}) & := & L_j + r_j v_j  - \frac{(v_j - v_{\min})^2}{2 a_{j, \min}} + (x_j - x_{j-1}), \qquad \ \forall \ j=1, \ldots, n.
\end{eqnarray*}
%
%
To simplify notation, we often write these functions as $ p_{j}$  when $(x_i, v_i)^n_{i=0}$ and  $u_0$ are fixed in the subsequent development.
The sequential feasibility holds if $\Wcal((x_i, v_i)^n_{i=0}, u_0)$ is nonempty for any given feasible $(x_i, v_i)^n_{i=0}$ and $u_0$,
 i.e., $a_{0, \min} \le u_0 \le a_{0,\max}$, $v_{\min}\le v_0 \le v_{\max}$, $v_{\min} \le v_0+ \tau u_0 \le v_{\max}$, $v_{\min} \le v_i \le v_{\max}$ and $p_{i}((x_i, v_i)^n_{i=0})  \le 0$ for all $i=1, \ldots, n$.

\begin{proposition} \label{prop:seq_feasibility_nonlinear}
Consider the nonlinear vehicle dynamics given by (\ref{eqn:d_model_longit}). Suppose the nonnegative constants $c_{2,i}, c_{3,i}$ are such that $c_{2, i}v_{\max}^{2} + c_{3, i} g \leq a_{i, \max} $  and $r_i \ge \tau$ for each $i=1, \ldots, n$. Then the system is sequentially feasible for an arbitrary feasible initial condition.
\end{proposition}

\begin{proof}
  We present some technical preliminaries first. For each given $v_{i}$ satisfying $v_{\min} \le v_{i} \le v_{\max}$, define the continuous function $q_{i}: \mathbb{R} \rightarrow \mathbb{R}$ as
  \begin{equation} \label{eqn:q_function}
  q_{i}(w) \, := \, v_{i} + \bigg(\frac{\tau}{2} + r_i \bigg) \cdot w  - \frac{v_{i} - v_{\min}}{a_{i, \min}} w  - \frac{\tau w^{2}}{2a_{i, \min}} , \quad \forall \hspace{2mm} i = 1, \ldots, n.
  \end{equation}
 Moreover, in view of the definition of $h_i$ given by (\ref{eqn:h_i_func}), we write $h_i$ as $h_i(a_i, a_{i-1})$ by slightly abusing notation.
  We  claim the following result:
  \begin{align*}
    \mbox{{\bf Claim}}: & \ \ \mbox{Given any feasible $(x_i, v_i)^n_{i=0}$ and $u_0$, assume that there exist $a_i$'s such that } \\
     & \mbox{ $v_i+\tau a_i \ge v_{\min}, \forall \, i=0,1, \ldots, n$. If $q_{i}(a_{i}) \le v_{\min}$, then $h_{i}(a_i, a_{i-1}) \leq 0$ for all $i = 1, \ldots, n$.}
  \end{align*}
   To prove this claim, we first show that $ \frac{v_{i-1} + (v_{i-1} + \tau a_{i-1})}{2} \ge v_{\min}$ for each $i=1, \ldots, n$. Clearly, we
    deduce via $v_{0} \ge v_{\min}$ and $v_{0} + \tau a_{0} \ge v_{\min}$ that $ \frac{v_{i-1} + (v_{i-1} + \tau a_{i-1})}{2} \ge v_{\min}$ when $i=1$. For each $i \ge 2$, it follows from the feasibility of $(x_i, v_i)^n_{i=0}$ that  $v_{i-1} \ge v_{\min}$. This, along with the assumption that $v_{i-1} + \tau a_{i-1} \ge v_{\min}$, yields $ \frac{v_{i-1} + (v_{i-1} + \tau a_{i-1})}{2} \ge v_{\min}$ for each $i \ge 2$. Further, for each $i=1, \ldots, n$,  we obtain, using $p_{i} \leq 0$, that
\begin{eqnarray*}
           h_{i}(a_i, a_{i-1}) & = &  L + r_i ( v_i + \tau a_i ) - \frac{(v_i + \tau a_i - v_{\min})^2}{2 a_{i, \min}} +
 (x_i - x_{i-1}) + \tau ( v_i - v_{i-1}) + \frac{\tau^{2}}{2}(a_{i} - a_{i-1})\\
             & = & \underbrace{\big[ L + r_i v_i - \frac{(v_{i} - v_{\min})^{2}}{2 a_{i, \min}}+ (x_i - x_{i-1}) \big]}_{=p_{i}}\\
             & & \quad  + \tau \Big[ \, r_i a_i + (v_i - v_{i-1}) + \frac{\tau}{2} (a_i - a_{i-1}) - \frac{\tau a_{i}^{2}}{2a_{i, \min}} - \frac{a_{i}}{a_{i,\min}}(v_{i} - v_{\min})\, \Big] \\
             & \le & \tau \Big[ \, v_i + \big(r_i + \frac{\tau}{2} \big) a_i - \frac{a_{i}}{a_{i, \min}} \big(v_{i} - v_{\min} \big) - \frac{\tau a_{i}^{2}}{2 a_{i, \min}} - \frac{v_{i-1} + (v_{i-1} + \tau a_{i-1})}{2} \, \Big] \\
             & = & \tau \Big[ \, q_{i}(a_{i})  - \frac{v_{i-1} + (v_{i-1} + \tau a_{i-1})}{2} \, \Big] \\ 
            %
             &\le & \tau \Big[ q_{i}(a_{i}) - v_{\min}\Big],
         \end{eqnarray*}
 where the last inequality follows from $ \frac{v_{i-1} + (v_{i-1} + \tau a_{i-1})}{2} \ge v_{\min}$ for each $i=1, \ldots, n$.
Since $q_{i}(a_{i}) \le v_{\min}$, we have $ h_{i}(a_i, a_{i-1}) \le 0$ for each $i=1, \ldots, n$, completing the proof of the claim.

 With the help of the above result, we prove that there exists $u=(u_i)^n_{i=1} \in \Wcal((x_i, v_i)^n_{i=0}, u_0)$ for any feasible $(x_i, v_i)^n_{i=0}$ and $u_0$. For this purpose,  consider the following choice of $u_i$'s for a feasible $v_i$: 
 \begin{equation} \label{eqn:feasible_u_i}
       u_i \, : = \, \left\{ \begin{array}{lcc} a_{i, \min} + c_{2,i} v^{2}_{i} + c_{3, i}g, & \mbox{ if } \  v_{i} + \tau a_{i, \min} \ge v_{\min} \\ \frac{v_{\min} - v_{i}}{\tau} + c_{2, i} v^{2}_{i} + c_{3, i} g, & \mbox{ if } \  v_{i} + \tau a_{i, \min} < v_{\min} \end{array} \right., \quad \forall \ i=1, \ldots, n.
 \end{equation}

We show that  the above $u =(u_i)^n_{i=1}\in \Wcal((x_i, v_i)^n_{i=0}, u_0)$ by considering the following two cases for each fixed $i$:
  \begin{itemize}
  \item [(C.1)] $v_{i} + \tau a_{i, \min} \ge v_{\min}$. In this case,  $a_{i}(u_i) = a_{i, \min}<0$.
\begin{itemize}
\item [(a)] Control constraint. In view of $c_{2, i}, c_{3, i} > 0$, we have $u_{i} > a_{i, \min}$. By the assumption that $c_{2, i} v^{2}_{\max} + c_{3, i}g \le a_{i, \max}$, we have $u_{i} = a_{i, \min} + c_{2, i} v^{2}_{i} + c_{3, i}g \le c_{2, i} v^{2}_{\max} + c_{3, i}g \le a_{i, \max}$. Thus $a_{i, \min} < u_{i} \le a_{i, \max}$.
\item [(b)] Speed constraint. By virtue of $a_i(u_i)=a_{i, \min}$, $\tau a_{i, \min} < 0$, and the assumption $v_{i} + \tau a_{i, \min} \ge v_{\min}$, we have $v_{\min} \le v_{i} + \tau a_{i, \min} < v_{i} \le v_{\max}$. Thus $v_{\min} \le  v_i + \tau a_i(u_i) = v_{i} + \tau a_{i, \min} < v_{\max}$.
%
%
\end{itemize}

  \item [(C.2)] $v_{i} + \tau a_{i, \min} < v_{\min}$. In this case,  $a_{i}(u_i) = \frac{v_{\min} - v_{i}}{\tau} \le 0$.
  \begin{itemize}
  \item [(a)] Control constraint. Since $c_{2, i}, c_{3, i} > 0$ and $v_{\min}> v_{i} + \tau a_{i, \min}$, we have $u_{i} > a_{i}(u_i) = \frac{v_{\min} - v_{i}}{\tau}  > \frac{v_{i} + \tau a_{i, \min} - v_{i}}{\tau} = a_{i, \min}$. By the assumption that $c_{2, i} v^{2}_{\max} + c_{3, i}g \le a_{i, \max}$ and using $a_i(u_i)=\frac{v_{\min} - v_{i}}{\tau} \le 0$, we also have $u_{i} \le c_{2, i} v^{2}_{i} + c_{3, i}g \le a_{i, \max} $. Thus $a_{i, \min} < u_{i} \le a_{i, \max}$.
  \item [(b)] Speed constraint. Since $v_{i} + \tau a_{i}(u_i) = v_{\min}$ and $a_i(u_i) \le 0$, we have $v_{\min} =v_{i} + \tau a_{i}(u_i) \le v_{\max}$.
  \end{itemize}
  \end{itemize}
 These results show that the $a_i(u_i)$'s satisfy $v_i+\tau a_i(u_i) \ge v_{\min}, \forall \, i=0,1, \ldots, n$ such that all the assumptions in the claim proved above hold.

 To show that the proposed $u$ given by (\ref{eqn:feasible_u_i}) satisfies the safety distance constraints, we consider cases (C.1) and (C.2) separately. For the former case, note that $q_{i}(a_{i}(u_i)) = q_{i}(a_{i, \min}) = r_i \cdot a_{i, \min} + v_{\min} < v_{\min}$. By the claim proved above,  $h_{i}(a_i, a_{i-1}) \le 0$. For the latter case, in light of $v_{i} = v_{\min} - \tau a_{i}(u_i)$, we have
  \[
    q_{i}(a_{i}(u_i)) = v_{\min} + (r_i - \tau)\cdot a_{i}(u_i) - \frac{v_{i} - v_{\min}}{a_{i, \min}} a_{i}(u_i) + \frac{\tau}{2} a_{i}(u_i) \Big[1 - \frac{a_{i}(u_i)}{a_{i, \min}} \Big].
    \]
   Since $r_i \ge \tau $, $a_{i}(u_i) \le 0$ and $v_{i} - v_{\min} \ge 0 $, we have $(r_i - \tau) \cdot a_{i}(u_i) \le 0$ and $\frac{-(v_{i} - v_{\min})}{a_{i, \min}} a_{i}(u) \le 0$. It has been shown in part (a) of (C.2) that $a_{i}(u) > a_{i, \min}$. This yields  $\frac{\tau}{2} a_{i}(u_i) [1 - \frac{a_{i}(u_i)}{a_{i, \min}}] \le 0$. We thus conclude that $q_{i}(a_{i}(u_i)) \le v_{\min}$, leading to $h_{i}(a_i, a_{i-1}) \le 0$.
  Consequently, $u \in \Wcal((x_i, v_i)^n_{i=0}, u_0)$.
\end{proof}

%
\subsection{Nonempty Interior of the Constraint Sets}

Consider the non-polyhedral constraint set arising from the nonlinear vehicle dynamics subject to the control, speed, and safety distance constraints. We show that under  mild assumptions, this constraint set has nonempty interior. This property is critical for the Slater's constraint qualification in optimization.

\begin{proposition} \label{prop:interior_nonlin_dynamics}
Consider the nonlinear vehicle dynamics (\ref{eqn:d_model_longit}). Suppose the nonnegative constants $c_{2,i}, c_{3,i}$ are such that $c_{2, i}v_{\max}^{2} + c_{3, i} g < a_{i, \max} $ and $r_i \ge \tau$ for each $i=1, \ldots, n$. For any feasible $(x_i, v_i)^n_{i=0}$ and $u_0$, if $v_0>v_{\min}$ and $v_0+\tau u_0>v_{\min}$, then $\Wcal((x_i, v_i)^n_{i=0}, u_0)$
        has nonempty interior.
\end{proposition}

\begin{proof}
Given a feasible $(x_i, v_i)^n_{i=0}$ and $u_0$ such that $a_{0, \min} \le u_0 \le a_{0, \max}$,  $v_{\min}< v_0 \le v_{\max}$,  $v_{\min} < v_0+ \tau a_0 \le v_{\max}$, $v_{\min} \le v_i\le v_{\max}$ and $p_{i} \le 0$ for all $i=1, \ldots, n$, we show  that for each $i=1, \ldots, n$, there exists $\wh u_i$ satisfying $a_{i, \min}< \wh u_i < a_{i, \max}$, $v_{\min} < v_i + \tau a_i(\wh u_i) < v_{\max}$, and $h_{i}(\wh u_{i-i}, \wh u_i )<0$, where $\wh u:=(\wh u_1, \ldots, \wh u_n)^T \in \mathbb R^n$. We prove this result by induction on $i$ as follows.

  Consider $i=1$ first. Let $u_1$ be given by (\ref{eqn:feasible_u_i}). By the given assumptions and the proof of Proposition~\ref{prop:seq_feasibility_nonlinear}, we obtain the following:
  \begin{itemize}
     \item [(C.1)] If $v_1+ \tau a_{1, \min} \ge v_{\min}$, then $a_{i, \min}<u_1\le c_{2, 1} v^2_{\max} + c_{3, 1} g < a_{i, \max}$, $v_{\min} \le v_1 + \tau a_1(u_1) = v_1 + \tau a_{1, \min} < v_{\max}$, and $q_1(a_1(u_1)) = r_1 a_{i, \min} + v_{\min}< v_{\min}$. Since $q_1(\cdot)$ given by (\ref{eqn:q_function}) is continuous, there exists a constant $\varepsilon>0$ such that $\wh u_1 := u_1 + \varepsilon$ satisfies $a_{i, \min}<\wh u_1 < a_{i, \max}$, $v_{\min} < v_1 + \tau a_1(\wh u_1) < v_{\max}$, and $q_1(a_1(\wh u_1))< v_{\min}$. It follows from the Claim given in the proof of Proposition~\ref{prop:seq_feasibility_nonlinear} that $h_{1}(\wh u_1)\le \tau [ q_1(a_1(\wh u_1))  - v_{\min}] <0$.

     \item [(C.2)] If $v_1+ \tau a_{1, \min} < v_{\min}$, then $a_{1}(u_1) = \frac{v_{\min} - v_{1}}{\tau}$,  $a_{i, \min}<u_1\le c_{2, 1} v^2_{\max} + c_{3, 1} g < a_{i, \max}$, $v_{\min} = v_1 + \tau a_1(u_1) < v_{\max}$, and $q_1(a_1(u_1)) \le v_{\min}$.
         As shown in the proof for the Claim in  Proposition~\ref{prop:seq_feasibility_nonlinear},
         \[
          h_{1}(u_1) \le \tau \Big[ \, q_{1}(a_{1} (u_1))  - \frac{v_{0} + (v_{0} + \tau u_0)}{2} \, \Big] < \tau \Big[ \, q_{1}(a_{1} (u_1))  - v_{\min} \, \Big] \le 0,
         \]
         where we use the assumptions that $v_0>v_{\min}$ and $v_0+\tau u_0> v_{\min}$. Hence, $h_1(u_1)<0$.
         By the continuity of $q_1(\cdot)$ and $h_1(\cdot)$, we see that there exists a small constant $\varepsilon>0$ such that $\wh u_1 := u_1 + \varepsilon$ satisfies $a_{i, \min}<\wh u_1 < a_{i, \max}$, $v_{\min} < v_1 + \tau a_1(\wh u_1) < v_{\max}$, and $h_{1}( \wh u_1)<0$.
   \end{itemize}

 Now assume that for each $i$ with $1\le i \le n-1$,  there exists $\wh u_j$ satisfying $a_{i, \min}< \wh u_j < a_{i, \max}$, $v_{\min} < v_j + \tau a_j(\wh u_j) < v_{\max}$, and $h_{j}(\wh u_j, \wh u_{j-1})<0$ for all $j=1, \ldots, i$. Consider $i+1$ as follows. As before, let $u_{i+1}$ be given in (\ref{eqn:feasible_u_i}), and consider two cases:
 \begin{itemize}
   \item [(C.1')] $v_{i+1}+ \tau a_{i, \min} \ge v_{\min}$. By the proof of Proposition~\ref{prop:seq_feasibility_nonlinear} and a similar argument for (C.1) given above, we deduce that there exists a constant $\varepsilon>0$ such that $\wh u_{i+1}:=u_{i+1} + \varepsilon$ satisfies the desired results.

   \item [(C.2')] $v_{i+1}+ \tau a_{i, \min} < v_{\min}$. Similarly, $a_{i, \min}<u_{i+1} < a_{i, \max}$, $v_{\min} = v_{i+1} + \tau a_{i+1}(u_{i+1}) < v_{\max}$, and $q_{i+1}(a_{i+1}(u_{i+1})) \le v_{\min}$, where $a_{i+1}(u_{i+1}) = \frac{v_{\min} - v_{i+1}}{\tau}$. Moreover,
       \[
          h_{i+1}(\wh u_{i}, u_{i+1}) \le \tau \Big[ \, q_{i+1}(a_{i+1} (u_{i+1}))  - \frac{v_{i} + (v_{i} + \tau a_i(\wh u_i))}{2} \, \Big] < \tau \Big[ \, q_{i+1}(a_{i+1} (u_{i+1}))  - v_{\min} \, \Big] \le 0,
         \]
        where the strict inequality follows from the assumption that $v_i\ge v_{\min}$ and the induction hypothesis $ v_{i} + \tau a_i(\wh u_i)) > v_{\min}$. Hence, $h_{i+1}(\wh u_{i}, u_{i+1}) <0$.
        Then by the similar argument for (C.2) above, we see that there exists a constant $\varepsilon>0$ such that $\wh u_{i+1}:=u_{i+1} + \varepsilon$ yields the desired results.
 \end{itemize}
 Consequently,  by the induction principle, there exists a vector $\wh u$ in the interior of $\Wcal((x_i, v_i)^n_{i=0}, u_0)$.
\end{proof}

In light of the above result, we make the following assumptions throughout the rest of the paper
unless otherwise stated:
\begin{itemize}
  \item [$\bf A.2$] For each $i=1, \ldots, n$, the nonnegative constants $c_{2,i}, c_{3,i}$ satisfy $c_{2, i}v_{\max}^{2} + c_{3, i} g < a_{i, \max} $ and the reaction time $r_i$ satisfies $r_i \ge \tau$. Further, $(v_0(k), u_0(k))$ is feasible with $v_0(k)>v_{\min}$ for all $k\in \mathbb Z_+$.
\end{itemize}
It will be shown in Corollary~\ref{coro:nonempty_interior_general_p} that under this assumption,  the constraint set of a general $p$-horizon model predictive control model has nonempty interior.

\mycut{
Using the above proposition, we show that under $\bf A.2$, the constraint set of the $p$-horizon MPC model has nonempty interior for a general MPC horizon $p \in \mathbb N$.

\begin{corollary} \label{coro:nonempty_interior_general_p}
   Suppose the assumption $\bf A.2$ holds. Then for any $p \in \mathbb N$,
  the constraint set of the $p$-horizon MPC optimization problem (\ref{eq:MPC_opt_model}) has nonempty interior  at each $k$.
\end{corollary}

\begin{proof}
 Fix an arbitrary $k \in \mathbb Z_+$. Since $v_0(k)>v_{\min}$, it follows from Proposition~\ref{prop:interior_nonlin_dynamics} that there exists a vector denoted by $\wh u(k)$ in the interior of the set $\Wcal((x_i(k), v_i(k))^n_{i=0}, u_0(k))$. Let $x_i(k+1)$ and $v_i(k)$ be generated by $\wh u(k)$ (and $(x_i(k), v_i(k))^n_{i=0}, u_0(k)$). Since $v_0(k+1)=v_0(k)>v_{\min}$, we deduce via Proposition~\ref{prop:interior_nonlin_dynamics}
  again that there exists a vector denoted by $\wh u(k+1)$ in the interior of the constraint set $\Wcal((x_i(k+1), v_i(k+1))^n_{i=0}, u_0(k+1))$. Continuing this process in $p$-steps, we derive the existence of an interior point in the constraint set of the $p$-horizon MPC model (\ref{eq:MPC_opt_model}).
\end{proof}
}


%
%
%

%
\section{Model Predictive Control for CAV Platooning} \label{sect:MPC_formulation}

We consider the model predictive control (MPC) approach for CAV platooning, and we use the same formulation given in \cite[Section 3]{ShenEswarDu_Arx20}. To be self-contained, this MPC formulation is presented as follows.
Let $\Delta$ be the desired constant spacing between two adjacent vehicles, and $(x_0, v_0, u_0)$ be the position, speed, and control input of the leading vehicle, respectively. Define the following vectors:
(i) the relative spacing error $z(k):=\big( x_0-x_1-\Delta, \ldots, x_{n-1}-x_n-\Delta\big)(k) \in \mathbb R^n$; (ii) the relative speed between adjacent vehicles $z'(k):=\big( v_0-v_1, \ldots, v_{n-1}-v_n\big)(k) \in \mathbb R^n$; and (iii) the control input $u(k):=\big(u_1, \ldots, u_n\big)(k)\in \mathbb R^n$.
Further, let $w_i(k):= u_{i-1}(k) - u_i(k)$ for each $i=1, \ldots, n$, and $w(k):= \big( w_1, \ldots, w_n \big)(k) \in \mathbb R^n$, representing the difference of control input between adjacent vehicles. Hence, for any $k \in \mathbb Z_+$, $u(k) = - S_n w(k) + u_0(k) \cdot \mathbf 1$, where $\mathbf 1 :=(1, \ldots, 1)^T$ is the vector of ones, and
\begin{equation} \label{eqn:matrix_S_n}
    S_n \, := \, \begin{bmatrix}
1&0&0&  \hdots &0\\
1 &1& 0&  \hdots &0\\
\vdots &\vdots & \ddots  & \ddots & \vdots\\
1 & 1 &\hdots& 1  & 0\\
1 & 1 &  \hdots & 1 &1 \\
\end{bmatrix} \in \mathbb R^{n\times n}, \qquad \quad S^{-1}_n \, = \, \begin{bmatrix} 1& & & & \\
 -1 & 1& & &   \\
 &  \ddots & \ddots &  &  \\
   & & -1 & 1 &  \\
   &  &  & -1 & 1\\
\end{bmatrix} \in \mathbb R^{n\times n}.
\end{equation}


Given a prediction horizon $p\in \mathbb N$, the $p$-horizon MPC control is determined by solving the following constrained optimization problem at each $k\in \mathbb Z_+$, involving all vehicles' control inputs for given feasible state $(x_i(k), v_i(k))^n_{i=1}$ and $(v_0(k), u_0(k))$ at time $k$ subject to the nonlinear vehicle dynamics (\ref{eqn:d_model_longit}):
\begin{align}
 & \ \mbox{minimize}  \ \ J(u(k), \ldots, u(k+p-1) )  :=  \label{eq:MPC} \\
 & \ \frac{1}{2} \sum^p_{s=1} \Big(  \underbrace{\tau^2 u^T(k+s-1) S^{-T}_n Q_{w, s} S^{-1}_n  u(k+s-1)}_{\mbox{ride comfort}} +   \underbrace{z^T(k+s) Q_{z, s} z(k+s) + (z'(k+s))^T Q_{z', s} z'(k+s)}_{\mbox{traffic stability and smoothness}} \Big)  \notag
\end{align}
subject to: for each $i=1, \ldots, n$ and each $s=1, \ldots, p$,
\begin{eqnarray}
  a_{i, \min}  & \le  u_i(k+s-1) \ \le  a_{i, \max}, & \qquad v_{\min}  \, \le \, v_i(k+s) \ \le \ v_{\max}, \label{eqn:MPC:u_constriant} \\
  &   x_{i-1}(k+s)-x_i(k+s)  & \ge   L_i + r_i \cdot v_i(k+s)  - \frac{(v_i(k+s)-v_{\min})^2}{2a_{i, \min}},  \label{eqn:MPC:safety_constraint_TypeIII}
 %
%
\end{eqnarray}
where $Q_{z, s}$, $Q_{z', s}$ and $Q_{w, s}$ are $n\times n$ symmetric positive semidefinite weight  matrices to be discussed soon. Note that when $p>1$, $(x_0(k+s+1), v_0(k+s+1), u_0(k+s))$  are unknown at time $k$ for $s=1, \ldots, p-1$. We  assume that $u_0(k+s)=u_0(k)$ for all $s=1, \ldots, p-1$ and use these $u_0(k+s)$'s and the vehicle dynamics model (\ref{eqn:d_model_longit}) to predict $(x_0(k+s+1), v_0(k+s+1))$ for $s=1, \ldots, p-1$.  Note that $u_0(k)$ represents the actual acceleration of the leading vehicle at time $k$.

%
%
%

The physical interpretation of the three terms of the objective function $J$ can be found in \cite[Remark 3.1]{ShenEswarDu_Arx20}. Further, The presence of the matrix $S^{-1}_n$ in the first term is due to the coupled vehicle dynamics through the CAV platoon; see \cite[Remark 3.1]{ShenEswarDu_Arx20}.
%
%
%
To develop fully distributed schemes for a general class of vehicle network topologies and to facilitate control design and analysis, we make the following assumption on the weight matrices $Q_{z, s}$, $Q_{z', s}$, and $Q_{w, s}$ through the rest of the paper:
\begin{itemize}
  \item [$\bf A.3$] For each $s=1, \ldots, p$, $Q_{z, s}$ and $Q_{z', s}$ are diagonal and positive semidefinite (PSD), and $Q_{w, s}$ is diagonal and positive definite (PD).
\end{itemize}

More discussions on the class of weight matrices specified in $\bf A.3$ can be found in \cite[Section 3]{ShenEswarDu_Arx20}. We show below that under the assumption $\bf A.2$, the constraint set of the $p$-horizon MPC model has nonempty interior at each $k$ for an arbitrary MPC horizon $p \in \mathbb N$.

\begin{corollary} \label{coro:nonempty_interior_general_p}
   Suppose the assumption $\bf A.2$ holds. Then for any $p \in \mathbb N$,
  the constraint set of the $p$-horizon MPC optimization problem (\ref{eq:MPC}) has nonempty interior  at each $k$.
\end{corollary}

\begin{proof}
The proof is similar to \cite[Corollary 3.1]{ShenEswarDu_Arx20}.
 Fix an arbitrary $k \in \mathbb Z_+$. Since $v_0(k)>v_{\min}$, by Proposition~\ref{prop:interior_nonlin_dynamics}, there exists a vector denoted by $\wh u(k)$ in the interior of the set $\Wcal((x_i(k), v_i(k))^n_{i=0}, u_0(k))$. Let $x_i(k+1)$ and $v_i(k)$ be generated by $\wh u(k)$ (and $(x_i(k), v_i(k))^n_{i=0}, u_0(k)$). Since $v_0(k+1)>v_{\min}$, we deduce via Proposition~\ref{prop:interior_nonlin_dynamics}
  again that there exists a vector denoted by $\wh u(k+1)$ in the interior of the constraint set $\Wcal((x_i(k+1), v_i(k+1))^n_{i=0}, u_0(k+1))$. Continuing this process in $p$-steps, we derive the existence of an interior point in the constraint set of the $p$-horizon MPC model (\ref{eq:MPC}).
\end{proof}

%
%

%
\subsection{Constrained Optimization Model under the Nonlinear Vehicle Dynamics} \label{subsect:optim_nonlinear_dynamics}

In this subsection, we discuss the constrained optimization model (\ref{eq:MPC})  arising from the MPC at each time $k$ under the nonlinear vehicle dynamics (\ref{eqn:d_model_longit}) with the positive parameters $c_{2,i}$ and $c_{3, i}$.
For notational simplicity, define the parameter vectors  $\varphibf_d:=(c_{2, 1}, \ldots, c_{2, n}) \in \mathbb R^n_+$ and $\varphibf_f:=(c_{3, 1}, \ldots, c_{3, n}) \in \mathbb R^{n}_+$, where the subscripts $d$ and $f$ denote the drag and friction respectively. Further, $\varphibf:=(\varphibf_d, \varphibf_f)\in \mathbb R^{2n}_+$. For notational convenience, we set $c_{2, 0}=c_{3, 0}=0$; this is because $u_0(k)$ is the actual acceleration of the leading vehicle instead of its control.

Consider the constrained MPC optimization model (\ref{eq:MPC}) at an arbitrary but fixed time $k \in \mathbb Z_+$.
%
%
Let $\ubld(k):=(\ubld_1(k), \ldots, \ubld_n(k))\in \mathbb R^{np}$ with $\ubld_i(k):=( u_i(k), \ldots, u_i(k+p-1))\in \mathbb R^{p}$.
Recall that for each $i=1, \ldots, n$ and $j=0, \ldots, p-1$,
\[
    a_i\big(k+j, u_i(k), \ldots, u_i(k+j) \big) \, = \, u_i(k+j) - c_{2, i} v^2_i(k+j) - c_{3, i} g,
\]
where we note that $v_i(k+j)$ depends on $u_i(k), \ldots, u_i(k+j-1)$ for $j\ge 1$. Specifically, for $p>1$,
\begin{align*}
  a_i(k, u_i(k)) & = u_i(k) - c_{2, i} v^2_i(k) - c_{3, i} g, \\
  a_i(k+1, u_i(k), u_i(k+1)) 
    &  = u_i(k+1) - c_{2, i} \big[ v_i(k) + \tau a_i(k, u_i(k)) \big]^2 - c_{3, i} g, \\
   \vdots \quad &  \qquad \vdots \qquad \qquad \qquad \vdots \\
  a_i\big(k+p-1, u_i(k), \ldots, u_i(k+p-1) \big) & = u_i(k+p-1) -  c_{2, i} \Big[ v_i(k) + \tau \sum^{p-2}_{s=0} a_i(k+s, u_i(k), \ldots, u_i(k+s) ) \Big]^2 \\
   & \qquad - c_{3, i} g
\end{align*}
By slightly abusing the notation, we may denote each $a_i\big(k+j, u_i(k), \ldots, u_i(k+j) \big)$ by $a_i(k+j, \mathbf u_i(k))$.

Define for each $i=1, \ldots, n$ and $j=0, 1, \ldots, p-1$,
\[
 b_i(k+j, \ubld_{i-1}(k), \ubld_i(k) ) \,  := \, a_{i-1}(k+j, \ubld_{i-1}(k)) - a_{i}(k+j,\ubld_i(k) ),
\]
where $a_0(k+j, \mathbf u_0(k)) := u_0(k)$ for all $j=0, 1, \ldots, p-1$ due to $\mathbf u_0(k) := u_0(k) \cdot \mathbf 1$.
It follows from the nonlinear vehicle dynamics (\ref{eqn:d_model_longit}) that for each $i=1, \ldots, n$ and $j=1, \ldots, p$,
\begin{eqnarray}
 z_i(k+j) &  = & z_i(k) + j \tau z'_i(k) + \tau^2 \sum^{j-1}_{s=0} \frac{2(j-s)-1}{2} b_i(k+s, \ubld_{i-1}(k), \ubld_i(k) ), \label{eqn:z_k_j} \\
 z'_i(k+j) & = & z'_i(k) + \tau \sum^{j-1}_{s=0} b_i(k+s, \ubld_{i-1}(k), \ubld_i(k) ). \label{eqn:z'_k_j}
\end{eqnarray}

For a fixed $k \in \mathbb Z_+$, define for each $i=1, \ldots, n$,
\[
 \abld_i(\ubld_i(k)):=\Big(a_i\big(k, u_i(k)\big), \, a_i\big(k+1, u_i(k), u_i(k+1) \big), \, \ldots,  \, a_i\big(k+p-1, u_i(k), \ldots, u_i(k+p-1) \big) \Big).
\]
 In what follow, we often omit $k$ in $\ubld_i(k)$  when $k$ is fixed. Further, define the function $\abld:\mathbb R^{np} \rightarrow \mathbb R^{np}$ as $\abld(\ubld):=\big( \abld_1(\ubld_1), \ldots, \abld_n(\ubld_n) \big)$. Note that if $\varphibf=(\varphibf_d, \varphibf_f)=(c_{2,i}, c_{3,i})^n_{i=1}=0$, then $\abld(\ubld) = \ubld$ for all $\ubld \in \mathbb R^{np}$.
We introduce more notation.  Define the following matrices:
\[
   \ol Q_w := \mbox{diag}\Big( \, Q_{w, 1}, \, \ldots \, Q_{w, p} \, \Big) \in \mathbb R^{np \times np}, \qquad
  \mathbf S^{-1} \, := \, \mbox{diag}\Big( \, \underbrace{S^{-1}_n, \, \ldots, \, S^{-1}_n }_{p-\mbox{copies}} \, \Big) \in \mathbb R^{np \times np}.
\]
Furthermore, let $E \in \mathbb R^{np \times np}$ be the permutation matrix whose $(i, j)$-entry is given by
\begin{equation} \label{eqn:E_matrix02}
  E_{i, j} = \left\{ \begin{array}{ll} 1 &  \mbox{if} \ \  i=n\cdot k+s, \ j=p\cdot (s-1)+k+1, \quad \mbox{for} \ \ k=0, \ldots, p-1, \ s=1,\ldots, n; \\ 0, & \mbox{otherwise}. \end{array} \right.
\end{equation}
Clearly,  $E=I_n$ when $p=1$, and
\[
    \begin{bmatrix} u(k) \\ u(k+1) \\ \vdots \\ u(k+p-1) \end{bmatrix} \, = \, E \begin{bmatrix} \mathbf u_1 \\ \mathbf u_2 \\ \vdots \\ \mathbf u_n \end{bmatrix} \, = \, E \ubld.
\]
Using these matrices, it is easy to verify that the following term in the objective function $J$ in (\ref{eq:MPC}) satisfies
\[
\left( \mathbf S^{-1}\begin{bmatrix} u(k)\\ \vdots \\ u(k+p-1) \end{bmatrix} \right)^T \, \ol Q_{w}  \, \left( \mathbf S^{-1} \begin{bmatrix} u(k)\\ \vdots \\ u(k+p-1) \end{bmatrix} \right) \, = \, \ubld^T \underbrace{E^T \mathbf S^{-T} \ol Q_w \mathbf S^{-1} E}_{:=\Psi} \, \ubld.
\]
where $\Psi \in \mathbb R^{np \times np}$ is symmetric PD when $\bf A.3$ holds. Therefore, the  objective function $J$ in the MPC model (\ref{eq:MPC}) becomes
\begin{eqnarray*}
  J(\ubld) & = &  J(u(k), \ldots, u(k+p-1) )   \\
  & = &
   \frac{1}{2} \Big[
    \sum^p_{s=1} z^T(k+s) Q_{z, s} z(k+s) + (z'(k+s))^T Q_{z', s} z'(k+s) \Big] + \frac{\tau^2}{2} \ubld^T \Psi \, \ubld \\
  & = & \frac{1}{2} \Big[
    \sum^p_{s=1} z^T(k+s) Q_{z, s} z(k+s) + (z'(k+s))^T Q_{z', s} z'(k+s) \Big] + \frac{\tau^2}{2}\abld^T(\ubld) \Psi \,\abld(\ubld) \\
    & & \qquad +     \frac{\tau^2}{2} \Big(\ubld^T \Psi \ubld - \abld^T(\ubld) \Psi \abld(\ubld)\Big).
\end{eqnarray*}
 In light of the expressions for $z(k+j)$ and $z'(k+j)$ given by (\ref{eqn:z_k_j})-(\ref{eqn:z'_k_j}), it follows from the similar argument in \cite[Section 3.1]{ShenEswarDu_Arx20} that the objective function
\[
   J(\ubld) \, = \, \frac{1}{2} \abld^T(\ubld) W \abld(\ubld) + c^T \abld(\ubld) + \gamma + \frac{\tau^2}{2} \Big(\ubld^T \Psi \ubld - \abld^T(\ubld) \Psi \abld(\ubld)\Big),
\]
where $W\in \mathbb R^{np\times np}$, $c\in \mathbb R^{np}$, and $\gamma \in \mathbb R$. In fact, $W=E^T \mathbf S^{-T}  \Theta \mathbf S^{-1} E$ for a symmetric PSD matrix $\Theta$ whose blocks are diagonal; see \cite[Section 3.1]{ShenEswarDu_Arx20} for the closed-form expression of $W$. In particular, under the assumption $\bf A.3$, $W$ is a positive definite (PD) matrix that only depends on $Q_{z, s}, Q_{z', s}$ and $Q_{w, s}$, $s=1, \ldots, p$ \cite[Lemma 3.1]{ShenEswarDu_Arx20}.
In addition, the linear term in $J(\ubld)$ can be written as
$
    c^T \abld(\mathbf u) \, = \, \sum^n_{i=1} c^T_{\Ical_i} \abld_i(\mathbf u_i),
$
where $c_{\Ical_i}$ is the subvector of $c$ corresponding to $\abld_i(\ubld_i)$. By \cite[Lemma 3.2]{ShenEswarDu_Arx20}, the subvector $c_{\Ical_i}$ depends only on $z_i(k), z'_i(k), z_{i+1}(k), z'_{i+1}(k)$'s for $i=1, \ldots, n-1$,
$c_{\Ical_n}$ depends only on $z_n(k), z'_n(k)$, and  only $c_{\Ical_1}$ depends on $u_0(k)$. These properties are important for developing fully distributed schemes later on.

To characterize the constraints, let the matrix $S_p \in \mathbb R^{p \times p}$ be defined in the same way as in (\ref{eqn:matrix_S_n}) with $n$ replaced by $p$, and $(S_p \mathbf u_i)_0 :=0$.
Recall that for each $i=1, \ldots, n$ and $j=1, \ldots, p$,
\[
v_i(k+j) = v_i(k) + \tau \sum^{j-1}_{s=0} a_i(k+s, \ubld_i(k))
 = v_i(k) + \tau \big(S_p \, \abld_i(\ubld_i) \big)_j.
\]
Further, $ x_{i-1}(k+j) - x_i(k+j) = z_i(k+j)+\Delta$  depends only on $\ubld_i(k)$ and $\ubld_{i-1}(k)$  as shown in (\ref{eqn:z_k_j}). Hence, we see that for each $i=1, \ldots, n$ and  each $j=1, \ldots, p$, the safety distance constraint is given by:
\[
  \big( H_i(\ubld_{i-1} (k), \ubld_i(k)) \big)_j := L_i  + r_i \cdot v_i(k+j) - \frac{ (v_i(k+j) - v_{\min} )^2}{ 2 a_{i, \min} } - [ x_{i-1}(k+j) - x_i(k+j)] \, \le \, 0.
\]
Note that $H_1(\cdot)$ depends only on $\ubld_1(k)$  although it is written in the above form for notational convenience.
%
%
%
Combining the above results, the MPC model (\ref{eq:MPC}) is formulated as the following optimization problem:
\begin{equation}
\begin{array}{ll}
\mbox{minimize} & J(\ubld):=  \frac{1}{2} \abld^T(\ubld) \big( W - \tau^2 \Psi \big) \abld(\ubld) + c^T \abld(\ubld) + \gamma+ \frac{\tau^2}{2} \ubld^T \Psi \ubld, \label{eq:MPC_opt_model} \\  [0.06in]
%
\mbox{subject to} &  \ubld_i \in \mathcal X_i, \quad v_{\min} \le v_i(k) + \tau \big( S_p \, \abld_i(\ubld_i) \big)_s \le v_{\max},  \\ [0.06in]
 & \quad (H_i( \ubld_{i-1}, \ubld_i ))_s \le 0, \quad \forall \, i=1,\ldots, n, \quad \forall \, s=1, \ldots, p,
\end{array}
\end{equation}
where $\mathcal X_i:=\{ \mathbf u_i \in \mathbb R^p \, | \, a_{i, \min} \mathbf 1 \le \mathbf u_i \le a_{i, \max} \mathbf 1 \}$ for each $i=1, \ldots, n$.
It can be shown via the expressions of $W$ and $\Psi$ given in \cite[Section 3.1]{ShenEswarDu_Arx20} that $W-\tau^2 \Psi$ is PSD.
 When $p=1$, (\ref{eq:MPC_opt_model}) is clearly a convex optimization problem.
When $p>1$, since all but the first entry functions of $\abld_i(\cdot)$ are nonconvex in $\ubld_i$ for each $i$, it is easy to verify that the objective function $J$ is nonconvex, and the velocity and safety distance constraints are nonconvex.
 Hence, when $p>1$, (\ref{eq:MPC_opt_model}) yields a nonconvex optimization problem.  
Since $J$ is continuous,  each $\mathcal X_i$ is compact, and the other constraints are defined by continuous functions, the optimization problem in (\ref{eq:MPC_opt_model}) has a (possibly non-unique) solution.
Moreover, the objective function $J$ is densely coupled, and the safety distance constraint function $\big( H_i(\ubld_{i-1}, \ubld_i) \big)_j$ not only depends on $\ubld_i$ but also on $\ubld_{i-1}$ of the $(i-1)$-th vehicle, and thus is locally coupled with its neighboring vehicles. This coupling structure, together with the nonconvexity of the optimization problem (\ref{eq:MPC_opt_model}), leads to many challenges in developing fully distributed schemes.

%

%
\section{Fully Distributed Algorithms for Coupled Nonconvex MPC Optimization Problem} \label{sect:dist_scheme}

In this section, we develop fully distributed algorithms for solving the underlying coupled, nonconvex optimization problem (\ref{eq:MPC_opt_model}) at each time $k \in \mathbb Z_+$. To achieve this goal, several new techniques are exploited: the formulation of locally coupled (albeit nonconvex) optimization, sequential convex programming, and operator splitting methods.

%
\subsection{Formulation of MPC Optimization Problem as Locally Coupled Optimization}

Since the safety distance constraint of each vehicle $i$ is coupled with its neighboring vehicle $(i-1)$ whereas the acceleration and velocity constraints are decoupled, the constraints of the MPC optimization problem (\ref{eq:MPC_opt_model}) are locally coupled \cite{HuXiaoLiu_CDC18}.
Motivated by distributed computation for locally coupled {\em convex} optimization \cite{HuXiaoLiu_CDC18, ShenEswarDu_Arx20}, we show below that (\ref{eq:MPC_opt_model}) can be formulated as a locally coupled {\em nonconvex} optimization problem.

%
%
%

Consider a multi-agent network of $n$ agents whose communication is characterized by a connected and undirect graph $\mathcal G(\mathcal V, \mathcal E)$, where $\mathcal V=\{1, \ldots, n \}$ is the set of agents, and $\mathcal E$ is the set of edges. Let $\mathcal N_i$ denote the set of neighbors of agent $i$, i.e., $\mathcal N_i =\{ j \, | \, (i, j) \in \mathcal E\}$. Let $\{ \Ical_1, \ldots, \Ical_n\}$ be a disjoint union of the index set $\{1, \ldots, N \}$. Hence,  $(x_{\Ical_i} )^n_{i=1}$ forms a partition of $x \in \mathbb R^N$, and each $x_{\Ical_i}$ is called a local variable of agent $i$.
For each $i$, define $\wh x_i :=\big(x_{\Ical_i}, (x_{\Ical_j})_{j \in \mathcal N_i} \big) \in \mathbb R^{n_i}$. Therefore, $\wh x_i$ contains the local variable $x_{\Ical_i}$ and the variables from its neighboring agents (or the so-called locally coupled variables). The locally coupled (possibly nonconvex) optimization problem is given by:
\begin{equation} \label{eqn:local_coupled_opt_formulation}
 (P): \qquad \min_{x \in \mathbb R^N} J(x) :=  \sum^n_{i=1} J_i(\wh x_i),, \qquad  \mbox{ subject to } \qquad \wh x_i \in \mathcal C_i, \quad \forall \ i=1, \ldots, n,
\end{equation}
where $J_i:\mathbb R^{n_i} \rightarrow \mathbb R $,  and $\mathcal C_i \subseteq \mathbb R^{n_i}$ is a closed set  for each $i$. When $(P)$ is a convex optimization problem,  the paper \cite{HuXiaoLiu_CDC18} formulates $(P)$ as a separable consensus convex optimization problem by
introducing copies of locally coupled variables for each agent and imposing certain consensus constraints on these copies; see \cite{HuXiaoLiu_CDC18} for details and \cite{ShenEswarDu_Arx20} for its application to CAV platooning control under linear vehicle dynamics.

%
%
%
%

%

The framework of a locally coupled optimization problem requires that both its objective function and constraints are expressed in a locally coupled manner satisfying the communication network topology constraint. However, the objective function in the underlying MPC optimization problem  (\ref{eq:MPC_opt_model}) is densely coupled. As indicated in \cite[Section 4]{ShenEswarDu_Arx20} for convex optimization, this difficulty can be overcome by using certain matrix decomposition techniques. Specifically, it is shown in \cite[Lemma 4.1]{ShenEswarDu_Arx20} that under the assumption $\bf A.3$, the PSD or PD matrix $W\in \mathbb R^{np \times np}$
 in (\ref{eq:MPC_opt_model}) can be decomposed as $W= \sum^n_{s=1} \wt W^s$, where all $\wt W^s \in \mathbb R^{np \times np}$ are PSD and satisfy  the following conditions:
   \[
      \wt W^1 \, = \, \begin{bmatrix} (\wt W^1)_{1,1} & (\wt W^1)_{1,2} &  & &  \\
     (\wt W^1)_{2, 1} & (\wt W^1)_{2, 2} &  & &  \\
     & &   0 & \cdots & 0 &  \\
      & &  \vdots  &  \cdots & \vdots & \\
     &  &  0 & \cdots & 0
    \end{bmatrix}, \quad
    \wt W^n \, = \, \begin{bmatrix}
       0 & \cdots & 0 &  & & \\
      \vdots  &  \cdots & \vdots & & & \\
       0 & \cdots & 0 & & \\
    &  & & (\wt W^n)_{n-1,n-1} & (\wt W^n)_{n-1,n}   \\
    &  & & (\wt W^n)_{n, n-1} & (\wt W^n)_{n, n}  \\
    \end{bmatrix},
   \]
%
%
 and  for each $s=2, \ldots, n-1$,
     \[
        \wt W^s \, = \, \begin{bmatrix}
         {\mathbf 0}_{(i-2)p\times (i-2)p}  & & & & \\
%
      &     (\wt W^s)_{s-1,s-1} & (\wt W^s)_{s-1,s} & 0 & &  \\
      &   (\wt W^s)_{s, s-1} & (\wt W^s)_{s, s} & (\wt W^s)_{s, s+1} & &  \\
      &   0 & (\wt W^s)_{s+1, s} & (\wt W^s)_{s+1, s+1} &  & &  \\
     & & &  &  0 & \cdots & 0 &  \\
     &  & & & \vdots  &  \cdots & \vdots & \\
     & &  &  & 0 & \cdots & 0
    \end{bmatrix}.
     \]
 To simplify the notation, we let $\wh W^s$ denote the possibly nonzero block in each $\wt W^s$, namely,
\[
  \wh W^1 := \begin{bmatrix} (\wt W^1)_{1,1} & (\wt W^1)_{1,2} \\ (\wt W^1)_{2, 1} & (\wt W^1)_{2, 2} \end{bmatrix} \in \mathbb R^{2p \times 2p}, \qquad \wh W^n :=  \begin{bmatrix} (\wt W^n)_{n-1,n-1} & (\wt W^n)_{n-1,n} \\ (\wt W^n)_{n, n-1} & (\wt W^n)_{n, n} \end{bmatrix} \in \mathbb R^{2p \times 2p},
\]
and for each $s=2, \ldots, n-1$,
\[
  \wh W^s := \begin{bmatrix} (\wt W^s)_{s-1,s-1} & (\wt W^s)_{s-1,s} & 0 \\  (\wt W^s)_{s, s-1} & (\wt W^s)_{s, s} & (\wt W^s)_{s, s+1} \\  0 & (\wt W^s)_{s+1, s} & (\wt W^s)_{s+1, s+1} \end{bmatrix} \in \mathbb R^{3 p\times 3p}.
\]
When $W$ is PD, it is shown in \cite[Lemma 4.1]{ShenEswarDu_Arx20} that there exist $\wt W^s$'s such that each $\wh W^s$ in the above decomposition is PD.

Since $\ol Q_w$ is diagonal and PD, it follows from the similar argument in \cite[Lemma 4.1]{ShenEswarDu_Arx20} that the PD matrix $\Psi \in \mathbb R^{np \times np}$ can be decomposed in the similarly way. Specifically, there exist matrices $\wt \Psi^s$ such that $\Psi = \sum^n_{s=1} \wt \Psi^s$, where $\wt \Psi^s$'s satisfy the abovementioned conditions with $\wt W^s$ (resp. $\wh W^s$) replaced by $\wt \Psi^s$ (resp. $\wh \Psi^s$). By setting $\gamma \equiv 0$ in (\ref{eq:MPC_opt_model}) without losing generality, the objective function $J(\ubld)$ in (\ref{eq:MPC_opt_model}) can be decomposed as
\[
  J(\mathbf u) \, = \,  J_1(\ubld_1, \ubld_2) + \sum^{n-1}_{i=2} J_i(\ubld_{i-1}, \ubld_i, \ubld_{i+1}) + J_n(\ubld_{n-1}, \ubld_n),
\]
where the functions $J_i$'s on the right hand side are given by
\begin{eqnarray}
   J_1(\ubld_1, \ubld_2) & := & \frac{1}{2} \begin{bmatrix} \abld^T_1(\ubld_1) & \abld^T_2(\ubld_2) \end{bmatrix} \Big(\wh W^1 - \tau^2 \wh \Psi^1 \Big)  \begin{bmatrix} \abld_1(\ubld_1) \\ \abld_2(\ubld_2) \end{bmatrix} + c^T_{\Ical_1} \abld_1(\ubld_1) \notag \\
   & &  \quad + \ \frac{\tau^2}{2}
   \begin{bmatrix} \ubld^T_{1} & \ubld^T_2 \end{bmatrix} \wh \Psi^1 \begin{bmatrix} \ubld_{1} \\ \ubld_2 \end{bmatrix}, \notag \\
   J_i (\ubld_{i-1}, \ubld_{i}, \ubld_{i+1}) & := & \frac{1}{2} \begin{bmatrix} \abld^T_{i-1}(\ubld_{i-1}) & \abld^T_i(\ubld_i) & \abld^T_{i+1}(\ubld_{i+1}) \end{bmatrix}
    \Big( \wh W^i -\tau^2 \wh \Psi^i \Big) \begin{bmatrix} \abld_{i-1}(\ubld_{i-1}) \\ \abld_i(\ubld_i) \\ \abld_{i+1}(\ubld_{i+1}) \end{bmatrix} + c^T_{\Ical_i} \abld_i(\ubld_i) \notag \\
   & & \quad \
  + \frac{\tau^2}{2} \begin{bmatrix} \ubld^T_{i-1} & \ubld^T_i & \ubld^T_{i+1} \end{bmatrix} \wh \Psi^i \begin{bmatrix} \ubld_{i-1} \\ \ubld_i \\ \ubld_{i+1} \end{bmatrix}, \qquad \forall \ i=2, \ldots, n-1, \label{eqn:J_i_decomposition} \\
   J_n(\ubld_{n-1}, \ubld_n) & := &  \frac{1}{2} \begin{bmatrix} \abld^T_{n-1}(\ubld_{n-1}) & \abld^T_n(\ubld_n) \end{bmatrix} \Big(\wh W^n - \tau^2 \wh \Psi^n \Big)  \begin{bmatrix} \abld_{n-1}(\ubld_{n-1}) \\ \abld_n(\ubld_n) \end{bmatrix} + c^T_{\Ical_n} \abld_n(\ubld_n) \notag \\
   & & \quad + \
   \frac{\tau^2}{2} \begin{bmatrix} \ubld^T_{n-1} & \ubld^T_n \end{bmatrix} \wh \Psi^n \begin{bmatrix} \ubld_{n-1} \\ \ubld_n \end{bmatrix}. \notag
\end{eqnarray}
In view of the assumption $\bf A.1$, the above decomposition of $J$ satisfies the communication network topology constraint.  Note that $\wh W^i - \tau^2 \wh\Psi^i$ may not be PSD or PD although $W-\tau^2 \Psi$ is PSD.


\begin{remark} \rm
 Another decomposition of $J$ is as follows. Note that $V:=W-\tau^2 \Psi$ is PSD, and it can be written as $V= E^T \mathbf S^{-T}  \Phi \mathbf S^{-1} E$ for a symmetric PSD matrix $\Phi$ whose blocks are all diagonal. Therefore, the similar decomposition can be made to $V$ whose corresponding $\wh V^i$ is PSD. In view of the objective function $J$ in (\ref{eq:MPC_opt_model}), we can decompose $J$ in a similar way by replacing $\wh W^i - \tau^2 \wh\Psi^i$ in the above decomposition by $\wh V^i$.
\end{remark}

In what follows, we use the above decomposition to formulate a locally coupled optimization problem by introducing copies of local variables.
We consider the cyclic like network topology through this subsection, although the proposed formulation and schemes can be easily extended to other network topologies satisfying the assumption $\bf A.1$. In this case, $\mathcal N_1 = \{ 2\}$, $\mathcal N_{n} = \{ n-1 \}$, and $\mathcal N_{i} = \{ i-1, i+1 \}$ for $i=2, \ldots, n-1$. Hence, each $J_i$ in the decomposition of $J$ can be written as $J_i (\mathbf u_i, (\mathbf u_j)_{j \in \mathcal N_i} )$.

Recall that for each $i=1, \ldots, n$, $\mathcal X_i:=\{ \mathbf u_i \in \mathbb R^p \, | \, a_{i, \min} \mathbf 1 \le \mathbf u_i \le a_{i, \max} \mathbf 1 \}$. Further, define
\begin{eqnarray}
 \mathcal Y_i & := & \left \{ \ubld_i \in \mathbb R^p \, \big | \, v_{\min} \le v_i(k) + \tau \big( S_p \, \abld_i(\ubld_i) \big)_s \le v_{\max}, \quad \forall \, s=1, \ldots, p  \, \right\},  \label{eqn:Y_i} \\
\mathcal Z_i & := & \left\{  \, (\ubld_{i-1}, \ubld_i) \in \mathbb R^p \times \mathbb R^p \, \big | \, (H_i( \ubld_{i-1}, \ubld_i ))_s \le 0, \qquad \forall \, s=1, \ldots, p \, \right\}.  \label{eqn:Z_i}
\end{eqnarray}
As indicated before, $\mathcal Z_1$ depends  only on $\mathbf u_1$ although it is written in the above form for notational convenience.
Let  $\deltabf_S$ denote the indicator function of a closed set $S$. Define, for each $i=1, \ldots, n$,
\begin{align*}
  \wh J_{i}( \mathbf u_i, (\mathbf u_j)_{j \in \mathcal N_i}) & \, := \,   J_i (\mathbf u_i, (\mathbf u_j)_{j \in \mathcal N_i} ) + \deltabf_{\mathcal X_i}(\ubld_i) + \deltabf_{\mathcal Y_i}(\ubld_i) + \deltabf_{\mathcal Z_i}(\mathbf u_{i-1}, \mathbf u_i).
\end{align*}
For each $i=1, \ldots, n$, define $\wh{\mathbf u}_i:= \big(\mathbf u_i, (\mathbf u_{i, j} )_{j \in \mathcal N_i} \big)$, where the new variables $\mathbf u_{i, j}$ represent the predicted values of $\mathbf u_j$ of vehicle $j$ in the neighbor $\mathcal N_i$ of vehicle $i$, and let $\wh {\mathbf u} := ( \wh{\mathbf u}_1, \ldots, \wh{\mathbf u}_n) \in \mathbb R^{N}$.
 Define the consensus subspace
\[
   \mathcal A \, := \, \Big\{ \, \wh{\mathbf u} \in \mathbb R^{N} \, \big | \,  \mathbf u_{i, j} = \mathbf u_j, \ \forall \, (i, j) \in  \Ecal \, \Big\}.
\]
Then the underlying optimization problem  (\ref{eq:MPC_opt_model}) can be equivalently written as the following locally coupled optimization problem:
\begin{equation} \label{eqn:MPC_model_locally_coupled}
\min_{\wh{\mathbf u} } \ \ \sum^n_{i=1} \wh J_{i}(\wh{\mathbf u}_i), \quad \mbox{ subject to } \quad \wh{\mathbf u} \in \mathcal A.
\end{equation}
In the above formulation, the functions $\wh J_i$'s are decoupled, and the consensus constraint $\mathcal A$ gives rise to the only coupling in this formulation.

%
\subsection{Sequential Convex Programming and Operator Splitting Method based Fully Distributed Algorithms for the MPC Optimization Problem} \label{subsect:SCP_scheme}

When the MPC horizon $p=1$, the underlying MPC optimization problem (\ref{eq:MPC_opt_model}) or (\ref{eqn:MPC_model_locally_coupled}) is a convex  quadratically constrained quadratic program (QCQP), for which the fully distributed schemes developed in \cite{ShenEswarDu_Arx20} can be applied. We consider $p>1$ from now on. In this case,
the underlying MPC optimization problem (\ref{eq:MPC_opt_model}) or (\ref{eqn:MPC_model_locally_coupled})
yields a non-convex minimization problem whose objective function and constraints are non-convex, whereas the coefficients $c_{2, i}>0$ and $c_{3, i}>0$ defining the nonlinearities are small. Therefore, it is expected that an optimal solution under the nonlinear vehicle dynamics is ``close'' to that under the linear vehicle dynamics.
The latter solution, which can be obtained using fully distributed schemes \cite{ShenEswarDu_Arx20},
may be used as an initial guess for a distributed  scheme for the nonlinear vehicle dynamics.
 We formally discuss this observation as follows.
%

%

The general theory of perturbed optimization can be found in the monograph \cite{BonnasShapiro_book00}; 
we consider a special case here.
Let $f:\mathbb R^n \times \mathbb R^q \rightarrow \mathbb R$ and $g_i:\mathbb R^n \times \mathbb R^q \rightarrow \mathbb R$ with $i=1, \ldots, m$ be all continuous functions. Let $\Omega \subset \mathbb R^n$ be a compact set, and  $\Theta \subseteq \mathbb R^q$ be a set of parameter vectors that contain the zero vector. Fix a parameter vector $\theta \in \Theta$, and define the parameter dependent constraint set
\[
   \mathcal W_\theta \, := \, \big \{ x \in \mathbb R^n \, \big| \, g_i(x, \theta) \le 0, \ \forall \, i=1, \ldots, m \big \}.
\]
We assume that for each parameter vector $\theta \in \Theta$, the set $\Omega \cap \mathcal W_\theta$ is nonempty. Since $g_i(\cdot, \theta)$ is continuous for a given $\theta$, $\Omega \cap \mathcal W_\theta$ is a nonempty compact set such that for a fixed $\theta \in \Theta$, the minimization problem
\[
P_\theta: \quad \min_{x \in \Omega \cap \mathcal W_\theta} f(x, \theta)
\]
 has a nonempty closed solution set denoted by $\mathcal S_\theta$.

For each $x \in \Omega\cap \W_0$, define the index set $\Ical(x):= \{ i \ | \ g_i(x, 0)=0 \}\subseteq \{1, \ldots, m\}$, which corresponds to the  index set of active inequality constraints. We introduce the following assumption on $\Omega \cap \W_0$.
\begin{itemize}
  \item [$\bf A.4$] For any $x^\diamond \in \Omega\cap \Wcal_0$ whose corresponding $\Ical(x^\diamond)$ is nonempty, there exists a sequence $(w^\ell)$ in $\Omega\cap \W_0$ such that: (i) for each $\ell$, $g_i(w^\ell, 0) < 0$ for all $i=1, \ldots, m$; and (ii) $(w^\ell)$ converges to $x^\diamond$.
\end{itemize}

%
%

The following lemma presents a sufficient condition related to the Slater's condition for $\bf A.4$ to hold.

\begin{lemma} \label{lem:condition_for_A4}
   Suppose each $g_i(\cdot, 0)$ is a convex function, $\Omega \cap \Wcal_0$ is a convex compact set, and there exists $z \in \Omega \cap \Wcal_0$ such that $g_i(z, 0) <0$ for all $i=1, \ldots, m$. Then $\bf A.4$ holds.
\end{lemma}

\begin{proof}
 Let $z\in \Omega \cap \Wcal_0$ be such that $g_i(z, 0) < 0$ for each $i=1, \ldots, m$,
  and consider $x \in \Omega \cap \Wcal_0$ whose index set $\Ical(x)$ is nonempty. Therefore, for any $\lambda \in (0, 1]$, $g_i(x + \lambda (z - x), 0) = g_i(\lambda z + (1-\lambda) x, 0 ) \le \lambda g_i(z, 0) + (1-\lambda) g_i(x, 0) \le  \lambda g_i(z, 0)<0$ for each $i=1, \ldots, m$. Therefore, $\bf A.4$ holds.
\end{proof}

\begin{proposition} \label{prop:solution_proximal}
  Suppose $P_0$ has the unique minimizer $x_*$, i.e., $\mathcal S_0=\{ x_*\}$. Then under the abovementioned assumptions (including $\bf A.4$), for any $\varepsilon>0$, there exists $\eta>0$ such that for all $\theta \in \Theta$ with $\| \theta \| \le \eta$, $\sup_{z \in \mathcal S_\theta} \| z - x_* \| < \varepsilon$.
\end{proposition}

\begin{proof}
 Suppose not. Then there exist $\varepsilon_0>0$ and a sequence $(\theta^k)$ in $\Theta$ with $\| \theta^k \| \rightarrow 0$ such that for each $k$, there exist $z^k \in \mathcal S_{\theta^k}$ with $\| z^k - x_{*} \| \ge \varepsilon_0$.
%
%
 Since $z^k$ belongs to the compact set $\Omega$, $(z^k)$ has a convergent subsequence whose limit $z^* \in \Omega$ satisfies $\| z^* - x_* \| \ge \varepsilon$. Without loss of generality, we may assume that this subsequences is $(z^k)$ itself.
%
%
  Since $g_i(z^k, \theta^k) \le 0$ for all $i=1, \ldots, m$, it follows from the continuity of each $g_i (\cdot, \cdot)$ that $g_i(z^*, 0) \le 0$. Hence, $z^* \in \Omega \cap \Wcal_0$. Consider a fixed but arbitrary $x \in \Omega \cap \Wcal_0$. Hence, either $\Ical(x)$ is empty or $\Ical(x)$ is nonempty. For the former case, we deduce via the continuity of $g_i(x, \cdot)$ that $g_i(x, \theta^k) < 0, i=1, \ldots, m$ for all large $k$.  Hence, $x \in \Omega \cap \Wcal_{\theta^k}$ for all large $k$. This shows that $f(x, \theta^k) \ge f(z^k, \theta^k)$ for all large $k$, and therefore $f(x, 0) \ge f(z^*, 0)$.
 For the latter case, it follows from the assumption $\bf A.4$ on $\Omega \cap \Wcal_0$ that there exists a sequence $(w^\ell)$  in $\Omega\cap \W_0$ which converge to $x$ such that $g_i(w^\ell, 0) < 0$ for all $\ell$ and all $i=1, \ldots, m$ . By the continuity of $g_i$'s and the fact that $(\theta^k) \rightarrow 0$, we see that for $\ell=1$, there exists an index $s_1$ such that $g_i(w^1, \theta^{s_1}) < 0$ for all $i=1, \ldots, m$. Then for $\ell=2$, there exists an index $s_2$ with $s_2>s_1$ such that $g_i(w^2, \theta^{s_2}) < 0$ for all $i=1, \ldots, m$.
 Continuing this process, we obtain a strictly increasing index sequence $(s_\ell)$ such that
 for each $\ell$, $g_i(w^\ell, \theta^{s_\ell}) < 0$ for all $i=1, \ldots, m$.
%
%
 Hence, each $w^\ell \in \Omega \cap \Wcal_{\theta^{s_\ell}}$ such that $f(w^\ell, \theta^{s_\ell})  \ge f(z^{s_\ell}, \theta^{s_\ell})$. Since $(\theta^{s_\ell})$ is a subsequence of $(\theta^k)$ and $(z^{s_\ell})$ is a subsequence of $(z^k)$, we have that $\theta^{s_\ell} \rightarrow 0$ and $z^{s_\ell} \rightarrow z^*$. This leads to $f(x, 0) \ge f(z^*, 0)$. Consequently, $f(x, 0) \ge f(z^*, 0)$ for all $x \in \Omega \cap \Wcal_0$. This implies that $z^*$ is a minimizer of $P_0$. Since $x_*$ is the unique minimizer of $P_0$, we must have $z^*=x_*$, yielding a contradiction to $\| z^* - x_* \| \ge \varepsilon_0$.
\end{proof}

We apply this proposition to the optimization problem (\ref{eq:MPC_opt_model}).
Recall that the parameter vector $\varphibf=(\varphibf_d, \varphibf_f)=(c_{2,i}, c_{3,i})^n_{i=1} \in \mathbb R^{2n}_+$. To emphasize the dependence of the objective function $J$ on $\varphibf$, we write it as $J(\ubld, \varphibf)$ by abusing the notation. Further, the constraints in (\ref{eq:MPC_opt_model}) can be written as $\mathcal X \cap \mathcal Y \cap \mathcal Z$, where $\mathcal X = \mathcal X_1 \times \cdots \times \mathcal X_n$ is a convex and compact set, and $\mathcal Y \cap \mathcal Z = \{ \ubld \, | \, g_i(\ubld, \varphibf) \le 0, \, i=1, \ldots, m \}$ for some real-valued functions $g_i$, which also depend on $\varphibf$. It is shown in \cite{ShenEswarDu_Arx20} that when $\varphibf =0$, $J(\ubld, 0)$ is a strongly convex quadratic function, and each $g_i(\ubld, 0)$ is an affine or a convex quadratic function. Hence, when $\varphibf =0$,  (\ref{eq:MPC_opt_model}) becomes a convex optimization problem which attains a unique optimal solution $\ubld_{*,0}$. Further, when $r_i \ge \tau$ for all $i$ and $v_0(k)> v_{\min}$, this convex optimization problem has non-empty interior \cite[Corollary 3.1]{ShenEswarDu_Arx20} such that $\bf A.4$ holds by Lemma~\ref{lem:condition_for_A4}. Therefore, letting $\mathcal S_{\varphibf}$ denote the solution set of (\ref{eq:MPC_opt_model}) corresponding to the parameter vector $\varphibf$, we obtain the following corollary from Proposition~\ref{prop:solution_proximal}.

\begin{corollary} \label{coro:optimal_solution_closeness}
 Consider the optimization problem (\ref{eq:MPC_opt_model}) with the parameter vector $\varphibf \in \mathbb R^{2n}_+$ at time $k$. Suppose $r_i \ge \tau$ for all $i$ and $v_0(k)> v_{\min}$. Then
 for any $\varepsilon>0$, there exists $\eta>0$ such that for all $\varphibf \in \mathbb R^{2n}_+$ with $\| \varphibf \| \le \eta$, $\sup_{\ubld \in \mathcal S_{\varphibf} } \| \ubld - \ubld_{*,0} \| < \varepsilon$.
\end{corollary}

\gap

%
%
To solve the coupled non-convex  optimization problem (\ref{eq:MPC_opt_model}) with $\varphibf \ne 0$, we exploit the sequential convex programming (SCP) method \cite{ZSLu_Technote2013}. To be self-contained, we provide a brief description of the SCP method for an important special case as follows.
%
%
%
Consider the nonlinear program
\begin{equation} \label{eqn:opt_model_SCP}
 (P'): \quad \min_{x \in \mathbb R^n} f(x) \quad \mbox{ subject to} \quad  x \in \mathcal P, \ \ \ g_i(x) - r_i(x) \le 0, \quad \forall \, i=1, \ldots, \ell,
\end{equation}
where $\Pcal \subseteq \mathbb R^n$ is a closed convex set, $f$ and each $g_i$ are $C^1$ (but not necessarily convex) functions,
and each $r_i$ is a  convex $C^1$-function. We assume that $\nabla f$ and $\nabla g_i$ are Lipschitz on $\mathcal P$, i.e. there exist constants $L_f>0$ and $L_{g_i}>0$ such that $\| \nabla f(x) - \nabla f(x') \|_2 \le L_f \|x - x'\|_2$ and $\| \nabla g_i(x) - \nabla g_i(x') \|_2 \le L_{g_i} \|x - x'\|_2$ for all $x, x' \in \mathcal P$ and $i=1, \ldots, \ell$.
%
%
Let $\wh x$ be a feasible point of $(P')$, i.e., $\wh x \in \mathcal P'$ and $g_i(\wh x) - r_i(\wh x)\le 0, \ i=1, \ldots, \ell$. Consider an approximation of the constraint set of $(P')$ at $\wh x$:
\begin{eqnarray*}
\lefteqn{   \mathcal C(\wh x, \{ \nabla g_i(\wh x) \}^\ell_{i=1}, \{ \nabla r_i(\wh x) \}^\ell_{i=1} ) }  \\
 & := & \Big \{ z \in \mathcal P \, | \, g_i(\wh x) + \nabla g_i(\wh x)^T (z - \wh x) + \frac{L_{g_i}}{2}\| z - \wh x\|^2_2 - [ r_i(\wh x) + \nabla r_i(\wh x)^T (z - \wh x) ] \le 0, \ i=1, \ldots, \ell \Big \}.
\end{eqnarray*}
It is shown in \cite[Lemma 3.3]{ZSLu_Technote2013} that $ \mathcal C(\wh x, \{ \nabla g_i(\wh x) \}^\ell_{i=1}, \{ \nabla r_i(\wh x) \}^\ell_{i=1} )$ is a nonempty closed convex set. The following lemma provides a simple sufficient condition for the Slater's condition to hold for the approximated constraint set; this condition is useful for convergence analysis of the SCP scheme.

\begin{lemma} \label{lem:Slater_CQ_SCP}
 Given a feasible point $\wh x$ of $(P')$,  suppose $\mathcal C(\wh x, \{ \nabla g_i(\wh x) \}^\ell_{i=1}, \{ \nabla r_i(\wh x) \}^\ell_{i=1} )$ is not singleton. Then the Slater's condition holds for $\mathcal C(\wh x, \{ \nabla g_i(\wh x) \}^\ell_{i=1}, \{ \nabla r_i(\wh x) \}^\ell_{i=1} )$, i.e., there exists $\wh z \in \mathcal P$ such that $g_i(\wh x) + \nabla g_i(\wh x)^T (\wh z - \wh x) + \frac{L_{g_i}}{2}\| \wh z - \wh x\|^2_2 - [ r_i(\wh x) + \nabla r_i(\wh x)^T (\wh z - \wh x) ] < 0, \forall \, i=1, \ldots, \ell$.
\end{lemma}

\begin{proof}
Clearly, $\wh x \in \mathcal C(\wh x, \{ \nabla g_i(\wh x) \}^\ell_{i=1}, \{ \nabla r_i(\wh x) \}^\ell_{i=1} )$. Since $\mathcal C(\wh x, \{ \nabla g_i(\wh x) \}^\ell_{i=1}, \{ \nabla r_i(\wh x) \}^\ell_{i=1} )$ is not singleton, there exists $x' \in \mathcal C(\wh x, \{ \nabla g_i(\wh x) \}^\ell_{i=1}, \{ \nabla r_i(\wh x) \}^\ell_{i=1} )$ and $x' \ne \wh x$. Let  $\Ical(\wh x):=\{ i \, | \, g_i(\wh x) - r_i(\wh x) = 0 \}$ denote the index set of  active constraints at $\wh x$ of $(P)$. Thus $x' \in \mathcal P$ and $[\nabla g_i(\wh x)- \nabla r_i(\wh x)]^T (x' - \wh x)+ \frac{L_{g_i}}{2}\| x' - \wh x\|^2_2 \le 0$ for all $i\in \Ical(\wh x)$. Since  $\frac{L_{g_i}}{2}\| x' - \wh x\|^2_2>0$ for all $i \in \Ical(\wh x)$, we have $[\nabla g_i(\wh x)- \nabla r_i(\wh x)]^T (x' - \wh x)<0$ for all $i \in \Ical(\wh x)$. Let $d:=x' - \wh x$. Since $\mathcal P$ is a closed convex set, $d \in \mathcal T_{\mathcal P}(\wh x)$ \cite[Lemma 3.13]{Ruszczynski_book06}, where $\mathcal T_{\mathcal P}(\wh x)$ denotes the tangent cone of $\mathcal P$ at $\wh x$. This shows that the (generalized) MFCQ holds. It thus follows from \cite[Proposition 3.5]{ZSLu_Technote2013} that the Slater's condition holds for $\mathcal C(\wh x, \{ \nabla g_i(\wh x) \}^\ell_{i=1}, \{ \nabla r_i(\wh x) \}^\ell_{i=1} )$.
%
%
\end{proof}

The SCP scheme solves $(P')$ in (\ref{eqn:opt_model_SCP}) as follows \cite{ZSLu_Technote2013}: consider an approximation of the objective function $f$ for a given feasible point $\wh x$: $\wt f(z; \wh x):=f(\wh x) + [\nabla f(\wh x) ]^T(z - \wh x) + \frac{L_f}{2} \|z - \wh x \|^2_2$. Clearly, $\wt f$ is a strongly convex function in $z$. At each step, the SCP scheme solves the convex optimization problem at $x^k$ using the convex approximation $\wt f (\cdot; x^k)$ over the  approximating convex constraint set $\mathcal C(x^k, \{ \nabla g_i(x^k) \}^\ell_{i=1}, \{ \nabla r_i(x^k) \}^\ell_{i=1} )$ to generate a unique optimal solution $x^{k+1}$. It then updates the gradients $\nabla f$, $\nabla g_i$, and $\nabla r_i$ using $x^{k+1}$, and formulates another convex optimization problem and solves it again. It is shown in \cite[Theorem 3.4]{ZSLu_Technote2013} that any accumulation point of the sequence $(x^k)$ generated by the SCP scheme is a KKT point of $(P')$, provided that the accumulation point $x^*$ satisfies the Slater's condition for $\mathcal C(x^*, \{ \nabla g_i(x^*) \}^\ell_{i=1}, \{ \nabla r_i(x^*) \}^\ell_{i=1} )$.

We now apply the SCP scheme to develop a fully distributed scheme for the non-convex  MPC optimization problem (\ref{eq:MPC_opt_model}). Consider the  locally coupled formulation (\ref{eqn:MPC_model_locally_coupled}) of the MPC optimization problem (\ref{eq:MPC_opt_model}).
Recall that $\wh{\mathbf u}_i:= \big(\mathbf u_i, (\mathbf u_{i, j} )_{j \in \mathcal N_i} \big)$, and $\wh {\mathbf u} := ( \wh{\mathbf u}_1, \ldots, \wh \ubld_n)$. For each $i=1,\ldots, n$, it follows from the velocity constraint $\mathcal Y_i$ in (\ref{eqn:Y_i}) and the safety distance constraint $\mathcal Z_i$ in  (\ref{eqn:Z_i}) that there are real-vauled smooth functions $g_{i, s}$ and convex quadratic functions $r_{i, s}$ for $s=1, \ldots, 3p$ such that $\wh \ubld_i \in \mathcal Y_i \cap \mathcal Z_i$ if and only if $g_{i, s}(\wh\ubld_i) - r_{i, s}(\wh \ubld_i) \le 0$ for $s=1, \ldots, 3p$;
specific choices of $g_{i, s}$ and $r_{i, s}$ are given in Sections~\ref{subsect:approximation_functions} and \ref{subsect:test_performance}.
  In view of the real-valued objective function $J (\wh \ubld) = \sum^n_{i=1} J_i( \wh \ubld_i)$, the problem (\ref{eqn:MPC_model_locally_coupled}) becomes
\[
   \min \sum^n_{i=1} J_i( \wh \ubld_i), \quad \mbox{ subject to } \wh \ubld \in \mathcal A, \ \
   \wh\ubld_i \in \mathcal X_i, \ \ g_{i, s}(\wh\ubld_i) - r_{i, s}(\wh \ubld_i) \le 0, \ \ \forall \,   i=1,\ldots, n, \ \ s=1, \ldots, 3p.
\]
Recall that $\mathcal X = \mathcal X_1 \times \cdots \times \mathcal X_n$ is a convex compact set.
Since $\mathcal X$ is compact and $\mathcal A$ is the consensus subspace, it is easy to show that there are positive Lipschitz constants $L_{J_i}$ and $L_{g_{i, s}}$ for the gradients of $J_i$ and $g_{i, s}$ on $\mathcal A \cap \mathcal X$, i.e., for all $\wh \ubld, \wh \ubld' \in \mathcal A \cap \mathcal X$,
\begin{eqnarray*}
    \| \nabla J_i( \wh \ubld_i) -  \nabla J_i( \wh \ubld'_i) \|_2 & \le & L_{J_i} \cdot \| \wh \ubld_i - \wh \ubld'_i \|_2, \qquad \forall \ i=1, \ldots, n, \\
     \| \nabla g_{i, s}(\wh \ubld_i ) -  \nabla g_{i, s}( \wh \ubld'_i) \|_2 & \le & L_{g_{i, s}} \cdot \| \wh \ubld_i - \wh \ubld'_i \|_2, \qquad \forall \  i=1, \ldots, n, \ \ s=1, \ldots, 3p.
\end{eqnarray*}

To develop a SCP based fully distributed scheme,
we introduce more notation. Given any $\wh \ubld =(\wh\ubld_i)^n_{i=1} \in \mathcal X$ and any vectors $d_{J_i}$, $d_{g_{i, s}}$, and $d_{r_{i, s}}$ for $i=1, \ldots, n$ and $s=1, \ldots, 3p$, consider the following function as a convex approximation of the original nonconvex objective function $J$, where $y=(y_1, \ldots, y_n) \in \mathbb R^N$ with each $y_i$ being a suitable subvector of $y$:
\[
   f(y; \, \wh\ubld, \{ d_{J_i} \}^n_{i=1} ) \, := \, \sum^n_{i=1} \Big( J_i(\wh\ubld_i) + d^T_{J_i}(\wh\ubld_i) ( y_i - \wh\ubld_i) + \frac{L_{J_i}}{2} \|   y_i - \wh\ubld_i \|^2_2 \Big),
\]
and the following sets as convex approximations of the original nonconvex constraint sets $\mathcal Y \cap \mathcal Z$:
\begin{eqnarray*}
  \lefteqn{ \mathcal C\big(\wh \ubld, \{ d_{g_{i, s}}, d_{r_{i, s}}, \, i=1, \ldots, n, \, s=1, \ldots, 3p  \} \, \big) } \\
   & := & \Big\{ \, y \in \mathcal X \, | \  g_{i, s}(\wh \ubld_i) + d^T_{g_{i, s}} (y_i - \wh\ubld_i) + \frac{L_{g_{i, s}}}{2} \| y_i - \wh \ubld_i \|^2_2  \\
    & & \qquad \qquad \qquad - \big[ r_{i, s}(\wh \ubld_i) + d^T_{r_{i, s}} ( y_i - \wh \ubld_i) \big] \le 0, \ \ \ i=1, \ldots, n, \ \ s=1, \ldots, 3p \ \Big\},
\end{eqnarray*}
Clearly, $f$ is a strongly convex quadratic function in $y$ and decoupled in $y_i$'s, and the convex set  $\mathcal C\big(\wh \ubld, \{ d_{g_{i, s}}, d_{r_{i, s}}, \, i=1, \ldots, n, \, s=1, \ldots, p  \} \big)$ is the Cartesian product of $\mathcal C_i$'s for $i=1, \ldots, n$, where each
\begin{eqnarray*}
 \mathcal C_{i}\big(\wh \ubld_i, \{ d_{g_{i, s}} \}^{3 p}_{s=1}, \{ d_{r_{i, s}} \}^{3p}_{s=1} \big)  & := & \Big\{ y_i \in \mathcal X_i \ | \ g_{i, s}(\wh \ubld_i) + d^T_{g_{i, s}} (y_i - \wh\ubld_i) + \frac{L_{g_{i, s}}}{2} \| y_i - \wh \ubld_i \|^2_2  \\
    & & \qquad \qquad \qquad - \big[ r_{i, s}(\wh \ubld_i) + d^T_{r_{i, s}} ( y_i - \wh \ubld_i) \big] \le 0,  \ \, s=1, \ldots, 3p \Big\}.
\end{eqnarray*}
%
%

Using the above notation, the iterative scheme of the SCP method is: for a  feasible initial guess $\wh \ubld^0$,
\begin{eqnarray}
   \wh \ubld^{k+1} & = & \argmin_y \Big\{ f(y; \wh\ubld^k, \{ \nabla J_i(\wh\ubld^k_i) \}^n_{i=1} ) \, \big | \, y \in \mathcal A, \ \ \mbox{and} \notag \\
    & & \qquad \qquad y \in \mathcal C\big(\wh \ubld^k, \{ \nabla g_{i, s}(\wh \ubld^k_i), \nabla r_{i, s}(\wh \ubld^k_i), \, i=1, \ldots,n, \, s=1, \ldots, 3 p \} \big) \Big\}. \label{eqn:SCP_k}
\end{eqnarray}
By virtue of Corollary~\ref{coro:optimal_solution_closeness}, the initial $\wh \ubld^0$ can be chosen as a solution to the problem (\ref{eq:MPC_opt_model}) or (\ref{eqn:MPC_model_locally_coupled}) whose objective function is $J$ with $\varphibf=0$ and the approximate constraints are polyhedral or quadratically constrained convex sets; see Section~\ref{subsect:approximation_functions} for details.
An efficient fully distributed scheme has been developed in \cite{ShenEswarDu_Arx20} to compute such $\wh \ubld^0$.
%
%
%
It is shown in \cite[Theorem 4.3]{ZSLu_Technote2013} that if $\wh\ubld^0$ is feasible, then $\wh \ubld^k$ is feasible for all $k$ and the constraint set in each step $k$ is a nonempty closed convex set \cite[Lemma 3.3]{ZSLu_Technote2013}.

The convex minimization problem (\ref{eqn:SCP_k}) at each step $k$ can be solved via  operator splitting method based fully distributed schemes. Fix $\wh \ubld^k=(\wh \ubld^k_i)^n_{i=1}$ and the related gradients evaluated at $\wh \ubld^k$. We write the objective function $f(y; \, \wh\ubld, \{ d_{J_i} \}^n_{i=1} )$ as $f(y)$ and
 the constraint sets $\mathcal C_i\big(\wh \ubld^k_i, \{ \nabla g_{i, s}(\wh \ubld^k_i), \nabla r_{i, s}(\wh \ubld^k_i), \, s=1, \ldots, 3 p \} \big)$ as $\mathcal C_{i}$'s for notational simplicity. Clearly, $\wh \ubld^k_i \in \mathcal C_i$ for each $i$.
If $\mathcal C_{i}$ is singleton for some $i$, i.e., $\mathcal C_{i} = \{ \wh\ubld^k_i\}$, then we have $\wh\ubld^{k+1}_i = \wh \ubld^k_i$ such that the optimization problem can be reduced to a simpler problem. When $\mathcal C_{i}$ is non-singleton, it follows from Lemma~\ref{lem:Slater_CQ_SCP} that the Slater's condition holds for that $\mathcal C_{i}$. Let $F(y) := f(y; \wh\ubld^k, \{ \nabla J_i(\wh\ubld^k_i) \} ^n_{i=1} )+ \deltabf \mathcal C(y) + \deltabf \mathcal A(y)$.
By  \cite[Corollary 23.8.1]{Rockafellar_book70}, $\partial F(y) = \{ \nabla f(y) \} +  \mathcal N_{\mathcal C} (y) + \mathcal N_{\mathcal A}(y)$.
As a result, several operator splitting method based fully distributed algorithms \cite{DavisYin_SVA17, HuXiaoLiu_CDC18} can be applied to solve the convex optimization problem (\ref{eqn:SCP_k}).

Motivated by \cite{ShenEswarDu_Arx20}, we consider the (generalized) Douglas-Rachford splitting method based distributed scheme. Specifically, define for each $i=1, \ldots, n$, $f_i(y_i):= J_i(\wh\ubld^k_i) + d^T_{J_i}(\wh\ubld^k_i) ( y_i - \wh\ubld_i) + \frac{L_{J_i}}{2} \|   y_i - \wh\ubld^k_i \|^2_2$, and $\wh f_i(y) := f_i(y_i) + \deltabf \mathcal C_i(y_i)$. Hence, the objective function $f(y)= \sum^n_{i=1} f_i(y_i)$.
For any constant $0<\alpha<1$ and $\rho>0$, the Douglas-Rachford splitting method based scheme is given by
\begin{align*}
   w^{t+1} \, = \, \Pi_{\mathcal A} (z^t), \quad
   z^{t+1} \, = \, z^t + 2 \alpha \cdot \Big[ \mbox{Prox}_{\rho \wh f_1+ \cdots + \rho \wh f_n}\big(2 w^{t+1} - z^t \big) - w^{t+1} \Big],  \ \ \forall \, t\in \mathbb Z_+,
\end{align*}
where $\mbox{Prox}_h$ denotes the proximal operator of a proper lower semicontinuous convex function $h$, and $\Pi_{\mathcal A}$ denotes the Euclidean projection onto $\mathcal A$.
Since $\mathcal A$ is the consensus subspace, it is shown that \cite[Section IV]{HuXiaoLiu_CDC18} that for any $\wh {\mathbf u} := ( \wh{\mathbf u}_1, \ldots, \wh{\mathbf u}_n)$ where $\wh{\mathbf u}_i:= \big(\mathbf u_i, (\mathbf u_{ij} )_{j \in \mathcal N_i} \big)$, $\ol  {\mathbf u} := \Pi_{\mathcal A}(\wh {\mathbf u})$ is given by:
\begin{equation} \label{eqn:consensus_projection}
   \ol {\mathbf u}_j = \ol {\mathbf u}_{ij} =  \frac{1}{1+|\mathcal N_j| } \Big(  \wh{\mathbf u}_j  + \sum_{k \in \mathcal N_j} \wh{\mathbf u}_{kj} \Big), \qquad \forall \, (i, j) \in \mathcal E.
\end{equation}
Furthermore, since $\wh f_i$'s are decoupled, a distributed version of the above algorithm is given by:
\begin{subequations} \label{eqn:DR_scheme}
\begin{align}
   w^{t+1}_i & \, = \, \ol z^t_i, \qquad i=1, \ldots, n; \\
   z^{t+1}_i & \, = \, z^t_i + 2 \alpha \cdot \Big[ \mbox{Prox}_{\rho \wh f_i} \big( 2 w^{t+1}_i - z^t_i\big ) - w^{t+1}_i \Big], \quad i=1, \ldots, n.
\end{align}
\end{subequations}
%
%
Note that the proximal operator in the second equation of (\ref{eqn:DR_scheme}) is given by
$\mbox{Prox}_{\rho \wh f_i}( 2 w^{t+1}_i - z^t_i) = \argmin_{y_i \in \mathcal C_i} f_i(y_i)+ \frac{1}{2\rho} \| y_i - (2 w^{t+1}_i - z^t_i) \|^2_2$, where $\mathcal C_i$ is the intersection of the polyhedral set $\mathcal X_i$ and a quadratically constrained convex set. Since $f_i$ is a convex quadratic  function, $\mbox{Prox}_{\rho \wh f_i}( 2 w^{t+1}_i - z^t_i)$ can be formulated as a second-order cone program or QCQP and solved by SeDuMi \cite{Sturm_OMS99}. See Algorithm~\ref{algo:distributed_SCP} for its pseudo-code.

%
%

\begin{algorithm}
\caption{Sequential Convex Programming and Douglas-Rachford Splitting Method based Fully Distributed Algorithm for $p \ge 2$}
\begin{algorithmic}[1]
\label{algo:distributed_SCP}
%
%
\STATE Choose constants $0<\alpha <1$ and $\rho>0$

\STATE Solve the problem (\ref{eqn:MPC_model_locally_coupled}) with $\varphibf =0$ via a fully distributed scheme and obtain a solution $\wh \ubld^{\text{lin}}$

\STATE Initialize $k=0$, and set an initial point $\wh\ubld^0=\wh \ubld^{\text{lin}}$

\WHILE{the stopping criteria is not met}

  \STATE Compute $\nabla J_i(\wh\ubld^k_i)$, $\nabla g_{i, s}(\wh \ubld^k_i)$, $\nabla r_{i, s}(\wh \ubld^k_i)$, and set  $z^0=\wh\ubld^k$ and $t=0$.

   \REPEAT

  \FOR {$i=1, \ldots, n$}

   \STATE Compute $\ol z^t_i$ using equation (\ref{eqn:consensus_projection}), and let $w^{t+1}_i \leftarrow \ol z^t_i$


  \ENDFOR

  \FOR {$i=1, \ldots, n$}

   \STATE {$z^{t+1}_i \leftarrow z^t_i + 2 \alpha \cdot \Big[ \mbox{Prox}_{\rho \wh f_i} \big( 2 w^{t+1}_i - z^t_i\big ) - w^{t+1}_i \Big]$}

  \ENDFOR

  \STATE  $t \leftarrow t+1$

   \UNTIL an accumulation point is achieved

   \STATE Set $\wh \ubld^{k+1} = w^t$ and $k\leftarrow k+1$

\ENDWHILE


\RETURN $\wh{\mathbf u}^* = \wh \ubld^k$


\end{algorithmic}
\end{algorithm}

Since $\mathcal X$ is a compact set, the numerical sequence $(\wh\ubld^k)$ generated by Algorithm~\ref{algo:distributed_SCP} always has an accumulation point denoted by $\wh\ubld^*$. It follows from \cite[Theorem 3.4]{ZSLu_Technote2013} that under very mild conditions, $\wh\ubld^*$ is feasible and is a KKT point of the nonconvex program (\ref{eq:MPC_opt_model}).
%
%
Our numerical experiences show that $(\wh\ubld^k)$ converges to $\wh\ubld^*$ which is a local minimizer of (\ref{eq:MPC_opt_model}). This coincides with the observation made in Corollary~\ref{coro:optimal_solution_closeness} when $c_{2, i}$ and $c_{3, i}$ are small.


%
\subsection{Approximation of the Objective Function and Constraint Functions} \label{subsect:approximation_functions}

When $p>1$, the underlying MPC optimization problem (\ref{eq:MPC_opt_model}) and its locally coupled formulation (\ref{eqn:MPC_model_locally_coupled})  give rise to non-convex optimization problems with complicated objective functions and constraints, especially the velocity and safety distance constraints for a relatively large $p$, due to highly sophisticated closed-form expressions for $\abld_i(\ubld_i)$'s.
To facilitate computation, particularly for real-time computation, we derive a simplified model to approximate the objective function and constraint functions below. We start with the constraints for $\ubld=(\ubld_1, \ldots, \ubld_n)$ first, where we omit $k$ since it is fixed.

The exact closed-form expressions for $\abld_i(\ubld_i)$'s are given in the recursive manner at the beginning of Section~\ref{subsect:optim_nonlinear_dynamics}. These expressions
are highly sophisticated especially for large $j$'s because of the nonlinear relation in aerodynamic drag. Since the  coefficients $c_{2, i}$'s and $c_{3,i}$'s are small, we only consider the terms in $\abld_i(\ubld_i)$ that are linearly in $c_{2, i}$'s and $c_{3,i}$'s while ignoring the terms involving higher orders of $c_{2, i}$'s and $c_{3,i}$'s. Such an approximation is accurate enough for transportation applications and facilitates numerical computation. Toward this goal, recall that
 the matrix $S_p \in \mathbb R^{p \times p}$ is defined in the same was as in (\ref{eqn:matrix_S_n}) with $n$ replaced by $p$, and $(S_p \mathbf u_i)_0 :=0$.
We then have, for each $i=1, \ldots, n$ and $s=1, \ldots, p$,
\[
   \big( \abld_i(\ubld_i) \big)_s \, \approx \, \big( \mathbf u_i \big)_s  - c_{2, i} \Big[ v_i(k) + \tau \big(S_p \mathbf u_i \big)_{s-1} \Big]^2 -  c_{3, i} g,
\]
where $\big( \abld_i(\ubld_i) \big)_s$ denotes the $s$-entry of $\abld_i(\ubld_i)$ that corresponds to $a_i(k+s-1, u_i(k), \ldots, u_i(k+s-1))$, and  $\big( \mathbf u_i \big)_s$ denotes the $s$-entry of $\ubld_i$ that corresponds to $u_i(k+s-1)$. Therefore,
\begin{equation} \label{eqn:a_i_approx}
   \abld_i(\ubld_i) \, \approx \, \ubld_i - c_{2, i} \Big[ v^2_i(k) \mathbf 1  + 2\tau v_i(k) \wt S_p \ubld_i  + \tau^2 \big( \wt S_p \ubld_i \big) \circ \big( \wt S_p \ubld_i \big) \Big] - c_{3,i} g \mathbf 1,
\end{equation}
where $\wt S_p := \begin{bmatrix} 0 & 0 \\ I_{p-1} & 0 \end{bmatrix} S_p \in \mathbb R^{p\times p}$, and $\circ$ denotes the Hadamard product of two vectors in $\mathbb R^p$. Slightly abusing notation, we let $\abld_i(\ubld_i)$ represent its approximation given on the right-hand side of (\ref{eqn:a_i_approx}).
It is easy to derive the Jacobian of $\abld_i(\ubld_i)$ as
\begin{equation} \label{eqn:Jacobian_a_i}
  \mathbf J \abld_i(\ubld_i) = I_p - 2 c_{2, i} \tau v_i(k) \wt S_p - 2  c_{2, i} \tau^2 \mbox{diag}(\wt S_p \ubld_i ) \wt S_p,
\end{equation}
where for a vector $v=(v_1, \ldots, v_p) \in \mathbb R^p$, $\mbox{diag}(v)$ denotes the $p\times p$ diagonal matrix whose diagonal entries are given by $v_1, \ldots, v_p$. Here we use the fact that the Jacobian of $(A x) \circ (A x)$ is given by $\mathbf J (A x) \circ (A x)= 2 \mbox{diag}(Ax) A$ for a matrix $A$.

\gap

\noindent {\bf Approximate speed constraint}. Using the approximated $\abld_i(\cdot)$, we have, for each $i=1, \ldots, n$ and $j=1, \ldots, p$,
\begin{equation} \label{eqn:v_i_approx}
  v_i(k+j) = v_i(k) + \tau\sum^{j}_{s=1} \big( \abld_i(\ubld_i) \big)_{s} \approx v_i(k) + \tau \left( \big( S_p \mathbf u_i \big)_j - j\cdot c_{3, i} g - c_{2, i} \sum^{j-1}_{s=0} \Big[ v_i(k) + \tau \big(S_p \mathbf u_i \big)_s \Big]^2 \right).
\end{equation}
%
It should be noted that the resulting approximation of $v_i(k+j)$ remains a nonlinear and nonconvex function in $u_i(k), \ldots, u_i(k+j-1)$ or $\mathbf u_i$. The approximated speed constraint for $\ubld_i$ is given by
\[
 \mathcal Y_i:= \left \{ \ubld_i \in \mathbb R^p \, \Big | \, v_{\min} \le v_i(k) + \tau \Big( \big( S_p \mathbf u_i \big)_j - j\cdot c_{3, i} g - c_{2, i} \sum^{j-1}_{s=0} \big[ v_i(k) + \tau (S_p \mathbf u_i \big)_s \big]^2 \Big) \le v_{\max}, \, j=1, \ldots, p  \, \right\}.
\]
For $j=1, \ldots, p$, define the function
\begin{equation} \label{eqn:_q_i_j}
q_{i, j}(\ubld_i) \, := \, v_i(k) + \tau \Big( \big( S_p \mathbf u_i \big)_j - j\cdot c_{3, i} g - c_{2, i} \sum^{j-1}_{s=0} \big[ v_i(k) + \tau (S_p \mathbf u_i \big)_s \big]^2 \Big).
\end{equation}
A straightforward calculation shows that the gradient of $q_{i, j}$ is given by
\[
\nabla q_{i, j}(\ubld_i)= \tau \left(  \big (S_p)_{j, \bullet} -  2\tau c_{2,i} \sum^{j-1}_{s=0} \big[ v_i(k) + \tau (S_p \mathbf u_i \big)_s \big] \big(S_p)_{s\bullet}  \right)^T.
\]


\noindent {\bf Approximate safety distance constraint}. For each $i=1, \ldots, n$ and $j=1, \ldots, p$, the safety distance constraint is given by
\[
  \big( H_i(\ubld_{i-1}, \ubld_i) \big)_j := L_i  + r_i \cdot v_i(k+j) - \frac{ (v_i(k+j) - v_{\min} )^2}{ 2 a_{i, \min} } - [ x_{i-1}(k+j) - x_i(k+j)] \, \le \, 0.
\]
%
%
To derive an approximation of $\big( H_i(\ubld_{i-1}, \ubld_i) \big)_j$, recall that the expression for $z_i(k+j)= x_{i-1}(k+j) - x_i(k+j)$ is given by (\ref{eqn:z_k_j}), where
\mycut{
 First, we have
\begin{align*}
 z_i(k+j) & = x_{i-1}(k+j) - x_i(k+j)  = z_i(k) + j \tau z'(k) + \tau^2 \sum^{j-1}_{s=0} \frac{2(j-s)-1}{2} b_i(k+s, \ubld_{i-1}, \ubld_i ), \\
 z'_i(k+j) & = v_{i-1}(k+j) - v_{i}(k+j) = z'_i(k) + \sum^{j-1}_{s=0} b_i(k+s, \ubld_{i-1}, \ubld_i ),
\end{align*}
where
}
for each $s=0, 1, \ldots, j-1$, it follows from  (\ref{eqn:v_i_approx}) that
\begin{align*}
\lefteqn{ b_i(k+s, \ubld_{i-1}, \ubld_i )  = \big( \abld_{i-1}(\ubld_{i-1})- \abld_i(\ubld_{i}) \big)_{s+1} }\\
   & = \big[ u_{i-1}(k+s) - c_{2, i-1} v^2_{i-1}(k+s) - c_{3, i-1} g \big]  - \big[ u_{i}(k+s) - c_{2, i} v^2_i(k+s) - c_{3, i} g \big] \\
   & \approx \big[ u_{i-1}(k+s) - u_{i}(k+s) \big] - \left( c_{2,i-1} \Big[ v_{i-1}(k) + \tau \big(S_p \ubld_{i-1}\big)_s \Big]^2 -  c_{2,i} \Big[ v_{i}(k) + \tau \big(S_p \ubld_{i}\big)_s \Big]^2 \right) \\
   & \qquad - \big(  c_{3, i-1} -c_{3, i} \big) g.
\end{align*}
Further, by using the approximation of $v_i(k+j)$ given in (\ref{eqn:v_i_approx}), we obtain
\begin{align}
  \big( H_i(\ubld_{i-1}, \ubld_i) \big)_j & \approx \ \ L_i + r_i \cdot \Big[
  v_i(k) + \tau \Big( \big( S_p \mathbf u_i \big)_j - j\cdot c_{3, i} g - c_{2, i} \sum^{j-1}_{s=0} \big[ v_i(k) + \tau  \big(S_p \mathbf u_i \big)_s \big]^2 \Big) \Big] \notag  \\
  & - \frac{1}{2 a_{i, \min}} \Bigg[
  v_i(k) - v_{\min} + \tau \Big( \big( S_p \mathbf u_i \big)_j - j\cdot c_{3, i} g - c_{2, i} \sum^{j-1}_{s=0} \big[ v_i(k) + \tau  \big(S_p \mathbf u_i \big)_s \big]^2 \Big) \Bigg]^2 \label{eqn:H_i_j_approx} \\
  & - \left\{ \, z_i(k) +\Delta + j \tau z'_i(k) + \tau^2 \sum^{j-1}_{s=0} \frac{2(j-s)-1}{2}
   \Big[ \, u_{i-1}(k+s) - u_{i}(k+s)   \right. \notag \\
   & \qquad  -  \Big( c_{2,i-1} \big[ v_{i-1}(k) + \tau \big(S_p \ubld_{i-1}\big)_s \big]^2   -  c_{2,i} \big[ v_{i}(k) + \tau  \big(S_p \ubld_{i}\big)_s \big]^2 \Big)  - \big(  c_{3, i-1} -c_{3, i} \big) g  \Big] \, \Bigg\}. \notag
\end{align}
By slightly abusing the notation, we let $\big( H_i(\ubld_{i-1}, \ubld_i) \big)_j$ denote its approximation given above. Clearly, this approximating function is smooth but nonconvex.
 Further, when $\varphibf=(\varphibf_d, \varphibf_f)=(c_{2,i}, c_{3,i})^n_{i=1}=0$, each $\big( H_i(\ubld_{i-1}, \ubld_i) \big)_j$ reduces to a convex quadratic function in $(\ubld_{i-1}, \ubld_i)$.
%
%
To compute the gradient of  $\big( H_i(\ubld_{i-1}, \ubld_i) \big)_j$, let
 the vector $\mathbf p$ and the matrix $R_p$ be given by
\[
  \mathbf p \, := \, \begin{bmatrix} \ 1 \ \\ \ 2 \ \\ \ \vdots \ \\ \ p \ \end{bmatrix} \in \mathbb R^p, \qquad
   R_p \, := \,  \begin{bmatrix}
1&0&0&  \hdots &0\\
3 &1& 0&  \hdots &0\\
\vdots &\vdots & \ddots  & \ddots & \vdots\\
2 p-3 & 2p-5 &\hdots& 1  & 0\\
2 p -1 & 2p-3 &  \hdots & 3 &1 \\
\end{bmatrix} \in \mathbb R^{p\times p}.
\]
It is noted that for each $i=1,\ldots, n$ and $j=1, \ldots, p$,
\[
 x_{i-1}(k+j) - x_i(k+j) = z_i(k+j) = z_i(k) + \Delta + \frac{\tau^2}{2} \Big[R_p \Big(\abld_{i-1}(\ubld_{i-1})- \abld_i(\ubld_{i}) \Big) \Big]_j
\]
Hence, using the gradient of $q_{i, j}$ and the Jacobian of $\abld_i$ given  by (\ref{eqn:Jacobian_a_i}), the gradients of  $\big( H_i(\ubld_{i-1}, \ubld_i) \big)_j$ with respect to  $\ubld_i$ and $\ubld_{i-1}$ (for $i\ge 2$) are respectively given by
\begin{align*}
\lefteqn{ \nabla_{\ubld_i} \big( H_i(\ubld_{i-1}, \ubld_i) \big)_j  = r_i \cdot \nabla q_{i, j}(\ubld_i) } \\
&  -  \frac{1}{a_{i, \min}} \Bigg[
  v_i(k) - v_{\min} + \tau \Big( \big( S_p \mathbf u_i \big)_j - j\cdot c_{3, i} g - c_{2, i} \sum^{j-1}_{s=0} \big[ v_i(k) + \tau  \big(S_p \mathbf u_i \big)_s \big]^2 \Big) \Bigg] \cdot \nabla q_{i, j}(\ubld_i) \\
  & + \frac{\tau^2}{2} \Big[ (R_p)_{j\bullet} \, \mathbf J \abld_i(\ubld_i) \Big]^T,
\end{align*}
and for $i\ge 2$,
\[
\nabla_{\ubld_{i-1}} \big( H_i(\ubld_{i-1}, \ubld_i) \big)_j =  - \frac{\tau^2}{2} \Big[ (R_p)_{j\bullet} \, \mathbf J \abld_{i-1}(\ubld_{i-1}) \Big]^T.
\]

\gap

\noindent {\bf Approximate objective function}. Consider the decomposition of the (central) objective function given by local objective functions $J_i$'s in (\ref{eqn:J_i_decomposition}). The approximate objective function $J_i$ can be easily obtained by substituting (\ref{eqn:a_i_approx}) into (\ref{eqn:J_i_decomposition}). In what follows, we compute the gradient of $J_i$. It follows from
$
J_1(\ubld_1, \ubld_2) := \frac{1}{2} \begin{bmatrix} \abld^T_1(\ubld_1) & \abld^T_2(\ubld_2) \end{bmatrix} \big(\wh W^1 - \tau^2 \wh \Psi^1 \big)  \begin{bmatrix} \abld_1(\ubld_1) \\ \abld_2(\ubld_2) \end{bmatrix} + c^T_{\Ical_1} \abld_1(\ubld_1)
+ \ \frac{\tau^2}{2} \begin{bmatrix} \ubld^T_{1} & \ubld^T_2 \end{bmatrix} \wh \Psi^1 \begin{bmatrix} \ubld_{1} \\ \ubld_2 \end{bmatrix}$ that
\begin{align*}
 \nabla_{\ubld_1} J_1(\ubld_1, \ubld_2) & = \begin{bmatrix} \big( \mathbf J \abld_1(\ubld_1) \big)^T & 0 \end{bmatrix} \big(\wh W^1 - \tau^2 \wh \Psi^1 \big)  \begin{bmatrix} \abld_1(\ubld_1) \\ \abld_2(\ubld_2) \end{bmatrix} + \big( \mathbf J \abld_1(\ubld_1) \big)^T c_{\Ical_1} + \tau^2\begin{bmatrix} I_p & 0 \end{bmatrix} \wh \Psi^1 \begin{bmatrix} \ubld_{1} \\ \ubld_2 \end{bmatrix}, \\
  \nabla_{\ubld_2} J_1(\ubld_1, \ubld_2) & = \begin{bmatrix}0 &  \big( \mathbf J \abld_2(\ubld_2) \big)^T  \end{bmatrix} \big(\wh W^1 - \tau^2 \wh \Psi^1 \big)  \begin{bmatrix} \abld_1(\ubld_1) \\ \abld_2(\ubld_2) \end{bmatrix} + \tau^2 \begin{bmatrix} 0 & I_p  \end{bmatrix} \wh \Psi^1 \begin{bmatrix} \ubld_{1} \\ \ubld_2 \end{bmatrix}.
\end{align*}
The similar result holds for $\nabla J_{n} (\ubld_{n-1}, \ubld_n)$ and $\nabla J_i (\ubld_{i-1}, \ubld_{i}, \ubld_{i+1})$ for $i=2, \ldots, n-1$.

\gap

%

\noindent {\bf Computation of a feasible $\wh \ubld^0$}: We discuss how to generate and compute a feasible initial guess $\wh \ubld^0$ of the SCP scheme using approximate constraint functions. In view of the approximate speed constraint function in (\ref{eqn:v_i_approx}) and $0\le v_{\min}<v_{\max}$, we see that any $\ubld_i \in \mathcal Y_i$ satisfies:
\begin{align*}
 0\le v_{\min} + c_{3, i} g & \le \ v_i(k) + \tau (S_p \ubld_i)_1 \ \le v_{\max}+c_{2, i} v^2_i(k) + c_{3, i} g := \breve v_{i, 1}, \\
0\le v_{\min} + 2 c_{3, i} g & \le \ v_i(k) + \tau (S_p \ubld_i)_2 \ \le v_{\max}+c_{2, i} [ v^2_i(k) + \breve v^2_{i, 1}] + 2 c_{3, i} g := \breve v_{i, 2},  \\
\vdots &  \qquad \qquad \vdots \qquad \qquad \qquad \qquad \qquad \vdots
\end{align*}
In general,  $0\le v_{\min} + j c_{3, i} g \le v_i(k) + \tau (S_p \ubld_i)_j \le v_{\max}+j c_{3, i}  g + c_{2, i} \sum^{j-1}_{s=0} \breve v^2_{i, s}:=\breve v_{i, j}$ for each $j=1, \ldots, p$, where $\breve v_{i, 0} :=v_i(k)$.
Using the above results, we consider the constraint set
\begin{align*}
 \wt {\mathcal Y}_i & := \left \{ \ubld_i \in \mathbb R^p \, \Big | \,  v_{\min}   \le v_i(k) + \tau \big( (\ubld_i)_1 - c_{3, i} g - \tau c_{2, i} v^2_{i}(k) \big)\le v_{\max}, \ \mbox{ and } \ \forall \ j=2, \ldots, p, \right. \\
& \quad \  v_{\min}  + \tau c_{2, i} \sum^{j-1}_{s=0} \breve v^2_{i, s} \le v_i(k) + \tau \Big( \big( S_p \mathbf u_i \big)_j -  j\cdot c_{3, i} g \Big)
 \le v_{\max}  +\tau c_{2, i} \big [ v^2_i(k) + \sum^{j-1}_{s=1} (v_{\min} + s \cdot c_{3, i} g )^2 \big]  \, \Bigg\},
\end{align*}
which is clearly polyhedral and thus convex. Further, it is easy to see that $\wt {\mathcal Y}_i \subseteq \mathcal Y_i$ for each $i=1, \ldots, n$.

Additionally, define $\grave v_{i, 0} := v_i(k)$ and $\grave v_{i, j}:= v_{\min} + j \cdot c_{3, i } g$ for $j=1, \ldots, p-1$.
In light of the approximating function for $\big(H_i(\ubld_{i-1}, \ubld_i) \big)_j$ given by (\ref{eqn:H_i_j_approx}), we define the following function for each $i=1, \ldots, n$ and $j=1, \ldots, p$,
\begin{align*}
  \big( \wt H_i(\ubld_{i-1}, \ubld_i) \big)_j & := \ L_i + r_i \cdot \Big[
  v_i(k) + \tau \Big( \big( S_p \mathbf u_i \big)_j - j\cdot c_{3, i} g - c_{2, i} \sum^{j-1}_{s=0} \grave v^2_{i, s} \Big) \Big] \notag  \\
  & - \frac{1}{2 a_{i, \min}} \Bigg[
  v_i(k) - v_{\min} + \tau \Big( \big( S_p \mathbf u_i \big)_j - j\cdot c_{3, i} g - c_{2, i} \sum^{j-1}_{s=0} \grave v^2_{i, s} \Big) \Bigg]^2  \\ 
  & - \left\{ z_i(k) + j \tau z'_i(k) + \tau^2 \sum^{j-1}_{s=0} \frac{2(j-s)-1}{2}
   \Big[ \, u_{i-1}(k+s) - u_{i}(k+s)   \right. \notag \\
   & \qquad \left. -  \Big( c_{2,i-1} \big[ v_{i-1}(k) + \tau \big(S_p \ubld_{i-1}\big)_s \big]^2   -  c_{2,i} \grave v^2_{i, s} \Big)  - \big(  c_{3, i-1} -c_{3, i} \big) g  \Big] \right\}. \notag
\end{align*}
Clearly,  $\big( \wt H_i(\ubld_{i-1}, \ubld_i) \big)_j$ is a convex quadratic function in $(\ubld_{i-1}, \ubld_i)$. Besides, for any $\ubld_i \in \wt {\mathcal Y}_i$, it is easy to verify that $\big[ v_i(k) + \tau  \big(S_p \mathbf u_i \big)_s \big]^2 \ge \grave v^2_{i, s}$ for all $s=0, 1, \ldots, p-1$, and for each $j=1, \ldots, p$,
\begin{align*}
0 & \le  v_i(k) - v_{\min} + \tau \Big( \big( S_p \mathbf u_i \big)_j - j\cdot c_{3, i} g - c_{2, i} \sum^{j-1}_{s=0} \big[ v_i(k) + \tau  \big(S_p \mathbf u_i \big)_s \big]^2 \Big) \\
& \le v_i(k) - v_{\min} + \tau \Big( \big( S_p \mathbf u_i \big)_j - j\cdot c_{3, i} g - c_{2, i} \sum^{j-1}_{s=0} \grave v^2_{i, s} \Big).
\end{align*}
 Thus for any $\ubld_{i-1}$ and any $\ubld_i \in \wt {\mathcal Y}_i$ , $\big( \wt H_i(\ubld_{i-1}, \ubld_i) \big)_j \le 0$ implies $\big( H_i(\ubld_{i-1}, \ubld_i) \big)_j \le 0$. Define the convex set
\[
\wt {\mathcal Z}_i :=  \left\{  \, (\ubld_{i-1}, \ubld_i) \in \mathbb R^p \times \mathbb R^p \, \big | \, (\wt H_i( \ubld_{i-1}, \ubld_i ))_s \le 0, \quad \forall \, s=1, \ldots, p \, \right\}.
\]
Thus $\wt {\mathcal Y_i} \cap \wt {\mathcal Z}_i \subseteq {\mathcal Y_i} \cap {\mathcal Z}_i$ for all $i=1, \ldots, n$.  Let $J^0_i$ denote the function $J_i$ when $\varphibf =0$.
Therefore, $\min \sum^n_{i=1} J^0_i(\wh \ubld_i)$ subject to $\wh\ubld_i \in \mathcal X_i \cap \wt {\mathcal Y_i} \cap \wt {\mathcal Z}_i, i=1, \ldots, n$ and $\wh \ubld \in \mathcal A$ is a convex QCQP and leads to a unique solution $\wh \ubld^0_{*}$ that is feasible to the nonconvex program (\ref{eqn:MPC_model_locally_coupled}) with approximate constraints, namely, $\wh \ubld^0_{*} \in \prod^n_{i=1} \mathcal X_i \cap {\mathcal Y_i} \cap  {\mathcal Z}_i$.
 Under the assumption $\bf A.2$, the convex QCQP is feasible for all small $\|\varphibf\|>0$. This implies that the nonconvex program (\ref{eqn:MPC_model_locally_coupled}) with approximate constraints is feasible for all small $\|\varphibf\|$.

%
\section{Control Design and Stability Analysis of Closed Loop Dynamics} \label{sect:control_analysis}

In this section, we discuss the choice of the weight matrices $Q_{z, s}$,  $Q_{z', s}$ and $Q_{w, s}$ to achieve the desired closed loop performance, including stability and traffic transient dynamics. For the similar reasons given in \cite[Section 5]{GShenDu_TRB16}, we focus on the constraint free case.
Recall that $\varphibf:=(\varphibf_d, \varphibf_f)\in \mathbb R^{2n}_+$, where $\varphibf_d:=(c_{2, 1}, \ldots, c_{2, n}) \in \mathbb R^n_+$ and $\varphibf_f:=(c_{3, 1}, \ldots, c_{3, n}) \in \mathbb R^{n}_+$. Further,  $c_{2, 0}=c_{3, 0}=0$ as indicated before.
%

%
%

%
%
%
%
%

%
\subsection{Review of the Closed-loop Stability Analysis under Linear Vehicle Dynamics} \label{susbsect:review_closed_loop_linear_dynamics}

When $\varphibf=0$, the nonlinear vehicle dynamics reduces to the linear vehicle dynamics given by (\ref{eqn:model_longit_double_integrator}), for which the closed loop stability of the MPC based platooning control with a general horizon $p$ has been analyzed in \cite[Section 5]{ShenEswarDu_Arx20}. We present a short review of these stability results as they pave a way for studying closed loop stability under nonlinear vehicle dynamics when $\| \varphibf\|$ is small.

Let $\mathbf w(k):= (w(k), \ldots, w(k+p-1))$. As before, we omit $k$ when $k$ is fixed. It is shown that under the linear vehicle dynamics, the objective function is  equivalently written as \cite[Section 5]{ShenEswarDu_Arx20}
\[
  J(\mathbf w) = \frac{1}{2} \mathbf w^T \mathbf H \mathbf w + \mathbf w^T \left( \mathbf G \begin{bmatrix} z(k) \\ z'(k) \end{bmatrix} - u_0(k) \mathbf g \right) + \wt \gamma,
\]
where $\wt \gamma$ is a constant. Here $\mathbf H \in \mathbb R^{pn\times pn}$ is a symmetric PD matrix given by
\begin{equation} \label{eqn:Hbld_matrix}
    \mathbf H \, = \,  \begin{bmatrix} \breve\Hbld_{1,1} +\tau^2 Q_{w, 1}  & \breve\Hbld_{1, 2} & \breve\Hbld_{1,3} & \cdots & \cdots & \breve\Hbld_{1, p} \\ \breve\Hbld_{2,1} & \breve\Hbld_{2,2} +\tau^2 Q_{w, 2}  & \breve\Hbld_{2, 3} & \cdots & \cdots & \breve\Hbld_{2, p} \\ \cdots &  & \cdots &  & \cdots & \\ \cdots  &  & \cdots &  & \cdots & \\ \breve\Hbld_{p,1} & \breve\Hbld_{p, 2} & \breve\Hbld_{p, 3} & \cdots & \cdots & \breve\Hbld_{p, p} +\tau^2 Q_{w, p}  \end{bmatrix} \in \mathbb R^{np \times np},
\end{equation}
where $\breve\Hbld_{i, j}$'s are diagonal PD matrices given by
\[
 \breve\Hbld_{i, j}  \, := \, \sum^p_{s=\max(i, j)} \Big( \frac{\tau^4}{4}[2(s-i)+1]\cdot [2 (s-j)+1] Q_{z, s} + \tau^2 Q_{z', s} \Big)  \in \mathbb R^{n\times n},
\]
and the matrix $\mathbf G$ and  constant vector $\mathbf g$ are
\begin{equation} \label{eqn:G_bld}
  \mathbf G \, := \, \begin{bmatrix} \mathbf G_{1, 1} &  \mathbf G_{1, 2} \\ \vdots & \vdots \\ \mathbf G_{p, 1} & \mathbf G_{p, 2} \end{bmatrix} \in \mathbb R^{pn \times 2 n}, \qquad
  \mathbf  g \, := \, \tau^2 \begin{bmatrix} Q_{w, 1} \mathbf e_1 \\ \vdots \\ Q_{w, p} \mathbf e_1 \end{bmatrix} \in \mathbb R^{pn}.
\end{equation}
where $\mathbf G_{i, 1}, \mathbf G_{i, 2} \in \mathbb R^{n\times n}$ are given by: for each $i=1, \ldots, p$,
\begin{eqnarray*}
 \mathbf G_{i, 1} & = & \tau^2 \sum^p_{s=i} \frac{ 2(s-i) + 1 }{2} Q_{z, s}, \qquad \quad
 \mathbf G_{i,2} \, = \, \tau^3 \sum^p_{s=i}s  \frac{ 2(s-i) + 1 }{2}  Q_{z, s} + \tau \sum^p_{s=i} Q_{z', s}.
\end{eqnarray*}
%
%
Hence, the optimal solution is  $\mathbf w_*=(w_*(k), w_*(k+1), \ldots, w_*(k+p-1))=-\mathbf H^{-1} \Big(\mathbf G \begin{bmatrix} z(k) \\ z'(k) \end{bmatrix} -  u_0(k) \mathbf g \Big)$, and
$w_*(k)= - \begin{bmatrix} I_n & 0 & \cdots & 0 \end{bmatrix} \mathbf H^{-1} \Big(\mathbf G \begin{bmatrix} z(k) \\ z'(k) \end{bmatrix} -  u_0(k) \mathbf g \Big)$.
Define the matrix $\mathbf K$ and the vector $\mathbf d$ as
\begin{equation} \label{eqn:K_mat_d_vec}
 \mathbf K := - \begin{bmatrix} I_n & 0 & \cdots & 0 \end{bmatrix} \mathbf H^{-1} \mathbf G \in \mathbb R^{n\times 2n}, \qquad \mathbf d:= \begin{bmatrix} I_n & 0 & \cdots & 0 \end{bmatrix} \mathbf H^{-1} \mathbf g \in \mathbb R^{n}.
\end{equation}
The closed loop dynamics under the MPC platooning control becomes
\begin{equation} \label{eqn:A_c_matrix}
   \begin{bmatrix} z(k+1) \\ z'(k+1) \end{bmatrix} = \underbrace{\left\{ \begin{bmatrix} I_n & \tau I_n \\ 0 & I_n \end{bmatrix} + \begin{bmatrix} \frac{\tau^2}{2} I_n \\ \tau I_n \end{bmatrix} \mathbf  K \right \} }_{A_{\mbox{c}}}  \begin{bmatrix} z(k) \\ z'(k) \end{bmatrix} +  \begin{bmatrix} \frac{\tau^2}{2} I_n \\ \tau I_n \end{bmatrix}  u_0(k) \cdot \mathbf d,
\end{equation}
where $A_{\mbox{c}}$ represents the closed loop dynamics matrix for the linear vehicle dynamics.

By the assumption $\bf A.3$, the diagonal matrices $Q_{z, s}, Q_{z', s}$ and $Q_{w, s}$ can be written as $Q_{z, s} = \mbox{diag}(\alphabf^s)$, $Q_{z', s} = \mbox{diag}(\betabf^s)$, and $Q_{w, s} = \mbox{diag}(\zetabf^s)$, where $\alphabf^s, \betabf^s \in \mathbb R^n_{+}$ and $\zetabf^s \in \mathbb R^n_{++}$ for all $s=1, \ldots, p$ with  $\alphabf^s =(\alpha^s_i)^n_{i=1}$, $\betabf^s=(\beta^s_i)^n_{i=1}$, and $\zetabf^s=(\zeta^s_i)^n_{i=1}$.
It is shown in \cite[Section 5]{ShenEswarDu_Arx20} that for each $p=1, 2, 3$, $A_{\mbox{c}}$ is Schur stable for any $\tau>0$, $(\alpha^1_i, \beta^1_i, \zeta^1_i) >0$,  $0\ne (\alpha^j_i, \beta^j_i, \zeta^j_i) \ge 0$ for each $j=2, \ldots, p$ and $i=1, \ldots, n$.
For $p>4$, numerical tests show that the same result is expected to hold. Throughout the rest of this section, we assume that for a given $p$, the weight matrices $Q_{z, s}, Q_{z', s}$ and $Q_{w, s}$ satisfying the assumption $\bf A.3$ are chosen such that $A_{\mbox{c}}$ is Schur stable.

%
%
%
%

%
\subsection{Reformulation of the Closed Loop Dynamics as a Tracking System}

Consider the nonlinear vehicle dynamics (\ref{eqn:d_model_longit}).  It follows from the definitions of $z(k), z'(k)$ and $w(k)$ that for $i=1, \ldots, n$,
%
\begin{subequations} \label{eqn:z_z'_closed_loop}
\begin{eqnarray}
  z_i(k+1) & = & z_i(k) + \tau z'_i(k)+ \frac{\tau^2}{2} \Big( w_i(k) - [c_{2, i-1} v^2_{i-1}(k) - c_{2, i} v^2_i(k) ] - [c_{3, i-1} -c_{3, i}] g\Big), \label{eqn:z_z'_closed_loop_line1}\\ 
  z'_i(k+1)  & = & z'_i(k) + \tau\Big( w_i(k) - [c_{2, i-1} v^2_{i-1}(k) - c_{2, i} v^2_i(k) ] - [c_{3, i-1} -c_{3, i}] g \Big). \label{eqn:z_z'_closed_loop_line2}
\end{eqnarray}
\end{subequations}

For given $(v_0(k), u_0(k)), k \in \mathbb Z_+$,  the equilibrium of the above discrete-time system is  $(z_e, z'_e)=(0, 0)$ such that $v_{e, i}(k)=v_0(k)$ for all $i=1, \ldots, n$.  Hence, the corresponding
\[
w_{e, i}(k)= [c_{2, i-1}-c_{2, i}] v^2_0(k)+[c_{3, i-1} -c_{3, i}] g, \qquad \forall \ i=1, \ldots, n.
\]
Let $w_e(k):=(w_{e, 1}(k), \ldots, w_{e, n}(k))^T$.
By shifting $w(k)$ from the time-varying $w_e(k)$, we define $\wh w(k) := w(k) - w_e(k)$.  Hence, this yields the following equations: for $i=1, \ldots, n$,
 \begin{eqnarray*} 
  z_i(k+1) \, = \, z_i(k) + \tau z'_i(k)+ \frac{\tau^2}{2} \Big( \wh w_i(k) + r_i(k) \Big), \quad 
  z'_i(k+1)  \, = \, z'_i(k) + \tau\Big( \wh w_i(k) + r_i(k) \Big),
\end{eqnarray*}
where   for each $i=1, \ldots, n$,
\[
  r_i(k) := w_{e, i}(k) -  [c_{2, i-1} v^2_{i-1}(k) - c_{2, i} v^2_i(k) ] - [c_{3, i-1} -c_{3, i}] g = c_{2, i-1} \big(  v^2_0(k) - v^2_{i-1}(k)\big) - c_{2, i} \big( v^2_0(k) - v^2_{i}(k) \big).
\]
In light of $v(k) = - S_n z'(k) + v_0(k) \cdot \mathbf 1$,  we have $v^2_0(k) - v^2_i(k)= (S_n z'(k))_i \cdot [ 2 v_0(k) - (S_n z'(k))_i ]$ for each $i$. We thus define the vector-valued smooth function $h:\mathbb R^n \times \mathbb R \rightarrow \mathbb R^n$ as $h(z', v_0) :=(h_1(z', v_0), \ldots, h_n(z', v_0))^T$, where $h_1(z', v_0):= c_{2, 1}(S_n z')_1 \big[ (S_n z')_1 - 2 v_0 \big]$, and for $i=2, \ldots, n$,
\[
 h_i(z', v_0):=  c_{2, i-1} (S_n z')_{i-1} \big[  2 v_0 - (S_n z')_{i-1} \big] - c_{2, i} (S_n z')_i\big[ 2 v_0 - (S_n z')_i  \big].
\]
Clearly, $h(0, v_0)=0$ for any $v_0$. Further, we decompose $h$ as the sum of the following two functions:
\begin{equation} \label{eqn:matrix_D}
 h(z', v_0) \, = \, v_0 \cdot \Bigg( \underbrace{2 \begin{bmatrix} -c_{2, 1} & & & \\ &  c_{2,1}-c_{2, 2}& &  \\ & & \ddots & \\ & & &  c_{2,n-1}-c_{2, n} \end{bmatrix} S_n}_{:=D(\varphibf_d)} \Bigg) \cdot  z' \, + \, \wt h(z'),
\end{equation}
where the vector-valued function $\wt h:=(\wt h_1, \ldots, \wt h_n)^T:\mathbb R^n \rightarrow \mathbb R^n$ is given by:
\[
    \wt h_i(z') \, := \, c_{2, i} [(S_n z' )_i]^2 - c_{2, i-1}[(S_n z' )_{i-1}]^2, \qquad \forall \ i=1, \ldots, n.
\]
or equivalently
\begin{equation} \label{eqn:wt_h_func}
    \wt h(z') \, := \, \begin{bmatrix} c_{2,1} & & &  \\ -c_{2,1} & c_{2, 2} & &  \\ & \ddots & \ddots &  \\ & &  -c_{2, n-1} & c_{2, n} \end{bmatrix} \Big[ \big(S_n z' \big) \circ  \big(S_n z' \big) \Big] \, = \, \underbrace{S^{-1}_n \mbox{diag}(\varphibf_d)}_{:=\wt D(\varphibf_d)} \Big[ \big(S_n z' \big) \circ  \big(S_n z' \big) \Big].
\end{equation}

%
Note that the elements of $D$ and $\wt D$ are linear in $\varphibf_d$ such that $D(\varphibf_d) =\wt D(\varphibf_d)=0$  when $\varphibf_d= 0$.
%
%
%
Using this notation, the nonlinear vehicle dynamics (\ref{eqn:d_model_longit})  is described by the following discrete-time system:
\begin{eqnarray*} 
\begin{bmatrix} z(k+1) \\ z'(k+1) \end{bmatrix} & = & \begin{bmatrix} I_n & \tau I_n \\ 0 & I_n \end{bmatrix} \begin{bmatrix} z(k) \\ z'(k) \end{bmatrix} +  \begin{bmatrix} \frac{\tau^2}{2} I_n \\ \tau I_n \end{bmatrix} \Big( \wh w(k) +  h(z'(k), v_0(k) ) \Big) \\
& = &  \left\{ \begin{bmatrix} I_n & \tau I_n \\ 0 &  I_n \end{bmatrix} +  v_0(k) \cdot \begin{bmatrix} \frac{\tau^2}{2} I_n \\ \tau I_n \end{bmatrix}  \begin{bmatrix} 0 & D(\varphibf_d) \end{bmatrix}   \right\} \begin{bmatrix} z(k) \\ z'(k) \end{bmatrix} +  \begin{bmatrix} \frac{\tau^2}{2} I_n \\ \tau I_n \end{bmatrix} \Big( \wh w(k) +  \wt h(z'(k)) \Big). \notag
\end{eqnarray*}
By slightly abusing notation, we also write the function $\wt h$ as $\wt h_{\varphibf_d}(z')$ to emphasize its dependence on $\varphibf_d$. Noting that $\wt h$ is linear in $\varphibf_d$ for any fixed $z'$, we see that $\wt h_0(z')\equiv 0$ for any given $z'\in \mathbb R^n$.

%
%
Define the following matrices:
\begin{equation} \label{eqn:closed_loop_matrices}
  A := \begin{bmatrix} I_n & \tau I_n \\ 0 &  I_n \end{bmatrix},
   \quad B:=\begin{bmatrix} \frac{\tau^2}{2} I_n \\ \tau I_n \end{bmatrix},
  \quad
   \Delta A(\varphibf_d) := B  \begin{bmatrix} 0 & D(\varphibf_d) \end{bmatrix}, \quad \wh A(k) := A + v_0(k) \cdot \Delta A(\varphibf_d).
\end{equation}
As before, we often write $\wh A(k)$ as $\wh A(v_0(k), \varphibf_d)$ to stress its dependence on $v_0(k)$ and $\varphibf_d$. Let $\zbld :=(z, z') \in \mathbb R^n \times \mathbb R^n$.
%
%
We obtain the following non-autonomous nonlinear dynamical system:
\begin{equation} \label{eqn:nonlin_closed_loop_02}
\zbld(k+1) \, = \, \wh A(k) \zbld(k) + B  \Big( \wh w_*(k) +  \wt h_{\varphibf_d}(z'(k)) \Big), \qquad \forall \ k \in \mathbb Z_+,
\end{equation}
where $\wh w_*(k)$ is an optimal solution to the unconstrained MPC optimization problem (\ref{eq:MPC_opt_model}) which implicitly depends on $\zbld(k), v_0(k)$ and $u_0(k)$.
For any fixed $\varphibf_d$, the closed loop system given by (\ref{eqn:nonlin_closed_loop_02}) yields a non-autonomous nonlinear dynamical system, since $\wt h$ is nonlinear in $z'$ and $v_0(k)$ is time varying. In what follows, we further discuss the non-autonomous system (\ref{eqn:nonlin_closed_loop_02}) for different MPC horizon $p$.


\gap

\noindent{\bf Case (i): $p=1$}.
In this case,  the closed-form expression of $\wh w_*(k)$ is derived below.
Letting $\wt w(k):= w(k) - u_0(k) \cdot \mathbf e_1 = \wh w(k) + d(k)$, where $d(k):=w_e(k) - u_0(k) \cdot \mathbf e_1$, the unconstrained MPC becomes
\[
   \min J(\wh w(k)) := \frac{1}{2} \Big\{ z^T(k+1) Q_z z(k+1) + (z'(k+1))^T Q_{z'} z'(k+1) + \tau^2 [\wh w(k) + d(k)]^T Q_{w}  [\wh w(k) + d(k)]^T \Big\}
\]
subject to $z(k+1) = z(k) + \tau z'(k)+ \frac{\tau^2}{2} \big[ \wh w(k) + h(z'(k), v_0(k))\big]$, and $z'(k+1)=z'(k) + \tau
\big[ \wh w(k) + h(z'(k), v_0(k))\big]$, where $Q_{z}:=Q_{z, 1}$, $Q_{z'}:=Q_{z', 1}$, and $Q_w :=Q_{w, 1}$. Hence,
\begin{eqnarray*}
 \nabla J(\wh w(k)) &  = &  \Big( \frac{\tau^4}{4} Q_z + \tau^2 Q_{z'} + \tau^2 Q_w \Big) \wh w(k) + \frac{\tau^2}{2} Q_z z(k) +  \Big( \frac{\tau^3}{2} Q_z + \tau Q_{z'}  \Big) z'(k) \\
   & & + \Big( \frac{\tau^4}{4} Q_z + \tau^2 Q_{z'}  \Big) h(z'(k), v_0(k)) + \tau^2 Q_w d(k).
\end{eqnarray*}
Define the matrix
\begin{equation} \label{eqn:W_matrix01}
  \wh W \, := \,  \left[ \frac{ \tau^2 Q_z }{4} +  Q_{z'}  + Q_w \right]^{-1}.
\end{equation}
Using this matrix, we obtain the closed form expression for the optimal solution $\wh w_*(k)$ as
\[
 \wh w_*(k) = - \wh W \cdot \Big[ \frac{Q_z}{2}  z(k) +  \Big( \frac{\tau Q_z}{2}  + \frac{Q_{z'}}{\tau}   \Big) z'(k) + \Big( \frac{\tau^2}{4} Q_z +  Q_{z'}  \Big) h(z'(k), v_0(k)) +  Q_w d(k) \Big].
\]
Substituting $\wh w_*(k)$ into (\ref{eqn:nonlin_closed_loop_02}), using the following matrix $A_{\mbox{c}}$ derived in \cite[Section 5]{GShenDu_TRB16} (which agrees with  the closed loop dynamics matrix in (\ref{eqn:A_c_matrix}) when $p=1$)
\begin{eqnarray} \label{eqn:linear_close_loop_p=1}
 A_{\mbox{c}} \, := \,  \begin{bmatrix} I_n - \frac{\tau^2}{4} \wh W Q_z  & \tau  I_n -  \wh W \Big( \frac{\tau^3}{4} Q_z + \frac{\tau}{2} Q_{z'} \Big)  \\ -\frac{\tau}{2} \wh W Q_z & I_n - \wh W \Big( \frac{\tau^2}{2} Q_z +  Q_{z'} \Big) \end{bmatrix}
\end{eqnarray}
and in view of $h(z', v_0) = v_0 D(\varphibf_d) z' + \wt h_{\varphibf_d}(z')$, the closed loop dynamics is characterized by:
\begin{eqnarray*}
%
\zbld(k+1) & = & A_{\mbox{c}} \, \zbld(k) 
+  B   \Big\{ - \wh W \Big( \frac{\tau^2}{4} Q_z +  Q_{z'}  \Big) h(z'(k), v_0(k)) -  \wh W Q_w d(k) + h(z'(k), v_0(k)) \Big\} \\
 & = & \Big(A_{\mbox{c}} +  v_0(k) \cdot \Delta {\bar A}(\varphibf_d) \Big) \zbld(k) -  B \wh W Q_w d(k) + \breve B \wt h_{\varphibf_d}(z'(k)),
\end{eqnarray*}
where the matrices $\Delta \bar A$ and $\breve B$ are given by
\begin{equation}
 %
 \breve B \, := \,  B \Big[ I_n- \wh W \big( \frac{\tau^2}{4} Q_z +  Q_{z'}  \big) \Big], \qquad
 \Delta \bar A(\varphibf_d) \, := \, \breve B \begin{bmatrix} 0 &  D(\varphibf_d) \end{bmatrix}.  \label{eqn:matrix_Delta_A_B}
%
\end{equation}
By using $d(k) = w_e(k) - u_0(k) \cdot \mathbf e_1$, the closed loop dynamics for $p=1$ becomes:
\begin{equation} \label{eqn:closed_loop_p=1}
  \zbld(k+1) \, = \, \Big(A_{\mbox{c}} +  v_0(k) \cdot \Delta {\bar A}(\varphibf_d) \Big) \zbld(k) + B \Big( u_0(k)  \mathbf d - \wh W Q_w w_e(k) \Big) + \breve B \wt h_{\varphibf_d} (z'(k)),
\end{equation}
where $\mathbf d = \wh W Q_w \mathbf e_1$ that agrees with what is given in (\ref{eqn:K_mat_d_vec}) for $p=1$, and $w_e(k)$ depends on $v_0(k)$ and $\varphibf$. Further, there exists a positive constant $\wt \varkappa$ such that $\| w_e(k) \| \le \wt \varkappa \cdot \| \varphibf\|$ for any $v_0(k) \in [v_{\min}, v_{\max}]$.

%
%
%

%
%

\gap

\noindent{\bf Case (ii): $p>1$}.
In this case, recall that for any fixed $k\in \mathbb Z_+$, $u_0(k+s)=u_0(k)$ for all $s=1, \ldots, p-1$ in the MPC model. Hence, $v_0(k+s)=v_0(k) + \tau s u_0(k)$ for all $s=1, \ldots, p-1$.
%
%
%
%
Define $\wh A(k+s) := A+ v_0(k+s) \cdot \Delta A(\varphibf_d)$ for all $s=0, 1, \ldots, p-1$.
Given $\wh A(k+s)$ with $s=0, \ldots, p-1$, define the state transition matrix as the following matrix product for any $s, s' \in \{0, \ldots, p\}$ with $s \le s'$,
\[
  \Phi_{\wh A}(k+s, k+s) \, := \, I; \qquad  \Phi_{\wh A}(k+s', k+s) \, := \, \wh A(k+s'-1)\times \cdots \times \wh A(k+s), \ \ \ \forall \, s'>s.
\]
Based upon the above notation, we obtain, for any fixed $k\in \mathbb Z_+$ and $s=1, \ldots, p$,
\begin{eqnarray}
 \lefteqn{ \zbld(k+s) \, = \, \Phi_{\wh A} (k+s, k) \zbld(k) + \sum^{s-1}_{i=0} \Phi_{\wh A}(k+s, k+i+1)  B \Big[ \wh w(k+i) + \wt h_{\varphibf_d}(z'(k+i)) \Big] } \label{eqn:zbld_z_k+s} \\
  & = & \Phi_{\wh A} (k+s, k)  \zbld(k) + \sum^{s-1}_{i=0} \Phi_{\wh A}(k+s, k+i+1) B \wh w(k+i)
   + \sum^{s-1}_{i=0} \Phi_{\wh A}(k+s, k+i+1) B \wt h_{\varphibf_d}(z'(k+i)). \notag
\end{eqnarray}

In light of (\ref{eqn:wt_h_func}) and (\ref{eqn:zbld_z_k+s}), the following lemma can be established via an induction argument on $s$ and straightforward calculations; its proof is omitted.

\begin{lemma} \label{lem:wt_h}
Fix an arbitrary $k \in \mathbb Z_+$. For each  $s=1, \ldots, p$, $\wt h_{\varphibf_d}(z'(k+s))$ is a vector-valued function
%
%
whose each entry is a multivariate polynomial in $(z'(k), u_0(k), v_0(k), \wh w(k), \ldots, \wh w(k+s-1) )$ and $\varphibf_d$.
\end{lemma}

%
%

Further,  in view of $u_0(k+s) =u_0(k)$ for any $s \ge 0$ and a fixed $k \in \mathbb Z_+$, we have for each $s=0, \ldots, p-1$,
\[
  \wt w(k+s) = w(k+s) - u_0(k+s)\mathbf e_1 = \wh w(k+s) + w_e(k+s) - u_0(k) \mathbf e_1 = \wh w(k+s) + d(k+s),
\]
where $d(k+s):= w_e(k+s) - u_0(k) \mathbf e_1$. Here we recall that for each $s=0, \ldots, p-1$,
\[
w_{e, i}(k+s)= [c_{2, i-1}-c_{2, i}] v^2_0(k+s)+[c_{3, i-1} -c_{3, i}] g, \qquad \forall \ i=1, \ldots, n,
\]
where $v_0(k+s) = v_0(k) + s\tau u_0(k)$.  Note that $w_e(k)$ depends on $\varphibf$ linearly.

Consider the unconstrained MPC model. Define the following augmented matrices and vector:
\[
   \ol Q_{\zbld, s} :=\begin{bmatrix} Q_{z, s} & \\ & Q_{z', s} \end{bmatrix}, \ \ s=1,\ldots, p; \quad
   \ol Q_w := \begin{bmatrix} Q_{w, 1} & & & \\ &  \ddots & \\ & & Q_{w, p} \end{bmatrix}; \quad \wt \dbld(k) := \begin{bmatrix} d(k)\\ \vdots \\ d(k+p-1) \end{bmatrix}.
\]
For any fixed $k \in \mathbb Z_+$, the objective function in the MPC model is written as
\begin{eqnarray*}
  J( \underbrace{\wh w(k), \ldots, \wh w(k+p-1)}_{:=\wh \wbld(k)} ) \, = \, \frac{1}{2} \Big(
    \sum^p_{s=1}  \zbld(k+s)^T \ol Q_{\zbld, s} \zbld(k+s) \Big) + \frac{\tau^2}{2} [\wh\wbld(k) + \wt\dbld(k)]^T \ol Q_{w}    [\wh \wbld(k) + \wt\dbld(k)].
\end{eqnarray*}
%
%

Substituting the expression for $\zbld(k+s)$ given by (\ref{eqn:zbld_z_k+s}) into the objective function  $J$, we obtain the objective function written as $J(\wh \wbld)$ for a fixed $k$. It follows from the previous development and Lemma~\ref{lem:wt_h} that $J$ is a polynomial function in $(\wh\wbld, \zbld(k), v_0(k), u_0(k), \varphibf)$.
Moreover, the Hessian of the objective function $J$ with respect to $\wh \wbld$ is given by
\[
   H J(\wh \wbld) \, =  \,\Bigg[ \, \frac{\partial J^2(\wh \wbld)}{\partial \wh \wbld_i \partial \wh \wbld_j} \, \Bigg]_{i, j} \, := \, \wh \Hbld(\wh \wbld, \zbld(k), v_0(k), u_0(k), \varphibf).
\]
%
%
%
When $k$ is fixed, we write this Hessian as $\wh \Hbld(\wh\wbld, \zbld, v_0, u_0, \varphibf)$ to emphasize its dependence on these variables. Clearly, $\wh \Hbld$ is an analytic, thus smooth, function, and for any $(\wh\wbld, \zbld, v_0, u_0)$, $\wh \Hbld(\wh\wbld, \zbld, v_0, u_0, \varphibf)|_{\varphibf=0} =\Hbld$, where $\Hbld$ is the constant PD matrix given by (\ref{eqn:Hbld_matrix}).
Moreover, when $\varphibf=0$, the objective function $J$ reduces to the one for the linear vehicle dynamics whose corresponding optimal solution is given in Section \ref{susbsect:review_closed_loop_linear_dynamics} as
\[
\wh \wbld_*(\zbld, v_0, u_0, \varphibf)|_{\varphibf=0} \, = \, -\mathbf H^{-1} \big(\mathbf G \cdot \zbld -  u_0 \cdot \mathbf g \big).
\]
However, unlike Case (i) where $p=1$, the closed form expression of a critical point or a local minimizer $\wh \wbld_*$ is unavailable  for $p>1$, and some implicit function theorems are needed. Instead of applying the classical local implicit function theorem, we consider results on non-local (or global) implicit functions to express $\wh \wbld_*$ in term of $\zbld, v_0, u_0$ and $\varphibf$, since the variables $\zbld, v_0, u_0$ can be non-local.

\begin{proposition} \label{prop:perturbed_H_uniform}
 Let $\mathcal U_{\zbld}$ be a bounded set in $\mathbb R^{2n}$, $\mathcal U_{0}$ be a bounded set containing $[a_{0,\min}, a_{0,\max}]$, and $\mathcal V_0$ be a bounded set containing $[v_{\min}, v_{\max}]$.
 Let $\mathcal U_{\wh \wbld}$ be a bounded set in $\mathbb R^{n p}$ containing all $\wh \wbld_*(\zbld, v_0, u_0, 0)$ for all $\zbld \in \mathcal U_{\zbld}$, $v_0\in\Vcal_0$, and $u_0 \in \mathcal U_{0}$.
%
%
  Then for any constant $\wt \lambda$ with $0 < \wt \lambda < \lambda_{\min}(\Hbld)$,
  there exists a positive constant $\mu_1>0$ such that    $\wh \Hbld(\wh \wbld, \zbld, v_0, u_0, \varphibf)$ is PD and $\lambda_{\min}(\wh \Hbld(\wh \wbld, \zbld, v_0, u_0, \varphibf)) \ge \wt \lambda$  for all $(\wh \wbld, \zbld, v_0, u_0, \varphibf) \in \Ucal_{\wh \wbld} \times \Ucal_{\zbld} \times \mathcal V_0 \times \mathcal U_{0} \times \mathcal B_{\infty}(0, \mu_1)$, where $\mathcal B_\infty(0, \mu_1):=\{ \varphibf \, | \, \| \varphibf \|_\infty < \mu_1 \}$.
\end{proposition}

\begin{proof}
As indicated before, each entry $\wh \Hbld_{i, j}(\cdot) $ is a real-valued polynomial function in $(\wh\wbld, \zbld, v_0, u_0, \varphibf)$ and is thus globally smooth. Let $\mu>0$ be any given positive constant. Define the following constant
\[
   M \, := \,  \max_{i, j\in\{1, \ldots, np\}} \Big( \sup_{ (\wh \wbld, \zbld, v_0, u_0, \varphibf) \in \Ucal_{\wh \wbld} \times \Ucal_{\zbld} \times \mathcal V_0 \times \mathcal U_{0} \times  B_{\infty}(0, \mu) } \big\| \nabla_{\varphibf}  \wh \Hbld_{i, j}( \wh \wbld, \zbld, v_0, u_0, \varphibf ) \big\| \Big).
\]
Since $\Ucal_{\wh \wbld} \times \Ucal_{\zbld} \times \mathcal V_0 \times \mathcal U_{0} \times  B_{\infty}(0, \mu)$ is bounded, we have $0 < M <\infty$.
Consider arbitrary indices $i, j \in \{1, \ldots, np\}$.
%
%
%
For any fixed $(\wh \wbld, \zbld, v_0, u_0) \in \Ucal_{\wh \wbld} \times \Ucal_{\zbld} \times \mathcal V_0\times \mathcal U_0$, we deduce via the Mean-value Theorem that for any $\varphibf \in \mathcal B_{\infty}(0, \mu)$, there exists $\varphibf'$ in the line segment joining $0$ and $\varphibf$ such that
\[
   \big \| \wh \Hbld_{i, j}( \wh \wbld, \zbld, v_0, u_0, \varphibf) - \wh \Hbld_{i, j}( \wh \wbld, \zbld, v_0, u_0, 0) \big \| \le \big\| \nabla_{\varphibf}  \wh \Hbld_{i, j}( \wh \wbld, \zbld, v_0, u_0, \varphibf' ) \big\| \cdot \|\varphibf\| \le M \cdot  \|\varphibf\|,
\]
and $\varphibf' \in B_{\infty}(0, \mu)$.
Hence, there exists a positive constant $M'$ such that
\[
\big\| \wh\Hbld(\wh\wbld, \zbld, v_0, u_0, \varphibf) - \Hbld \big\| = \big \| \wh\Hbld(\wh\wbld, \zbld, v_0, u_0, \varphibf) - \wh \Hbld(\wh\wbld, \zbld, v_0, u_0, 0) \big\| \le M' \cdot \| \varphibf\|
\]
for all  $(\wh \wbld, \zbld, v_0, u_0, \varphibf) \in \Ucal_{\wh\wbld} \times \Ucal_{\zbld} \times \mathcal V_0 \times \mathcal U_{0} \times \mathcal B_{\infty}(0, \mu)$. Since $\Hbld$ is a constant PD matrix,  it is easy to see that the desired result holds by choosing a small positive constant $\mu_1$ with $\mu_1\le \mu$.
\end{proof}

To obtain an implicit form of $\wh\wbld_*(k)$ in terms of $\zbld(k), v_0(k), u_0(k)$ from the MPC optimization problem for small $\|\varphibf\|$, we shall exploit the following global implicit function theorems \cite{Ichiraku_TCS85, Sandberg_TCS81}.

\begin{theorem} \cite[Theorem 2]{Ichiraku_TCS85} \label{thm:global_implicit_function_thm03}
Consider the sets $\Ucal \subseteq \mathbb R^n $ and $\Vcal \subseteq \mathbb R^m$, where $\Ucal$ is connected ($\Ucal$ and $\Vcal$ are not necessarily open).
Let $f:\Ucal' \times \Vcal' \rightarrow \mathbb R^m$ be a $C^r$-function with $r \ge 1$,
  where $\Ucal' \subseteq \mathbb R^n $ and $\Vcal' \subseteq \mathbb R^m$ are open sets containing $\Ucal$ and $\Vcal$ respectively. Further, suppose the following hold:
 \begin{itemize}
   \item [(i)] For some $x_*\in \Ucal$, there exists exactly one $y_* \in \Vcal$ such that $f(x_*, y_*)=0$;
   \item [(ii)] For any $(x, y)\in \mathcal G'_f:=\{ (x, y) \in \Ucal' \times \Vcal': f(x, y)=0 \}$, $D_y f(x, y)$ is invertible;
  %
   \item [(iii)] For any sequence $\big( (x_k, y_k) \big) \in \mathcal G_f:=\{ (x, y) \in \Ucal \times \Vcal: f(x, y)=0 \}$ with $(x_k) \rightarrow x_*$, there exists a subsequence $(y_{k'})$ of $(y_k)$ such that $(y_{k'})$ converges to a point in $\Vcal$.
  %
 \end{itemize}
 Then there exists a unique $C^r$-function $g:\Ucal \rightarrow \Vcal$ such that $f(x, g(x))=0, \forall \, x \in \Ucal$.
\end{theorem}

An easily verified condition in replace of condition (iii) in the above theorem is given by the following theorem; its proof resembles that of \cite[Theorem 5]{Ichiraku_TCS85}, which exploits the covering map argument.

\begin{theorem} \label{thm:global_implicit_function_thm}
 Let $\Ucal \subseteq \mathbb R^n$ be a connected set, and $\Vcal \subseteq \mathbb R^m$ be a closed set. Let $f:\Ucal' \times \Vcal' \rightarrow \mathbb R^m$ be a $C^r$-function with $r \ge 1$,
  where $\Ucal' \subseteq \mathbb R^n $ and $\Vcal' \subseteq \mathbb R^m$ are open sets containing $\Ucal$ and $\Vcal$ respectively.
 Suppose  the following hold:
 \begin{itemize}
   \item [(i)] For some $x_*\in \Ucal$, there exists exactly one $y_* \in \Vcal$ such that $f(x_*, y_*)=0$;
   \item [(ii)] For any $(x, y)\in \mathcal G'_f:=\{ (x, y) \in \Ucal' \times \Vcal': f(x, y)=0 \}$, $D_y f(x, y)$ is invertible;
   \item [(iii)] There is a positive constant $\rho$ such that $\| (D_y f(x, y))^{-1} \| \cdot \| D_x f(x, y) \| \le \rho$ for all $(x, y) \in \mathcal G'_f$.
 \end{itemize}
 Then there exists a unique $C^r$ function $g:\Ucal \rightarrow \Vcal$ such that $f(x, g(x))=0, \forall \, x \in \Ucal$.
\end{theorem}

\begin{proof}
It suffices to show that condition (iii) of Theorem~\ref{thm:global_implicit_function_thm03} holds under conditions (ii) and (iii).  To show this result, let
 $\big( (x_k, y_k) \big)$ be an arbitrary sequence in $\mathcal G_f:=\{ (x, y) \in \Ucal \times \Vcal: f(x, y)=0 \}$ such that $(x_k) \rightarrow x_*$. Define $\Scal:=\{ x_k \, : \, k \in \mathbb N \} \cup \{ x_* \}$. Hence, $\Scal$ is a compact set in $\Ucal$ and thus is bounded.  It follows from the assumption (iii) and a similar argument for \cite[Theorem 5]{Ichiraku_TCS85} that the set  $\Wcal':=\{ y \in \Vcal' \, : \, f(x, y)=0, x \in \mathcal S \}$ is bounded. Note that $\Wcal:=\{ y \in \Vcal \, : \, f(x, y)=0, x \in \mathcal S \} \subseteq \Wcal'$ is also bounded. Hence, $(y_k)$ attains a convergent subsequence $(y_{k'})$ whose limit is denoted by $y_*$. Since $\Vcal$ is closed and $y_{k'} \in \Vcal$, we have $y_* \in \Vcal$. This yields the desired result.
\end{proof}

%
%

Using the above theorem, we establish a result on global implication function for $\wh \wbld_*$ as follows.

\begin{proposition} \label{prop:implicit_function_wh_wbld}
  Let $\mathcal U_{\zbld}$ be a bounded open convex  set in $\mathbb R^{2n}$, let $\mathcal U_{0}$ be a bounded open convex set containing $[a_{0,\min}, a_{0,\max}]$, and let $\mathcal V_0$ be a bounded open convex set containing $[v_{\min}, v_{\max}]$.
 %
  %
  Let $\mathcal U_{\wh \wbld}$ be a compact set in $\mathbb R^{n p}$ containing all $\wh \wbld_*(\zbld, v_0, u_0, 0)$ for all $\zbld \in \mathcal U_{\zbld}$, $v_0\in\Vcal_0$, and $u_0 \in \mathcal U_{0}$.  Then there exist a positive constant $\mu_2>0$ and a unique smooth function $\mathbf h: \Ucal_{\zbld} \times \Vcal_0 \times \Ucal_{0} \times \mathcal B_\infty(0, \mu_2) \rightarrow \mathcal U_{\wh \wbld}$ such that $\wh \wbld_* = \mathbf h(\zbld, v_0, u_0, \varphibf)$ for all $(\zbld, v_0, u_0, \varphibf) \in \Ucal_{\zbld} \times \Vcal_0 \times \Ucal_{0} \times \mathcal B_\infty(0, \mu_2)$.
  %
\end{proposition}

\begin{proof}
 For any given $(\zbld, v_0, u_0, \varphibf)$, a local minimizer $\wh \wbld_*$ satisfies $\nabla J_{\wh \wbld}( \wh \wbld_*, \zbld, v_0, u_0, \varphibf)=0$.
%
%
 Define the function $\mathbf f( \wh \wbld, \zbld, v_0, u_0, \varphibf) := \nabla J_{\wh \wbld}( \wh \wbld, \zbld, v_0, u_0, \varphibf)$. Hence, $\mathbf f(\cdot)$ is a globally smooth function. In what follows, we verify the three conditions in Theorem~\ref{thm:global_implicit_function_thm} by recognizing $x=(\zbld, v_0, u_0, \varphibf)$,  $y=\wh \wbld$, $\Ucal = \Ucal_{\wh\wbld}$ and $\Vcal=\Ucal_{\zbld} \times \Vcal_0 \times \Ucal_{0} \times \mathcal B_\infty(0, \mu_2)$ for some suitable $\mu_2>0$.
 First, note that  for any given $(\zbld, v_0, u_0, \varphibf)$ with $\varphibf=0$, $\wh \wbld_*= -\mathbf H^{-1} \big(\mathbf G \cdot \zbld -  u_0 \mathbf g \big)$ is the unique solution to $\mathbf f( \wh \wbld_*, \zbld, v_0, u_0, 0)=0$. Secondly, letting $\Ucal'_{\wh \wbld} $ be a bounded open set containing $\Ucal_{\wh \wbld}$,
 it follows from
 Proposition~\ref{prop:perturbed_H_uniform} that there exist positive constants $\mu_2$ and $\wt \lambda$ such that $D_{\wh\wbld} \mathbf f(\wh \wbld, \zbld, v_0, u_0, \varphibf ) = \wh \Hbld(\wh \wbld, \zbld, v_0, u_0, \varphibf)$ is PD with $\lambda_{\min}\big( \Hbld(\wh \wbld, \zbld, v_0, u_0, \varphibf) \big) \ge \wt \lambda$
 for all $(\wh \wbld, \zbld, v_0, u_0, \varphibf) \in \Ucal'_{\wh \wbld} \times \Ucal_{\zbld} \times \mathcal V_0 \times \mathcal U_{0} \times \mathcal B_{\infty}(0, \mu_2)$. This
  implies that $D_{\wh\wbld} \mathbf f(\wh \wbld, \zbld, v_0, u_0, \varphibf )$ is invertible and
  $\| D^{-1}_{\wh\wbld} \mathbf f ( \wh \wbld, \zbld, v_0, u_0, \varphibf) \|_2 \le 1/\wt \lambda$ for all $(\wh\wbld, \zbld, v_0, u_0, \varphibf) \in \Ucal'_{\wh\wbld} \times \Ucal_{\zbld} \times \Vcal_0 \times \Ucal_{0} \times \mathcal B_\infty(0, \mu_2)$.
 Obviously, the same result holds on
 $\mathcal G'_{\mathbf f} := \{  (\wh \wbld, \zbld, v_0, u_0, \varphibf)\in  \Ucal'_{\wh\wbld} \times \Ucal_{\zbld} \times \Vcal_0 \times \Ucal_{0} \times \mathcal B_\infty(0, \mu_2) \, | \,
 \mathbf f( \wh \wbld, \zbld, v_0, u_0, \varphibf) = 0 \}$.
 %
 %
  Lastly,
 since $\| D_{(\zbld, v_0, u_0, \varphibf )} \mathbf f(\cdot)\|_2$ is a globally continuous function, it achieves its  maximum on any compact set containing $\Ucal'_{\wh\wbld} \times\Ucal_{\zbld} \times \Vcal_0 \times \Ucal_{0} \times \mathcal B_\infty(0, \mu_2)$. Hence, there exists a positive constant $\rho$ such that $\| D^{-1}_{\wh\wbld} \mathbf f(\cdot)\|_2\cdot \| D_{(\zbld, v_0, u_0, \varphibf )} \mathbf f(\cdot)\|_2 \le \rho$ on  $\Ucal'_{\wh\wbld} \times \Ucal_{\zbld} \times \Vcal_0 \times \Ucal_{0} \times \mathcal B_\infty(0, \mu_2)$. This result also holds on $\mathcal G'_{\mathbf f}$.
 %
 %
 Clearly, $\Ucal_{\zbld} \times \Vcal_0 \times \Ucal_{0} \times \mathcal B_\infty(0, \mu_2)$ is convex and thus connected, and $\Ucal_{\wh\wbld}$ is compact and thus closed.
 Consequently, in view of the smoothness of $\mathbf f$, we conclude that there exists a unique smooth function $\mathbf h$ such that $\wh \wbld_* = \mathbf h(\zbld, v_0, u_0, \varphibf)$ for all $(\zbld, v_0, u_0, \varphibf) \in \Ucal_{\zbld} \times \Vcal_0 \times \Ucal_{0} \times \mathcal B_\infty(0, \mu_2)$.
\end{proof}

The above proposition implies that the unconstrained nonconvex optimization problem $\min J(\wh \wbld)$ has a unique local minimizer $\wh \wbld_*$ on $\Ucal_{\wh \wbld}$ for any fixed $(\zbld, v_0, u_0, \varphibf) \in \Ucal_{\zbld} \times \Vcal_0 \times \Ucal_{0} \times \mathcal B_\infty(0, \mu_2)$. Hence, for any $(\zbld(k), v_0(k), u_0(k), \varphibf) \in \Ucal_{\zbld} \times \Vcal_0 \times \Ucal_{0}$ at each $k$,
\[
   \wh \wbld_*(k)  = \mathbf h (\zbld(k), v_0(k), u_0(k), \varphibf), \qquad
   \wh w_*(k) =  \begin{bmatrix} I_n & 0 & \cdots & 0 \end{bmatrix}  \mathbf h (\zbld(k), v_0(k), u_0(k), \varphibf).
\]
Moreover,  note that $\mathbf h(\zbld, v_0, u_0, 0) = -\mathbf H^{-1} \big(\mathbf G \cdot \zbld -  u_0 \mathbf g \big)$ for any fixed $(\zbld, v_0, u_0) \in \Ucal_{\zbld} \times \Vcal \times \Ucal_{0}$. Define
\[
  \Delta \wh{\mathbf h} (\zbld, v_0, u_0, \varphibf) \, := \, \begin{bmatrix} I_n & 0 & \cdots & 0 \end{bmatrix} \Big( \mathbf h (\zbld, v_0, u_0, \varphibf) - \mathbf h(\zbld, v_0, u_0, 0) \Big).
\]
Since $\Ucal_{\zbld} \times \Vcal_0 \times \Ucal_{0} \times \mathcal B_\infty(0, \mu_2)$ is an open convex  set, it follows from the Mean-value Theorem that for
any fixed $(\zbld, v_0, u_0, \varphibf) \in \Ucal_{\zbld} \times \Vcal_0 \times \Ucal_{0}
\times \mathcal B_\infty(0, \mu_2)$,
\[
   \Delta \wh {\mathbf h} (\zbld, v_0, u_0, \varphibf) = \int^1_{0} D_{\varphibf} \wh {\mathbf h}(\zbld, v_0, u_0, t \varphibf) dt \cdot \varphibf.
\]
Therefore, there exists a positive constant $\varkappa$ such that $\| \Delta \wh{\mathbf h} (\zbld, v_0, u_0, \varphibf) \| \le \varkappa \| \varphibf \|_{\infty}$ for all $(\zbld, v_0, u_0, \varphibf) \in \Ucal_{\zbld} \times \Vcal_0 \times \Ucal_{0} \times \mathcal B_\infty(0, \mu_2)$.

Substituting the above results to the closed loop dynamics (\ref{eqn:nonlin_closed_loop_02}), we obtain
\begin{eqnarray*}
  \zbld(k+1) & = & \wh A(k) \zbld(k) + B  \Big( \wh w_*(k) +  \wt h_{\varphibf_d}(z'(k)) \Big) \\
  & = &
    \Big(A + v_0(k) \Delta A(\varphibf_d) \Big) \zbld(k) + B \Big[ -\begin{bmatrix} I_n & 0 & \cdots & 0 \end{bmatrix} \mathbf H^{-1} \big(\mathbf G \cdot \zbld(k) -  u_0(k) \mathbf g\big) \Big] \\
    & & \qquad + B \Big( \Delta \wh {\mathbf h} (\zbld, v_0, u_0, \varphibf) + \wt h_{\varphibf_d}(z'(k)) \Big) \\
   & = & \Big[ \underbrace{ \big(A +B \mathbf K )}_{A_{\mbox{c}}}  + v_0(k) \cdot \Delta A(\varphibf_d) \Big] \zbld(k) + B \Big( u_0(k) \cdot  \mathbf d + \Delta \wh {\mathbf h} (\zbld, v_0, u_0, \varphibf)  + \wt h_{\varphibf_d}(z'(k)) \Big),
\end{eqnarray*}
where the constant matrix $\mathbf K $ and  the  constant vector $\mathbf d$ are given by (\ref{eqn:K_mat_d_vec}), and $A_{\mbox{c}}$ is the closed loop dynamics matrix for the linear vehicle dynamics given by (\ref{eqn:A_c_matrix}). This leads to the closed loop dynamics for $p>1$:
\begin{equation} \label{eqn:closed_loop_p>1}
  \zbld(k+1) \, = \, \Big(A_{\mbox{c}} +  v_0(k) \cdot \Delta A(\varphibf_d) \Big) \zbld(k) + B \Big[ u_0(k)  \mathbf d +   \Delta \wh {\mathbf h} (\zbld(k), v_0(k), u_0(k), \varphibf) \Big]  + B \wt h_{\varphibf_d}(z'(k))
\end{equation}
for all $(\zbld, v_0, u_0, \varphibf) \in \Ucal_{\zbld} \times \Vcal_0 \times \Ucal_{0} \times \mathcal B_\infty(0, \mu_2)$, where $\mathcal U_{\zbld}$ is a bounded open  convex set in $\mathbb R^{2n}$,  $\mathcal U_{0}$ is a bounded open convex set containing $ [a_{0, \min}, a_{0, \max}]$, and $\mathcal V_0$ is a bounded open convex set containing $[v_{\min}, v_{\max}]$.


%
\subsection{Input-to-state Stability of Discrete-time Dynamical Systems: An Overview}

We give a brief overview of (local) input-to-state stability. Consider the discrete-time system on $\mathbb R^n$:
\begin{equation} \label{eqn:discrete_time_sys}
  x(k+1) \, = \, f(x(k), u(k), k), \qquad \forall \ k \in \mathbb Z_+,
\end{equation}
where $f:\mathbb R^n \times \mathbb R^m \times \mathbb Z_+ \rightarrow \mathbb R^n$, and $f(\cdot, \cdot, k)$  is continuous for any fixed $k\in \mathbb Z_+$. Let $\ubld
:=(u(0), u(1), \ldots )$ be a sequence of vectors in $\mathbb R^m$ that represents an input function on $\mathbb Z_+$.
%
%
We assume that $f(0, 0, k) =0$ for all $k \in \mathbb Z_+$ such that $x_e=0$ is an equilibrium of the system (\ref{eqn:discrete_time_sys}) under the $0$-input, i.e., $\ubld=0$. We let $\| \ubld\|_\infty:=\sup\{ \| u(k) \| \, : \, k \in \mathbb Z_+\}$. Hence, for any $\ubld \in \ell^m_{\infty}$, $\|\ubld\|_\infty<\infty$.
For a given initial condition $\xi \in \mathbb R^n$ and an input function $\ubld$, let $x(k, \xi, \ubld)$ denote the trajectory of the system (\ref{eqn:discrete_time_sys}).

We introduce the notions of $\mathcal K$ class of functions \cite[pp. 135]{Khalil_book96}. A  continuous function $\alpha:\mathbb R_+ \rightarrow \mathbb R_+$ is a $\mathcal K$-function if it is strictly increasing on $[0, \infty)$ and $\alpha(0)=0$. It is a $\mathcal K_\infty$-function if it is a $\mathcal K$-function and $\alpha(t) \rightarrow \infty$ as $t \rightarrow \infty$. It is a {\em positive definite} function if $\alpha(t)>0$ for all $t>0$ and $\alpha(0)=0$. A function $\beta:\mathbb R_+ \times \mathbb R_+ \rightarrow \mathbb R_+$ is a $\mathcal {KL}$-function if (i) for any fixed $t \ge 0$, the function $\beta(\cdot, t)$ is a $\mathcal K$-function; and (ii) for any fixed $s \ge 0$, the function $\beta(s, \cdot)$ is decreasing and $\beta(s, t) \rightarrow 0$ as $t \rightarrow \infty$.

\begin{definition} \rm \label{def:local_ISS}
 The time-varying discrete-time system (\ref{eqn:discrete_time_sys}) is {\em locally input-to-state stable} (ISS) if there exist a $\mathcal {KL}$-function $\beta:\mathbb R_+\times \mathbb R_+ \rightarrow \mathbb R_+$, a $\mathcal K$-function $\gamma:\mathbb R_+ \rightarrow \mathbb R_+$, and two positive constants $\theta_x, \theta_u$ such that for all $\xi$ with $\| \xi \| \le \theta_x$ and $\ubld\in \ell^m_\infty$ with $\|\ubld\|_\infty\le \theta_u$, the following holds:
 \[
    \| x(k, \xi, \ubld) \| \, \le \, \beta(\| \xi\|, k) + \gamma (\|\ubld \|_\infty), \qquad \forall \ k \in \mathbb Z_+.
 \]
\end{definition}
The above definition follows from \cite[Definition 5.2]{Khalil_book96} for continuous-time systems and \cite[Definition 3.1]{JiangWang_Auto01} for global ISS of discrete-time systems. Also see \cite{ShenHu_SICON12, HuShenP_SICON16, Sontag_bookchapter08} for details. In what follows, we extend the Lyapunov approach for global ISS in \cite[Lemma 3.5]{JiangWang_Auto01} and \cite{JiangWang_SCL02} to local input-to-state stability (ISS) for the time-varying system (\ref{eqn:discrete_time_sys}); see \cite[Lemma 2.3]{JiangLiWang_Auto04} for a similar local version of the ISS.

\begin{theorem} \label{thm:local_ISS_general}
 Consider the time-varying discrete-time system (\ref{eqn:discrete_time_sys}) defined by $f:\mathbb R^n \times \mathbb R^m \times \mathbb Z_+ \rightarrow \mathbb R^n$. Suppose there exists a local ISS-Lyapunov function $V:\mathbb R^n \times \mathbb Z_+ \rightarrow \mathbb R_+$ for the system (\ref{eqn:discrete_time_sys}), namely, there exist two sets $\mathcal D_x :=\{ x \in \mathbb R^n \, | \, \| x \| \le r \}$ and $\mathcal D_u :=\{ u \in \mathbb R^m \, | \, \| u\| \le r_u \}$ for some positive constants $r$ and $r_u$,
 where $r_u$ can be $+\infty$,  such that  the following hold:
 \begin{itemize}
   \item [(i)]  There exist two $\mathcal K_\infty$-functions $\alpha_1$ and $\alpha_2$ such that $\alpha_1(t) \le \alpha_2(t), \forall \, t \ge 0$ and $\alpha_1(\|x\|) \le V(x, k) \le \alpha_2(\|x\|)$ for all $x \in \mathcal D_x$ and all $k \in \mathbb Z_+$;
 %
   \item [(ii)] There exist a $\mathcal K_\infty$-function $\alpha_3$ and a $\mathcal K$-function $\sigma$ such that $V(f(x, u, k), k+1) - V(x, k) \le -\alpha_3(\|x\|) + \sigma(\|u\|)$ for all $x \in \mathcal D_x$ and $u\in \mathcal D_u$ and all $k \in \mathbb Z_+$.
 \end{itemize}
 Then there exist positive constants $\theta_x>0$ and $\theta_u>0$ such that the following hold:
 \begin{itemize}
   \item [(i)] For any $\xi$ with $\| \xi \|\le \theta_x$ and
     $\mathbf u=(u(k))_{k\in \mathbb Z_+} \in \ell^m_\infty$ with $\| \mathbf u\|_\infty \le \theta_u$, $x(k, \xi, \ubld) \in \mathcal D_x$ for all $k \in \mathbb Z_+$;

   \item [(ii)] The system (\ref{eqn:discrete_time_sys}) is locally input-to-state stable in terms of the positive constants $\theta_x$ and $\theta_u$ given in Definition~\ref{def:local_ISS}.
 \end{itemize}
%
\end{theorem}

%
\subsection{Local Input-to-state Stability of the Closed Loop System}


Since the closed loop dynamics for $p=1$ given by (\ref{eqn:closed_loop_p=1}) and that for $p>1$ given by (\ref{eqn:closed_loop_p>1}) share the similar structure except that the latter holds on a restricted set, we provide a unified proof for local input-to-state stability under the assumption that the closed loop dynamic matrix $A_{\mbox{c}}$ under the linear vehicle dynamics given in (\ref{eqn:A_c_matrix}) is Schur stable.

%
%

\begin{theorem}
Fix $p\in \mathbb N$. Suppose the weight matrices $Q_{z, s}, Q_{z', s}$ and $Q_{w, s}$  satisfying $\bf A.1$ are such that $A_{\mbox{c}}$ given in (\ref{eqn:A_c_matrix}) is Schur stable. Then there exist positive constants $\mu$ and $\nu$ such that for all $\varphibf$ with $\|\varphibf\|_\infty \le \mu$, any $v_0(k) \in[v_{\min}, v_{\max}]$ and any $u_0(k)$ with $|u_0(k)|\le \nu$ for all $k\in\mathbb Z_+$, the closed loop dynamics given by (\ref{eqn:closed_loop_p=1}) or (\ref{eqn:closed_loop_p>1}) is locally input-to-state stable.
\end{theorem}

\begin{proof}
Consider $p>1$ first. For the given bounded open sets $\mathcal U_{\zbld}$ containing the zero vector,  $\mathcal U_{0}$ containing $ [a_{0, \min}, a_{0, \max}]$, and $\mathcal V_0$ containing $[v_{\min}, v_{\max}]$, the closed loop dynamics is given by (\ref{eqn:closed_loop_p>1}) as shown by Proposition~\ref{prop:implicit_function_wh_wbld}.
Since $A_{\mbox{c}}$ is Schur stable, there exist constants $\kappa_{\mbox{c}}>0$ and $r \in (0, 1)$ such that $\|(A_{\mbox{c}})^k \| \le \kappa_{\mbox{c}} \cdot r^k$ for all $k \in \mathbb Z_+$. Consider the time-varying discrete linear   system on $\mathbb R^{2n}$:
\begin{equation} \label{eqn:LTV_system}
 \zbld (k+1) \, = \, \Big ( A_{\mbox{c}} + v_0(k) \cdot \Delta A(\varphibf_d) \Big) \zbld(k), \qquad \forall \, k \in \mathbb Z_+.
\end{equation}
In view of the expressions of $\Delta A(\varphibf_d)$ given by (\ref{eqn:closed_loop_matrices}) and $D(\varphibf_d)$ by (\ref{eqn:matrix_D}) and $0\le v_{\min}\le v_0(k) \le v_{\max}$ for all $k \in \mathbb Z_+$, we deduce that there exists a positive constant $\kappa_{\Delta A}$ such that
  $\|  v_0 \cdot \Delta A(\varphibf_d) \|_2 \le \kappa_{\Delta A} \cdot v_{\max} \cdot \|\varphibf_d\|_\infty, \forall \, v_0 \in [v_{\min}, v_{\max}]$.
  Define the positive constant
  \[
   \wt \mu_3 \, :=  \, - \frac{r}{\kappa_{\mbox{c}} \cdot \kappa_{\Delta A} \cdot v_{\max}} \ln(r) > 0.
  \]
 Hence, for all $\varphibf_d$ with $\|\varphibf_d\|_\infty < \wt \mu_3$, we have $\kappa_{\Delta A} \cdot v_{\max} \cdot \|\varphibf_d\|_\infty \le - \frac{r}{\kappa_{\mbox{c}}} \ln(r)$.
  Then it follows from \cite[Theorem 3]{ZhouZhao_TAC17} that the discrete linear system (\ref{eqn:LTV_system}) is uniformly exponentially stable for any $v_0(k) \in [v_{\min}, v_{\max}], \forall \, k \in \mathbb Z_+$ and all $\varphibf_d$ with $\|\varphibf_d\|_\infty < \wt \mu_3$.
Define $\wh A_{\mbox{c}}(k) := A_{\mbox{c}} + v_0(k) \Delta A(\varphibf_d)$ for all $k \in \mathbb Z_+$.
 (Rigorously speaking, it should be written as $\wh A_{\mbox{c}}(v_0(k), \varphibf_d)$. For the sake of notational simplicity, we write it in this way.)
By \cite[Theorem  23.3]{Rugh_book96}, there exist a matrix sequence $\{ P(k) \}_{k\in \mathbb Z_+}$ with $P(k)=P^T(k) \in \mathbb R^{2n \times 2n}$ for each  $k$ and positive constants $\theta_2\ge \theta_1>0$ and $\theta_3>0$ such that
 for all $v_0(k) \in [v_{\min}, v_{\max}], \forall \, k \in \mathbb Z_+$ and all $\varphibf_d$ with $\|\varphibf_d\|_\infty < \wt \mu_3$,
 $\theta_1 I_{2n} \preccurlyeq P(k) \preccurlyeq \theta_2 I_{2n}$ and $\wh A^T_{\mbox{c}}(k) P(k+1) \wh A_{\mbox{c}}(k) - P(k) \preccurlyeq - \theta_3 I_{2n}$ for all $k \in \mathbb Z_+$, where $\preccurlyeq$ denotes the positive semi-definite order. Clearly, $\| P(k) \|_2 \le \theta_2$ for all $k\in \mathbb Z_+$.


 Given  any $v_0(k) \in [v_{\min}, v_{\max}], \forall \, k \in \mathbb Z_+$ and any $\varphibf_d$ satisfying $\|\varphibf_d\|_\infty < \wt \mu_3$, define the function $f_{\varphibf_d}:\mathbb R^{2n} \times \mathbb R^n \times \mathbb Z_+ \rightarrow \mathbb R^{2n}$ as:
\begin{equation*} 
   f_{\varphibf_d}(\zbld, d, k) \, := \, \wh A_{\mbox{c}}(k) \zbld + B d + B \wt h_{\varphibf_d}(z'),
\end{equation*}
where $\zbld=(z, z') \in \mathbb R^{2n}$.  Consider the time-varying Lyapunov function $V:\mathbb R^{2n} \times \mathbb Z_+ \rightarrow \mathbb R_+$ given by
\begin{equation} \label{eqn:V_ISS_Lypap_func}
    V(\zbld, k ) \, :=  \, \zbld^T P(k) \zbld, \qquad \quad k \in \mathbb Z_+.
\end{equation}
In light of $\wh A^T_{\mbox{c}}(k) P(k+1) \wh A_{\mbox{c}}(k) - P(k) \preccurlyeq - \theta_3 I_{2n}$, we have that for any $k \in \mathbb Z_+$,
\begin{eqnarray*} 
  V(f_{\varphibf_d}(\zbld, d, k), k+1) - V(\zbld, k) & \le & -\theta_3 \|\zbld\|^2_2 + 2 \Big[ \wh A_{\mbox{c}} (k) \zbld \Big]^T P(k+1) B \Big[ d + \wt h_{\varphibf_d}(z') \Big] \\
  & & \quad + \ \Big[d+ \wt  h_{\varphibf_d}(z') \Big]^T B^T P(k+1) B\Big[d+ \wt  h_{\varphibf_d}(z') \Big].
\end{eqnarray*}
Let
\[
\eta_1 \, := \, \sup_{\| \varphibf_d\|_\infty \le \wt \mu_3} \Big( \| A_{\mbox{c}} \|_2+ v_{\max} \cdot \| \Delta A(\varphibf_d)\|_2 \Big) \, > \, 0.
\]
 Hence, $\| \wh A_{\mbox{c}}(k) \| \le \eta_1$ for all $k\in \mathbb Z_+$. Moreover, it follows from (\ref{eqn:wt_h_func}) that there exists a positive constant $\eta_2$ such that $\| \wt
  h_{\varphibf_d}(z')\|_2 \le \eta_2 \cdot \|\varphibf_d\|_\infty \cdot \|z' \|^2_2 \le \eta_2 \cdot \|\varphibf_d\|_\infty \cdot \| \zbld \|^2_2$ for all $\varphibf_d$ and $\zbld$.
Let $ \wt \eta_2:=\eta_2 \big(\sup_{\zbld \in \Ucal_{\zbld}} \|\zbld \|\big)$.
Therefore, for all $\zbld \in \Ucal_{\zbld}$, we have
\[
\| \wt h_{\varphibf_d}(z') \|_2 \, \le \, \eta_2 \big(\sup_{\zbld \in \Ucal_{\zbld}} \|\zbld \|\big) \cdot \|\varphibf_d\|_\infty \cdot \|\zbld\| = \wt \eta_2 \cdot  \|\varphibf_d\|_\infty \cdot \|\zbld\|.
\]
Consequently,  for all $\zbld \in \Ucal_{\zbld}$, $v_0(k) \in [v_{\min}, v_{\max}], \forall \, k\in \mathbb Z_+$, $\varphibf_d \in \mathcal B(0, \wt \mu_3)$, and $d \in \mathbb R^n$, we have,
\begin{eqnarray*}
  \big[ \wh A_{\mbox{c}} (k) \zbld \big]^T P(k+1) B \big [ d + \wt h_{\varphibf_d}(z') \big] & \le & \theta_2 \| \wh A_{\mbox{c}} (k) \zbld \|_2 \cdot \|B\|_2 \big( \|d \|_2 + \| \wt
  h_{\varphibf_d}(z') \|_2 \big)
  \\
  & \le & \theta_2 \eta_1 \|B\|_2 \cdot \|\zbld\|_2 \cdot \big(  \|d \|_2 + \wt \eta_2 \|\varphibf_d \|_\infty \cdot \|\zbld \|_2 \big)
\end{eqnarray*}
and
\begin{eqnarray*}
  \big[d+ \wt h_{\varphibf_d}(z') \big ]^T B^T P(k+1) B \big[d+ \wt h_{\varphibf_d}(z') \big] & \le &  \theta_2 \cdot \|B\|^2_2 \cdot \big\|d + \wt h_{\varphibf_d}(z') \big\|^2_2 \\
   & \le &  2 \theta_2  \cdot \|B\|^2_2 \cdot \Big( \|d\|^2_2 + \big( \wt \eta_2 \|\varphibf_d \|_\infty \big)^2 \cdot   \| \zbld \|^2_2 \Big).
\end{eqnarray*}
Combining the above results, we deduce that there exists a constant $\mu_3$ with $0<\mu_3 \le \min(\wt \mu_3, \mu_2)$, where $\mu_2$ is given in Proposition~\ref{prop:implicit_function_wh_wbld}, such that for all $\|\varphibf \|_\infty \le \mu_3$, $\zbld \in \Ucal_{\zbld}$, $v_0(k)\in [v_{\min}, v_{\max}], \forall \, k \in \mathbb Z_+$, %
%
and $d \in \mathbb R^n$,
\[
  V(f_{\varphibf_d}(\zbld, d, k), k+1) - V(\zbld, k) \, \le \, -\frac{2\theta_3 }{3}\|\zbld\|^2_2   + 2 \eta_3 \|d \|_2 \cdot \|\zbld \|_2 + \eta_4 \| d\|^2_2,
\]
where $\eta_3:= \theta_2 \eta_1 \|B\|_2/2$, and $\eta_4 := 2 \theta_2 \|B\|^2_2$.
%
%
%
%
%
Consequently, for all $\|\varphibf \|_\infty \le \mu_3$, $\zbld \in \Ucal_{\zbld}$,
$v_0(k)\in [v_{\min}, v_{\max}], \forall \, k \in \mathbb Z_+$,
%
%
 and $d \in \mathbb R^n$, we have, for all $k \in \mathbb Z_+$,
\begin{eqnarray*}
 V(f_{\varphibf_d}(\zbld, d, k), k+1 ) - V(\zbld, k) & \le & -\frac{2\theta_3 }{3}\|\zbld\|^2_2   + 2 \eta_3 \cdot \|d \|_2 \cdot \|\zbld \|_2 + \eta_4 \cdot \| d\|^2_2 \\
  & = & -\frac{\theta_3}{6} \|\zbld\|^2_2 -  \frac{\theta_3}{2}  \|\zbld\|^2_2  + 2 \eta_3 \|d \|_2 \cdot \|\zbld \|_2 + \eta_4 \| d\|^2_2 \\
  & = & -\frac{\theta_3}{6} \|\zbld\|^2_2 -  \frac{\theta_3}{2}  \Big(\|\zbld\|_2 - \frac{2 \eta_3}{\theta_3} \|d\|_2\Big)^2 + \Big(\frac{2 \eta^2_3}{\theta_3} + \eta_4 \Big) \|d\|^2_2 \\
  & \le & -\frac{\theta_3}{6} \|\zbld\|^2_2 + \Big(\frac{2 \eta^2_3}{\theta_3 } + \eta_4 \Big) \|d\|^2_2.
\end{eqnarray*}
%
%

%
%
 Define the functions $\sigma_1(t):=\theta_1 t^2$, $\sigma_2(t):=\theta_2 t^2$, $\sigma_3(t):= \frac{\theta_3}{6} t^2$, and $\sigma(t):=\Big(\frac{2 \eta^2_3}{\theta_3} + \eta_4 \Big) t^2$. Clearly, these function are $\mathcal K_\infty$-functions. Let $\mathcal D_{\zbld}$ be the largest closed ball centered at the origin that is contained in $\Ucal_{\zbld}$ (such the closed ball exists since $\Ucal_{\zbld}$ is a bounded open set containing 0), and $\mathcal D_{d}=\mathbb R^n$. Hence,
 the function $V$ given in (\ref{eqn:V_ISS_Lypap_func}) is a local ISS-Lyapunov function on $\mathcal D_{\zbld}\times \mathcal D_d$ for
 the discrete time system $\zbld(k+1) = f_{\varphibf_d}(\zbld(k), d(k), k)$, for
  all $\varphibf_d \in \mathcal B_{\infty}(0, \mu_3)$ and all $v_0(k) \in [v_{\min}, v_{\max}], \forall \, k \in \mathbb Z_+$.
 It follows from Theorem~\ref{thm:local_ISS_general} that
there exist two positive constants $\nu_{\zbld}$ and $\nu_d$ such that for  any $\xi$ with $\| \xi \|\le \nu_\zbld$ and  $\ol{\mathbf d}=(d(k))_{k\in \mathbb Z_+} \in \ell^m_\infty$ with $\| \ol{\mathbf d}\|_\infty \le \nu_d$, $\zbld(k, \xi, \ol{\mathbf d}) \in \mathcal \Ucal_{\zbld}$ for all $k \in \mathbb Z_+$.
In view of the right-hand side of the  closed loop dynamics given by (\ref{eqn:closed_loop_p>1}), we see that
$d(k) = u_0(k) \cdot \mathbf d + \Delta \wh {\mathbf h} (\zbld(k), v_0(k), u_0(k), \varphibf)$ for all $k\in \mathbb Z_+$, where $\mathbf d$ is the constant vector given by (\ref{eqn:K_mat_d_vec}), and $\| \Delta \wh{\mathbf h} (\zbld, v_0, u_0, \varphibf) \| \le \varkappa \| \varphibf \|_\infty$ for all $(\zbld, v_0, u_0, \varphibf) \in \Ucal_{\zbld} \times \Vcal_0 \times \Ucal_{0} \times \mathcal B_\infty(0, \mu_2)$.
For an arbitrary but fixed $\varepsilon \in (0, 1)$, define the positive constants
\[
   \mu_4 \, := \, \min \Big( \mu_3, \frac{ \varepsilon \cdot \nu_d }{\varkappa} \Big), \qquad \nu_u \, := \, \min \Big( |a_{0, \min}|, \, a_{0,\max}, \, \frac{(1-\varepsilon) \nu_d}{\|\mathbf d\|} \Big).
\]
Hence,  $u_0 \in [-a_{0, \min}, a_{0, \max}] \subset \Ucal_0$ and $\| u_0 \mathbf d \| \le (1-\varepsilon) \nu_d$ for any $u_0$ with $| u_0 | \le \nu_u$. (The condition $u_0 \in \Ucal_0$ is needed to derive the closed loop dynamics as shown in Proposition~\ref{prop:implicit_function_wh_wbld}.)
Further, for all $\varphibf$ with $\|\varphibf \|_\infty \le \mu_4$, $u_0$ with $| u_0 | \le \nu_u$, $v_0\in[v_{\min}, v_{\max}]$ and $\zbld \in \Ucal_{\zbld}$, it is easy to show that $\| u_0 \mathbf d + \Delta \wh {\mathbf h} (\zbld, v_0, u_0, \varphibf) \| \le \nu_d$.
It can be further shown via induction on $k$ that for all $\varphibf$ with $\|\varphibf \|_\infty \le \mu_4$, $u_0(k)$ with $| u_0(k) | \le \nu_u, \forall \, k \in \mathbb Z_+$, $v_0(k) \in[v_{\min}, v_{\max}], \forall \, k \in \mathbb Z_+$, and any $\xi$ with $\| \xi \|\le \nu_\zbld$,
 $\zbld(k, \xi, \ol{\mathbf d}) \in \mathcal \Ucal_{\zbld}$ and
 $\| d(k) \|= \| u_0(k) \cdot \mathbf d + \Delta \wh {\mathbf h} (\zbld(k), v_0(k), u_0(k), \varphibf) \| \le \nu_d$ for all $k \in \mathbb Z_+$. In view of Theorem~\ref{thm:local_ISS_general} again, we deduce that the closed loop dynamics given by (\ref{eqn:closed_loop_p>1}) is locally input-to-state stable for all $\varphibf$ with $\|\varphibf \|_\infty \le \mu_4$, $v_0(k) \in[v_{\min}, v_{\max}], \forall \, k \in \mathbb Z_+$, and $u_0(k)$ with $| u_0(k) | \le \nu_u, \forall \, k \in \mathbb Z_+$.
\mycut{
Note that for the original closed loop dynamics for $p>1$ given by (\ref{eqn:closed_loop_p>1}), we can write $d(k) = u_0(k) \cdot \mathbf d + \Delta \wh {\mathbf h} (\zbld(k), v_0(k), u_0(k), \varphibf)$ for all $k\in \mathbb Z_+$, where $\| \Delta \wh{\mathbf h} (\zbld, v_0, u_0, \varphibf) \| \le \varkappa \| \varphibf \|_\infty$ for all $(\zbld, v_0, u_0, \varphibf) \in \Ucal_{\zbld} \times \Vcal \times \Ucal_{u_0} \times \mathcal B_\infty(0, \mu_2)$.
Hence, for an arbitrary but fixed $\varepsilon \in (0, 1)$, define the positive constant
\[
   \mu_4 \, := \, \min \Big( \mu_3, \frac{ \varepsilon \nu_u }{\kappa} \Big).
\]
Then for all $\varphibf$ with $\|\varphibf \|_\infty \le \mu_4$, $u_0(k)$ satisfying $\| u_0(k) \mathbf d\| \le (1-\varepsilon) \nu_u$, it is easy to show via induction on $k$ that for any $\xi$ with $\| \xi \|\le \nu_\zbld$, $\zbld(k, \xi, \ubld) \in \mathcal \Ucal_{\zbld}$ and
 $\| d(k) \|= \| u_0(k) \cdot \mathbf d + \Delta \wh {\mathbf h} (\zbld(k), v_0(k), u_0(k), \varphibf) \| \le \nu_u$ for all $k \in \mathbb Z_+$. In view of Theorem~\ref{thm:local_ISS_general} again, we deduce that the closed loop dynamics given by (\ref{eqn:closed_loop_p>1}) is locally input-to-state stable for all $\varphibf$ with $\|\varphibf \|_\infty \le \mu_4$, $v_0(k) \in \mathcal V_0$ and $u_0(k)$ satisfying $\| u_0(k) \mathbf d\| \le (1-\varepsilon) \nu_u$.
}


We next consider $p=1$ whose closed loop dynamics is given by (\ref{eqn:closed_loop_p=1}). Its proof is similar to that for $p>1$, and we sketch its key steps as follows.
First, consider the time-varying discrete time system
\[
  \zbld(k+1) \, = \, \Big(A_{\mbox{c}} +  v_0(k) \cdot \Delta {\bar A}(\varphibf_d) \Big) \zbld(k).
\]
Since $A_{\mbox{c}}$ is Schur stable, it follows from the definition of $\Delta {\bar A}(\varphibf_d)$ given in (\ref{eqn:matrix_Delta_A_B}) that there exists a constant $\mu_1>0$ such that the above discrete-time system is uniformly exponentially stable for all $v_0(k) \in[v_{\min}, v_{\max}], \forall \, k \in \mathbb Z_+$ and all $\varphibf_d$ with $\| \varphibf_d \|_\infty \le \mu_1$. Let $\wh A_{\mbox{c}}(v_0(k), \varphibf_d):= A_{\mbox{c}} +  v_0(k) \cdot \Delta {\bar A}(\varphibf_d)$, which is written as $\wh A_{\mbox{c}}(k)$ for simplicity.
 It follows from the similar argument for the case of $p>1$ that there exist a symmetric matrix sequence $\{ P(k) \}_{k\in \mathbb Z_+}$ and positive constants $\theta_2\ge \theta_1>0$ and $\theta_3>0$ such that
 for all $v_0(k) \in [v_{\min}, v_{\max}], \forall \, k \in \mathbb Z_+$ and all $\varphibf_d$ with $\|\varphibf_d\|_\infty < \mu_1$,
 $\theta_1 I_{2n} \preccurlyeq P(k) \preccurlyeq \theta_2 I_{2n}$ and $\wh A^T_{\mbox{c}}(k) P(k+1) \wh A_{\mbox{c}}(k) - P(k) \preccurlyeq - \theta_3 I_{2n}$ for all $k \in \mathbb Z_+$.
This leads to the Lyapunov function $V(\zbld, k):=\zbld^T P(k) \zbld$ for $k \in \mathbb Z_+$. Let the function $f_{\varphibf_d}(\zbld, d, k) \, := \, \wh A_{\mbox{c}}(v_0(k), \varphibf_d) \zbld + B d + \breve B \wt h_{\varphibf_d}(z')$, and $\Ucal_{\zbld}$ be a bounded open set in $\mathbb R^{2n}$ containing the zero vector. As before, there exists a constant $\wt \eta_2>0$ such that
$\| \wt h_{\varphibf_d}(z') \|_2 \le \wt \eta_2 \cdot  \|\varphibf_d\| \cdot \|\zbld\|$ for all $\zbld \in \Ucal_{\zbld}$. By the similar argument for the case of $p>1$, there exist positive constants $\mu_2, \eta_3, \eta_4$ such that
for all $\|\varphibf \|_\infty \le \mu_2$, $\zbld \in \Ucal_{\zbld}$, $v_0(k)\in [v_{\min}, v_{\max}], \forall \, k \in \mathbb Z_+$, and $d \in \mathbb R^n$,
\begin{align*}
  V(f_{\varphibf_d}(\zbld, d, k), k+1) - V(\zbld, k) \, \le  -\frac{2\theta_3}{3}\|\zbld\|^2_2   + 2 \eta_3 \|d \|_2 \cdot \|\zbld \|_2 + \eta_4 \| d\|^2_2
   \, \le -\frac{\theta_3}{6} \|\zbld\|^2_2 + \Big(\frac{2 \eta^2_3}{\theta_3 } + \eta_4 \Big) \|d\|^2_2.
\end{align*}
By Theorem~\ref{thm:local_ISS_general},
there exist two positive constants $\nu_{\zbld}$ and $\nu_d$ such that for  any $\xi$ with $\| \xi \|\le \nu_\zbld$ and  $\ol{\mathbf d}=(d(k))_{k\in \mathbb Z_+} \in \ell^m_\infty$ with $\| \ol{\mathbf d}\|_\infty \le \nu_d$, $\zbld(k, \xi, \ol{\mathbf d}) \in \mathcal \Ucal_{\zbld}$ for all $k \in \mathbb Z_+$. Recall that $d(k)= u_0(k)  \mathbf d - \wh W Q_w w_e(k)$, where $\mathbf d = \wh W Q_w \mathbf e_1$, and $w_e(k):=(w_{e, 1}(k), \ldots, w_{e, n}(k))^T$ with $w_{e, i}(k)= [c_{2, i-1}-c_{2, i}] v^2_0(k)+[c_{3, i-1} -c_{3, i}] g$. It is easy to derive via an argument similar to that for $p>1$ that there exist positive constants $\mu_3$ and $\nu_u$ such that
for all $\varphibf$ with $\|\varphibf \|_\infty \le \mu_3$, $u_0(k)$ with $| u_0(k) | \le \nu_u, \forall \, k \in \mathbb Z_+$, $v_0(k) \in[v_{\min}, v_{\max}], \forall \, k \in \mathbb Z_+$, and any $\xi$ with $\| \xi \|\le \nu_\zbld$,
 $\zbld(k, \xi, \ol{\mathbf d}) \in \mathcal \Ucal_{\zbld}$ and
 $\| d(k) \| \le \nu_d$ for all $k \in \mathbb Z_+$. By Theorem~\ref{thm:local_ISS_general} again, the closed loop dynamics given by (\ref{eqn:closed_loop_p=1}) is locally input-to-state stable for all $\varphibf$ with $\|\varphibf \|_\infty \le \mu_3$, $v_0(k) \in[v_{\min}, v_{\max}], \forall \, k \in \mathbb Z_+$, and $u_0(k)$ with $| u_0(k) | \le \nu_u, \forall \, k \in \mathbb Z_+$.
%
%
%
\end{proof}

%
\section{Numerical Results} \label{sect:numerical_results}

%
\subsection{Numerical Experiment Setup and Weight Matrix Design} \label{subsect:Numerical_test_description}

Numerical tests are carried out to evaluate the performance of the proposed fully distributed schemes and the platooning control for a possibly heterogeneous CAV platoon. We consider a platoon of an uncontrolled leading vehicle labeled by the index 0 and ten CAVs, i.e., $n=10$. The sample time $\tau=1 s$, and  the speed limits  $v_{\max}=27.78 m/s$ and $v_{\min}=10 m/s$.
Since the physical parameters of CAV platoons as well as algorithm and control design depend heavily on vehicle types, we consider the following three types of CAV platoons: (i) a homogeneous small-size CAV platoon; (ii) a heterogeneous medium-size CAV platoon; and (iii) a homogeneous large-size CAV platoon. Identical minimal (resp. maximal) values of nonlinear dynamics coefficients $c_{2, i}$'s and $c_{3, i}$'s are chosen for the homogeneous small-size (resp. large-size) CAV platoon, whereas inhomogeneous values of $c_{2, i}$'s and $c_{3, i}$'s are chosen for the heterogeneous medius-size CAV platoon. Other parameters for the CAVs and their constraints, i.e., the vehicle length $L_i$, the reaction time $r_i$, the acceleration and deceleration limits $a_{i, \max}$ and $a_{i,\min}$, and the desired spacing $\Delta$, are chosen accordingly. See Tables~\ref{Table01}-\ref{Table02} for the values of these parameters \cite{GShenDu_TRB16, ZhengLiBorrelliH_TCS16}.

\begin{table}[h!]
\caption{Physical parameters for homogeneous small-size and large-size CAV platoons  }
\newcommand{\tabincell}[2]{\begin{tabular}{@{}#1@{}}#2\end{tabular}}
\begin{center}
\begin{tabular}{|c|c|c|c|c|c|c|c|}
\hline
 & $L_i (m)$  & $r_i (s)$ & $a_{i, \min} (m/s^2)$ & $a_{i, \max}(m/s^2)$ & $c_{2, i} (\times 10^{-4})$ & $c_{3, i} (\times 10^{-2})$ & $\Delta (m)$ \\
\hline

Small-size &  $5$  &  $1.0 $   &  $-8$  &  $1.4$   & $2.5$  & $0.6$ & $50$  \\
\hline
Large-size & $10$ & $1.25$ & $-6.8$ & $1.4$ & $4.5$ & $1.5$ & $65$ \\
\hline
%
%
\end{tabular}
\end{center}
\label{Table01}
\end{table}

\begin{table}
\caption{Physical parameters for a heterogeneous medium-size CAV platoon with $\Delta = 60 m$}
\newcommand{\tabincell}[2]{\begin{tabular}{@{}#1@{}}#2\end{tabular}}
\begin{center}
\begin{tabular}{|c|c|c|c|c|c|c|c|c| c| c|}
\hline
 & $i=1$  & $i=2$ & $i=3$ & $i=4$ & $i=5$ & $i=6$ & $i=7$ & $i=8$ & $i=9$ & $i=10$ \\
\hline

$L_i (m)$ &  $7 $  &  $7 $   &  $7$  &  $7$   & $7$  & $7$ & $7$ & $7$  & $7$ & $7$ \\
\hline
$r_i (s)$ & $1.21$ & $1.155$ & $0.99$ & $1.045$ & $1.21$ & $1.155$ & $0.99$ & $1.045$ & $1.155$ & $1.045$ \\
\hline
$a_{i, \min} (m/s^2)$ & $-8.14$ & $-7.77$ & $-6.66$ & $-7.03$ & $-8.14$ & $-7.77$ & $-6.66$ & $-7.03$ & $-7.77$ & $-7.03$ \\ \hline
$a_{i, \max} (m/s^2)$ & $1.4$ & $1.4$ & $1.4$ & $1.4$ & $1.4$ & $1.4$ & $1.4$ & $1.4$ & $1.4$ & $1.4$ \\ \hline
$c_{2, i} (\times 10^{-4}) $ &  $3.85 $  &  $3.675$   &  $3.15$  &  $3.325$   & $3.85$  & $3.675$ & $3.15$ & $3.325 $ & $3.675$ & $3.325$  \\ \hline
$c_{3, i} (\times 10^{-2})$ &  $1.155$  &  $1.103$   &  $0.945$  &  $0.998$   & $1.155$  & $1.103$ & $0.945$ & $0.998$ & $1.103$ & $0.998$  \\ \hline
%
%
\end{tabular}
\end{center}
\label{Table02}
\end{table}

The initial state of each CAV platoon is $z(0)=z'(0)=0$ and $v_i(0)=25 m/s$ for all $i=0, 1, \ldots, n$.
The cyclic-like graph is considered for the vehicle communication network, i.e., the  bidirectional edges of the graph are $(1, 2), (2, 3), \ldots, (n-1, n) \in \mathcal E$.
Following the discussions in \cite[Section 6]{ShenEswarDu_Arx20}, we choose the MPC horizon $p$ as $1\le p\le 5$.

We present the choices of weight matrices for each of the abovementioned three CAV platoons.
%
%
Define
\begin{eqnarray*}
 \wt \alphabf & := & \big( 38.85, 40.2, 41.55,   42.90,   44.25,   45.60,  46.95, 48.30,   49.65, 51.00 \big) \in \mathbb R^{10}, \\
  \wt \betabf & := & \big( 130.61,   136.21,   141.82,   147.42,   153.03,   158.64,   164.24,   169.85,  175.46,  181.06 \big) \in \mathbb R^{10}, \\
  \wt \zetabf & := & \big( 62,    74,    90,    92,   106,   194,   298,   402,   454,   480 \big) \in \mathbb R^{10}.
\end{eqnarray*}
For all the three CAV platoons, $\alphabf^1 = 6 \wt \alphabf$, $\betabf^1=\wt \betabf$, and $\zetabf^1=0.5\wt \zetabf$ when $p=1$. %
\begin{itemize}
 \item Homogeneous small-size and heterogenous medium-size CAV platoons: for $p \ge 2$, $\alphabf^1 = 9(\wt\alphabf - \mathbf 1)$, $\betabf^1 = \wt \betabf - \mathbf 1$, $\zetabf^1 = 0.5(\wt\zetabf - \mathbf 1)$, and
\[
   \alphabf^s = \frac{0.1368}{(s-1)^4} \times \wt \alphabf, \quad
   \betabf^s = \frac{0.044}{(s-1)^4} \times \wt \betabf, \quad
   \zetabf^s = \frac{0.0013}{(s-1)^4} \times \wt \zetabf, \quad s=2, \ldots, \min(p, 3).
\]
 \item Homogeneous large-size CAV platoon: for $p \ge 2$, $\alphabf^1 = 6(\wt\alphabf - \mathbf 1)$, $\betabf^1 = \wt \betabf - \mathbf 1$, $\zetabf^1 = 0.5(\wt\zetabf - \mathbf 1)$, and
\[
   \alphabf^s = \frac{0.0684}{(s-1)^4} \times \wt \alphabf, \quad
   \betabf^s = \frac{0.044}{(s-1)^4} \times \wt \betabf, \quad
   \zetabf^s = \frac{0.0013}{(s-1)^4} \times \wt \zetabf, \quad s=2, \ldots, \min(p, 3).
\]
\end{itemize}
And for all the above CAV platoons: for $p=4, 5$,
\[
   \alphabf^s = \frac{0.0228}{(s-1)^4} \times \wt \alphabf, \quad
   \betabf^s = \frac{0.044}{(s-1)^4} \times \wt \betabf, \quad
   \zetabf^s = \frac{0.0026}{(s-1)^4} \times \wt \zetabf, \quad s=4, \ldots, p.
\]

The above vectors $\alphabf^s, \betabf^s, \zetabf^s$ define the weight matrices $Q_{z, s}, Q_{z', s}, Q_{w, s}$ for $s=1, \ldots, 5$, which further yield the closed loop dynamics matrix $A_{\mbox{c}}$; see the discussions below (\ref{eqn:K_mat_d_vec}).
It is shown that when these weights are used, $A_{\mbox{c}}$ is Schur stable for each $p=1, \ldots, 5$ and each CAV platoon.

\gap

\mycut{
\noindent{\bf Discussion on the selection of MPC horizon}. We discuss the choice of the MPC prediction horizon $p$ based on numerical tests as follows. Our numerical experience shows that for $p>1$, the weight matrices $Q_{z, 1}, Q_{z', 1}$ and $Q_{w, 1}$ play a more important role for the closed loop dynamics. For fixed $Q_{z, 1}, Q_{z', 1}$ and $Q_{w, 1}$ with the large penalties in $Q_{z, s}, Q_{z', s}$ and $Q_{w, s}$ for $s>1$, the closed loop dynamics may be mildly improved but at the expense of undesired large control. Hence, we choose smaller penalties in $Q_{z, s}, Q_{z', s}$ and $Q_{w, s}$ for $s>1$, which only lead to slightly better closed loop performance compared with the case of $p=1$. Further, when a large $p$ is used, the underlying optimization problem has a larger size, resulting in longer computation time and slow convergence of the proposed distributed scheme. Besides, the current MPC model assumes that the future $u_0(k+s)=u_0(k)$ for all $s=1, \ldots, p-1$ at each $k$. This assumption is invalid when the true $u_0(k+s)$ is substantially different from $u_0(k)$, which implies that the prediction performance is poor for a large $p$.
For these reasons, it is recommended that a smaller $p$ be used, for example, $p \le 5$.
}

The following scenarios  are used to evaluate the proposed CAV platooning control.

\noindent $\bullet$ {\bf Scenario 1}: The leading vehicle performs instantaneous deceleration/acceleration and then maintains a constant speed for a certain time of period. This scenario aims to test if a platoon can keep stable spacing and speed when the leading vehicle yields acceleration or deceleration disturbances. Specifically, the motion profile of the leading vehicle is: it decelerates from $k=51s$ to $k=54s$ with the deceleration $-2m/s^2$, and maintains a constant speed till $k=100s$. After $k=100s$, it restores to its original speed $25m/s$ with the acceleration $1m/s^2$.

\gap

\noindent $\bullet$ {\bf Scenario 2}: The leading vehicle performs periodical acceleration/deceleration. This scenario aims to test whether the proposed control scheme can reduce periodical spacing and speed fluctuation. The motion profile of the leading vehicle in this scenario is as follows: the leading vehicle periodically changes its acceleration and deceleration from $k=51s$ to $k=100s$ with the period $T=4s$ and acceleration/deceleration $\pm 1m/s^2$. Then  it maintains its original constant speed $25m/s$ after $k=100 s$.

\gap

\noindent $\bullet$ {\bf Scenario 3}:  This scenario aims to test the performance of the proposed control scheme in a real traffic environment, particularly when the leading vehicle undergoes traffic oscillations. We use real world trajectory data from an oscillating traffic flow to generate the leading vehicle's motion profile. Specifically, we consider NGSIM data on eastbound I-80 in San Francisco Bay area in California, and the data of position and speed of a real vehicle, which is treated as a leading vehicle, is used to generate control input at each sample time. Since the maximum of acceleration of this vehicle is nearly $1.8 m/s^2$, we let
$a_{i, \max}= 1.8 m/s^2$ for all $i$; the other parameters or physical limits are same as before.
%

%
\subsection{Performance of the Proposed Fully Distributed Scheme} \label{subsect:test_performance}

As indicated in Section~\ref{subsect:SCP_scheme}, when $p=1$, the underlying MPC optimization problem  (\ref{eqn:MPC_model_locally_coupled}) is a convex QCQP, for which the generalized Douglas-Rachford splitting method based fully distributed algorithm developed in \cite{ShenEswarDu_Arx20} is used. In what follows, we focus on $p>1$.

When $p>1$, the underlying MPC optimization problem  (\ref{eqn:MPC_model_locally_coupled}) is nonconvex, and the sequential convex programming and Douglas-Rachford splitting method based fully distributed scheme is applied (cf. Algorithm~\ref{algo:distributed_SCP}). To apply this algorithm, we discuss the choices of the smooth functions $g_{i, s}$ and the convex function $r_{i, s}$ for the (approximated) nonconvex constraint sets $\mathcal Y_i$ and $\mathcal Z_i$, where $i=1, \ldots, n$. In view of the definition of $\mathcal Y_i$ given before (\ref{eqn:_q_i_j}), we see that $\mathcal Y_i = \{ \ubld_i \, | \,  v_{\min} - q_{i, j}(\ubld_i) \le 0, \ q_{i, j}(\ubld_i) - v_{\max} \le 0, \ j=1, \ldots, p \}$, where $q_{i, j}(\cdot)$ is given by (\ref{eqn:_q_i_j}). Define
$
 g_{i, s}(\ubld_i) :=  v_{\min} - q_{i, j}(\ubld_i)$, and $ r_{i, s}(\ubld_i) :\equiv 0$ for $s=1, \ldots, p$;  $g_{i, s}(\ubld_i) : \equiv 0$, and $r_{i, s}(\ubld_i) :=  -  q_{i, j}(\ubld_i) + v_{\max}$ for $s=p+1, \ldots, 2p$. Then $\mathcal Y_i = \{ \ubld_i \, | \, g_{i, s}(\ubld_i) - r_{i, s}(\ubld_i) \le 0, \ s=1, \ldots, 2p \}$. Similarly, let $g'_{i, s}(\ubld_{i-1}, \ubld_i)$ be the right hand side of (\ref{eqn:H_i_j_approx}), and $r'_{i, s}(\ubld_{i-1}, \ubld_i) \equiv 0$. Then $\mathcal Z_i = \{ \wh \ubld_i \, | \,  g'_{i, s}(\wh \ubld_i) - r'_{i, s}(\wh \ubld_i) \le 0, \ s=1, \ldots, p \}$. The gradient of these functions are given in Section~\ref{subsect:approximation_functions}. Furthermore, the Lipschitz constants $L_{J_i}$'s and $L_{g_{i, s}}$'s are given by $\nu_p \| H J_i (\wh \ubld_i)\|_2$ and $0.9 \| H g_{i, s} (\wh \ubld_i)\|_2$, where $\nu_p=0.8$ for $p=2, 3$ and $\nu_p=0.9$ for $p=4, 5$ respectively, and  $H f$ denotes the Hessian of a  real-valued smooth function $f$. The reasons for each Hessian scaled by these factors  are twofold: (i) the 2-norm of Hessian is conservative; and (ii) the scaled Hessian leads to faster convergence.

\gap

\noindent {\bf Initial guess warm-up}. \ To achieve real-time computation of the proposed distributed scheme (i.e., Algorithm~\ref{algo:distributed_SCP}), we exploit the initial guess warm-up technique for both the linear stage (cf. Line 2) and the inner loop of the SCP-DR stage (cf. Lines 6-14). For the former stage, see \cite[Section 6.2]{ShenEswarDu_Arx20} for its warm-up scheme. We discuss a warm-up scheme for the latter stage. Recall that the inner loop solves the following convex optimization problem: $\min_{y=(y_i) \in \mathcal A } \sum^n_{i=1} f_i(y_i) + \deltabf \mathcal C_i(y_i)$, where for each $i=1, \ldots, n$, $f_i(y_i):= J_i(\wh\ubld^k_i) + d^T_{J_i}(\wh\ubld^k_i) ( y_i - \wh\ubld_i) + \frac{L_{J_i}}{2} \|   y_i - \wh\ubld^k_i \|^2_2$, and $\mathcal C_i$ is the intersection of the box-constraint set $\mathcal X_i$ corresponding to the control constraint and a quadratically constrained convex set corresponding to the (approximated) velocity and safety distance constraints; see Section~\ref{subsect:SCP_scheme} for details. Since the (approximated) velocity and safety distance constraints are often inactive, we replace $\mathcal C_i$ by $\mathcal X_i$ in a warm-up scheme. Further,
 the generalized Douglas-Rachford scheme given by (\ref{eqn:DR_scheme}) is used to solve $\min_{y=(y_i) \in \mathcal A } \sum^n_{i=1} f_i(y_i) + \deltabf \mathcal X_i(y_i)$ in a fully distributed manner by replacing $\mathcal C_i$ by $\mathcal X_i$. Since $f_i$ and the box constraint set $\mathcal X_i$ are fully decoupled, solving the proximal operator based optimization problem in this scheme boils down to solving finitely many decoupled univariate optimization problems of the form: $\min_{t\in [c, d]} a t^2 + b t+ e$, where $t \in \mathbb R$, and $a, b, c, d, e \in \mathbb R$ are given constants with $a>0$. Such a univariate optimization problem has a simple closed-form solution, which considerably reduces computation load of the Douglas-Rachford scheme. Numerical tests show that the proposed warm-up scheme significantly improves computation time and solution quality.

\gap


\begin{table}
\caption{Error tolerances for outer and inner loops at different MPC horizon $p$'s}
\newcommand{\tabincell}[2]{\begin{tabular}{@{}#1@{}}#2\end{tabular}}
\begin{center}
\begin{tabular}{|c|c|c|c|c|c|}
\hline
 & \ \ $p=1$ \ \ & $p=2$ & $p=3$ & $p=4$ & $p=5$ \\
\hline
Outer loop &  $2.5\times 10^{-3}$   &  $6.5\times 10^{-3}$   &  $7.5\times 10^{-3}$  &  $1.0\times 10^{-2}$   &  $1.25\times 10^{-2}$  \\ \hline
Inner loop & NA    & $4.0\times 10^{-3}$    & $5.0\times 10^{-3}$  & $7.5\times 10^{-3}$    &   $1.0\times 10^{-2}$  \\
\hline
%
\end{tabular}
\end{center}
\label{Table03}
\end{table}


\noindent{\bf Performance of distributed schemes}.
We implement the proposed fully distributed algorithms via MATLAB on a computer with 4-cores processor: Intel(R) Core(TM) i7-8550U CPU @ $1.80 GHz$ and RAM: $16.0 GB$.
These distributed  algorithm are tested for the three types CAV platoons, namely homogeneous small-size and large-size CAV platoons and a heterogeneous medium-size CAV platoon, on Scenarios 1-3 for different MPC horizon $p$'s. The proposed initial guess warm-up schemes are used with the error tolerance give by $10^{-7}$  for all the cases. Moreover, we choose $\alpha =0.9$ and $\rho=0.1$ for the proximal operator based Douglas-Rachford scheme in all of these algorithms. 
Further,  the stopping criteria are characterized by the minimum of absolute and relative errors of two neighboring iterates for $p=2, 3$, whereas  for $p=4, 5$, these criteria are characterized by absolute errors of two neighboring iterate.
The list of error tolerances for the outer and  inner loop (for all the three types of CAV platoons) at $p$'s is shown in Table~\ref{Table03}. Note that there is no inner loop when $p=1$, since the underlying MPC optimization problem is a convex QCQP and solved via the fully distributed scheme given in  \cite{ShenEswarDu_Arx20}.
%
%
%
%
 A summary of mean and variance of computation time per CAV for different CAV platoons with different $p$'s on the three scenarios is displayed in Tables~\ref{Table04}-\ref{Table06}.
Moreover, to evaluate the numerical accuracy of the proposed schemes for $p=1$, we compute the relative error between the numerical solution from the distributed schemes and that from a high precision centralized scheme when the latter solution, treated as a true solution, is nonzero; see Tables~\ref{Table08}. Note that for $p \ge 2$, a true solution is hard to compute even in a centralized manner.


\begin{table}[h!]
\caption{Scenario 1: computation time per CAV (sec)}
\newcommand{\tabincell}[2]{\begin{tabular}{@{}#1@{}}#2\end{tabular}}
\begin{center}
\begin{tabular}{|c|c|c|c|c|c|c|}
\hline
\multirow{2}{*}{MPC horizon} &
\multicolumn{2}{c|}{Small-size} &
\multicolumn{2}{c|}{Medium-size} &
\multicolumn{2}{c|}{Large-size}\\
\cline{2-7}
   & \ \ Mean \ \ & Variance & \ \ Mean \ \ & \ Variance \  & \ \ Mean \ \ & \ Variance \ \\
\hline
$p=1$ &   0.0952   & $3.43\times 10^{-4}$ &  0.1333 & \ $1.44\times 10^{-4}$ \ & 0.1087 & $2.32\times 10^{-4}$  \\
\hline
$p=2$ & 0.1616 & $1.6\times 10^{-3}$ & 0.2795 & $4.5\times 10^{-3}$ & 0.2759 & $1.6\times 10^{-3}$ \\
\hline
$p=3$ & 0.1721 & $1.40\times 10^{-3}$ & 0.2673 & $4.11\times 10^{-3}$ & 0.2667 & $2.35 \times 10^{-3}$ \\
\hline
$p=4$ & 0.1665 & $6.33\times 10^{-4}$ & 0.2535   & $2.02\times 10^{-3}$ & 0.3038 & 0.1440 \\
\hline
$p=5$ & 0.2243 &  0.2340 & 0.3056 & $0.4440$ & 0.3296 & 0.4240 \\
\hline
\end{tabular}
\end{center}
\label{Table04}
\end{table}

\begin{table}[ht!]
\caption{Scenario 2: computation time per CAV (sec)}
\newcommand{\tabincell}[2]{\begin{tabular}{@{}#1@{}}#2\end{tabular}}
\begin{center}
\begin{tabular}{|c|c|c|c|c|c|c|}
\hline
\multirow{2}{*}{MPC horizon} &
\multicolumn{2}{c|}{Small-size} &
\multicolumn{2}{c|}{Medium-size} &
\multicolumn{2}{c|}{Large-size} \\
\cline{2-7}
   & \ \ Mean \ \ & Variance & \ \ Mean \ \ & Variance & \ \ Mean \ \ & Variance \\
\hline
$p=1$ & \  0.1098 \ & $7.33\times 10^{-4}$ & \ 0.1421 \ & \ $2.55\times 10^{-4}$ & 0.1204 & $7.54\times 10^{-4}$\\
\hline
$p=2$ & 0.1771 & $1.2\times 10^{-3}$ & 0.2857 & $6.7\times 10^{-3}$ & 0.2814 & $3.3\times 10^{-3}$ \\
\hline
$p=3$ & 0.1939 & $4.83\times 10^{-4}$ & 0.2804 & $2.78\times 10^{-3}$ & 0.2734 & $1.62\times 10^{-3}$ \\
\hline
$p=4$ & 0.2241 & $9.58\times 10^{-3}$ & 0.3165 & $5.93\times 10^{-3}$ & 0.3681 & 0.0241 \\
\hline
$p=5$ & 0.2418 & 0.0113 & 0.3051 & 0.0109 &  0.3449 & 0.0150 \\
\hline
\end{tabular}
\end{center}
\label{Table05}
\end{table}

\begin{table}[ht!]
\caption{Scenario 3: computation time per CAV (sec) }
\newcommand{\tabincell}[2]{\begin{tabular}{@{}#1@{}}#2\end{tabular}}
\begin{center}
\begin{tabular}{|c|c|c|c|c|c|c|}
\hline
\multirow{2}{*}{MPC horizon} &
\multicolumn{2}{c|}{Small-size} &
\multicolumn{2}{c|}{Medium-size} &
\multicolumn{2}{c|}{Large-size} \\
\cline{2-7}
   & \ \ Mean \ \  & Variance & \ \ Mean \ \ & Variance & \ \ Mean \ \ & Variance \\
\hline
$p=1$ &  0.1109  & $1.2\times 10^{-3}$ & \ 0.1408 \ & \ $4.09\times 10^{-4}$ \ & 0.1125 & $8.38\times 10^{-4}$ \\
\hline
$p=2$ & 0.1559 & $2.3\times 10^{-3}$ & 0.2528 & $6.3\times 10^{-3}$ & 0.2503 & $4.8\times 10^{-3}$ \\
\hline
$p=3$ & 0.2257 & 0.1320 & 0.2398 & $4.91\times 10^{-3}$ & 0.2437 & $3.98\times 10^{-3}$\\
\hline
$p=4$ & 0.2053 & $7.73\times 10^{-3}$ & 0.2883 & $9.73\times 10^{-3}$ & 0.3216 & 0.0132 \\
\hline
$p=5$ & 0.2256 & 0.0136 & 0.2882 & 0.0135 & 0.3250 & 0.0249 \\
\hline
\end{tabular}
\end{center}
\label{Table06}
\end{table}


\begin{table}
\caption{Relative numerical error for $p=1$}
\newcommand{\tabincell}[2]{\begin{tabular}{@{}#1@{}}#2\end{tabular}}
\begin{center}
\begin{tabular}{|c|c|c|c|c|c|c|}
\hline
\multirow{2}{*}{Scenarios} &
\multicolumn{2}{c|}{Small-size} &
\multicolumn{2}{c|}{Medium-size} &
\multicolumn{2}{c|}{Large-size}\\
\cline{2-7}
   & \ \ Mean \ \ & Variance & \ \ Mean \ \ & \ Variance \  & \ \ Mean \ \ & \ Variance \ \\
\hline
Sc. 1 &   $1.07\times 10^{-3}$  & $1.44\times 10^{-7}$ &  $5.66\times 10^{-4}$ & \ $1.24\times 10^{-6}$ \ & $5.29\times 10^{-4}$ & $5.14\times 10^{-8}$  \\
\hline
Sc. 2 & $9.11\times 10^{-4}$ & $3.82\times 10^{-6}$ & $1.11\times 10^{-3}$ & $7.54 \times 10^{-6}$ & $4.38 \times 10^{-4}$ & $2.73\times 10^{-8}$ \\
\hline
Sc. 3 & $1.47\times 10^{-3}$ & $3.51\times 10^{-6}$ & $6.85\times 10^{-4}$ & $8.41\times 10^{-7}$ & $5.85\times 10^{-4}$ & $2.45 \times 10^{-7}$ \\
\hline
\end{tabular}
\end{center}
\label{Table08}
\end{table}

The numerical results show that for each $p$ and each CAV platoon type, the mean computation time is less than $0.369 s$  and thus less than the reaction time $r_i$ or sample time $\tau$ with overall fairly small variances, for all the three scenarios. Indeed, the computation time for $p=1$ is the least and becomes larger for a higher $p$ for most cases. 
Further, the computation times vary for different CAV vehicle types. In particular, the computation time for the small-size platoon is less than that of the medium-size and the large-size platoons.
%
%
This is because the nonlinear effects play an increasing  role in the latter types of platoons when $c_{2, i}$ and $c_{3, i}$ become larger. Additionally, the heterogeneous dynamics in the middle-size platoon also require more computation.
%
%
%
Besides, the numerical accuracy is satisfactory for $p=1$ as shown in Table~\ref{Table08}.
 Hence, we conclude that the proposed distributed schemes are suitable for real-time computation of a heterogenous or homogeneous CAV platoon with satisfactory numerical precision.

\subsection{Performance of CAV Platooning Control}


\begin{table}
\caption{Maximum steady state error of spacing (m)}
\newcommand{\tabincell}[2]{\begin{tabular}{@{}#1@{}}#2\end{tabular}}
\begin{center}
\begin{tabular}{|c|c|c|c|c|c|c|}
\hline
\multirow{2}{*}{MPC horizon} &
\multicolumn{3}{c|}{Scenarios 1-2 (among 10 vehicles)} &
\multicolumn{3}{c|}{Scenario 3 (among 9 vehicles)} \\
\cline{2-7}
   &  Small-size & Medium-size & Large-size &  Small-size & Medium-size & Large-size \\
\hline
$p=1$ & 0.0571 & 0.0941 & 0.1138 & 0 &  0.0431 & 0  \\
\hline
$p=2$ & 0.1107 & 0.1507 & 0.2738 & 0.1003 & 0.1277 & 0.2199\\
\hline
$p=3$ & 0.1122 & 0.1530 & 0.2793 & 0.1051 & 0.1369 & 0.2374 \\
\hline
$p=4$ & 0.1321 & 0.1687 & 0.2831 & 0.1142 & 0.1363 & 0.2106 \\
\hline
$p=5$ & 0.1438 & 0.1880 & 0.3052 & 0.1438 & 0.1538 & 0.2342 \\
\hline
\end{tabular}
\end{center}
\label{Table07}
\end{table}

We evaluate the closed-loop performance of the proposed CAV platooning control with different MPC horizon $p$'s for different CAV platoons on the three scenarios mentioned before. For each CAV platoon and scenario, we consider the spacing between two neighboring vehicles (i.e., $S_{i-1, i}(k):=x_{i-1}(k) - x_{i}(k)=z_{i}(k)+\Delta$), the vehicle speed $v_i(k)$, and the control input $u_i(k), \, i=1, \ldots, n$ for $p=1, 2, 3, 4, 5$.

\gap

\noindent {\bf Steady state error}. When $(c_{2, i}, c_{3, i}) \ne 0$ and $u_0(k)=0$ and $v_0(k)=v_{0,\infty}>0$ for  all large $k$,  it is observed from the numerical tests that  when the closed-loop dynamics of the CAV platoon reaches its steady state $(z_{ss}, z'_{ss}) \in \mathbb R^n \times \mathbb R^n$, i.e., $(z(k), z'(k))$ becomes the constant vector $(z_{ss}, z'_{ss})$  for all large $k$, $z_{ss}$ is nonzero. Physically, the nonzero steady state is due to nonlinear vehicle dynamics and the PD-like control structure of the MPC control scheme. To illustrate this phenomenon, consider the closed-loop dynamics in (\ref{eqn:closed_loop_p=1}) for $p=1$. (It is noted that only for $p=1$, we have a closed-form expression for the closed-loop dynamics.) Since $(z_i(k), z'_i(k))$ is constant for all large $k$ when a CAV platoon reaches its steady state, it follows from (\ref{eqn:z_z'_closed_loop}) that $
 w_i(k) - [c_{2, i-1} v^2_{i-1}(k) - c_{2, i} v^2_i(k) ] - [c_{3, i-1} -c_{3, i}] g=0$ for all large $k$. It is easy to see from (\ref{eqn:z_z'_closed_loop}) that $z'_{ss}=0$. Let $\zbld_{ss} := (z_{ss}, z'_{ss})=(z_{ss}, 0)$.  In view of $u_0(k)=0$ for all large $k$, we deduce via (\ref{eqn:closed_loop_p=1}) that $\zbld_{ss} = \big( A_{\mbox{c}} + v_{0, \infty} \Delta \bar A (\varphibf_d) \big) \zbld_{ss} - B\wh W Q_w w_{e, \infty}$, where $w_{e, \infty}$ is defined in the way of $w_e(k)$ by setting $v_0(k) \equiv v_{0, \infty}$. Since $\zbld_{ss}=(z_{ss}, 0)$, we see that $(I -  A_{\mbox{c}}) \zbld_{ss} = - B \wh W Q_w  w_{e, \infty}$. By the expression for $A_{\mbox{c}}$ given in (\ref{eqn:linear_close_loop_p=1}), we obtain $\frac{1}{2} B \wh W Q_z z_{ss} = - B \wh W Q_w w_{e, \infty}$. Since $B \wh W$ has full column rank and $Q_z, Q_w$ are diagonal PD, we further have $\frac{1}{2} Q_z z_{ss} = - Q_w w_{e,\infty}$ or equivalently $z_{ss} = -2 Q^{-1}_{z}Q_w w_{e,\infty}$. Since $w_{e,\infty} \ne 0$, we conclude that $z_{ss} \ne 0$. Note that $w_{e, \infty}$ depends on $c_{2, i-1}-c_{2, i}$ and $c_{3, i-1} -c_{3, i}$ and $v_{0, \infty}$ with  $c_{2,0}=c_{3, 0}=0$. Hence, if a CAV platoon is homogeneous, then $w_{e, \infty}$, and thus $z_{ss}$, is a multiple of $\mathbf e_1$. Moreover, in light of $Q^{-1}_{z}Q_w = \mbox{diag}(\frac{\zeta_1}{\alpha_1}, \ldots, \frac{\zeta_n}{\alpha_n})$, it is observed that $|(z_{ss})_i|$ will be smaller for a large $\alpha_i$ and  a small $\zeta_i$. This observation agrees with numerical results. In addition, when $p \ge 2$, similar observations are made although the closed form expression of $z_{ss}$ is hard to obtain. The maximum steady state error of spacing, i.e., $\|z_{ss}\|_\infty$, is displayed in Table~\ref{Table07} for different CAV platoons, different $p$'s, and the three scenarios. Note that in Scenario 3, we consider the steady state errors for $S_{i-1, i}, \, i=2, 3, \ldots, 9$ since $S_{0, 1}$ does not reach its steady state in this scenario.



%

We present the closed-loop performance only for $p=1$ and $p=5$  for each type of CAV platoons in each scenario because of the length limit; see  Figures~\ref{Fig:S1_small}-\ref{Fig:S3_large}.
%
%
%
The closed-loop performance in each scenario is commented as follows:
\begin{itemize}
  \item [(i)] Scenario 1.
  %
  %
   Figures~\ref{Fig:S1_small}, \ref{Fig:S1_medium}, and \ref{Fig:S1_large} show the MPC control performance of the homogenous small-size CAV platoon, the heterogeneous medium-size CAV platoon, and the homogenous large-size CAV platoon in Scenario 1, respectively. It can be seen that the spacing between the leading vehicle and the first CAV, i.e., $S_{0, 1}$, in all the CAV platoons  has small deviations (less than $0.5m$) from the desired spacing $\Delta$ when the leading vehicle takes instantaneous acceleration or deceleration. Further, when $p=1$, the spacings between the other CAVs in the two homogeneous CAV platoons  remain the desired constant $\Delta$, and there are small deviations from the desired spacing $\Delta$ for the other CAVs in the heterogeneous CAV platoon or the two homogeneous CAV platoons when $p=5$. In all the cases, the convergence to the steady states is fast (within 15 secs) and the steady state errors in spacing are nonzero but are small; see Table~\ref{Table07}. In fact, the maximum steady state errors increase as $p$ becomes larger;  compared with the desired spacing $\Delta=50 m$, $60 m$ or $65 m$, the largest relative error $\frac{\|z_{ss}\|_\infty }{\Delta } \le 0.47\%$ for all the three types of CAV platoons.
   %
   %
    %
   Lastly, the time history of  speed and control input demonstrates satisfactory performance. In particular, it is observed that all the CAVs show the same speed change and almost identical control, implying that the CAV platoon performs a nearly coordinated motion under the proposed platooning control.
   %

 \item [(ii)]
Scenario 2.   Figures~\ref{Fig:S2_small}, \ref{Fig:S2_medium}, and \ref{Fig:S2_large} display the MPC control performance of the homogenous small-size, the heterogeneous medium-size, and the homogenous large-size CAV platoons in Scenario 2, respectively, where the leading vehicle undertakes  periodic acceleration/deceleration. In all the cases, $S_{0, 1}$ demonstrates the largest fluctuations whose maximum magnitude of deviations is $0.25m$ when $\Delta = 50m$, $0.3 m$ when $\Delta = 60m$,
and $0.5 m$ when $\Delta = 65m$. Besides, all the CAV platoons demonstrates nearly coordinated motions. For example, when $p=1$,  the spacings $S_{i-1, i}$ for $i=2, \ldots, 10$ remain the desired constant for the two homogeneous CAV platoons, and they have small deviations from the desired spacing for the heterogeneous CAV platoon and  $p=5$ of the two homogeneous CAV platoons.
Moreover, the fluctuations of $S_{0, 1}$ and other $S_{i, i+1}$'s  quickly converge to their steady states within $15s$ when the  leading vehicle stops its periodical acceleration.  The steady state errors in spacing are as same as those in Scenario 1.
 The time history of speed and control input shows nearly identical behaviors for all the CAVs in each case.
%

%
%

  \item [(iii)] Scenario 3. Figures~\ref{Fig:S3_small}, \ref{Fig:S3_medium}, and \ref{Fig:S3_large} show the control performance of the homogenous small-size, the heterogeneous medium-size, and the homogenous large-size CAV platoons in Scenario 3, respectively, where the leading vehicle
      undergoes various traffic oscillations through the time window of $45s$. It is observed that $S_{0, 1}$ demonstrates the largest spacing variations with the maximum magnitude less than or equal to $0.25 m$ when $\Delta=50 m$, $0.3 m$ when $\Delta=60 m$, and $0.46 m$ when $\Delta=65 m$; the other spacings $S_{i-1, i}, \, i=2, \ldots, 10$  either are the desired constant or demonstrate nearly constant deviations with maximum magnitude less than $0.14 m$, in spite of the oscillation of $S_{0, 1}$. Further, the spacings $S_{i-1, i}, \, i=2, \ldots, 10$ almost reach steady states between $5s$ and $25 s$ and after $k=35$. The maximum steady state errors of these spacings are shown in Table~\ref{Table07}.
      It is seen that the maximum steady state error often appears at $S_{1,2}$. Compared with the desired spacing $\Delta=50 m$, $60 m$ or $65 m$, the largest relative error $\frac{\|z_{ss}\|_\infty }{\Delta } \le 0.37\%$ for all the three types of CAV platoons in Scenario 3. Finally, all the CAV platoons demonstrates  nearly coordinated motions.

\end{itemize}
%


Consequently, the proposed platooning control effectively mitigates traffic oscillations of the spacing and vehicle speed of the CAV platoons of different types with small or almost negligible steady state errors. In fact, it achieves nearly consensus motions of the entire CAV platoons even under some perturbations.
%

%
%

%
\begin{figure}[htbp]
\centering
\subfigure[Time history of spacing changes.]{ \label{S1_p1_spacing}
\includegraphics[width=0.48\columnwidth]{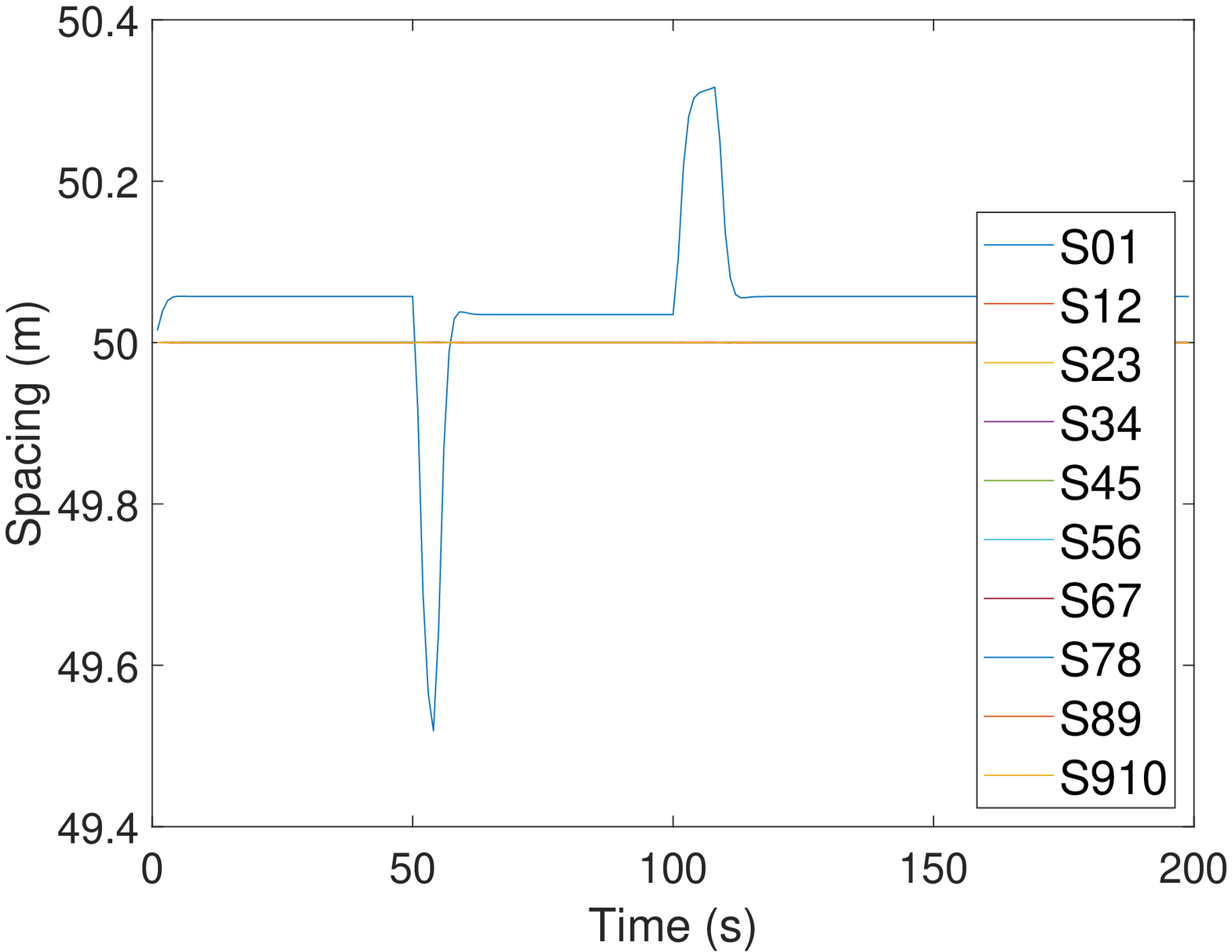}
}
\subfigure[Time history of spacing changes.] { \label{S1_p5_spacing}
\includegraphics[width=0.48\columnwidth]{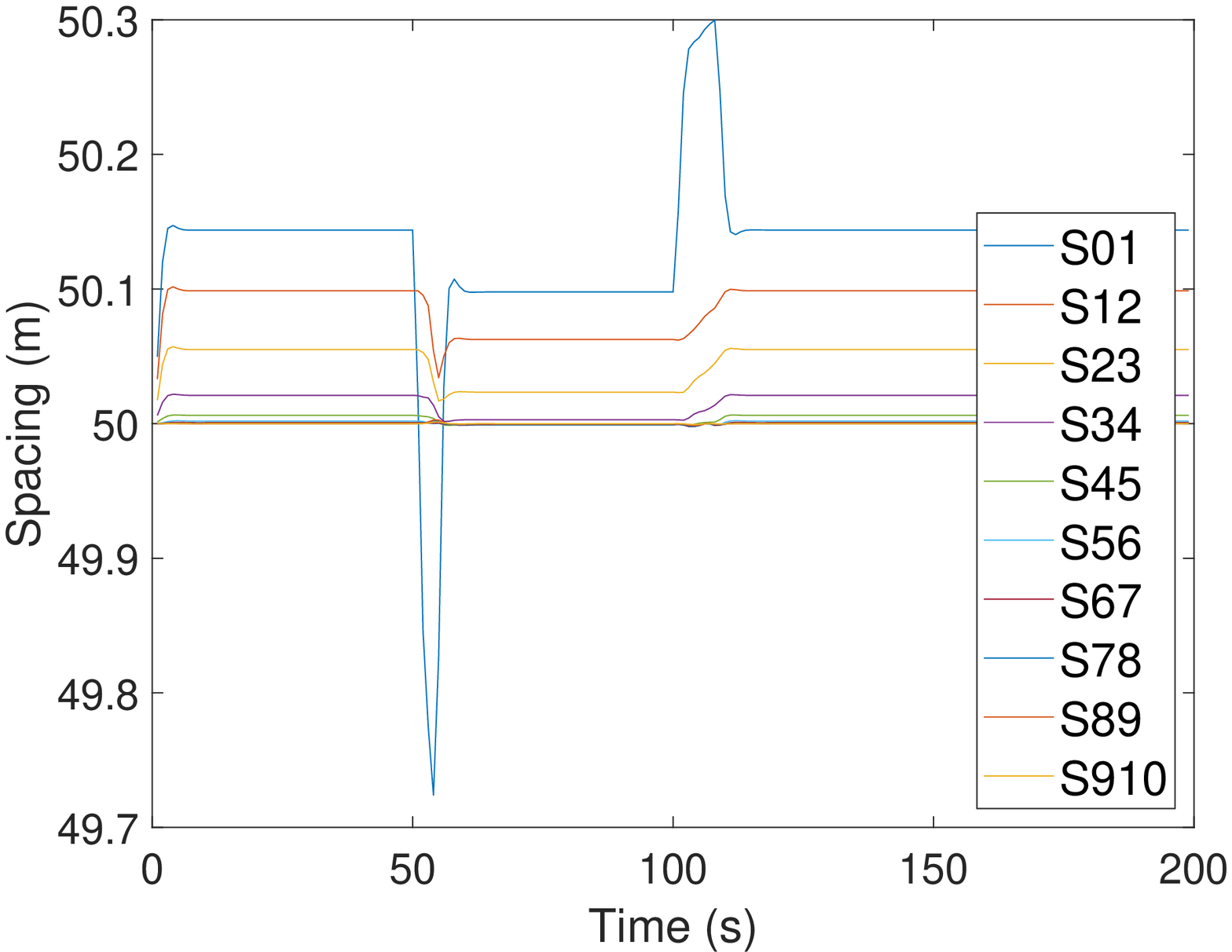}
}
\subfigure[Time history of vehicle speed. ]{ \label{S1_p1_speed}
\includegraphics[width=0.48\columnwidth]{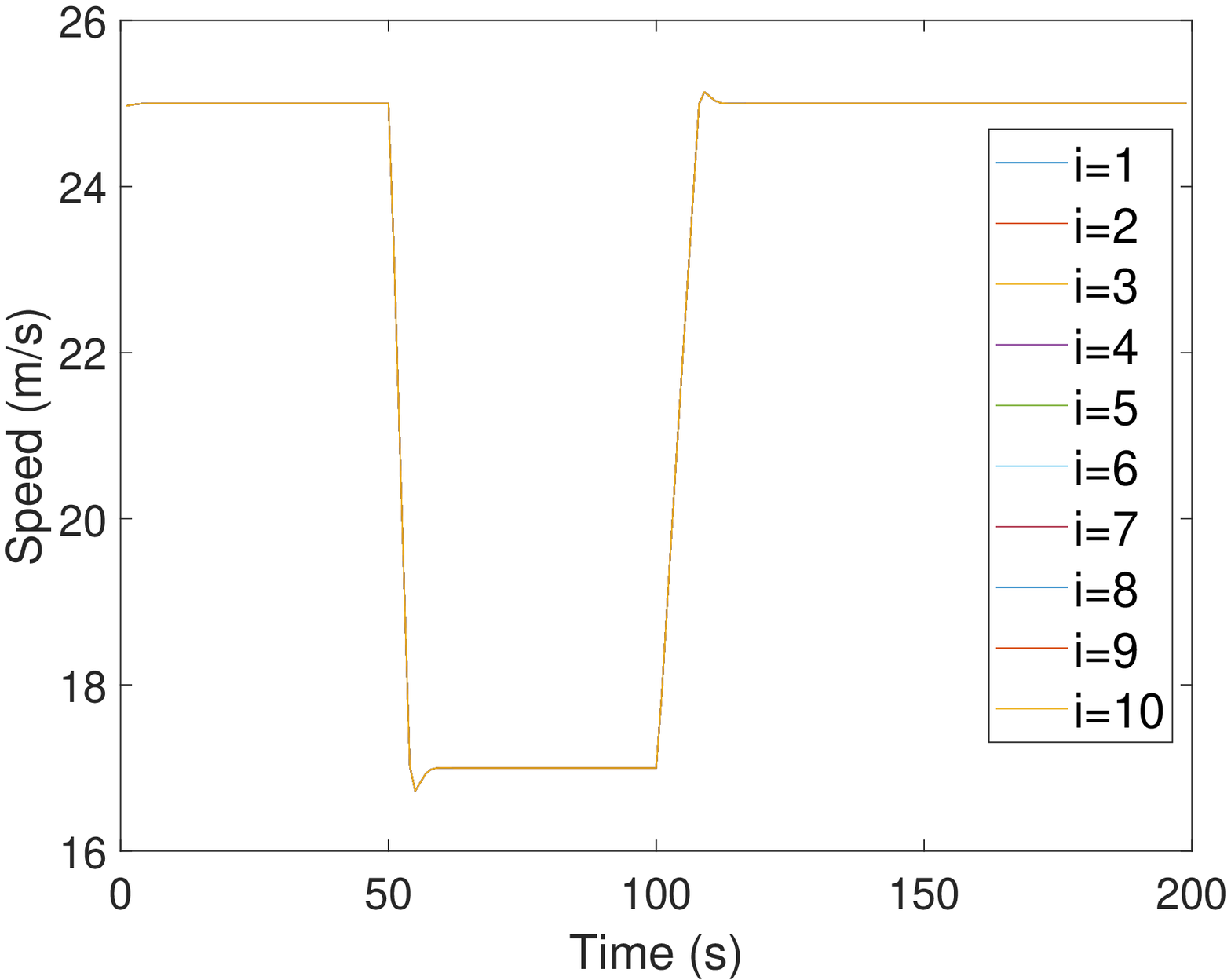}
}
\subfigure[Time history of vehicle speed.] { 
\includegraphics[width=0.48\columnwidth]{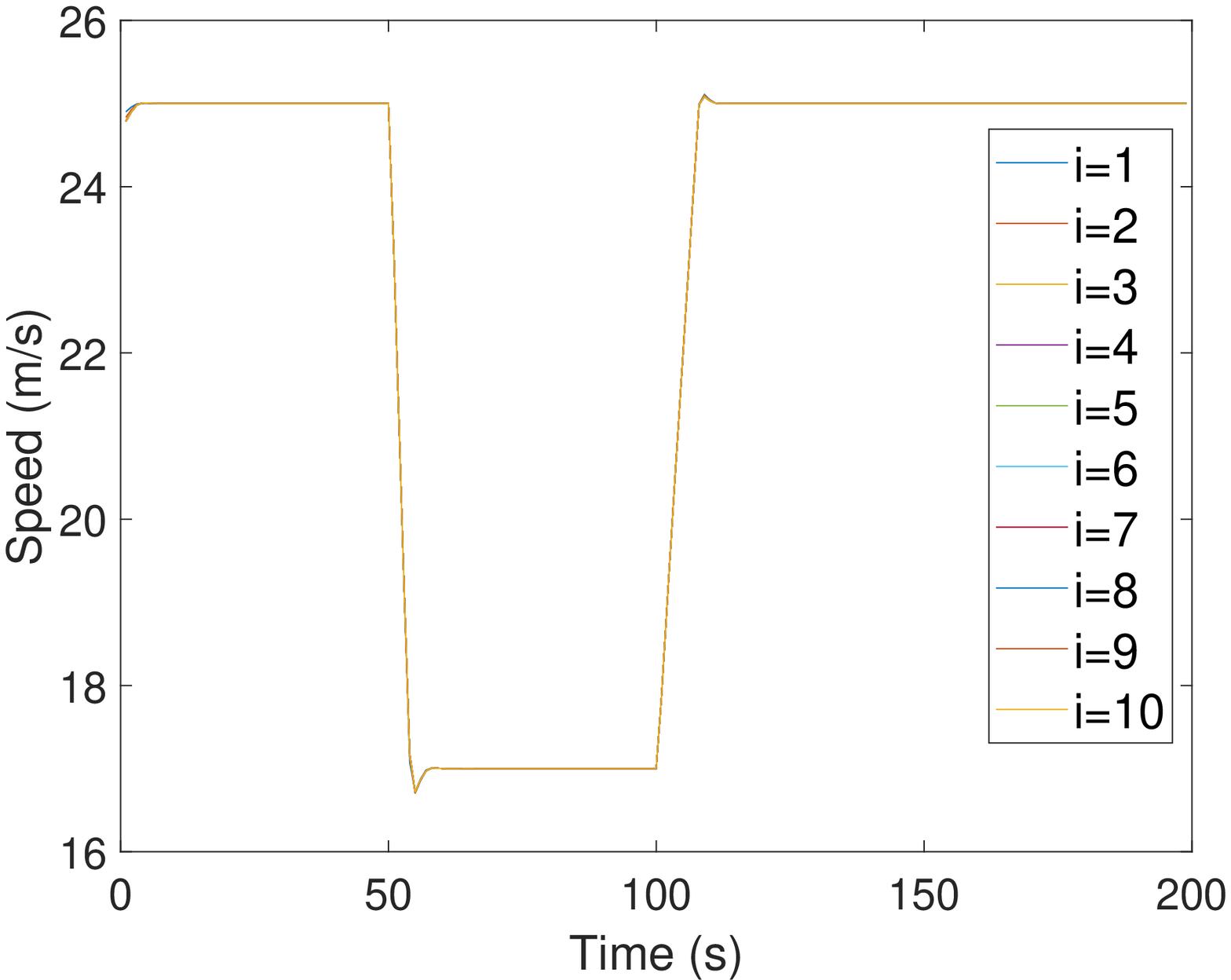}
}
\subfigure[Time history of control input.]{ 
\includegraphics[width=0.48\columnwidth]{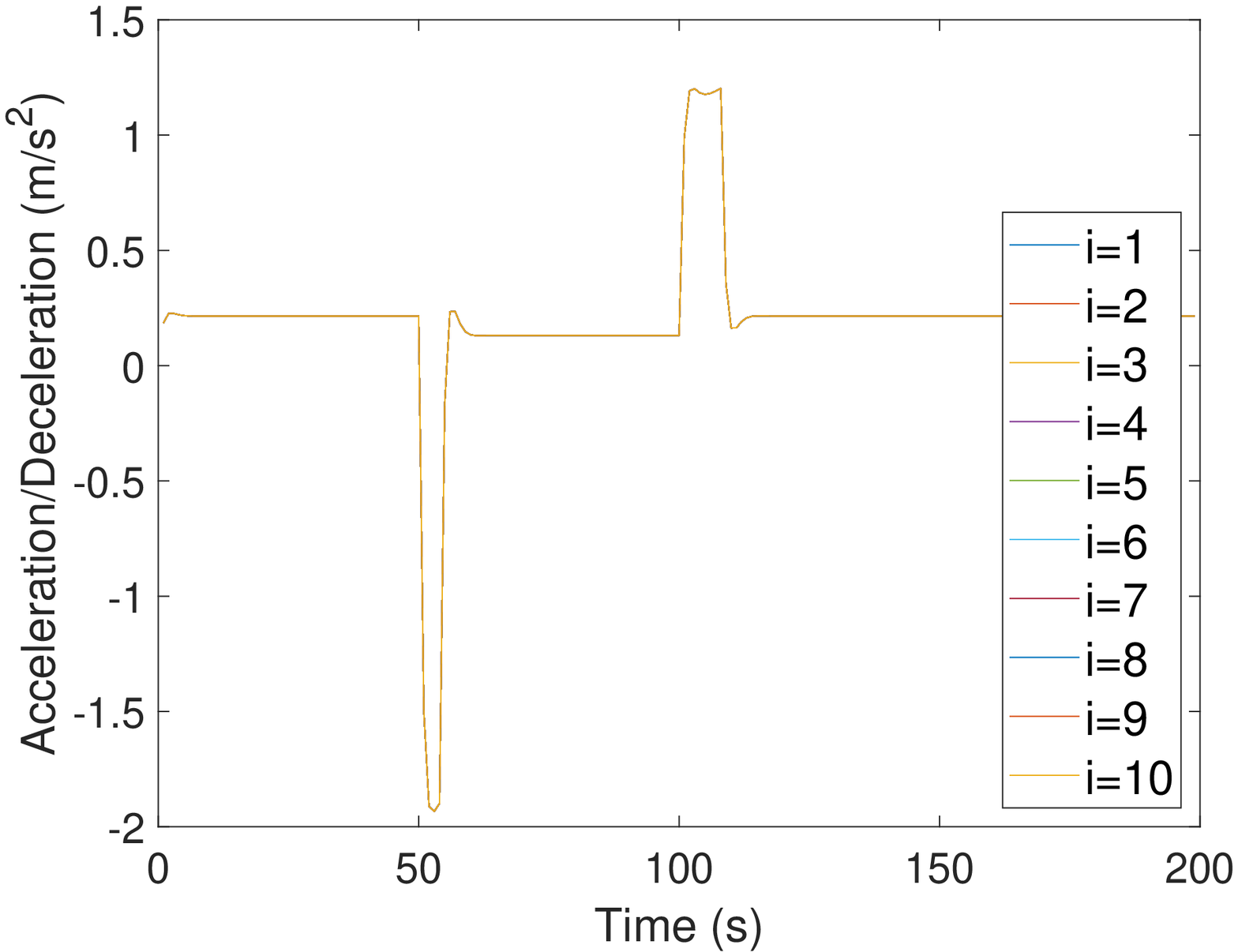}
}
\subfigure[Time history of control input] { 
\includegraphics[width=0.48\columnwidth]{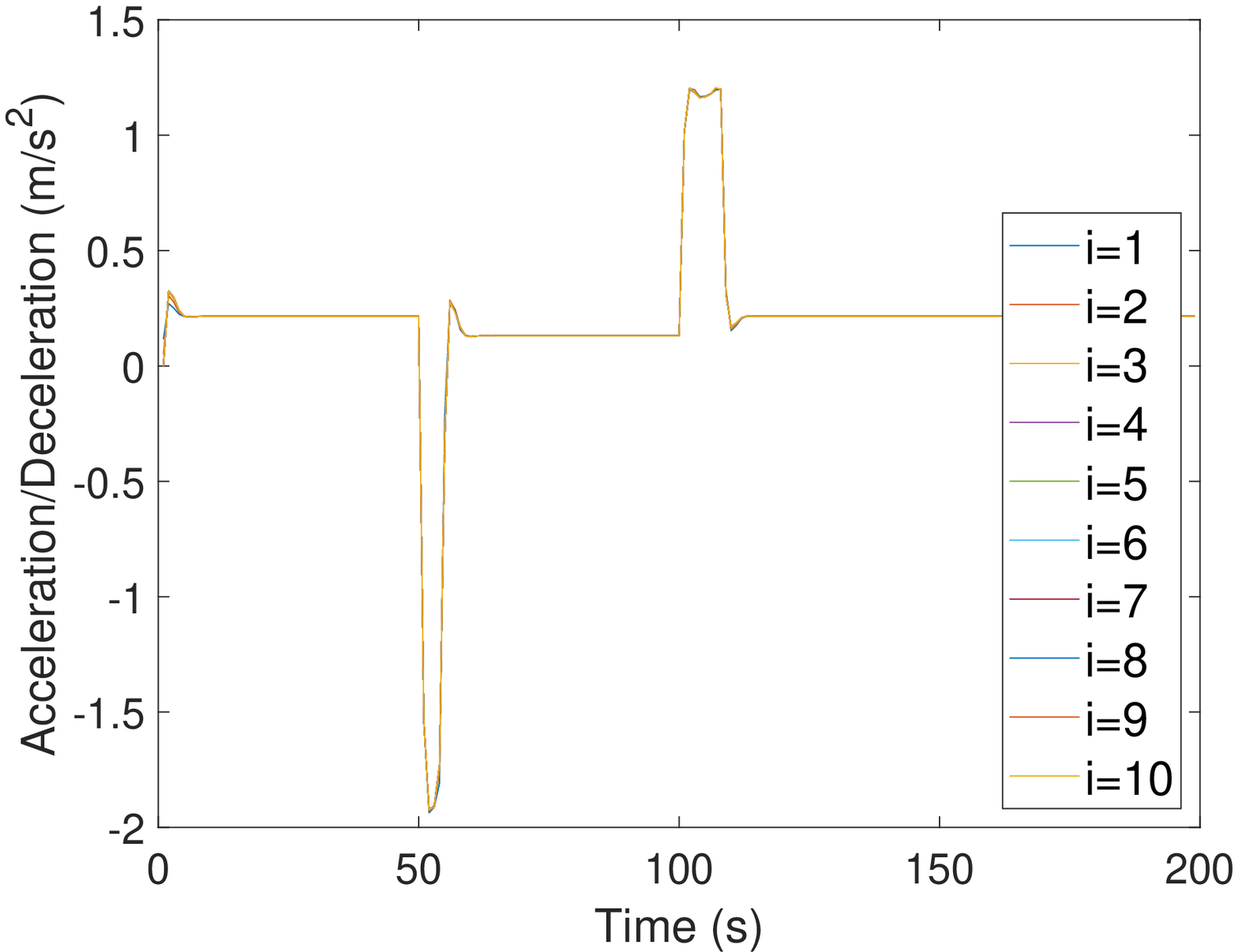}
}
\caption{Scenario 1 for the homogeneous small-size CAV platoon: platooning control with $p=1$ (left column) and $p=5$ (right column).}
\label{Fig:S1_small}
\end{figure}


\begin{figure}[htbp]
\centering
\subfigure[Time history of spacing changes.]{ \label{S1_p1m_spacing}
\includegraphics[width=0.48\columnwidth]{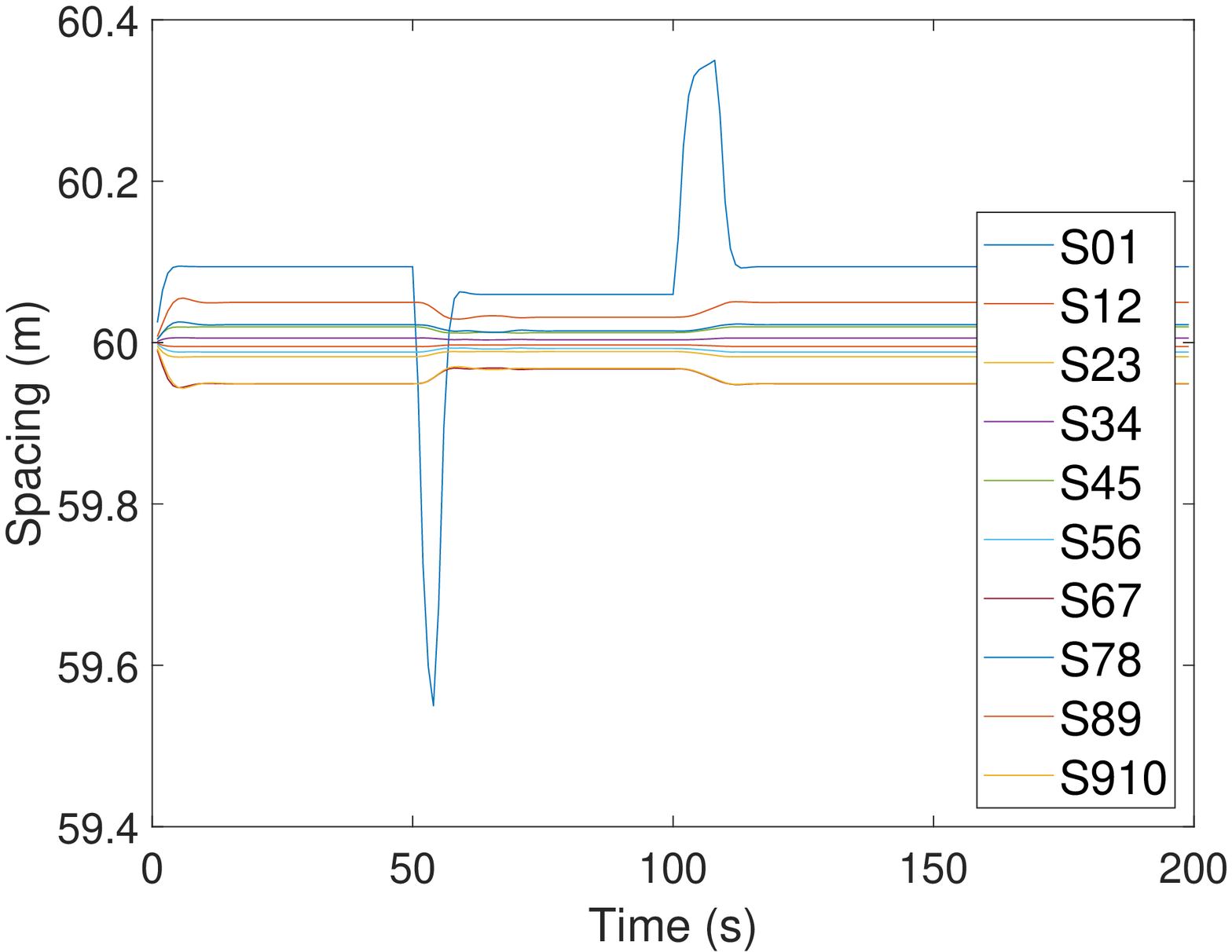}
}
\subfigure[Time history of spacing changes.] { \label{S1_p5m_spacing}
\includegraphics[width=0.48\columnwidth]{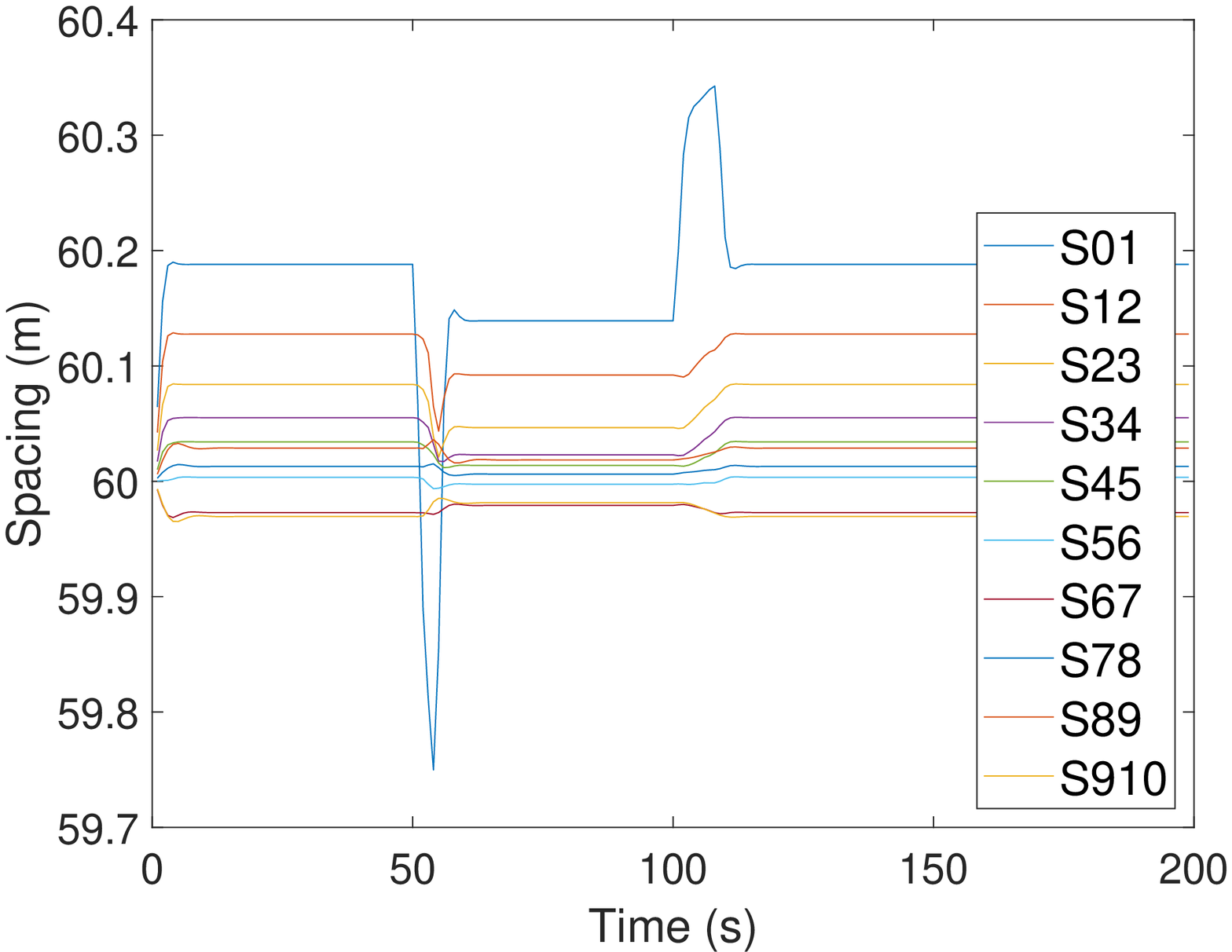}
}
\subfigure[Time history of vehicle speed. ]{ \label{S1_p1m_speed}
\includegraphics[width=0.48\columnwidth]{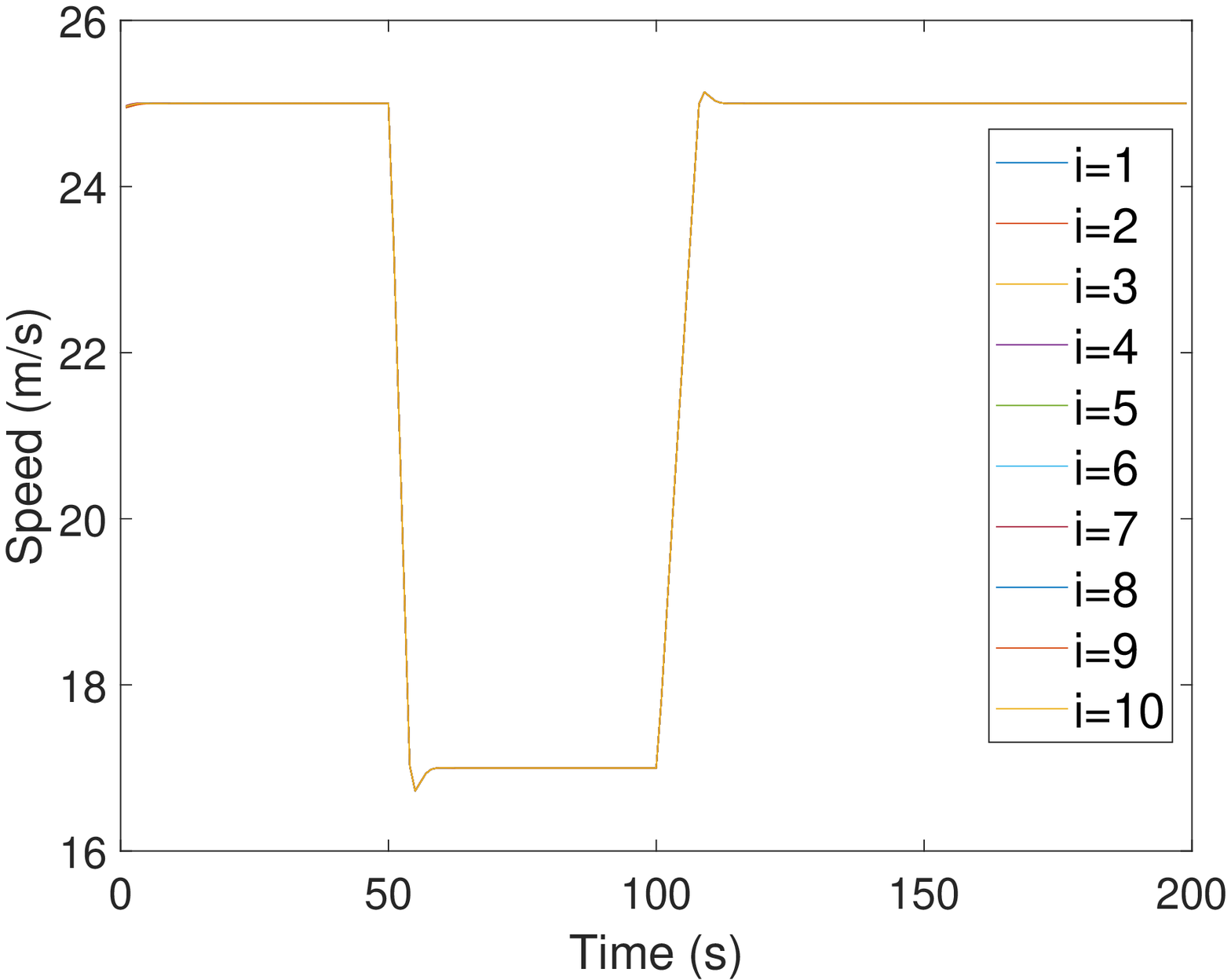}
}
\subfigure[Time history of vehicle speed.] { 
\includegraphics[width=0.48\columnwidth]{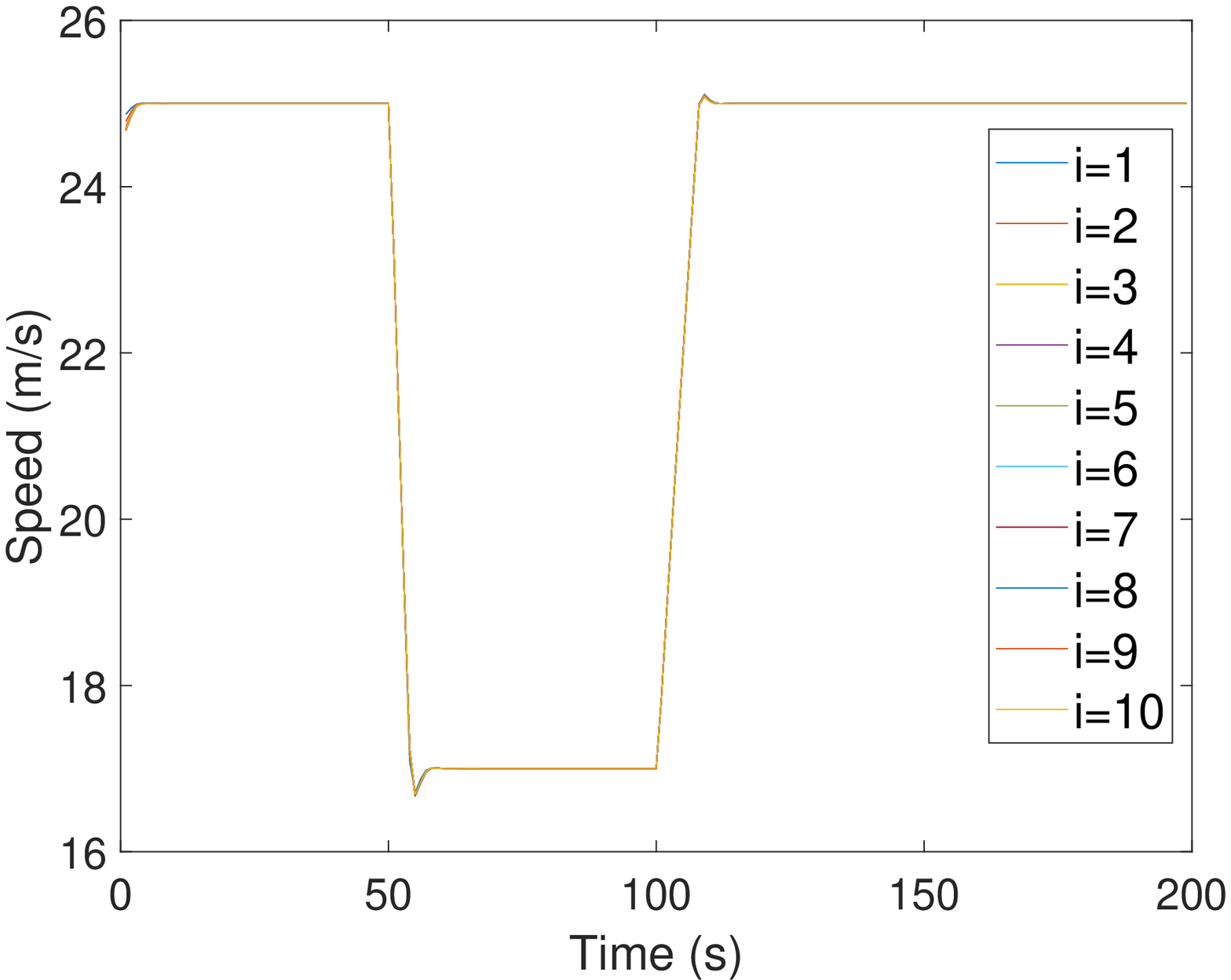}
}
\subfigure[Time history of control input.]{ 
\includegraphics[width=0.48\columnwidth]{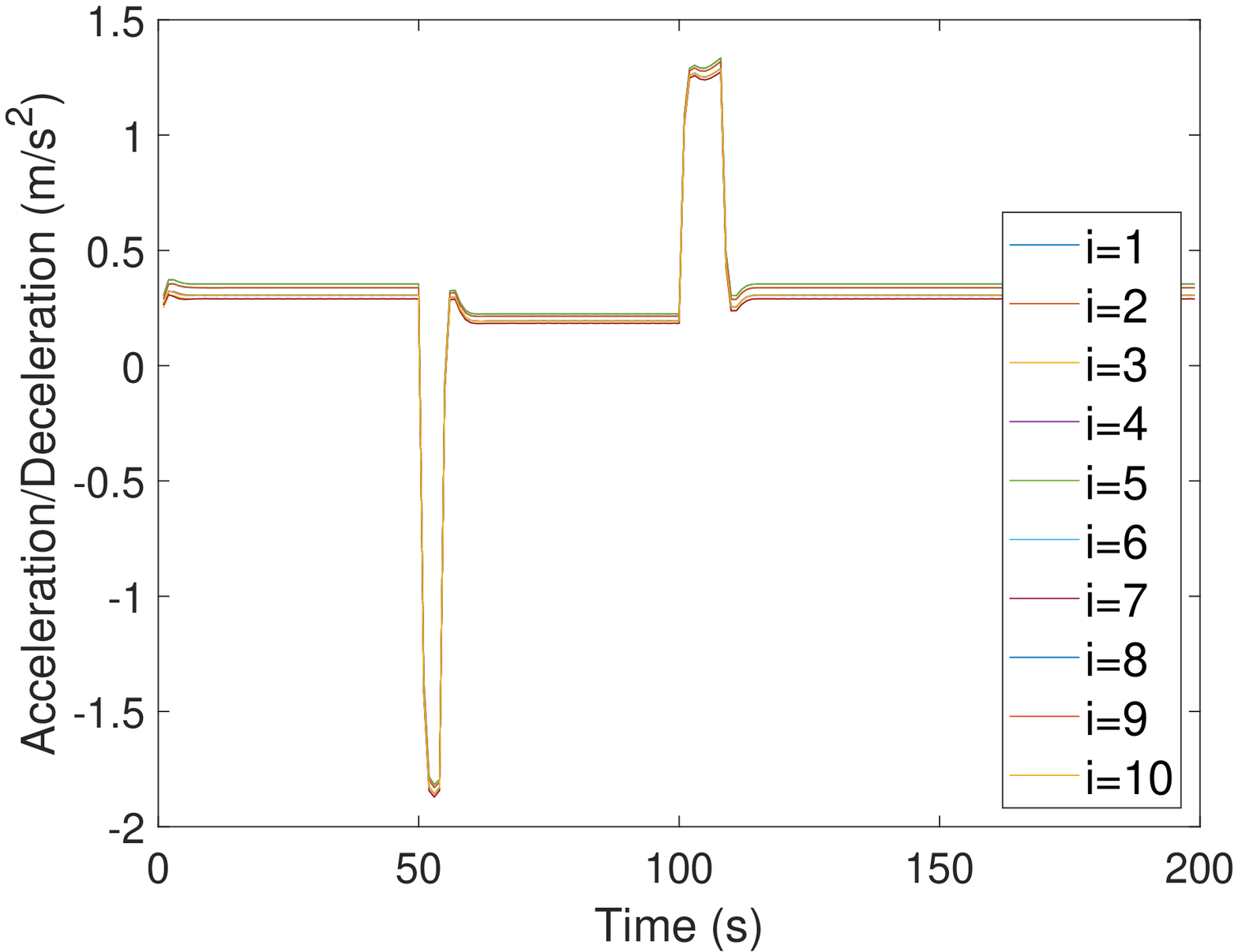}
}
\subfigure[Time history of control input] { 
\includegraphics[width=0.48\columnwidth]{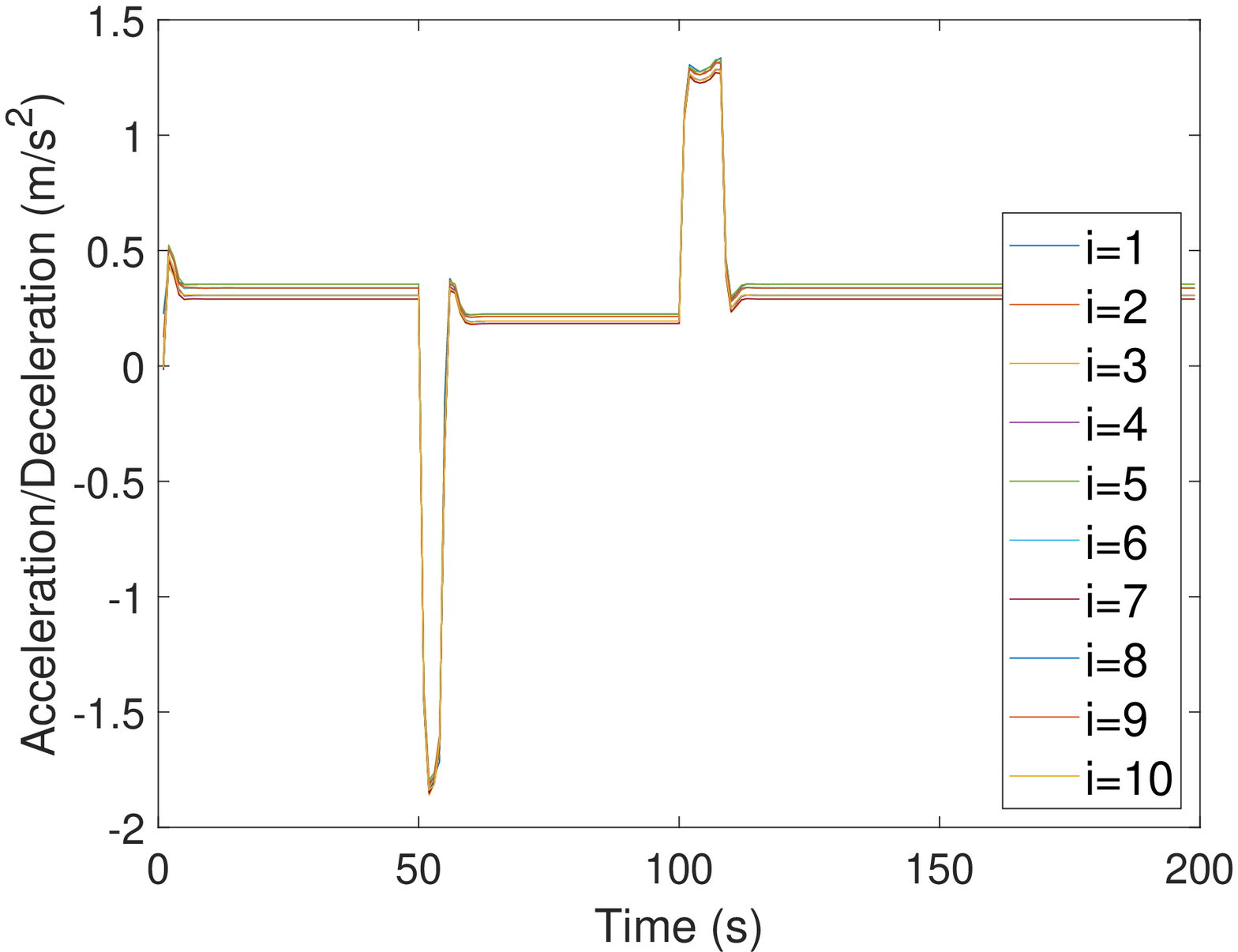}
}
\caption{Scenario 1 for the heterogeneous medium-size CAV platoon: platooning control with $p=1$ (left column) and $p=5$ (right column).}
\label{Fig:S1_medium}
\end{figure}


\begin{figure}[htbp]
\centering
\subfigure[Time history of spacing changes.]{ \label{S1_p1max_spacing}
\includegraphics[width=0.48\columnwidth]{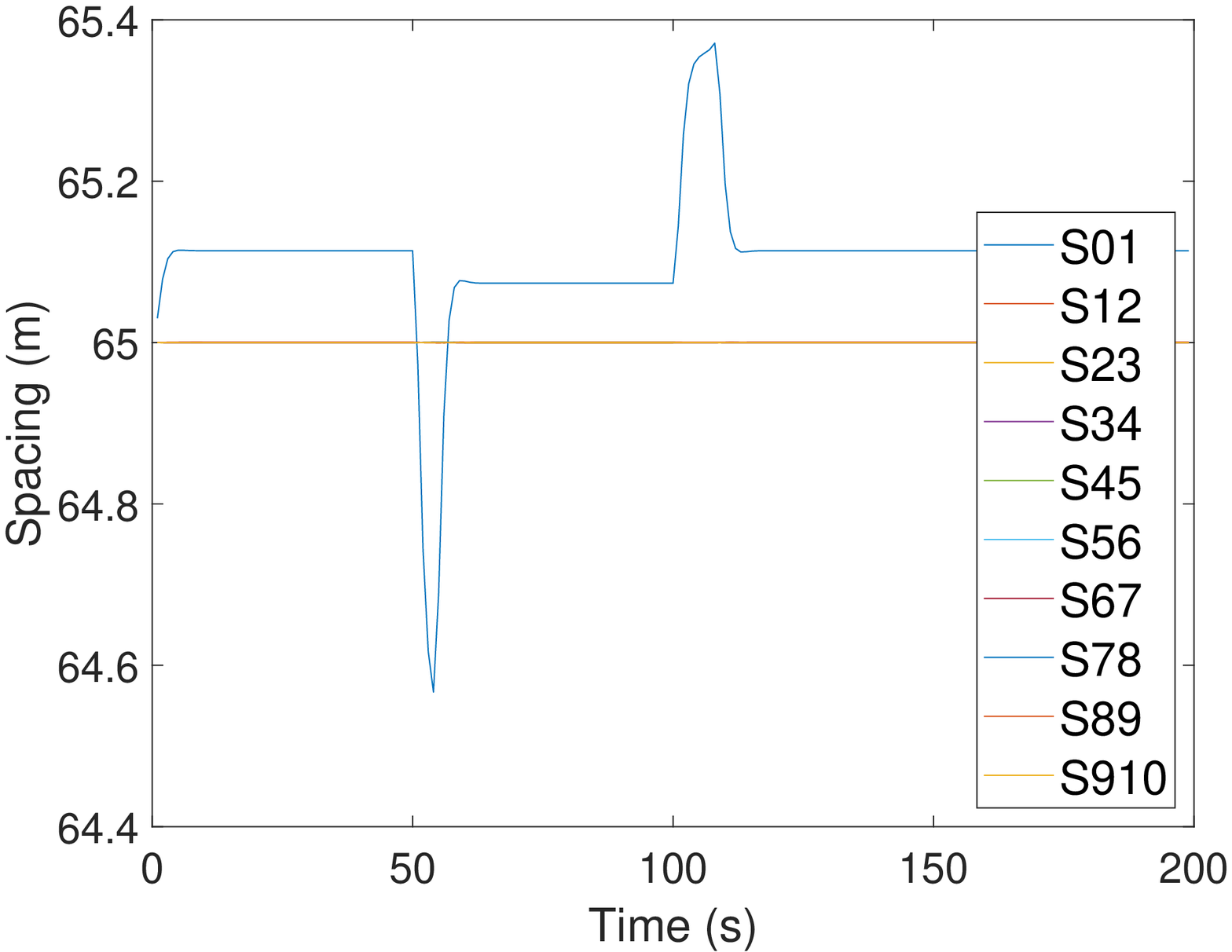}
}
\subfigure[Time history of spacing changes.] { \label{S1_p5max_spacing}
\includegraphics[width=0.48\columnwidth]{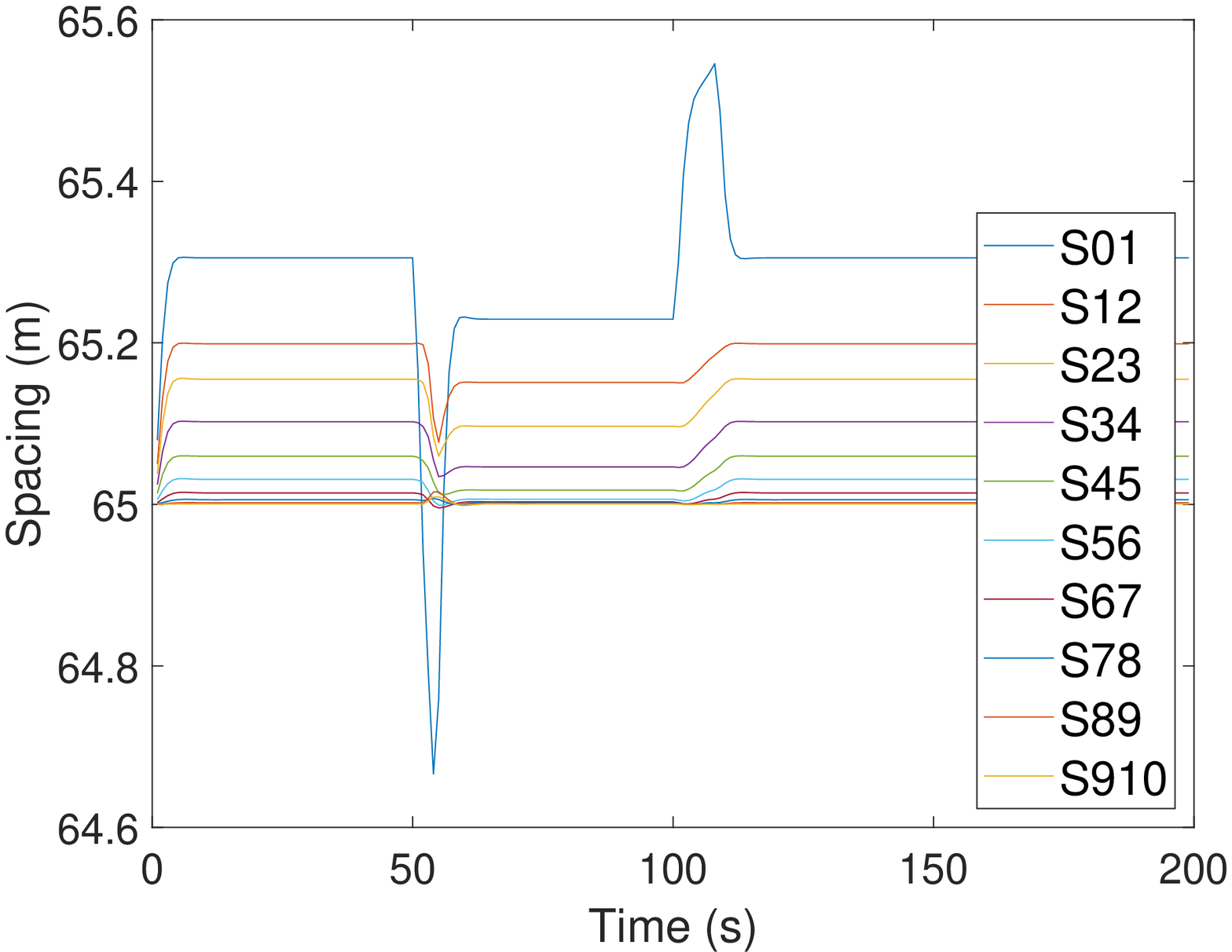}
}
\subfigure[Time history of vehicle speed. ]{ \label{S1_p1max_speed}
\includegraphics[width=0.48\columnwidth]{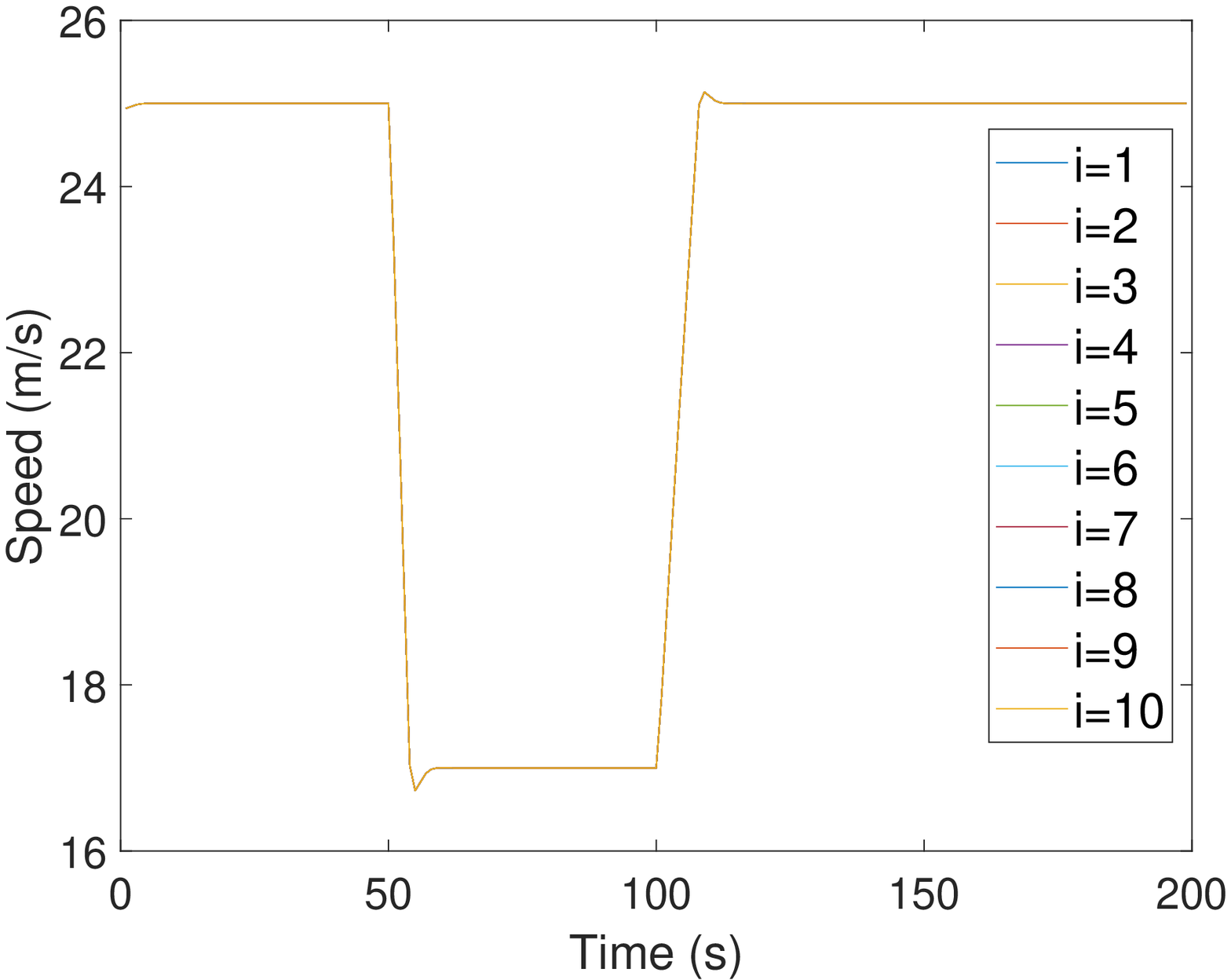}
}
\subfigure[Time history of vehicle speed.] { 
\includegraphics[width=0.48\columnwidth]{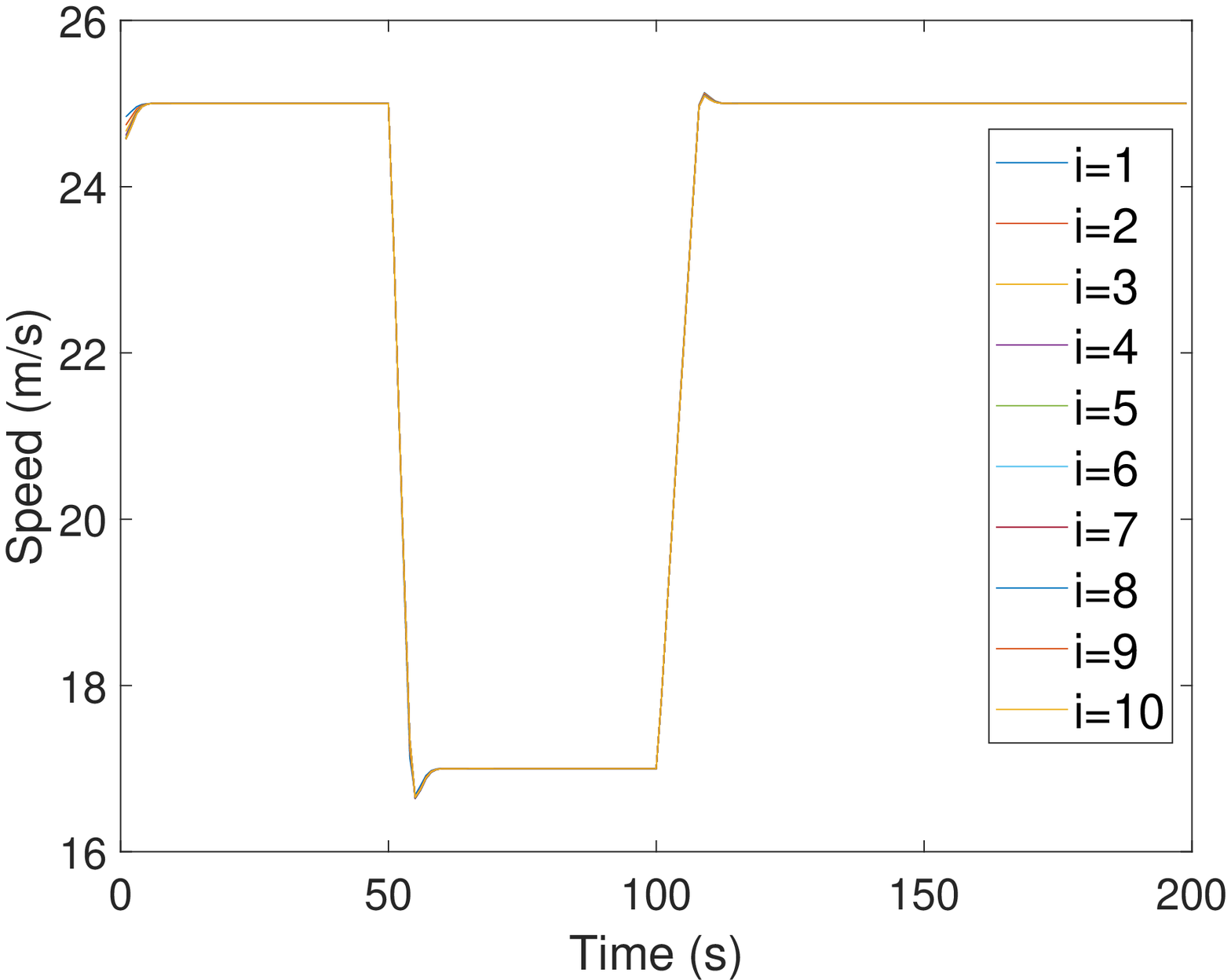}
}
\subfigure[Time history of control input.]{ 
\includegraphics[width=0.48\columnwidth]{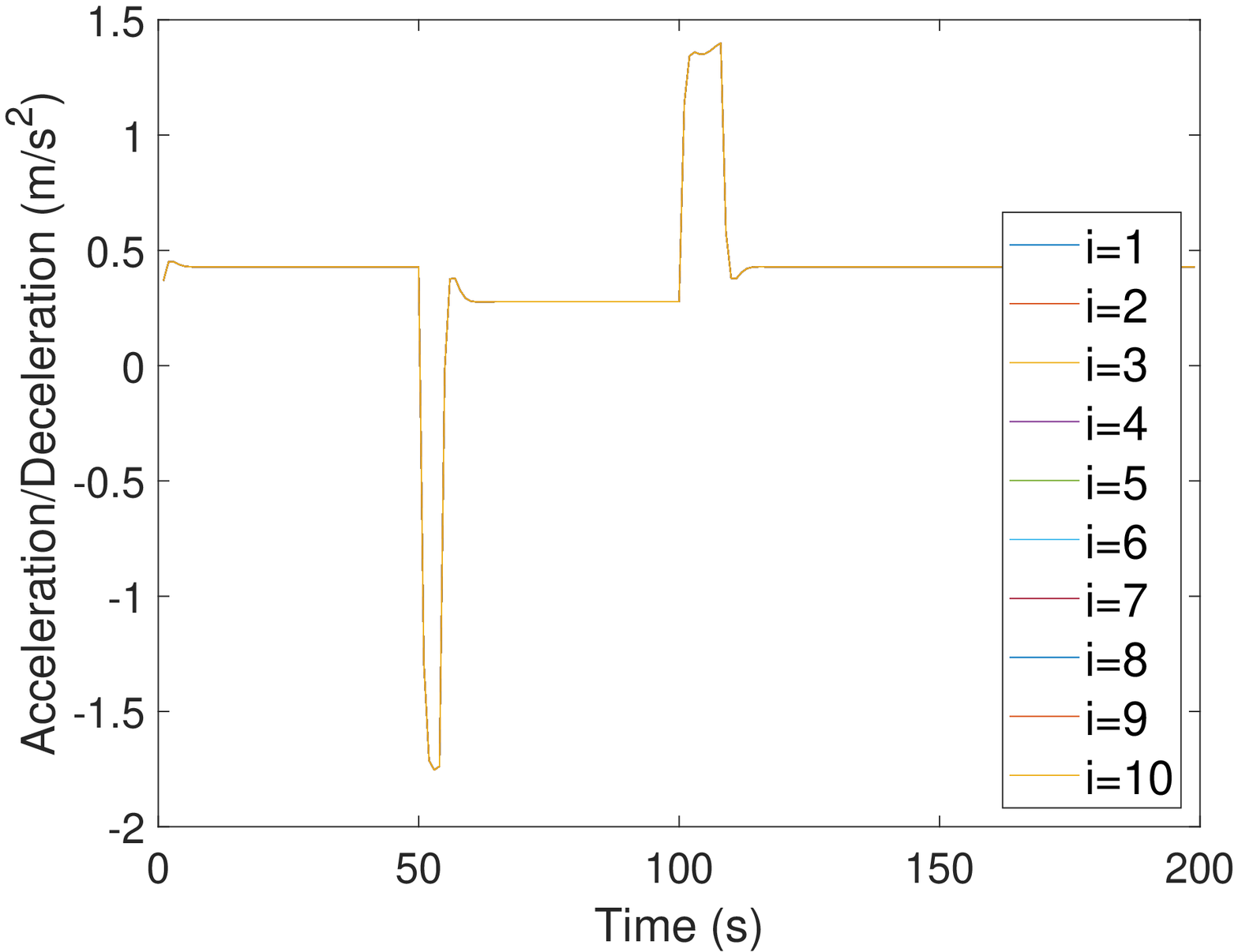}
}
\subfigure[Time history of control input] { 
\includegraphics[width=0.48\columnwidth]{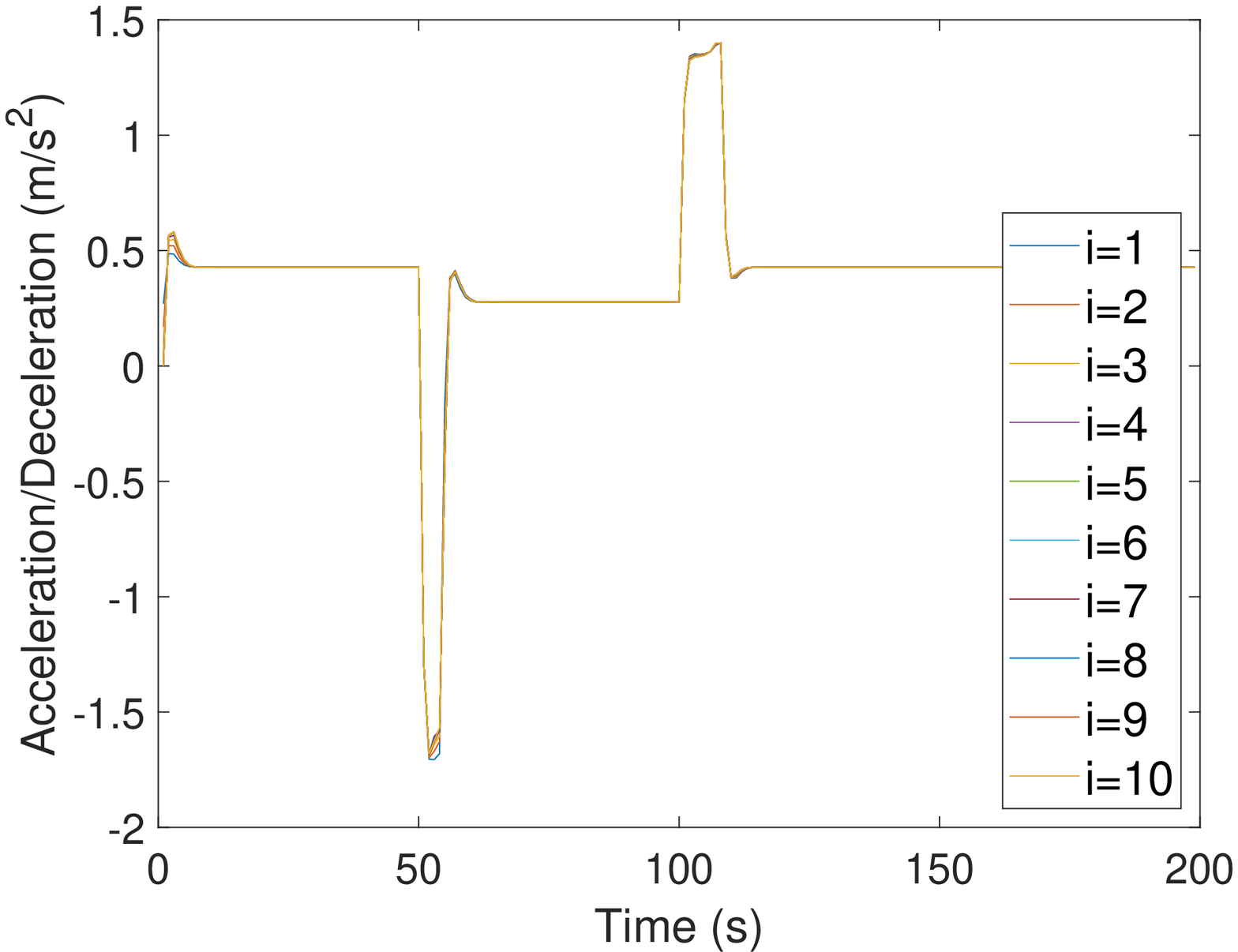}
}
\caption{Scenario 1 for the homogeneous large-size CAV platoon: platooning control with $p=1$ (left column) and $p=5$ (right column).}
\label{Fig:S1_large}
\end{figure}


\begin{figure}[htbp]
\centering
\subfigure[Time history of spacing changes.]{ 
\includegraphics[width=0.48\columnwidth]{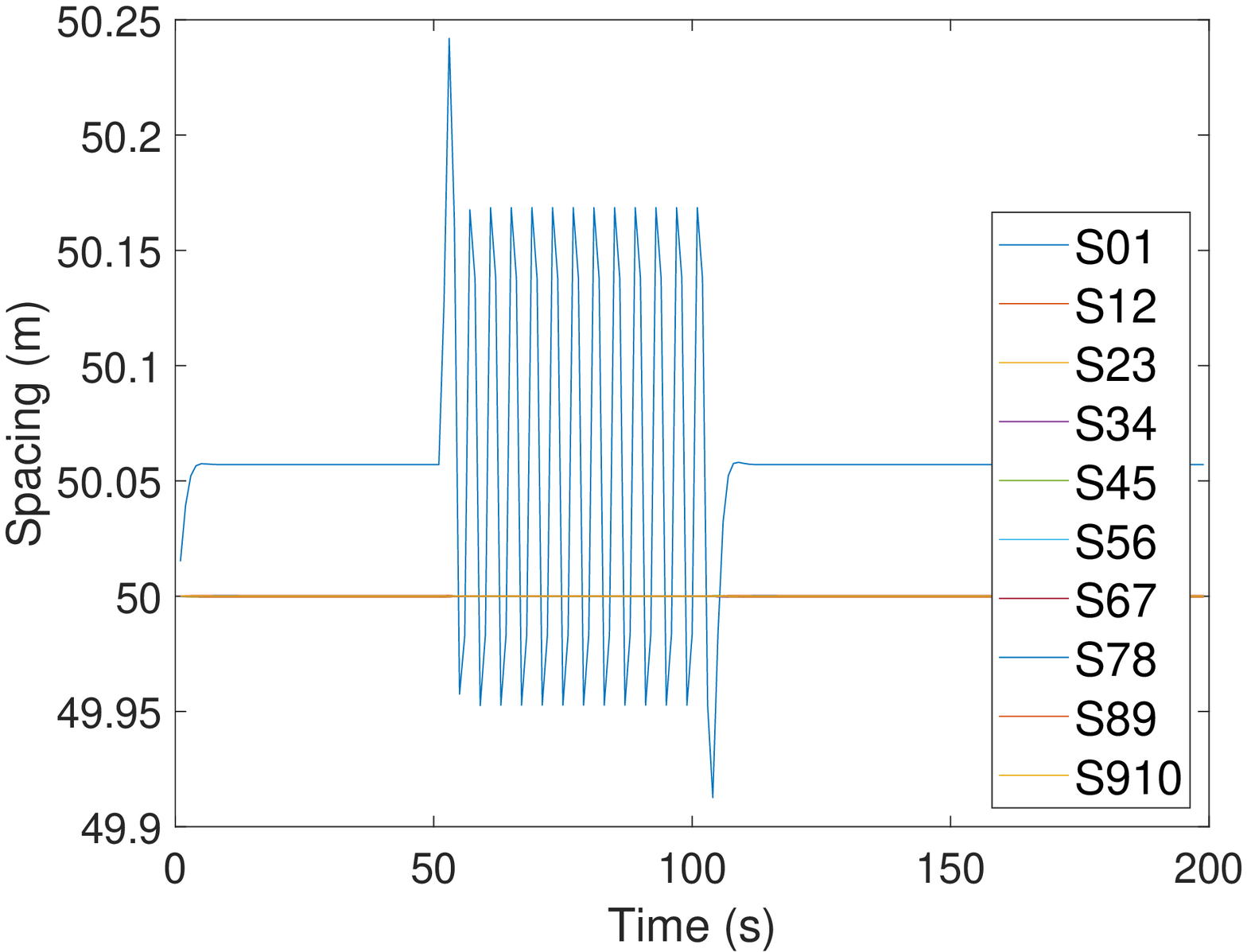}
}
\subfigure[Time history of spacing changes.] { 
\includegraphics[width=0.48\columnwidth]{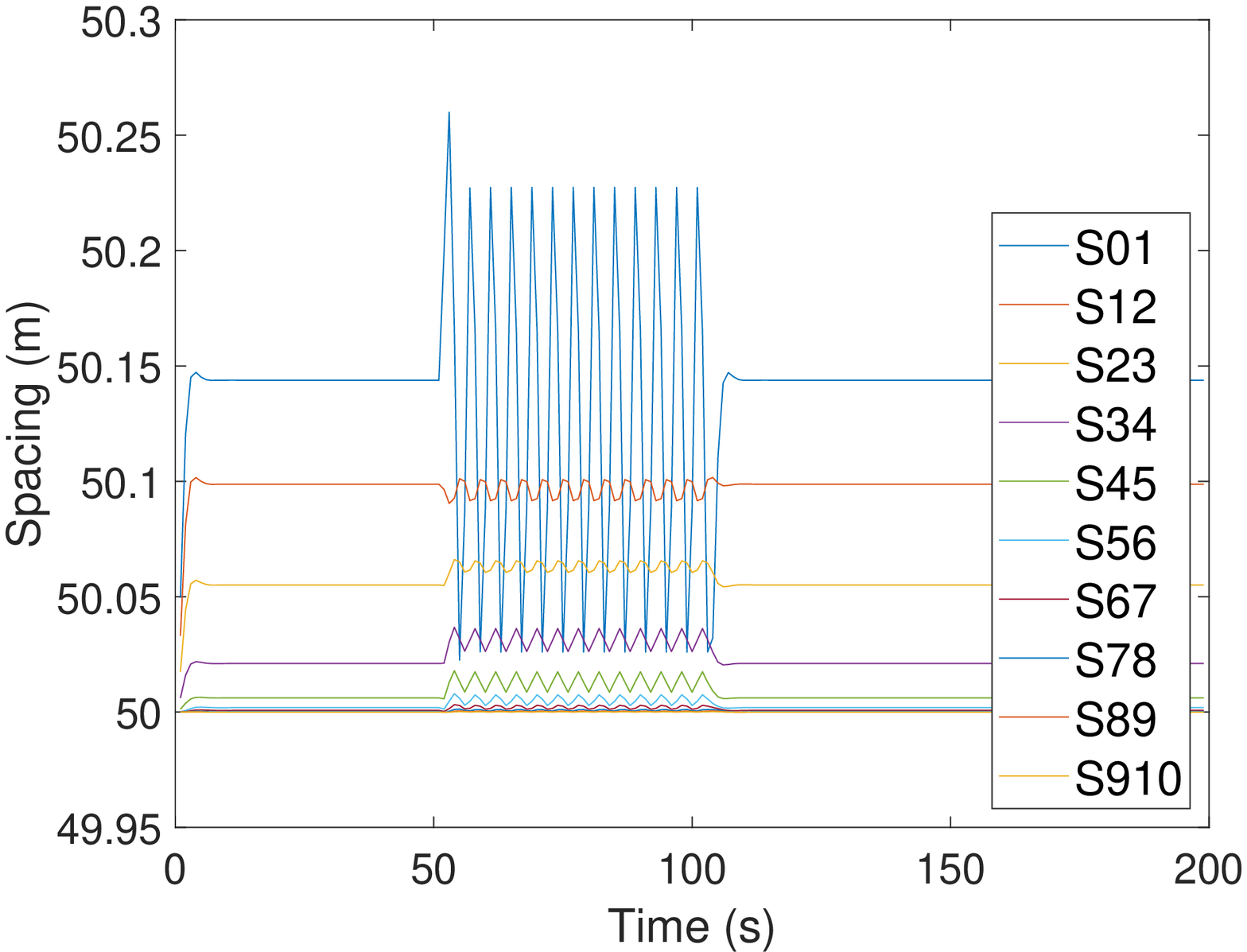}
}
\subfigure[Time history of vehicle speed. ]{ 
\includegraphics[width=0.48\columnwidth]{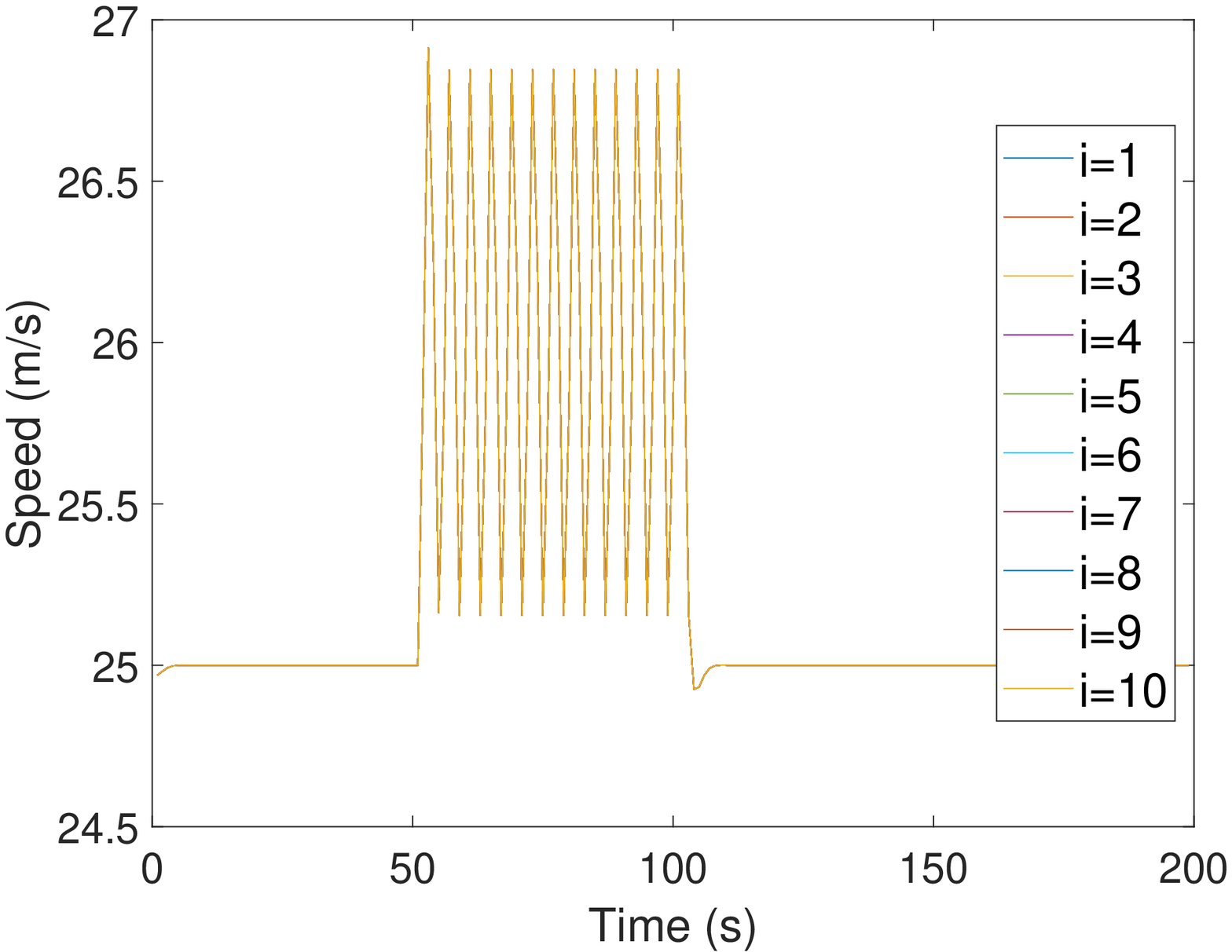}
}
\subfigure[Time history of vehicle speed.] { 
\includegraphics[width=0.48\columnwidth]{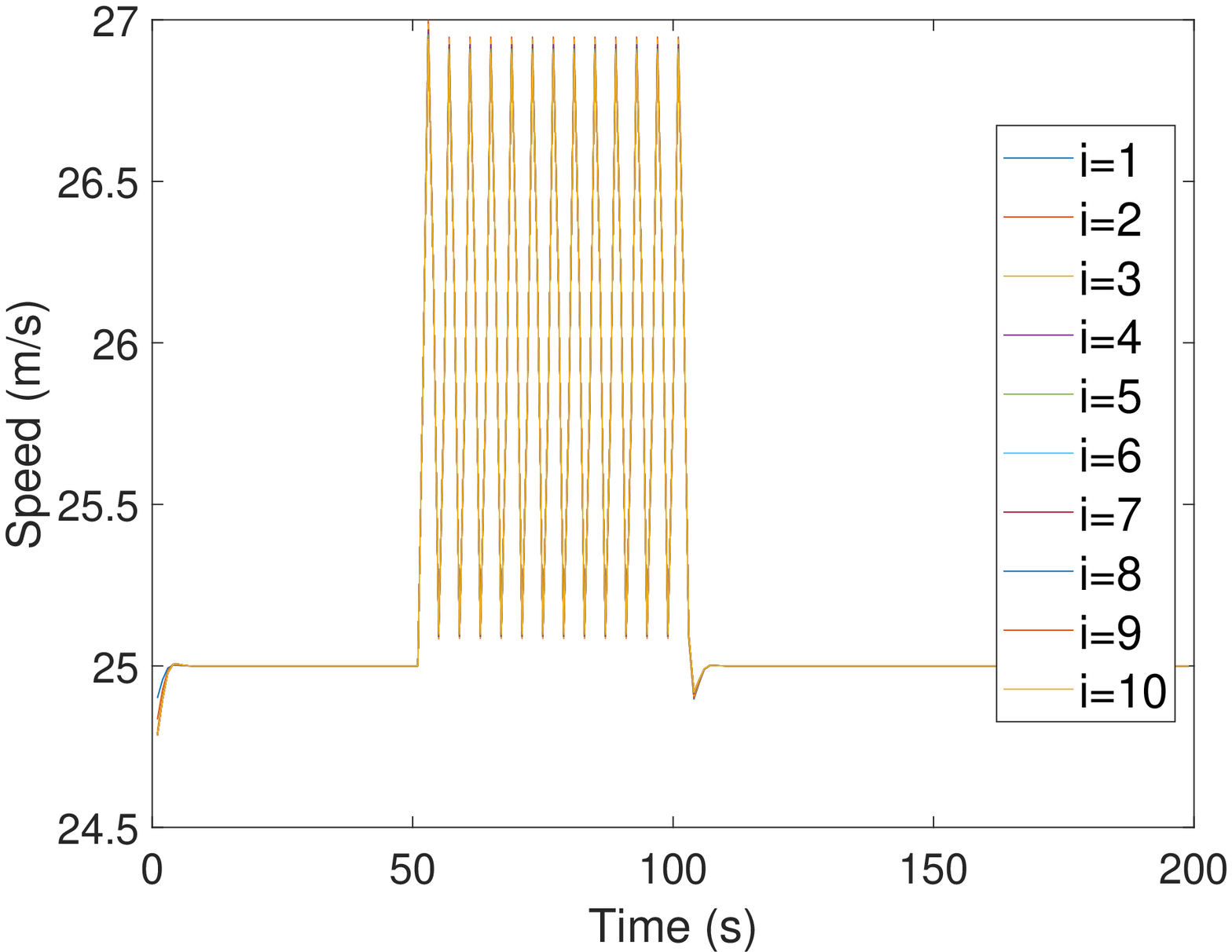}
}
\subfigure[Time history of control input.]{ 
\includegraphics[width=0.48\columnwidth]{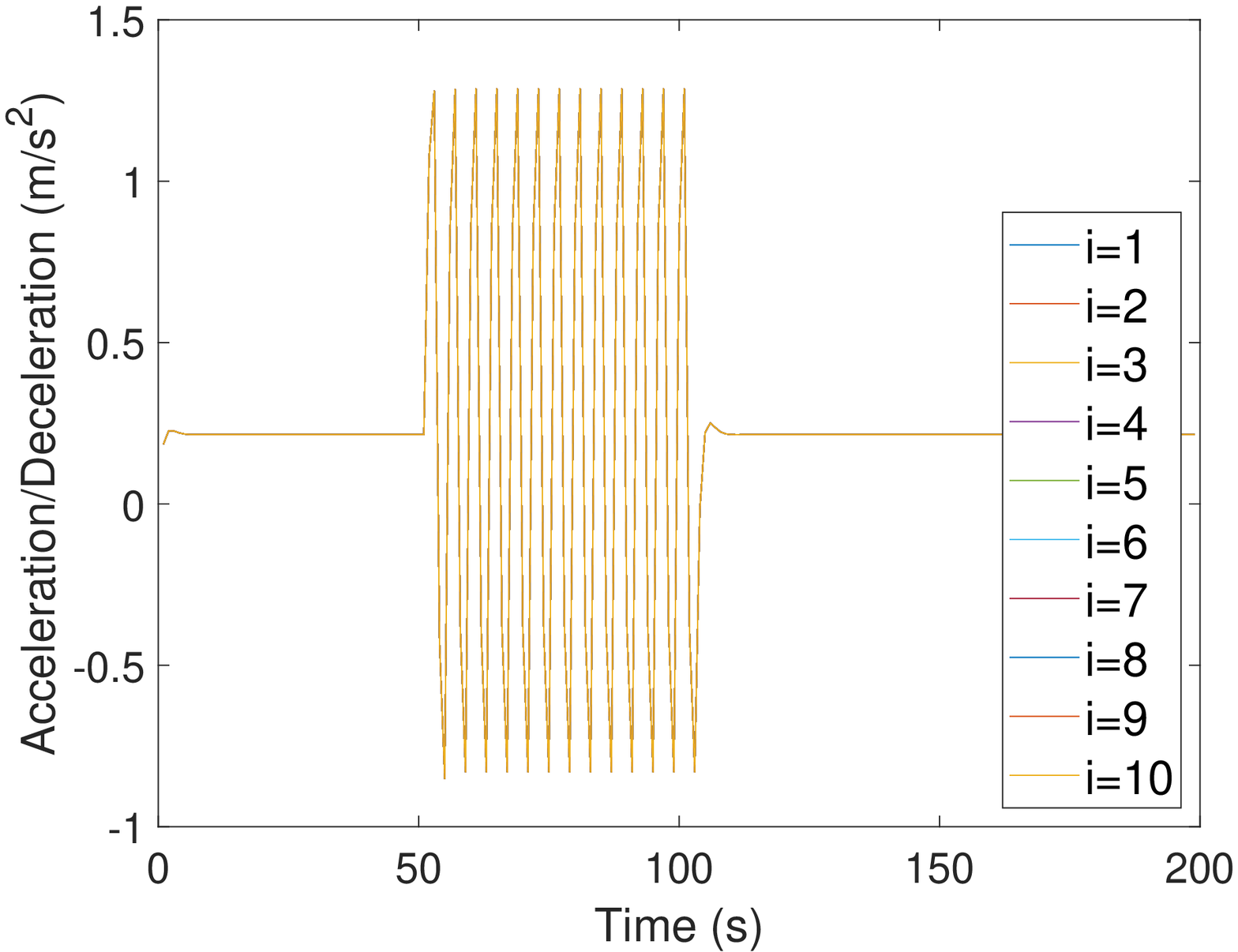}
}
\subfigure[Time history of control input] { 
\includegraphics[width=0.48\columnwidth]{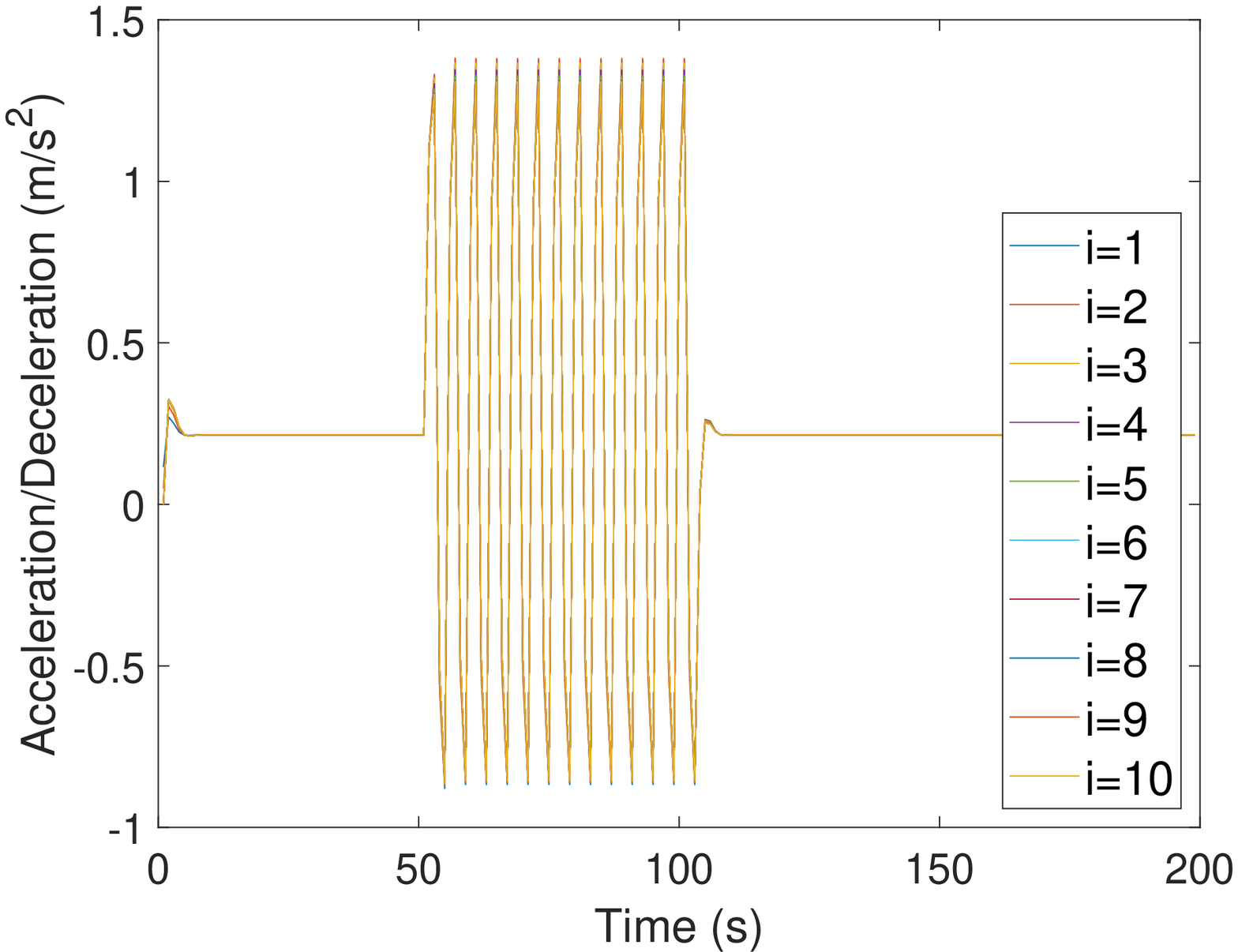}
}
\caption{Scenario 2 for the homogeneous small-size CAV platoon: platooning control with $p=1$ (left column) and $p=5$ (right column).}
\label{Fig:S2_small}
\end{figure}


\begin{figure}[htbp]
\centering
\subfigure[Time history of spacing changes.]{ 
\includegraphics[width=0.48\columnwidth]{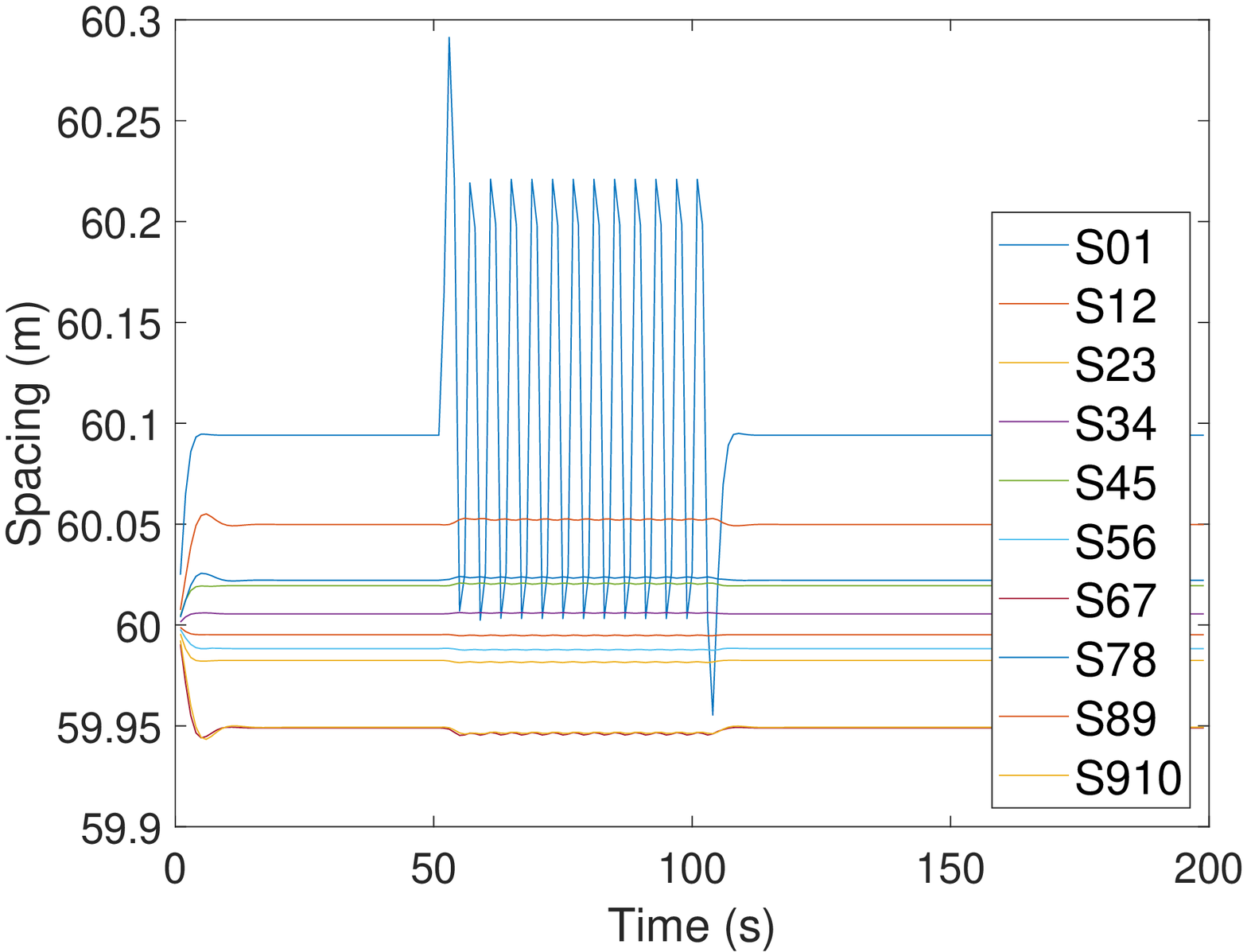}
}
\subfigure[Time history of spacing changes.] { 
\includegraphics[width=0.48\columnwidth]{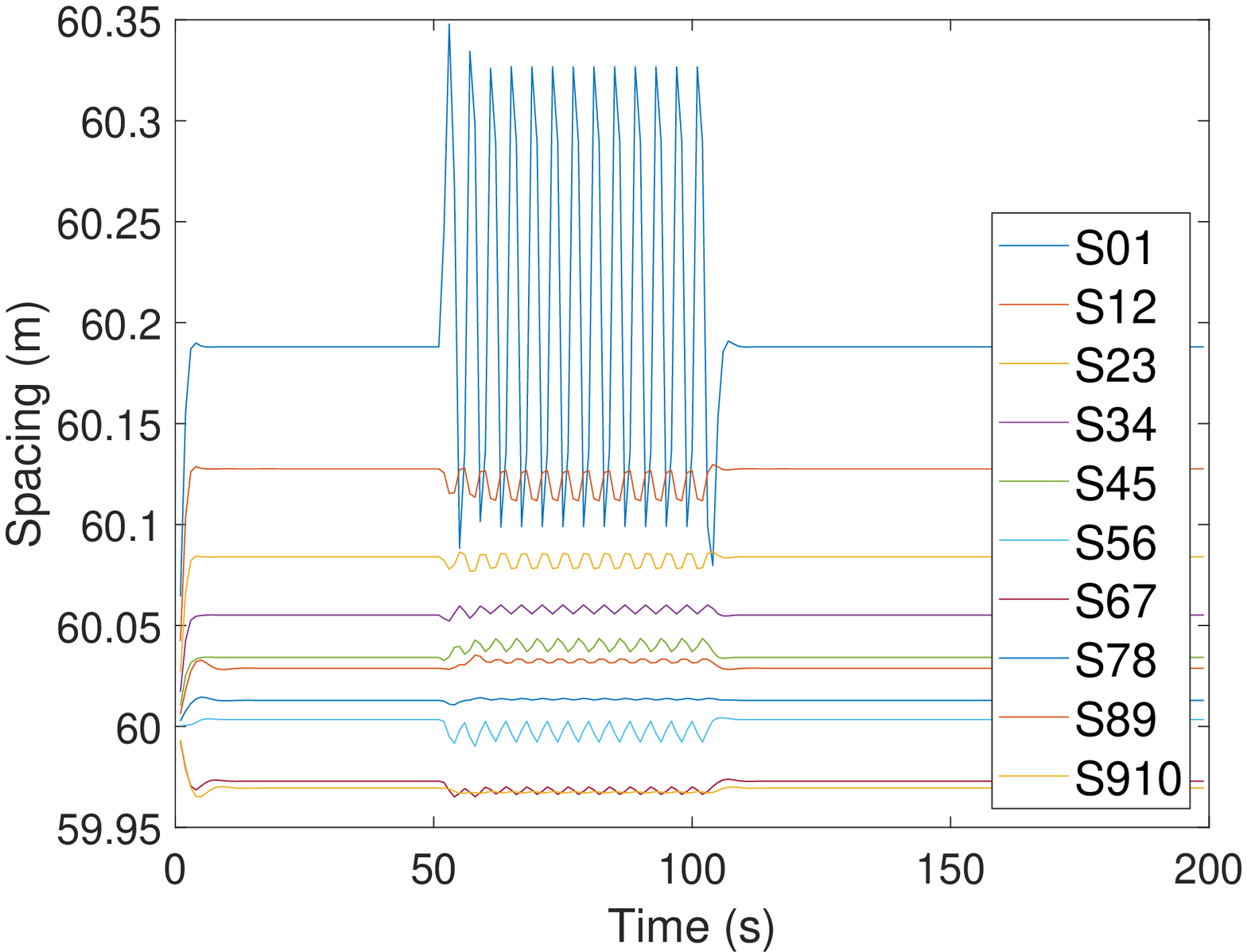}
}
\subfigure[Time history of vehicle speed. ]{ 
\includegraphics[width=0.48\columnwidth]{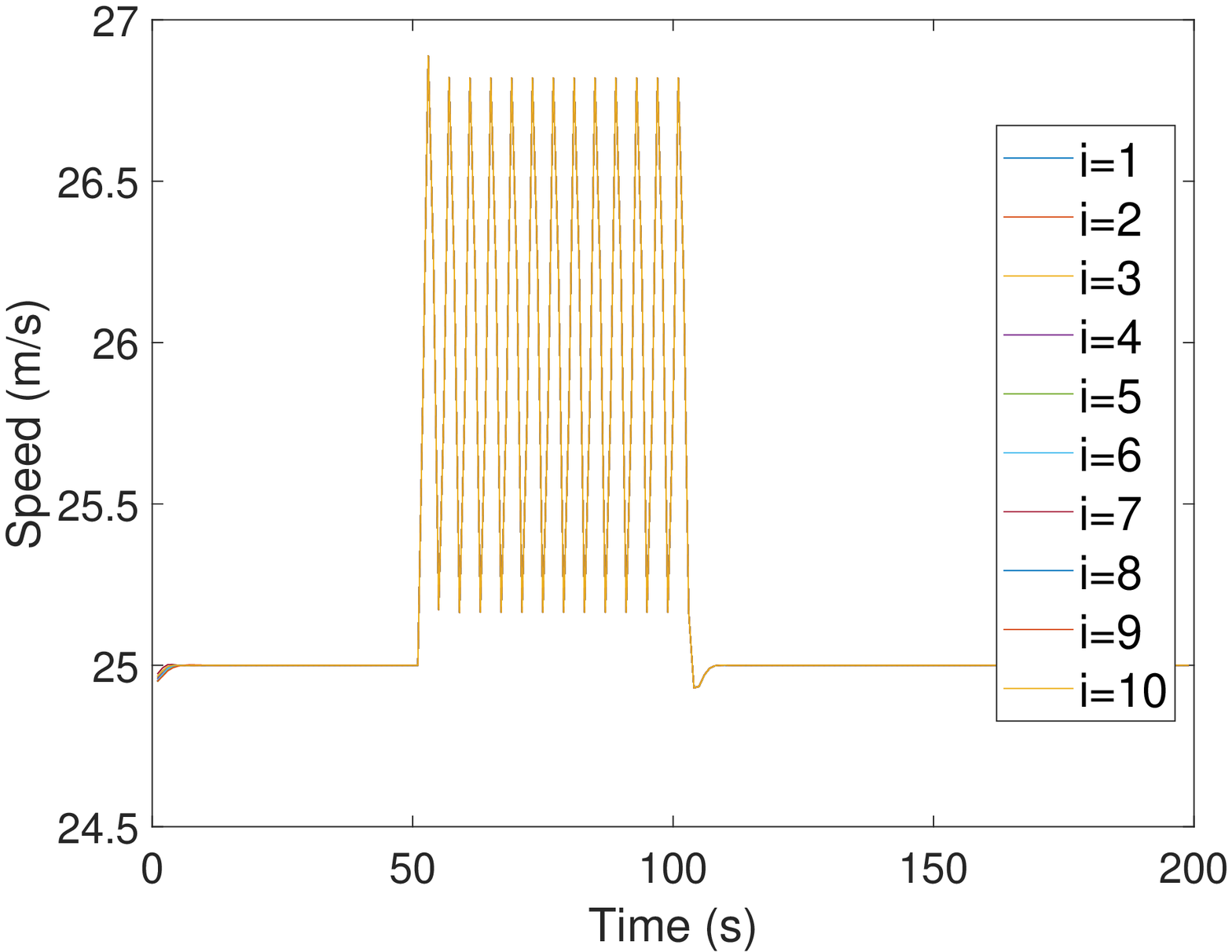}
}
\subfigure[Time history of vehicle speed.] { 
\includegraphics[width=0.48\columnwidth]{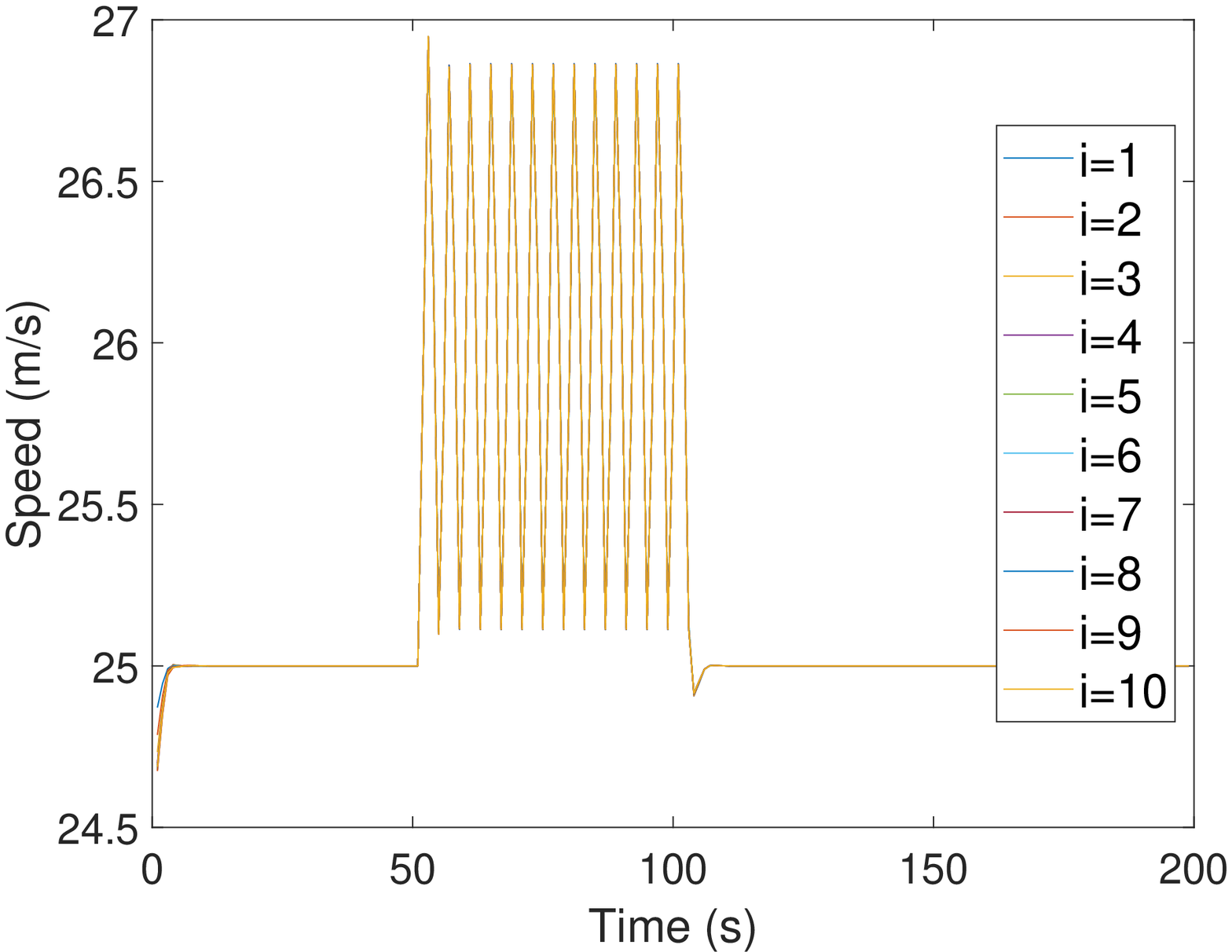}
}
\subfigure[Time history of control input.]{ 
\includegraphics[width=0.48\columnwidth]{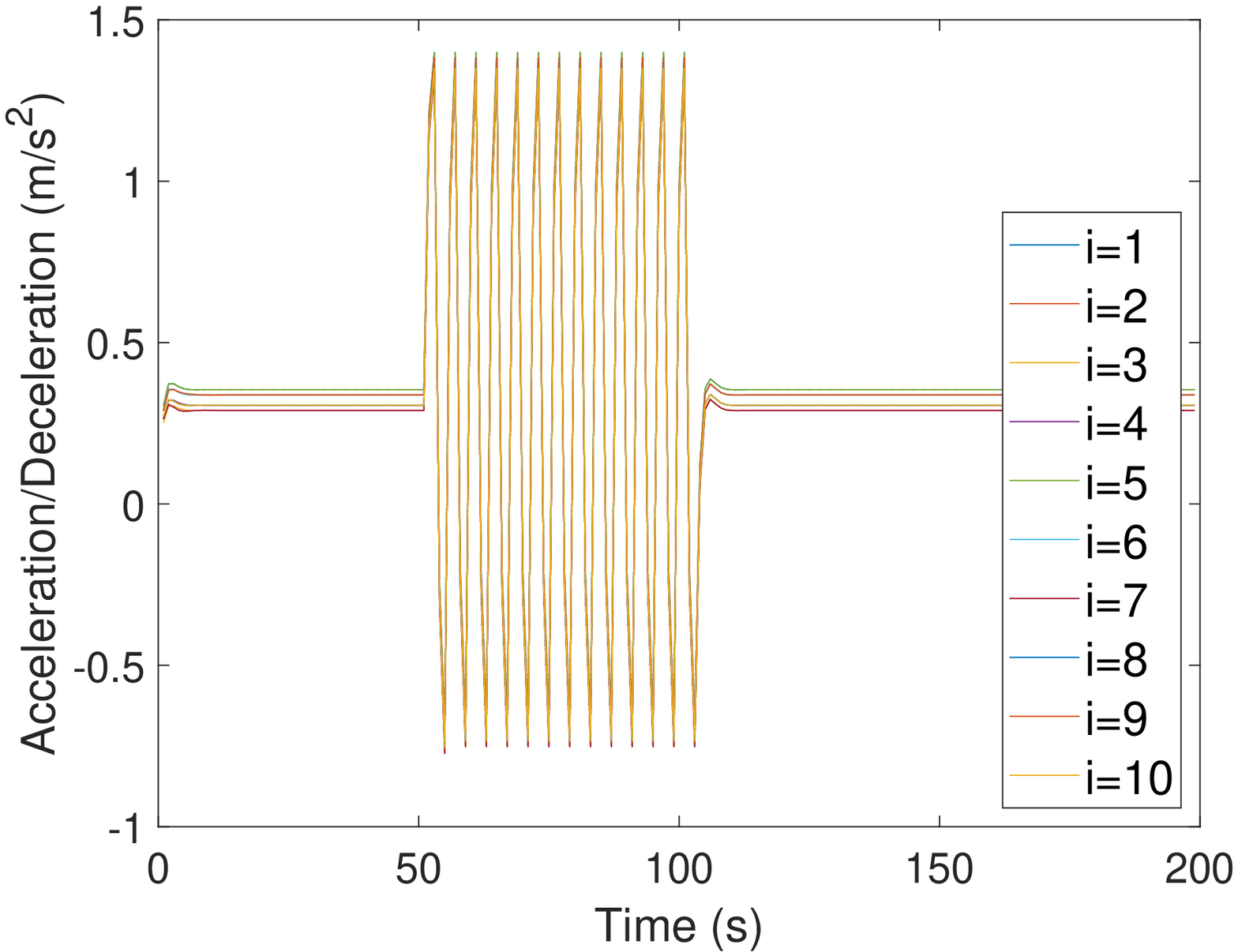}
}
\subfigure[Time history of control input] { 
\includegraphics[width=0.48\columnwidth]{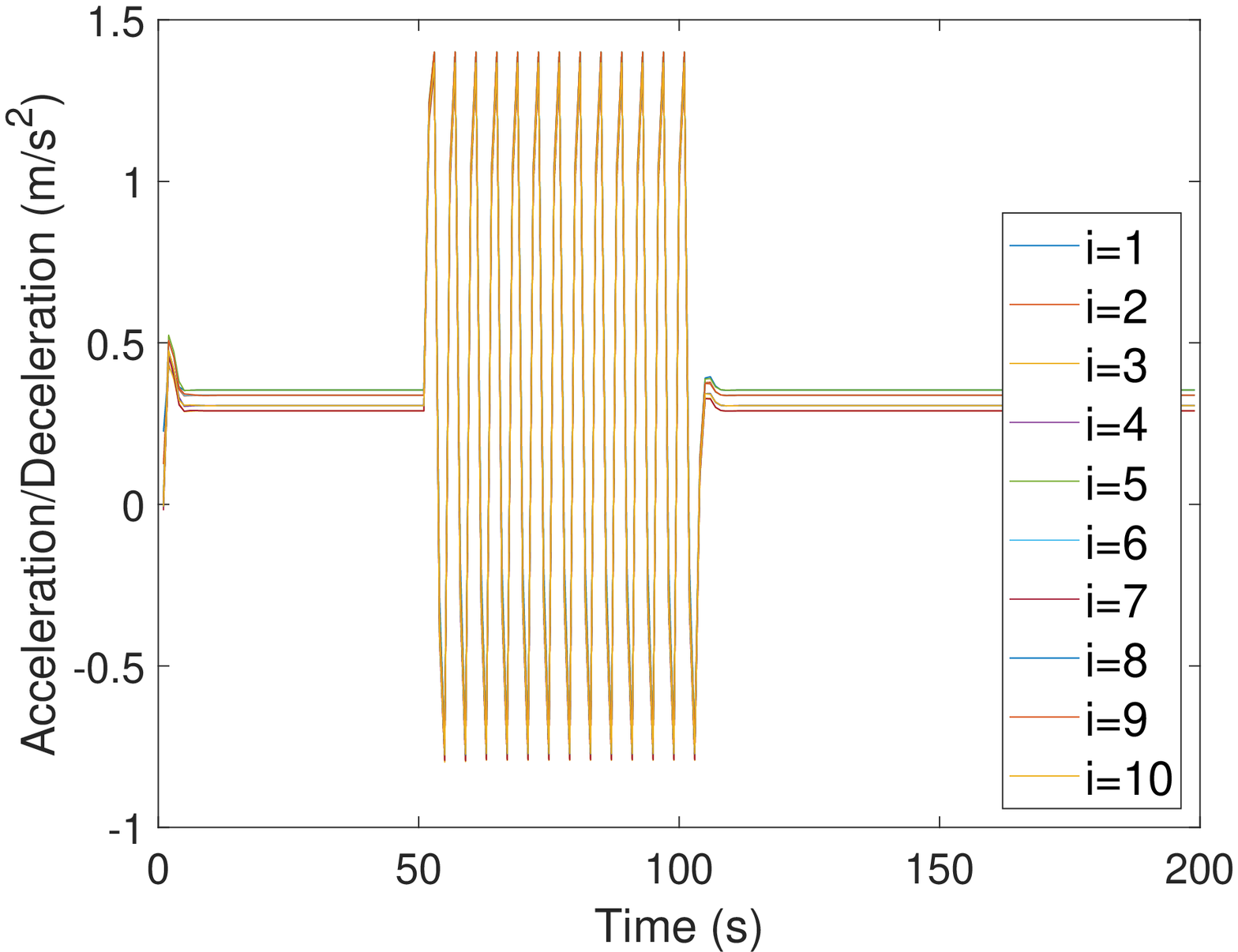}
}
\caption{Scenario 2 for the heterogeneous medium-size CAV platoon: platooning control with $p=1$ (left column) and $p=5$ (right column).}
\label{Fig:S2_medium}
\end{figure}


\begin{figure}[htbp]
\centering
\subfigure[Time history of spacing changes.]{ 
\includegraphics[width=0.48\columnwidth]{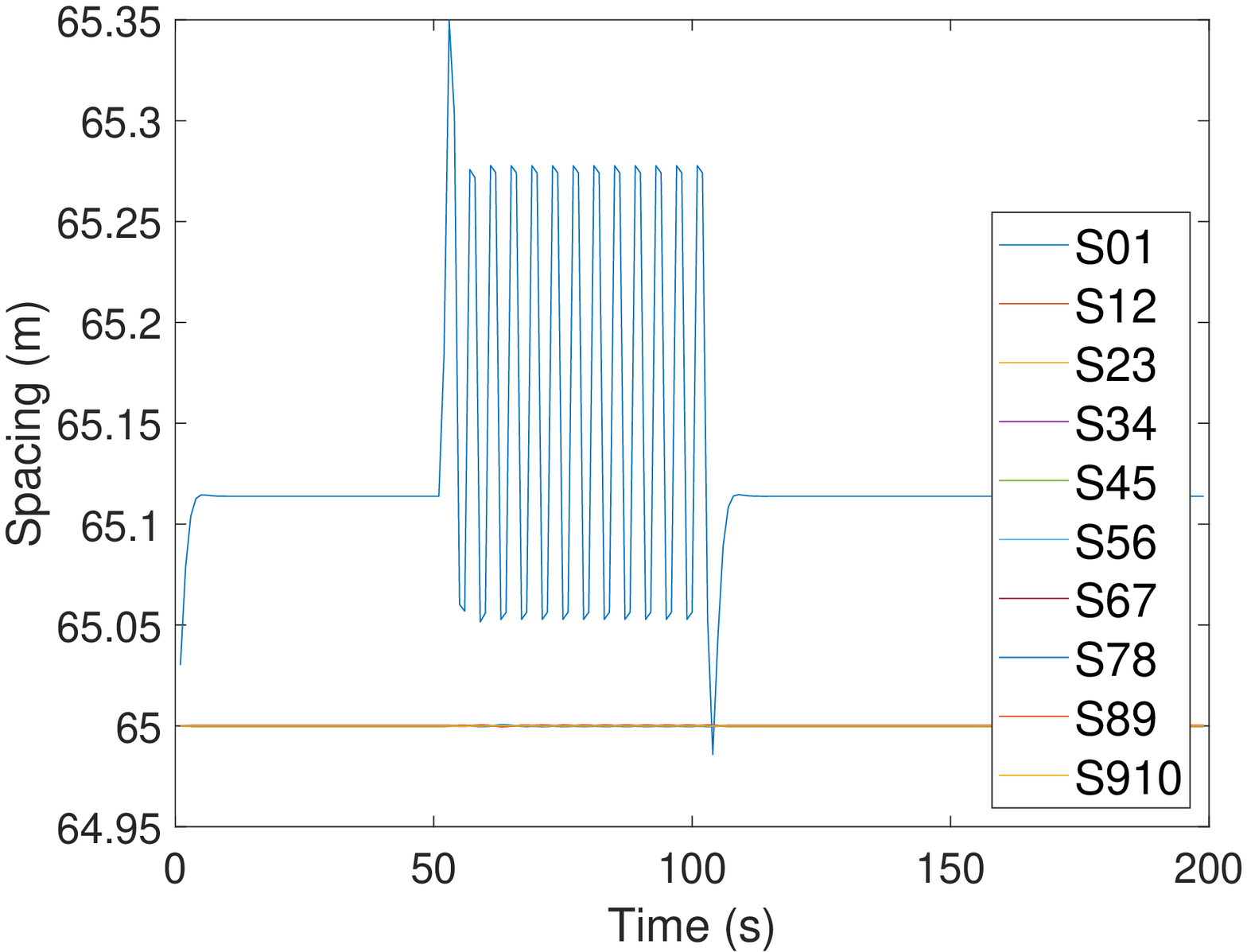}
}
\subfigure[Time history of spacing changes.] { 
\includegraphics[width=0.48\columnwidth]{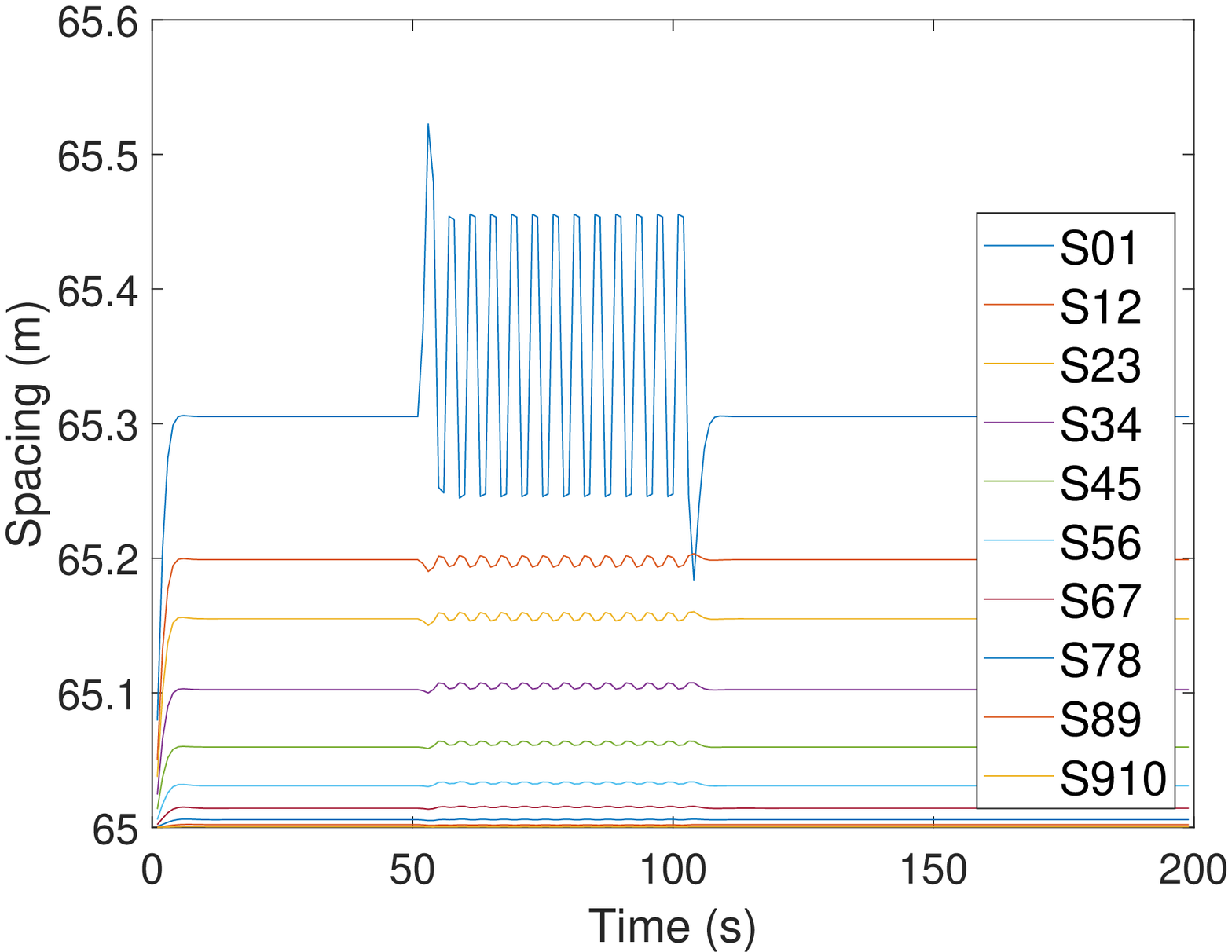}
}
\subfigure[Time history of vehicle speed. ]{ 
\includegraphics[width=0.48\columnwidth]{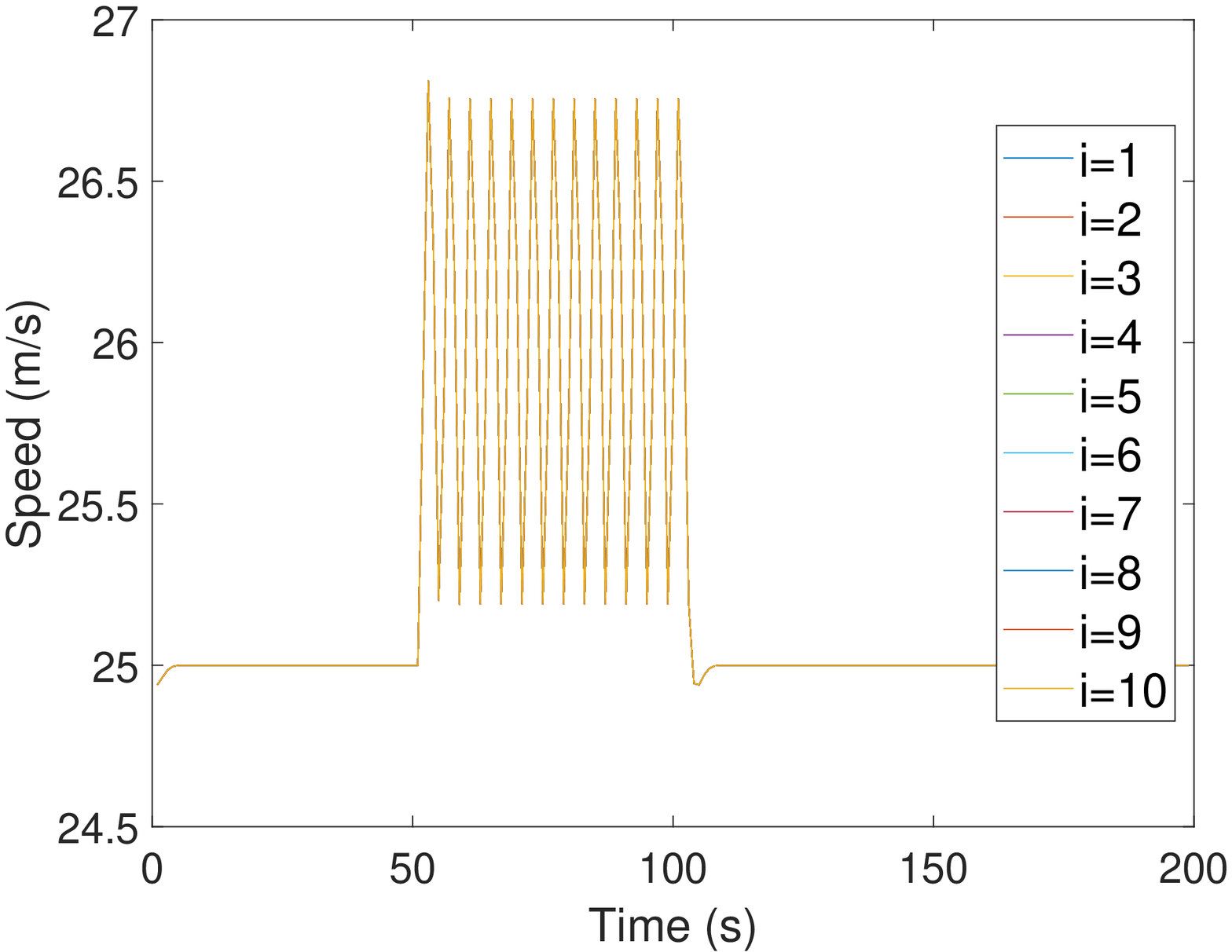}
}
\subfigure[Time history of vehicle speed.] { 
\includegraphics[width=0.48\columnwidth]{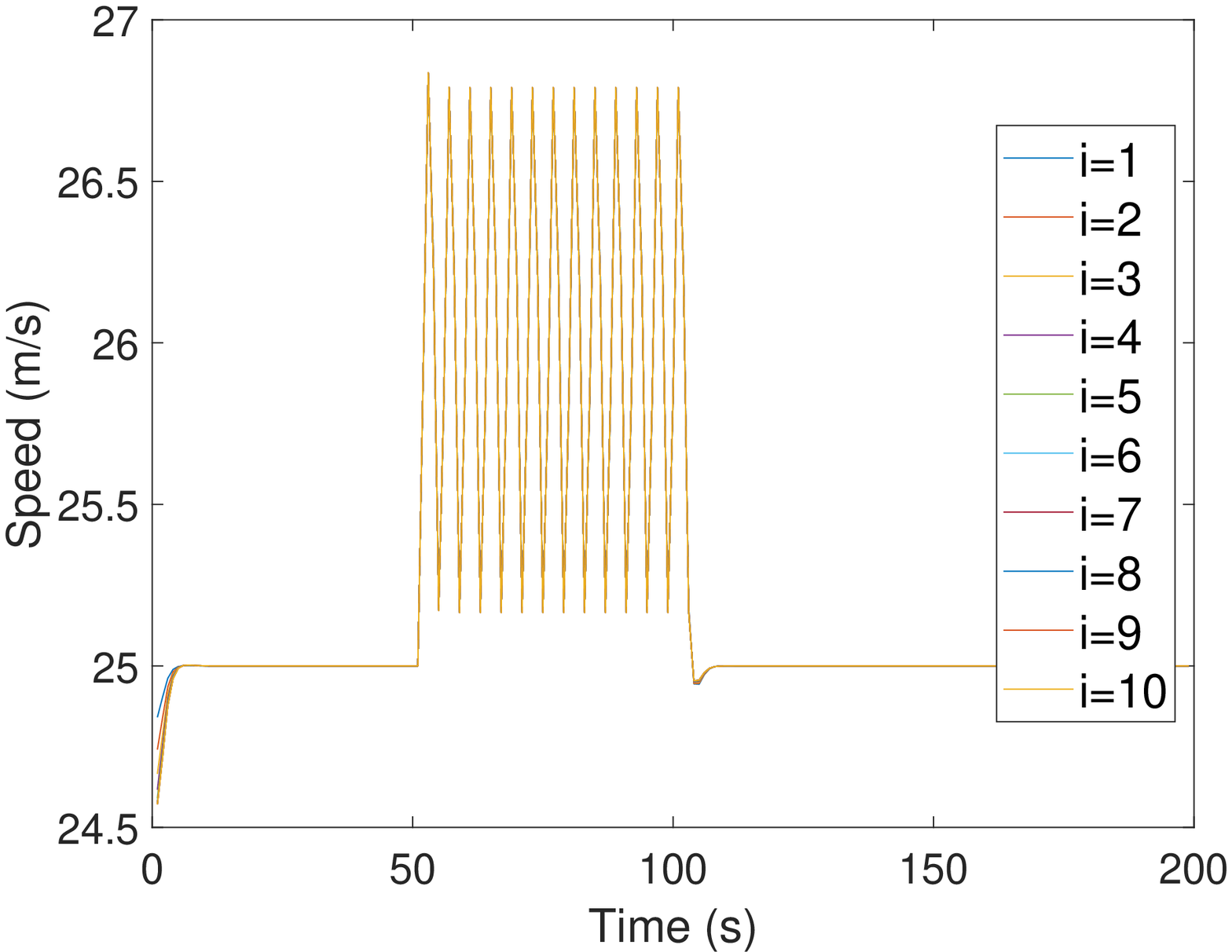}
}
\subfigure[Time history of control input.]{ 
\includegraphics[width=0.48\columnwidth]{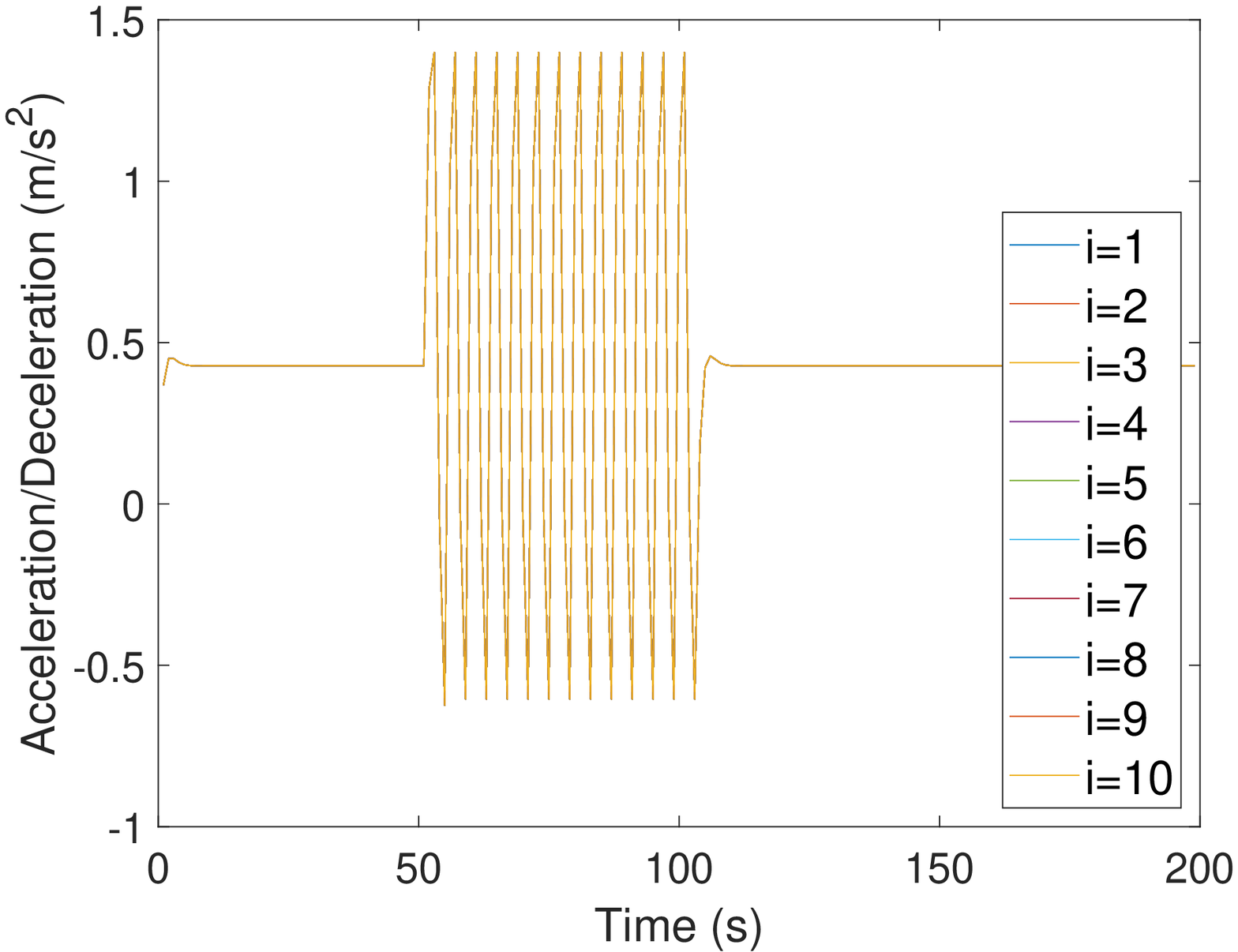}
}
\subfigure[Time history of control input] { 
\includegraphics[width=0.48\columnwidth]{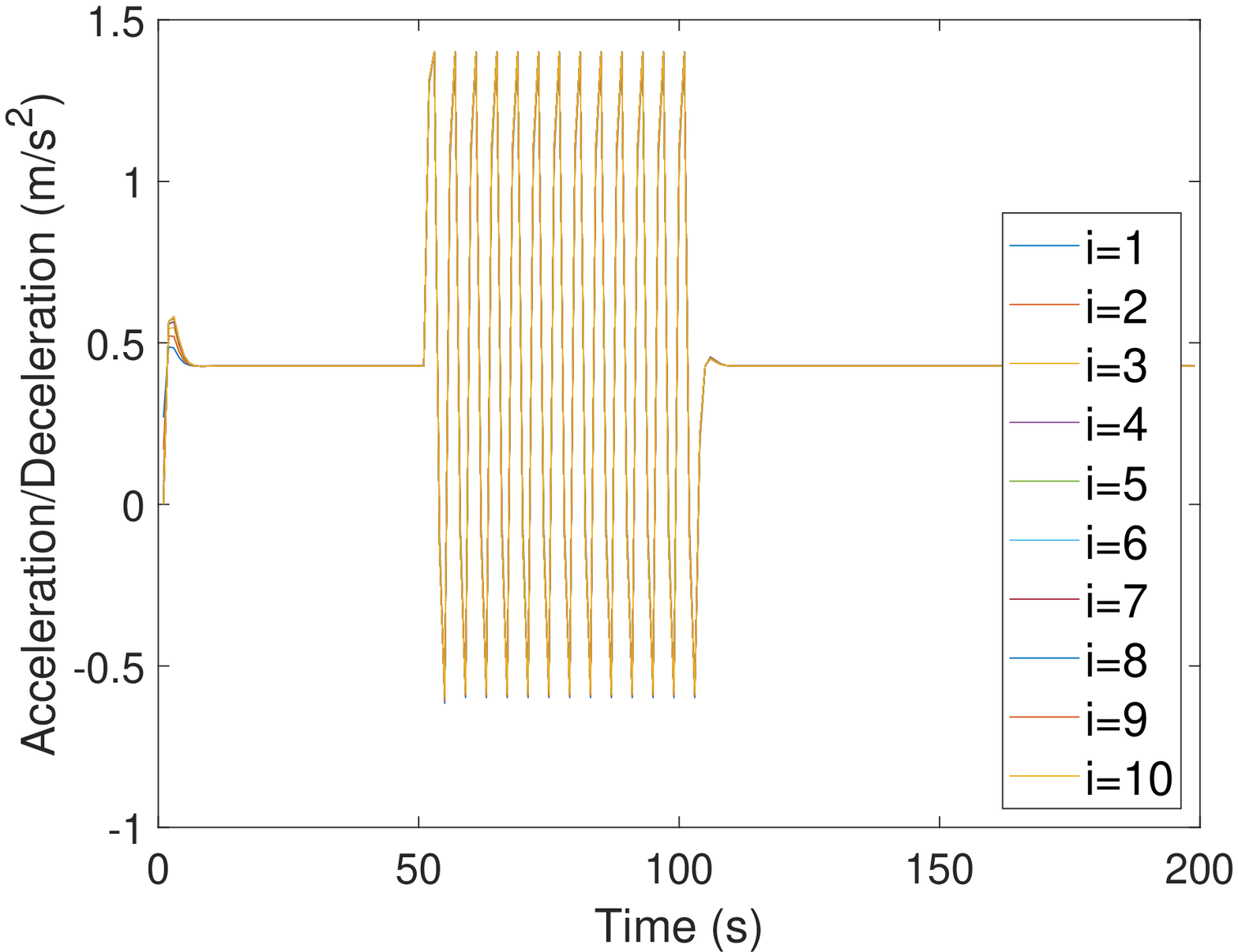}
}
\caption{Scenario 2 for the homogeneous large-size CAV platoon: platooning control with $p=1$ (left column) and $p=5$ (right column).}
\label{Fig:S2_large}
\end{figure}


\begin{figure}[htbp]
\centering
\subfigure[Time history of spacing changes.]{ 
\includegraphics[width=0.48\columnwidth]{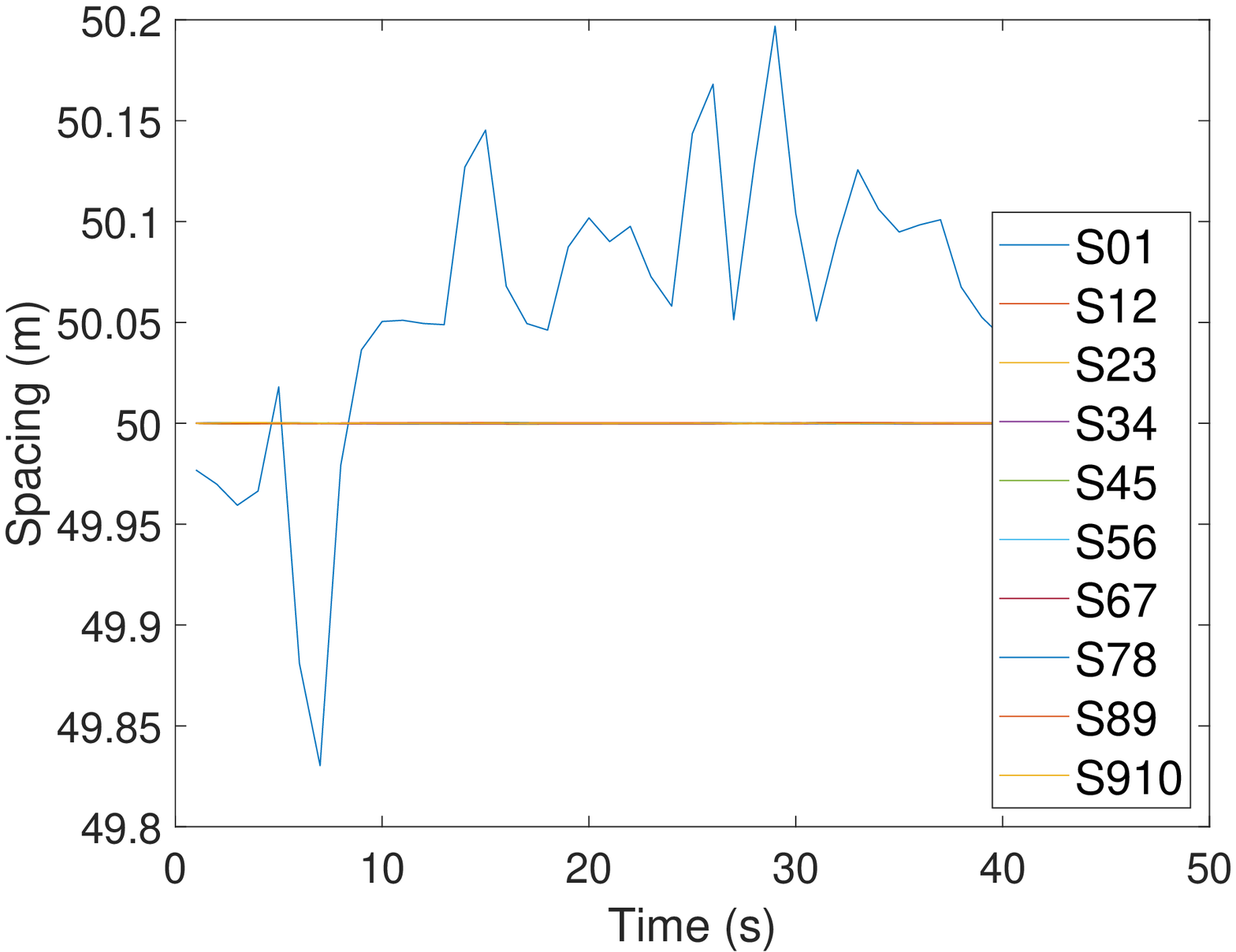}
}
\subfigure[Time history of spacing changes.] { 
\includegraphics[width=0.48\columnwidth]{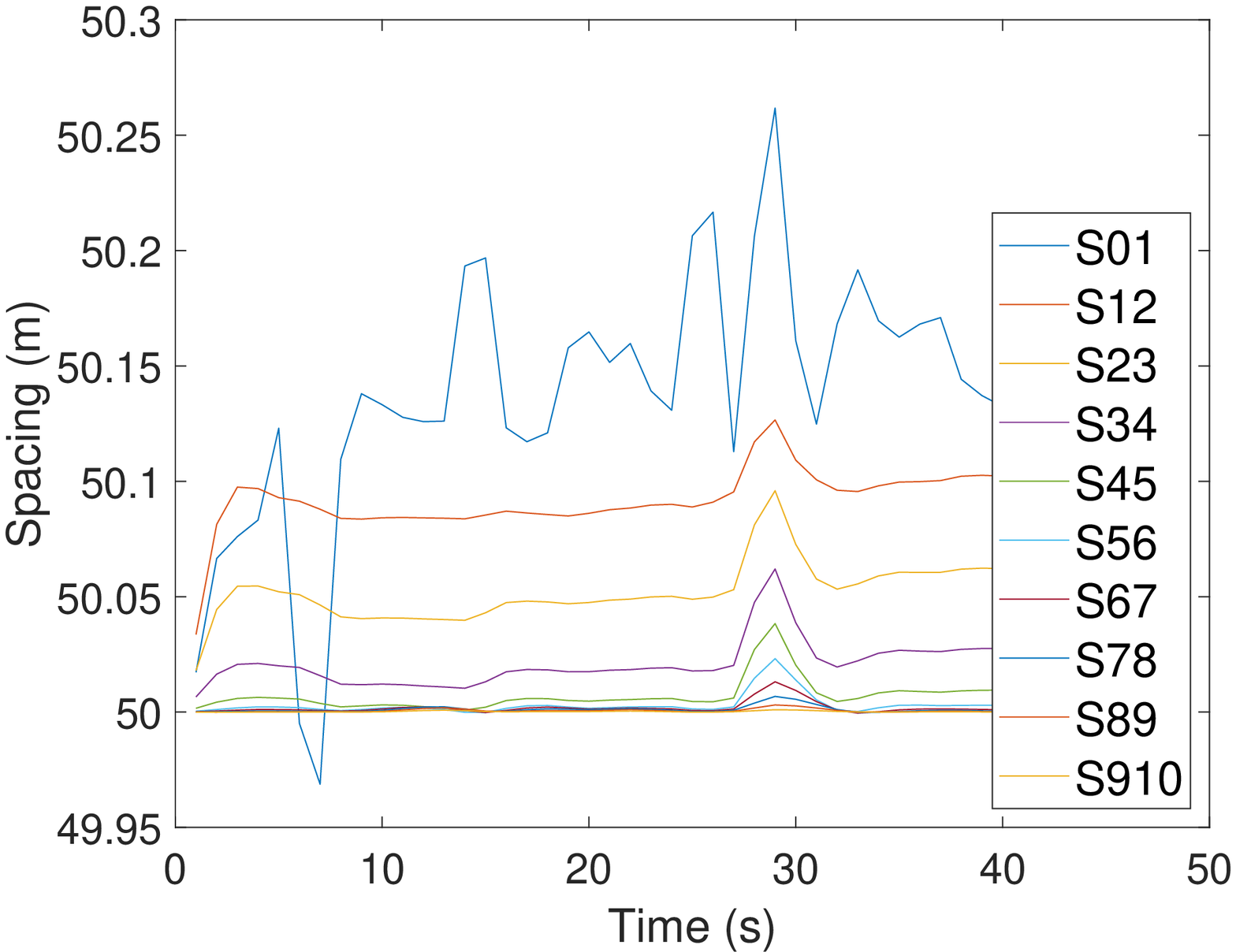}
}
\subfigure[Time history of vehicle speed. ]{ 
\includegraphics[width=0.48\columnwidth]{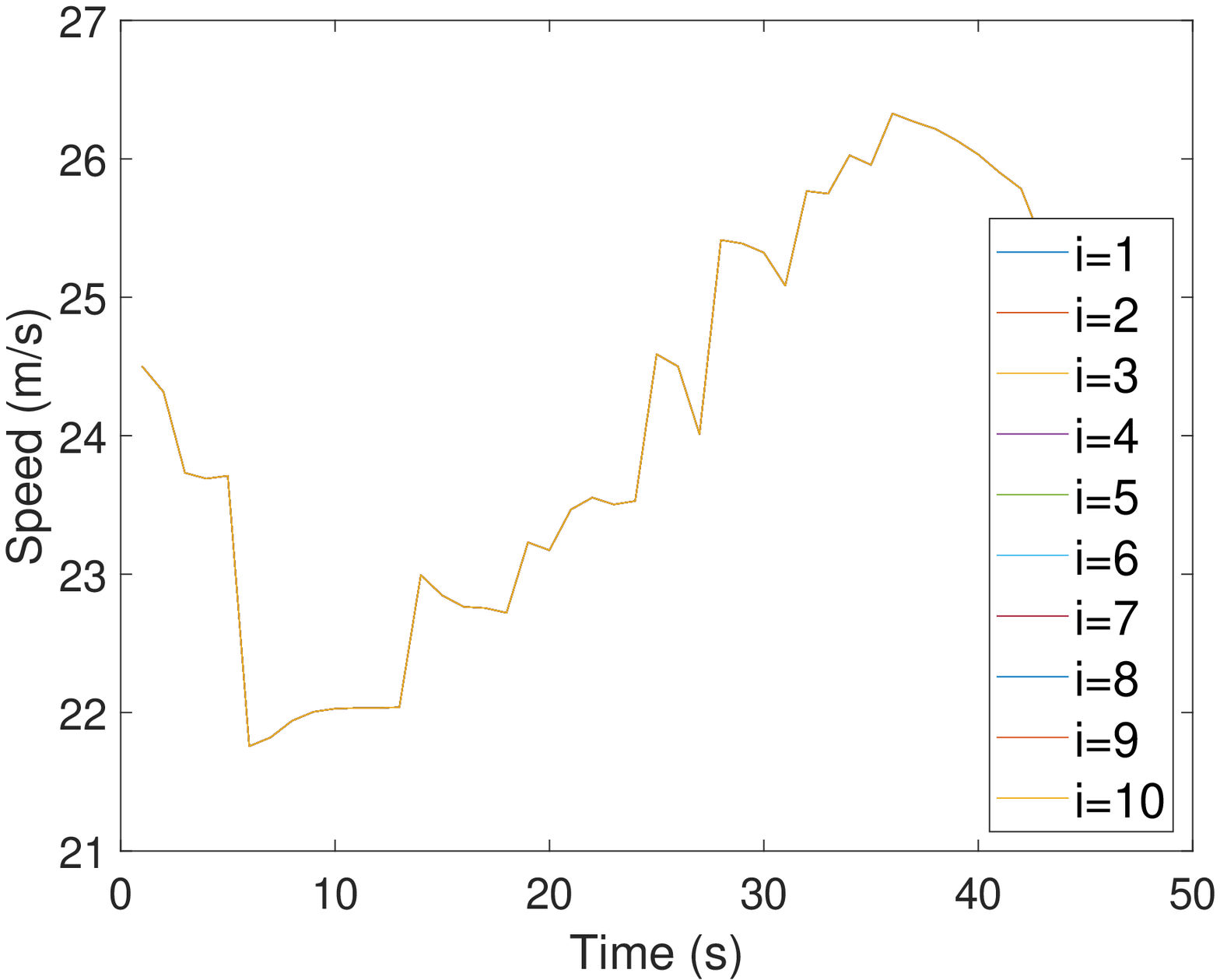}
}
\subfigure[Time history of vehicle speed.] { 
\includegraphics[width=0.48\columnwidth]{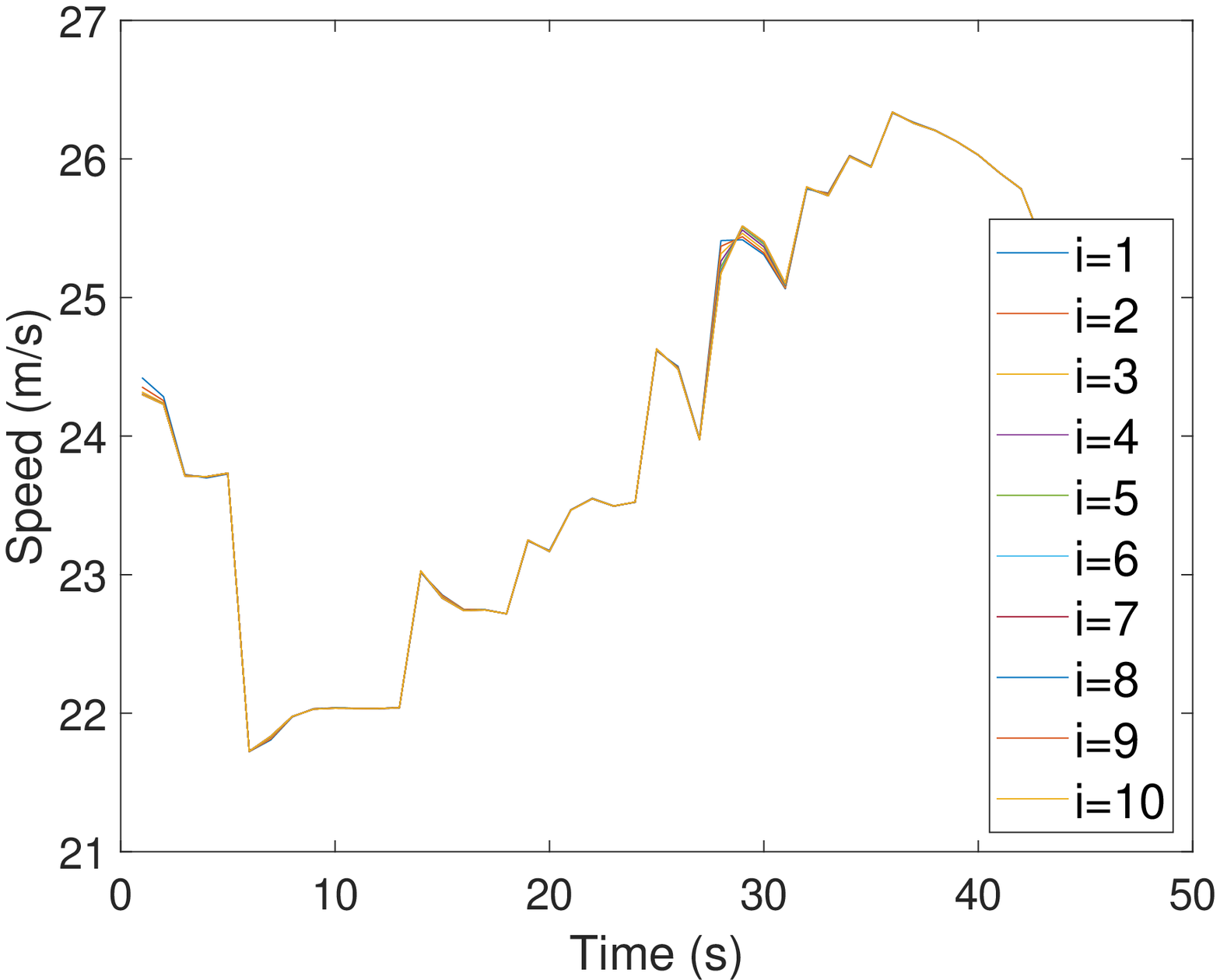}
}
\subfigure[Time history of control input.]{ 
\includegraphics[width=0.48\columnwidth]{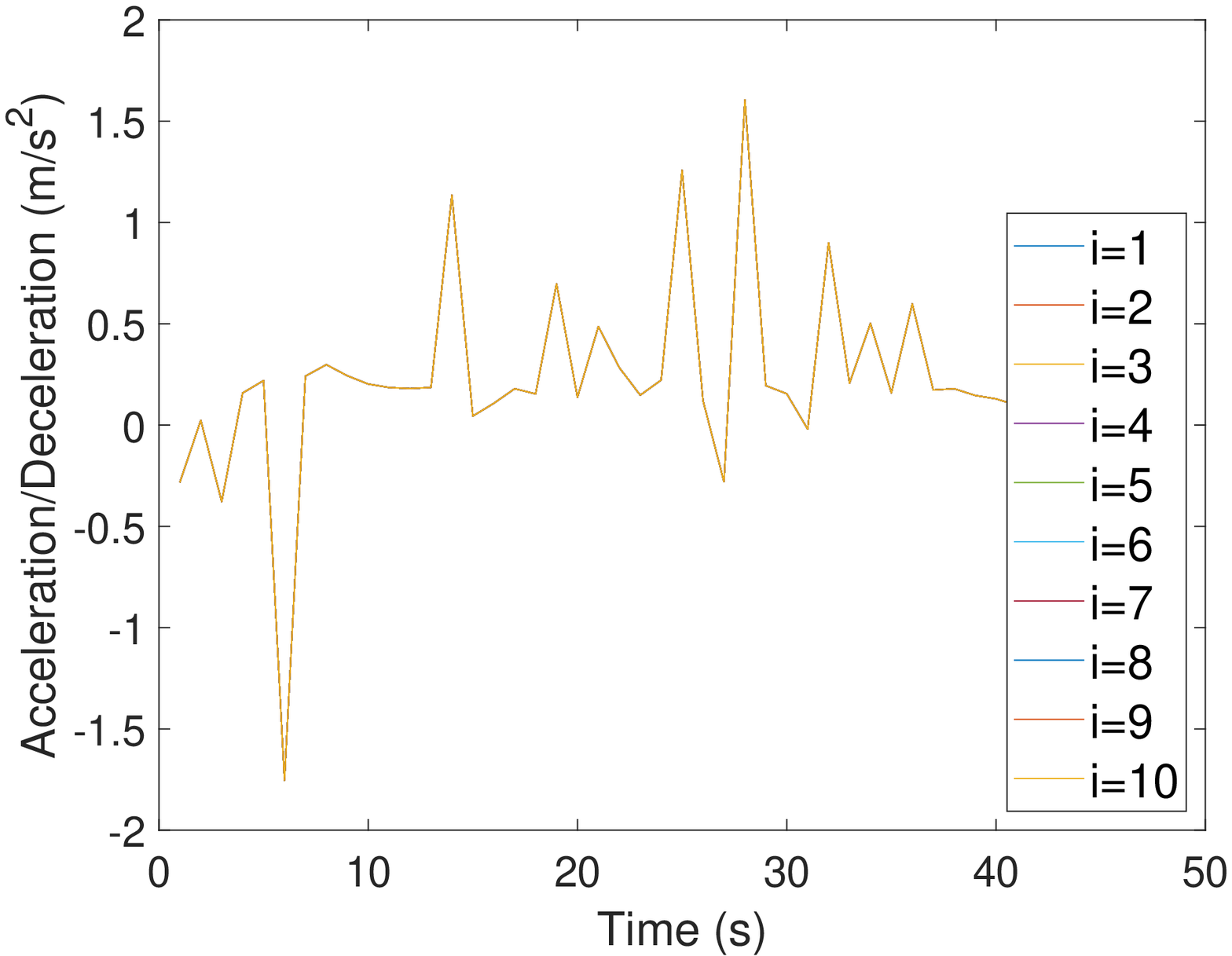}
}
\subfigure[Time history of control input] { 
\includegraphics[width=0.48\columnwidth]{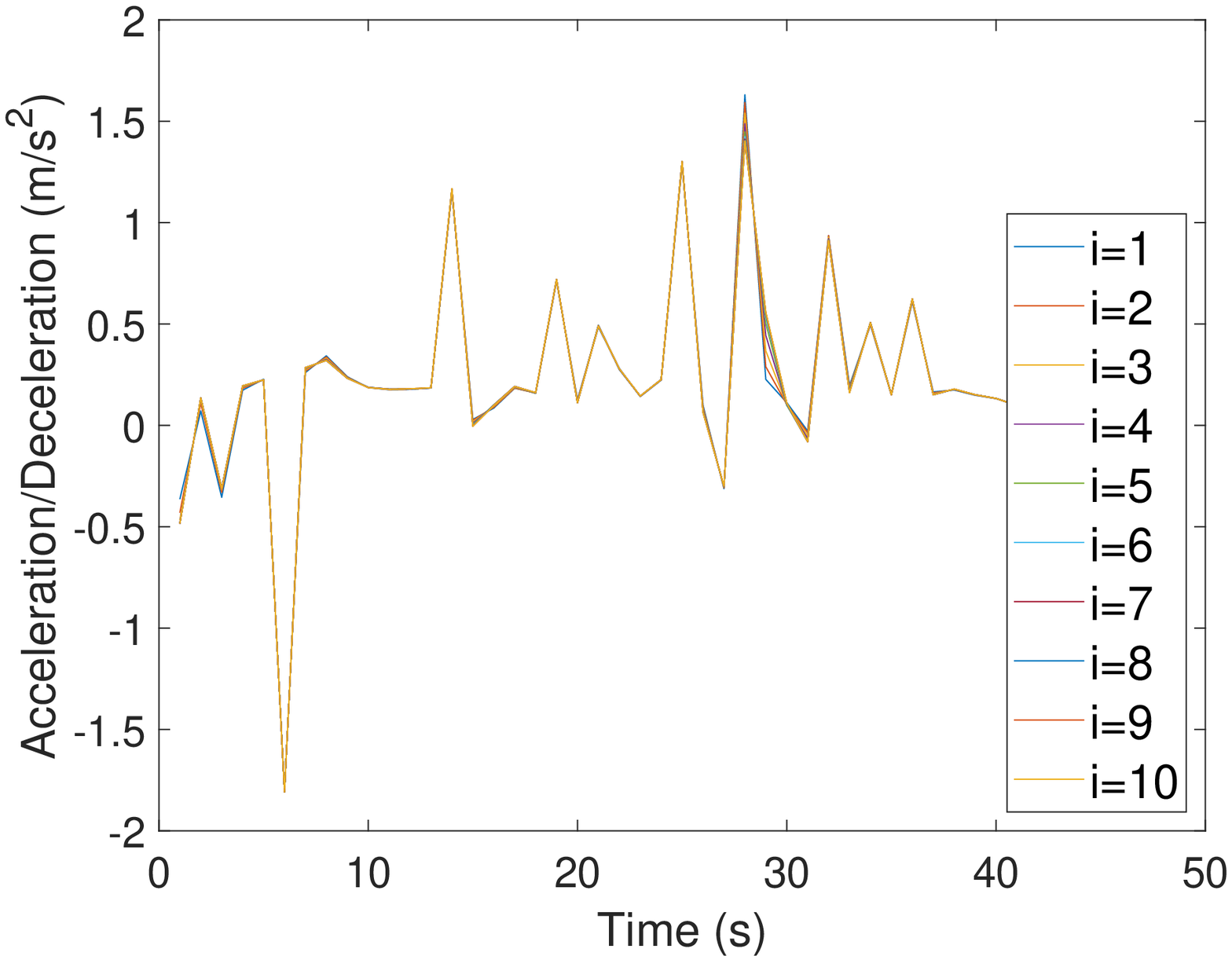}
}
\caption{Scenario 3 for the homogeneous small-size CAV platoon: platooning control with $p=1$ (left column) and $p=5$ (right column).}
\label{Fig:S3_small}
\end{figure}


\begin{figure}[htbp]
\centering
\subfigure[Time history of spacing changes.]{ 
\includegraphics[width=0.48\columnwidth]{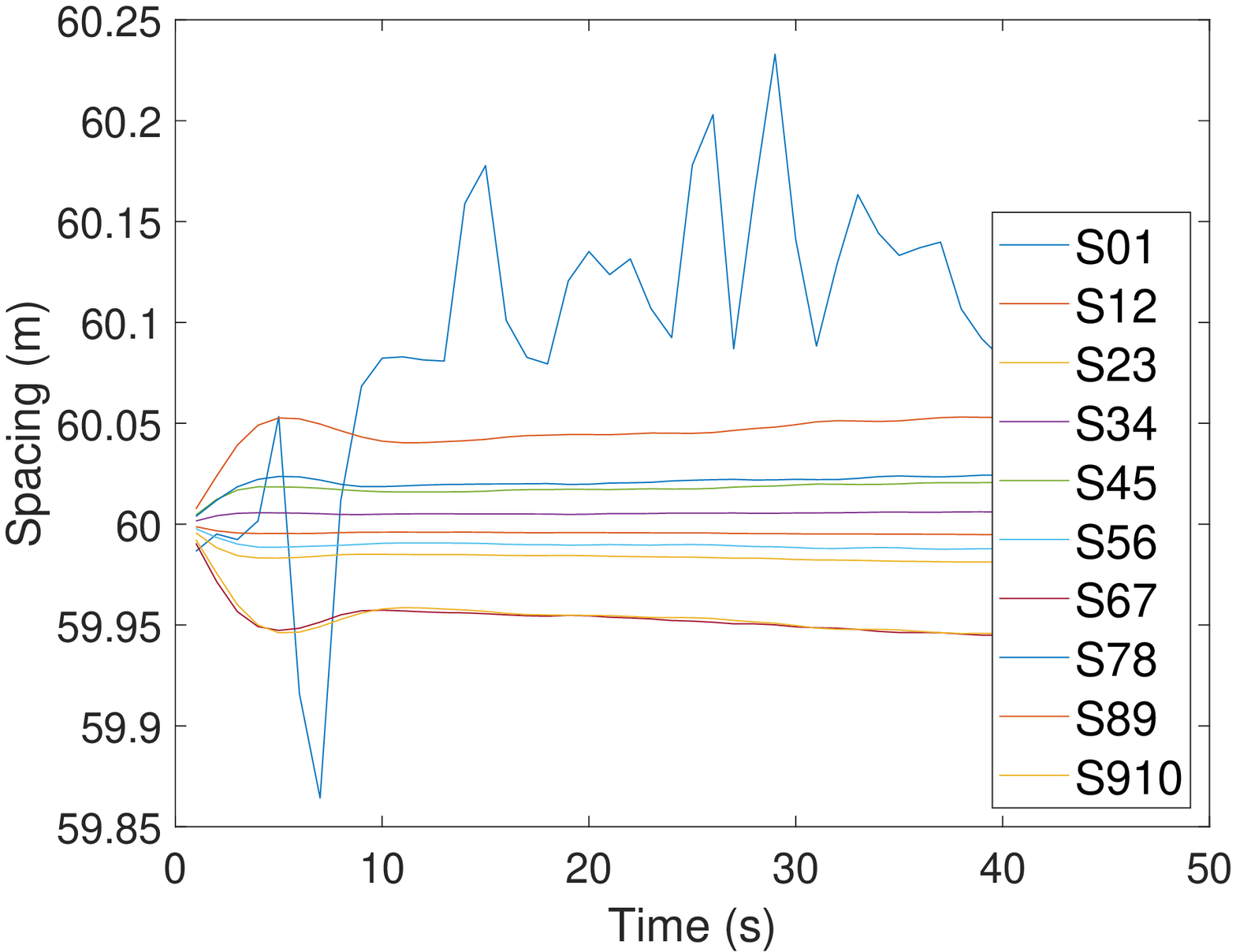}
}
\subfigure[Time history of spacing changes.] { 
\includegraphics[width=0.48\columnwidth]{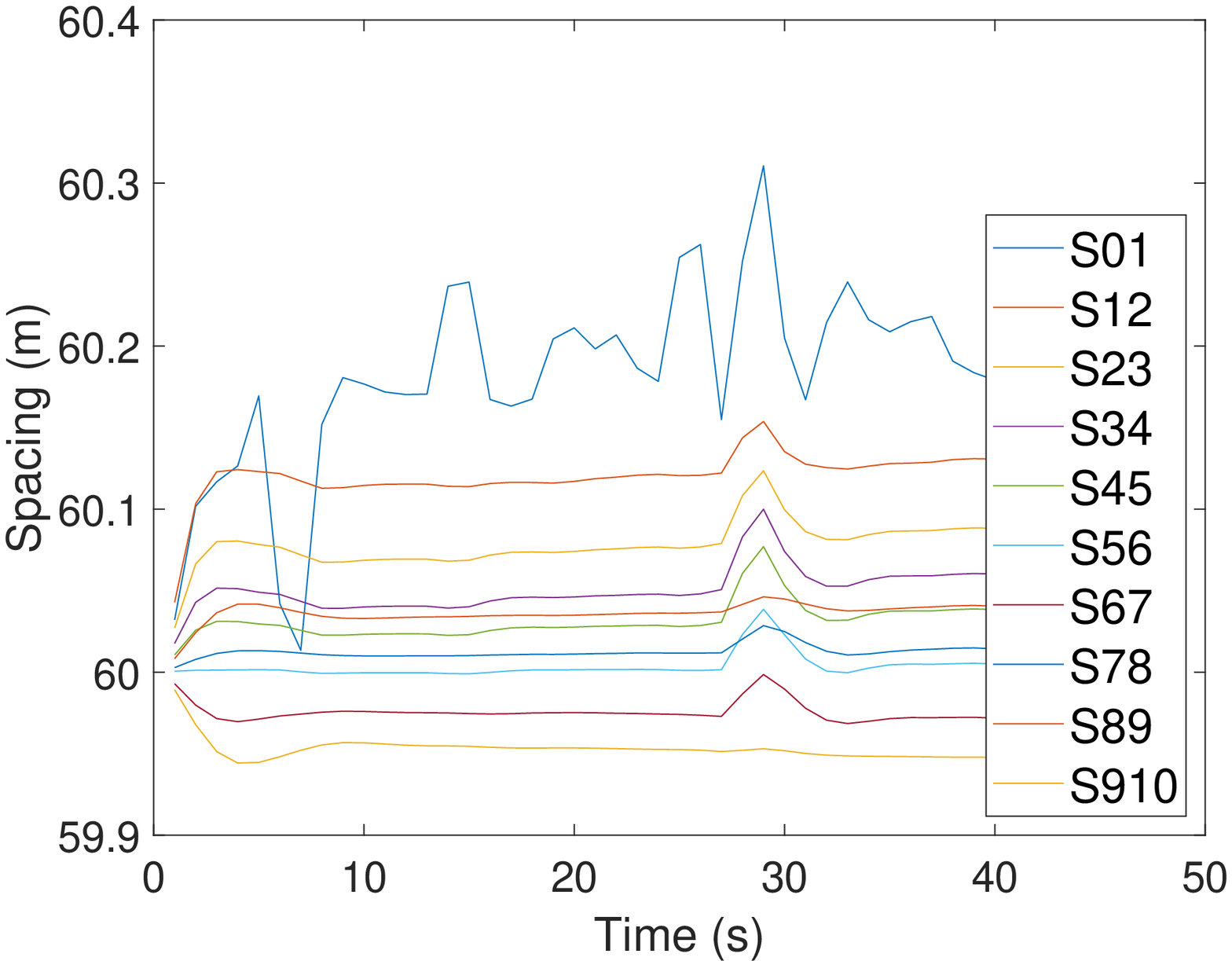}
}
\subfigure[Time history of vehicle speed. ]{ 
\includegraphics[width=0.48\columnwidth]{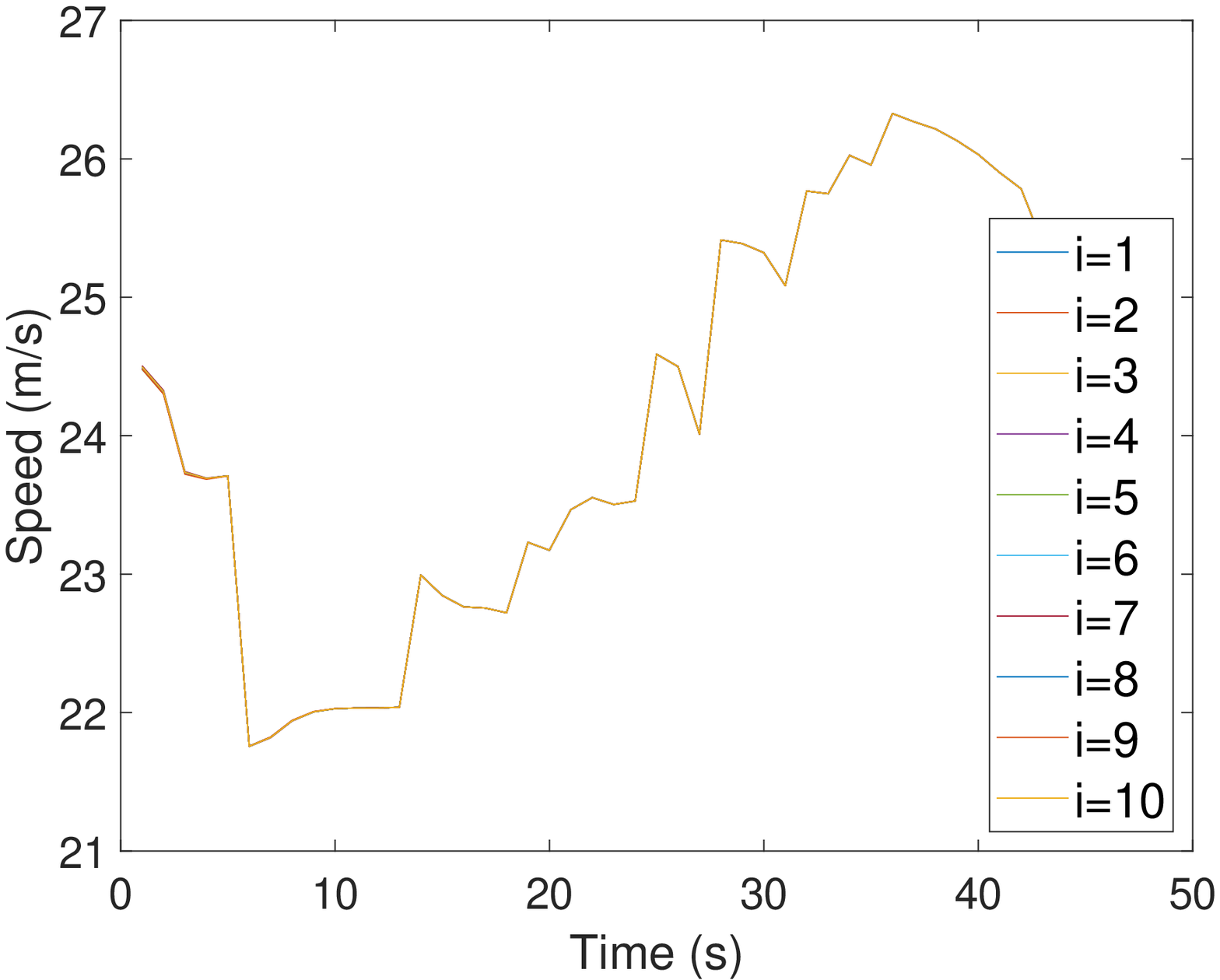}
}
\subfigure[Time history of vehicle speed.] { 
\includegraphics[width=0.48\columnwidth]{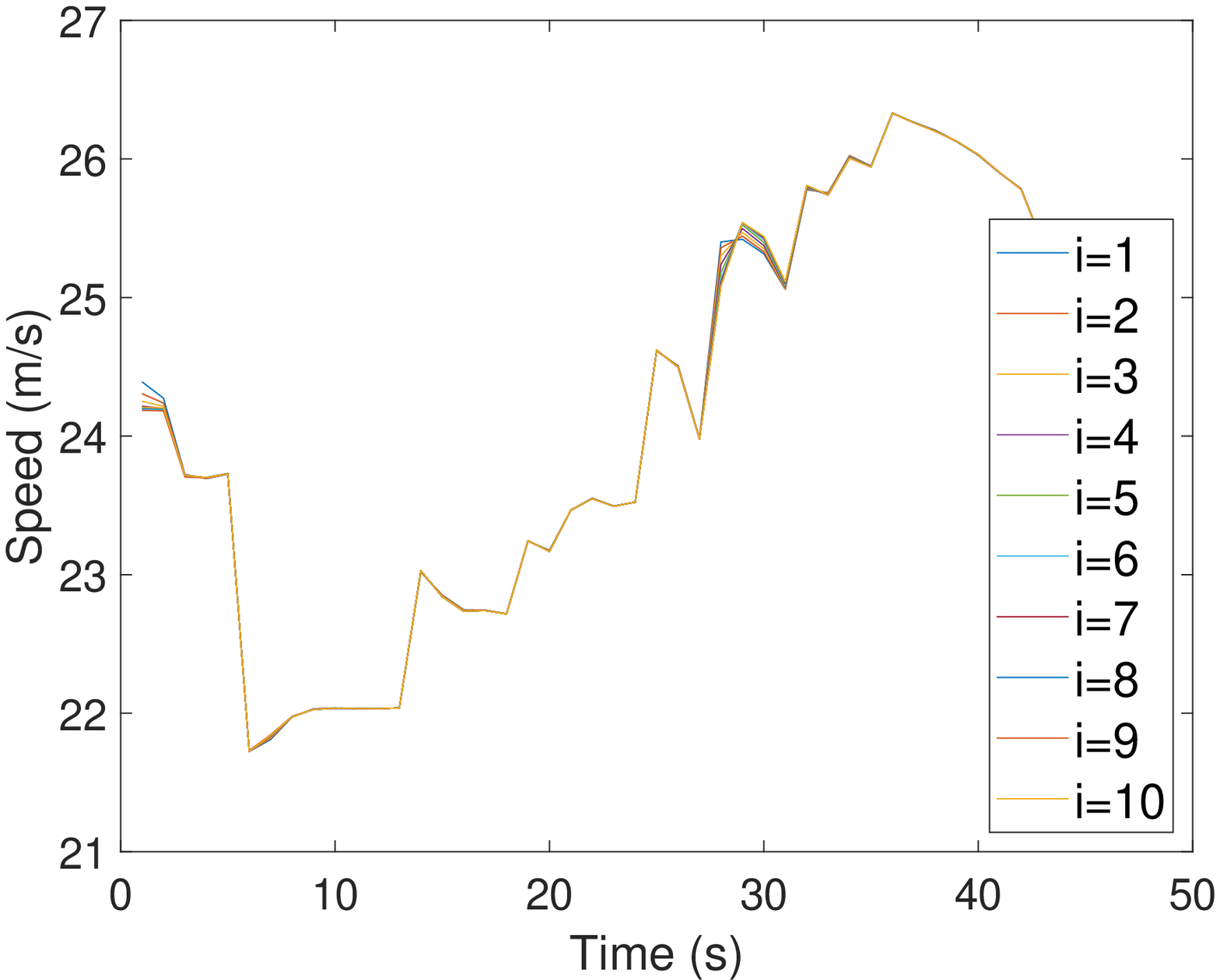}
}
\subfigure[Time history of control input.]{ 
\includegraphics[width=0.48\columnwidth]{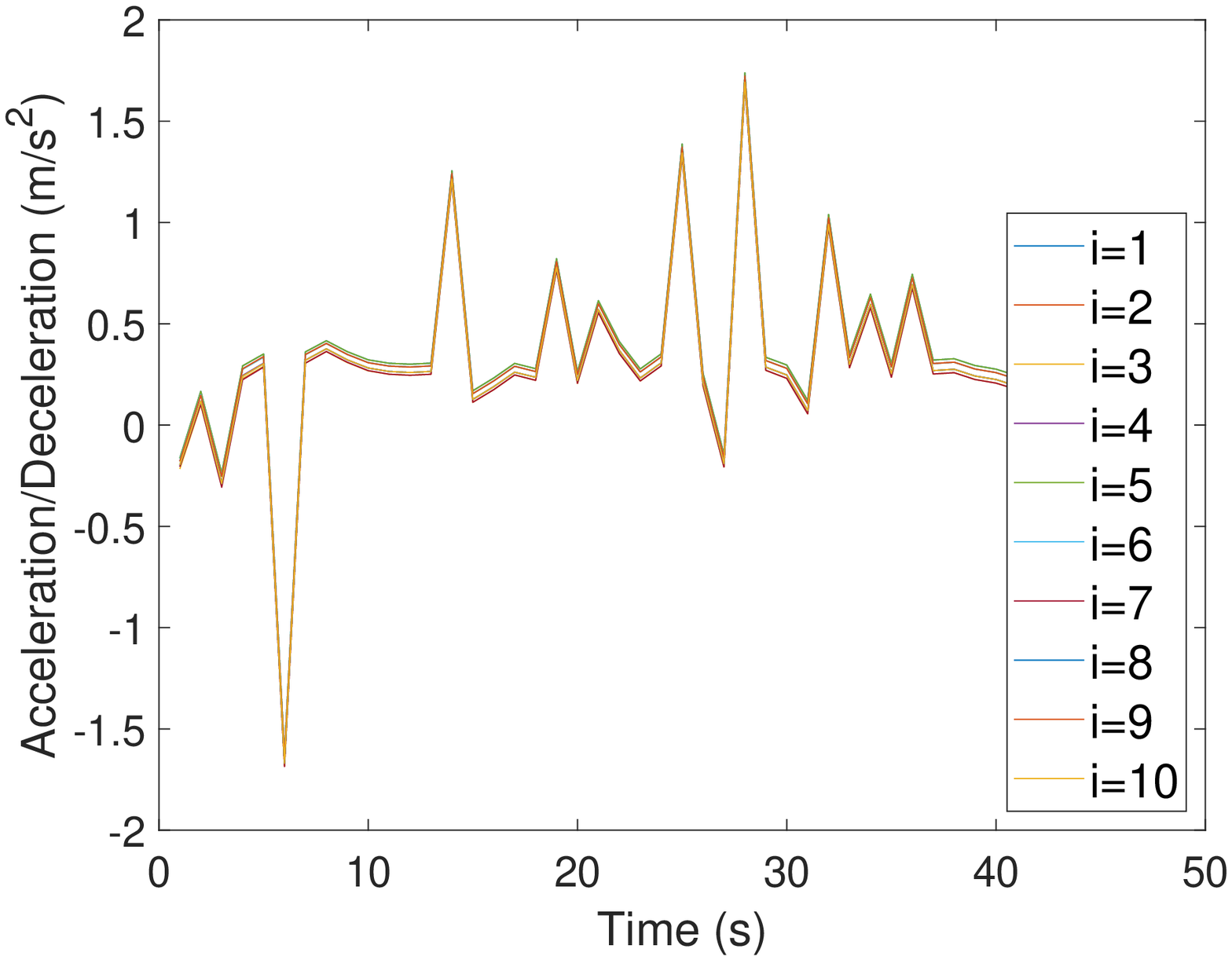}
}
\subfigure[Time history of control input] { 
\includegraphics[width=0.48\columnwidth]{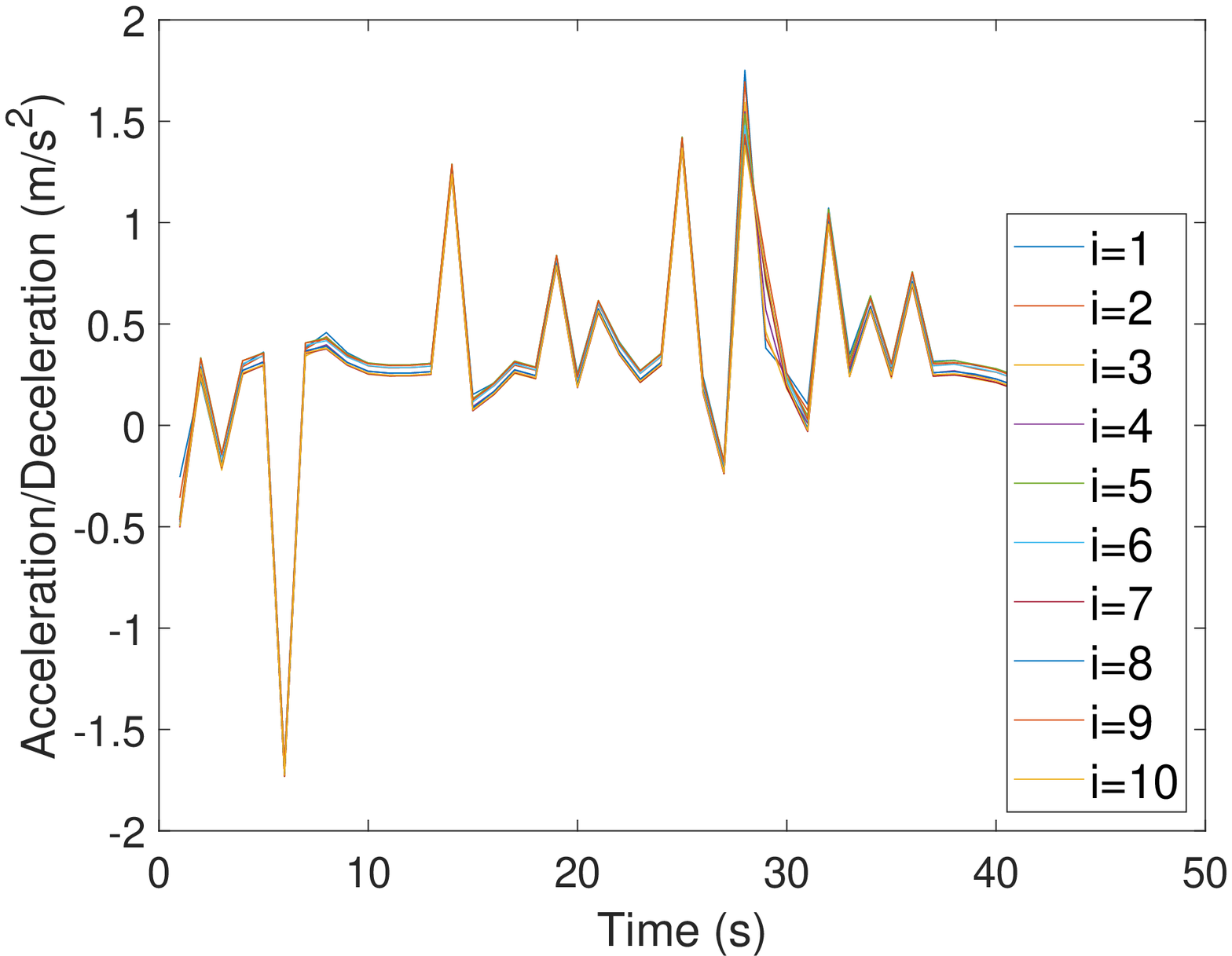}
}
\caption{Scenario 3 for the heterogeneous medium-size CAV platoon: platooning control with $p=1$ (left column) and $p=5$ (right column).}
\label{Fig:S3_medium}
\end{figure}


\begin{figure}[htbp]
\centering
\subfigure[Time history of spacing changes.]{ 
\includegraphics[width=0.48\columnwidth]{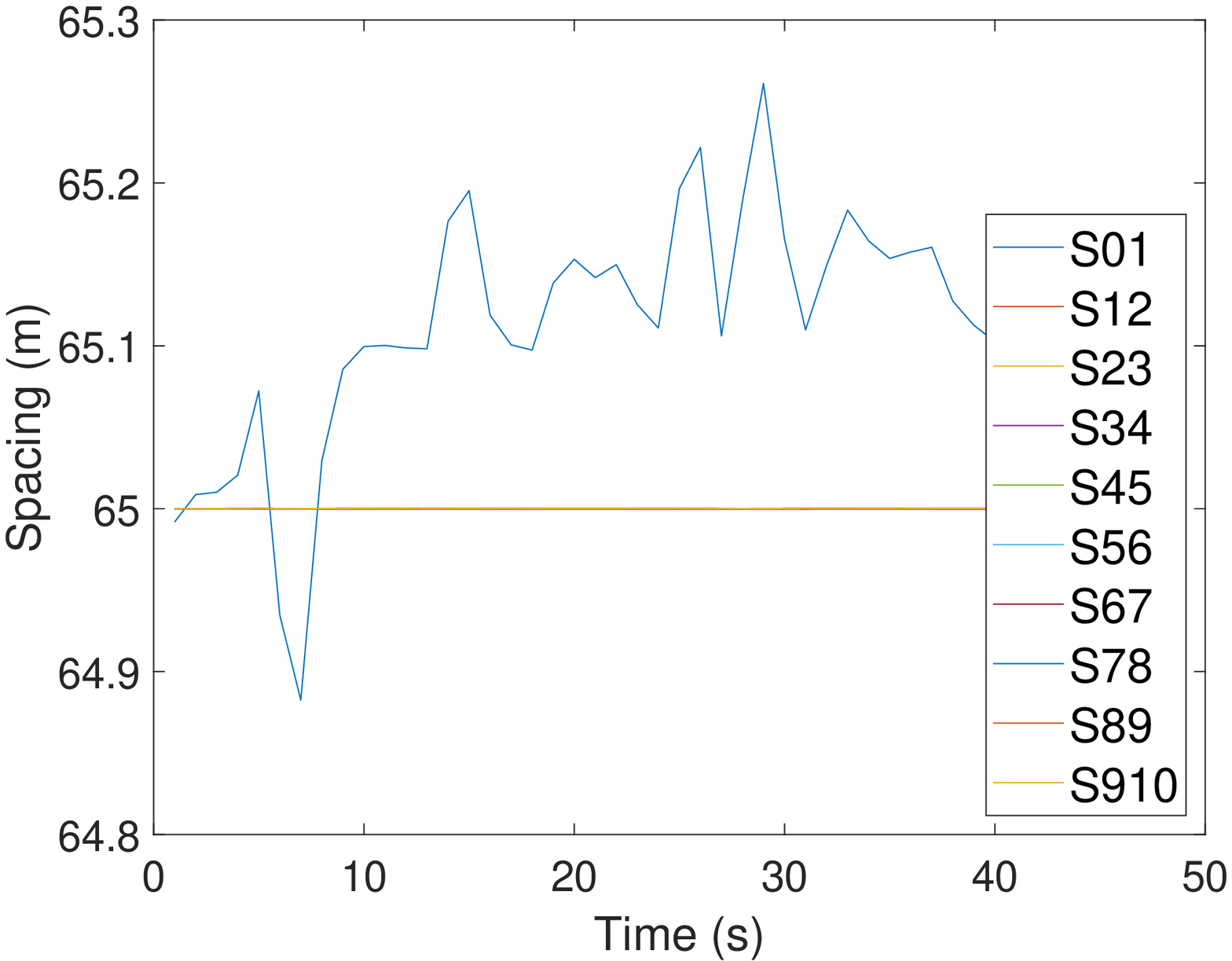}
}
\subfigure[Time history of spacing changes.] { 
\includegraphics[width=0.48\columnwidth]{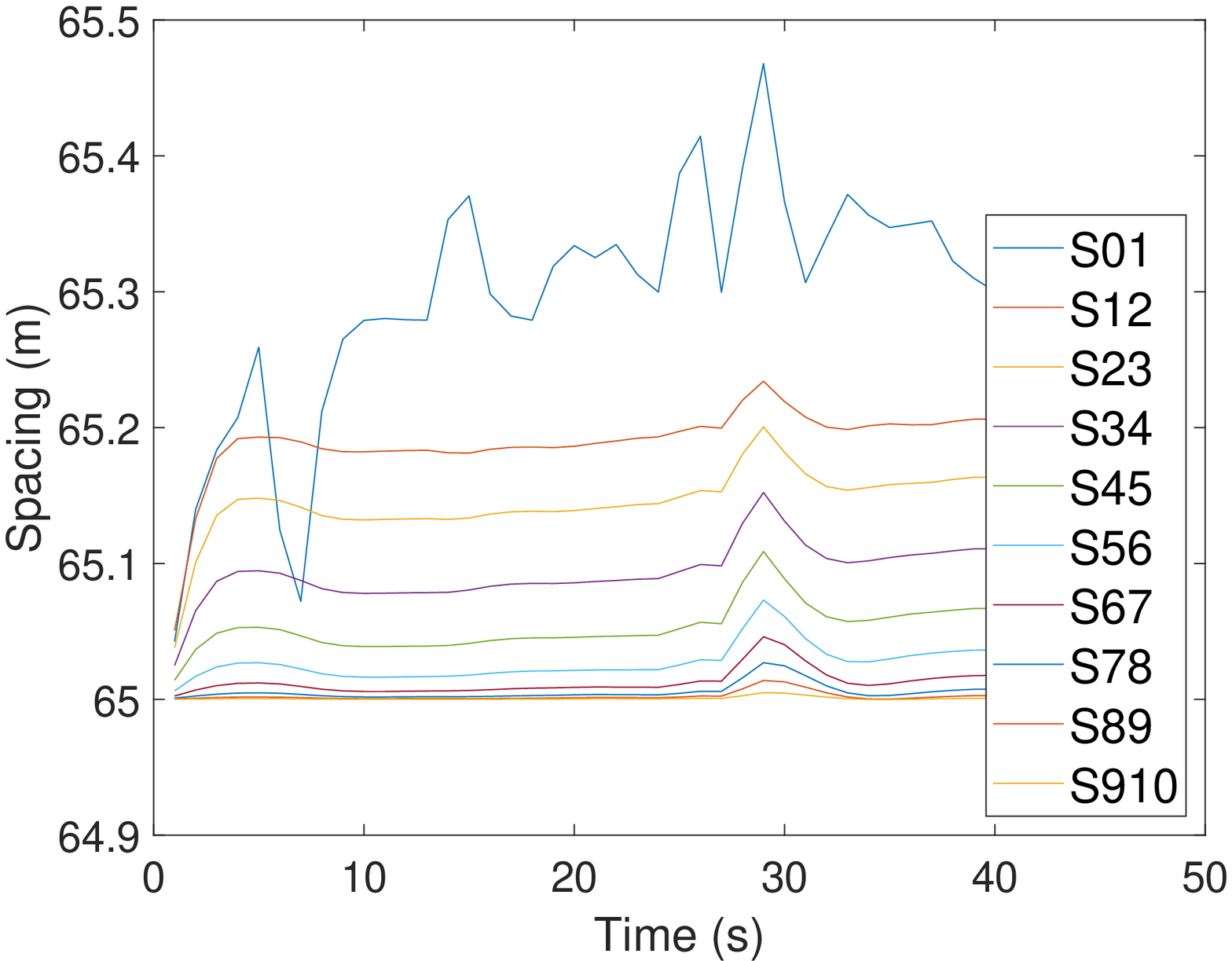}
}
\subfigure[Time history of vehicle speed. ]{ 
\includegraphics[width=0.48\columnwidth]{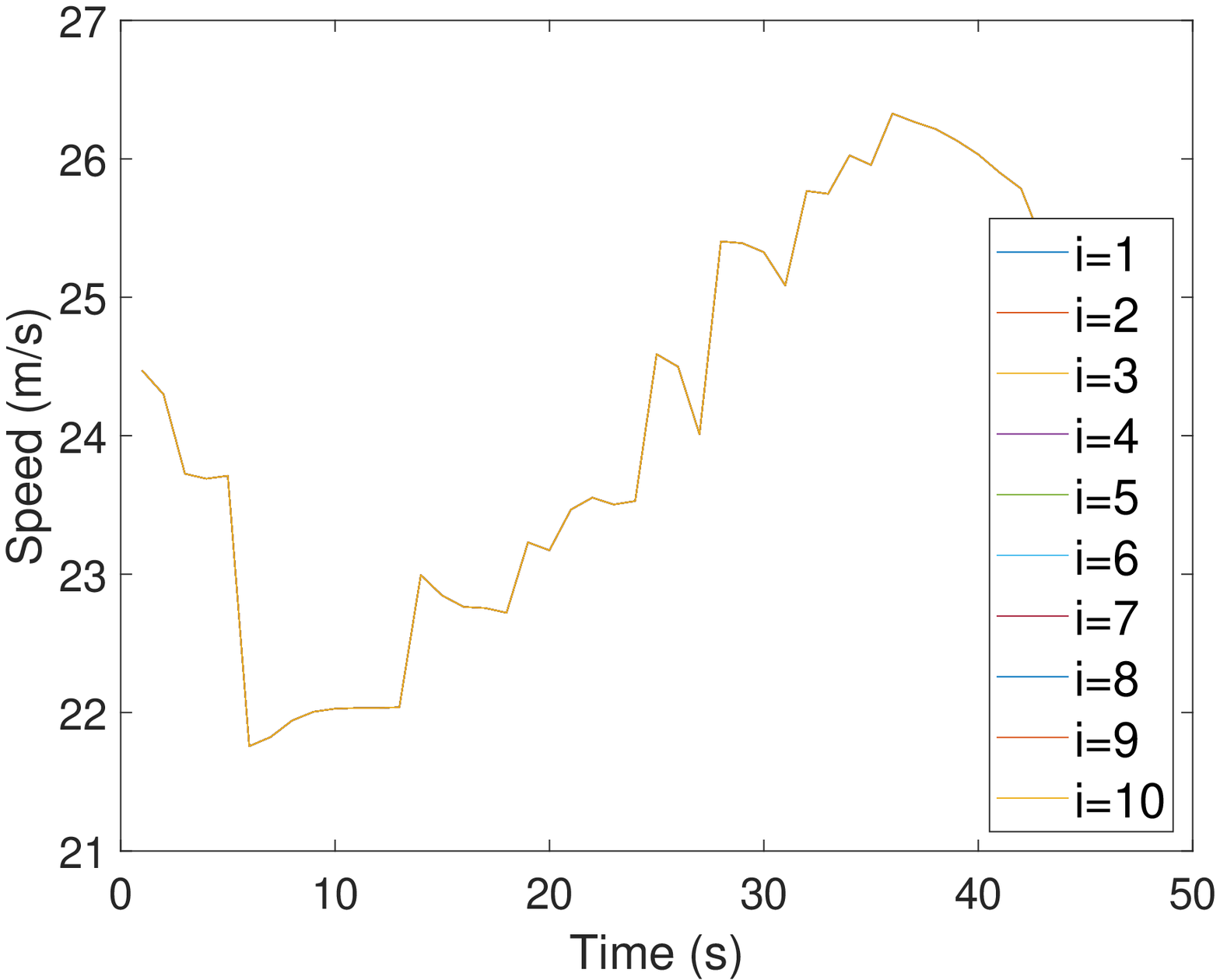}
}
\subfigure[Time history of vehicle speed.] { 
\includegraphics[width=0.48\columnwidth]{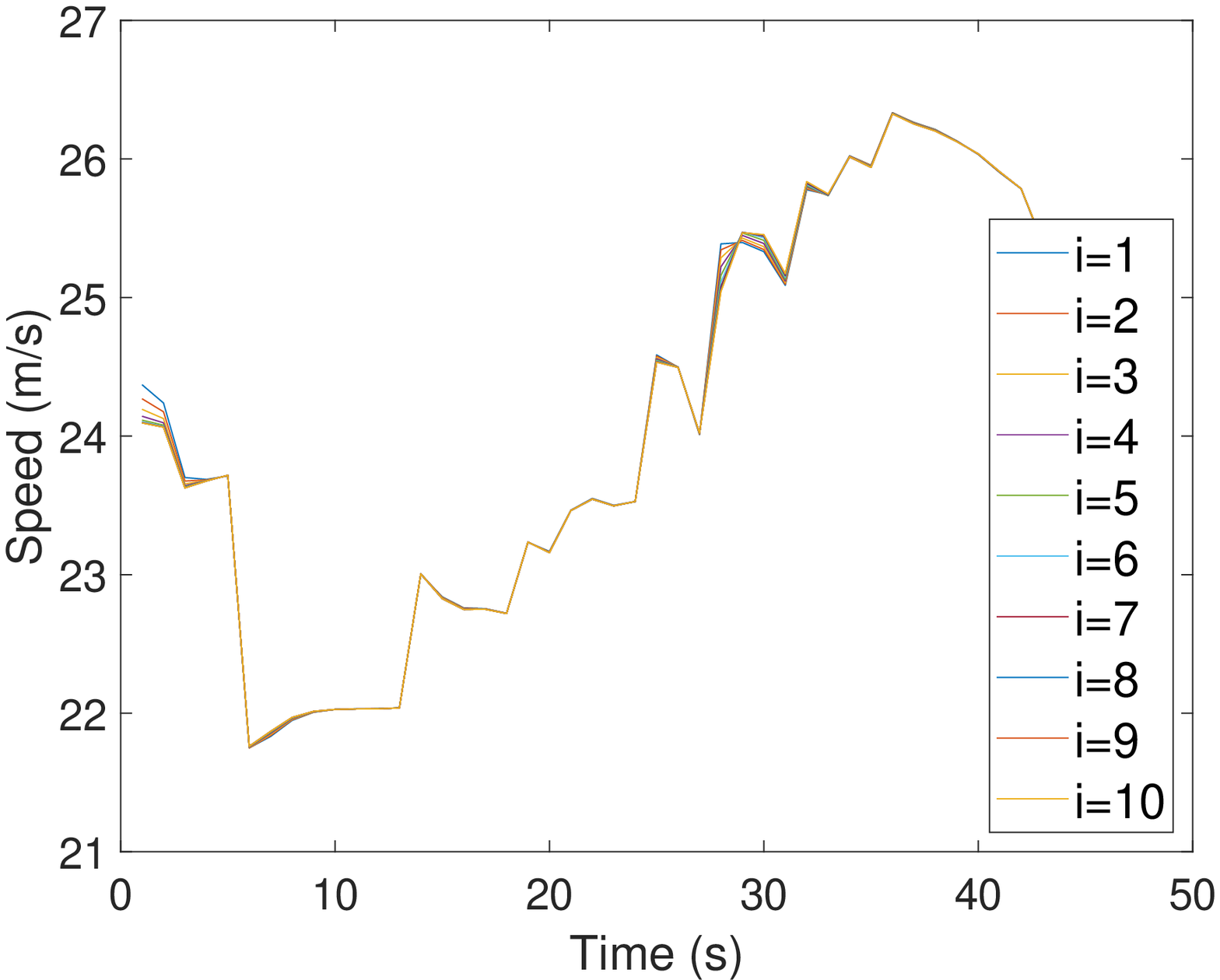}
}
\subfigure[Time history of control input.]{ 
\includegraphics[width=0.48\columnwidth]{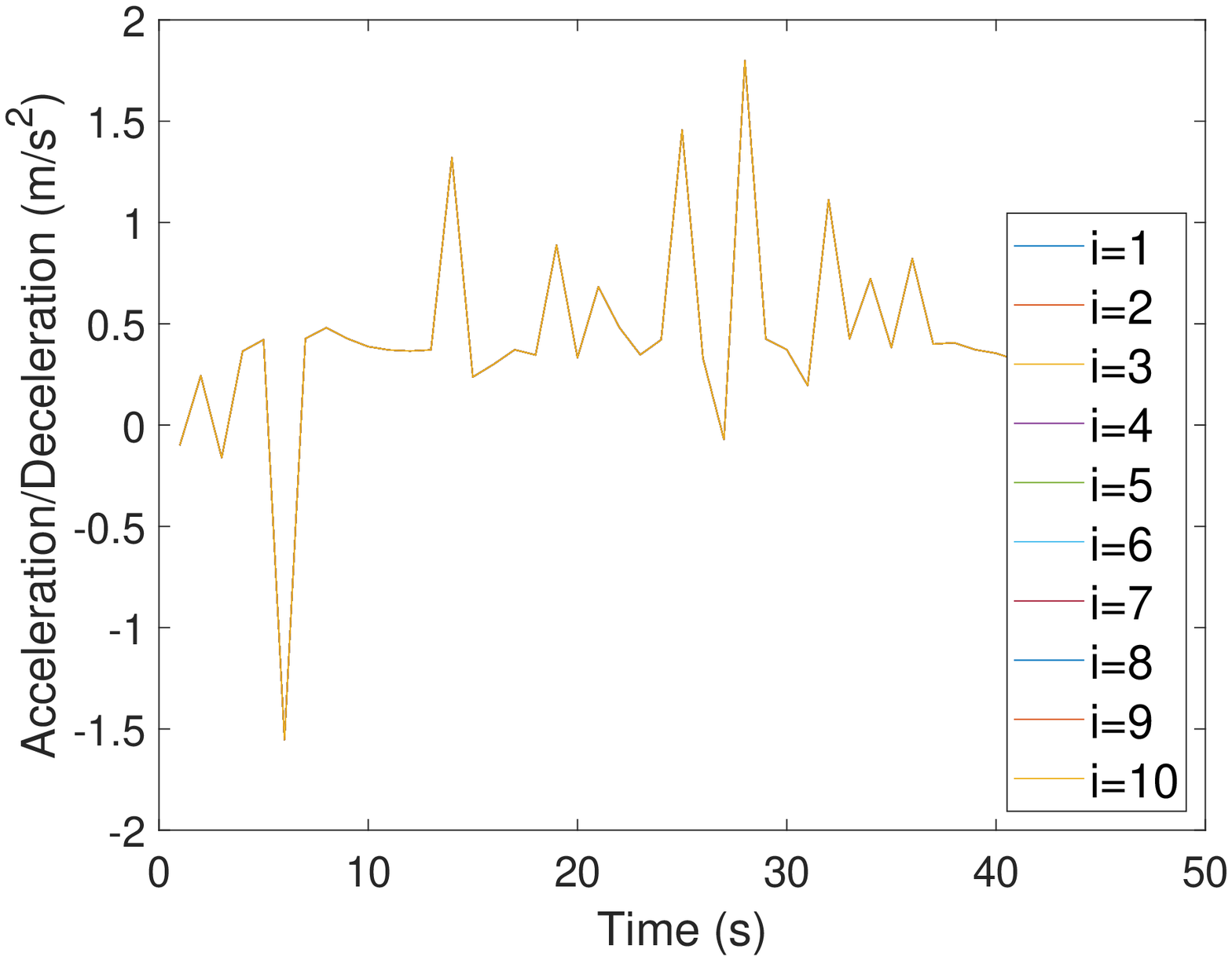}
}
\subfigure[Time history of control input] { 
\includegraphics[width=0.48\columnwidth]{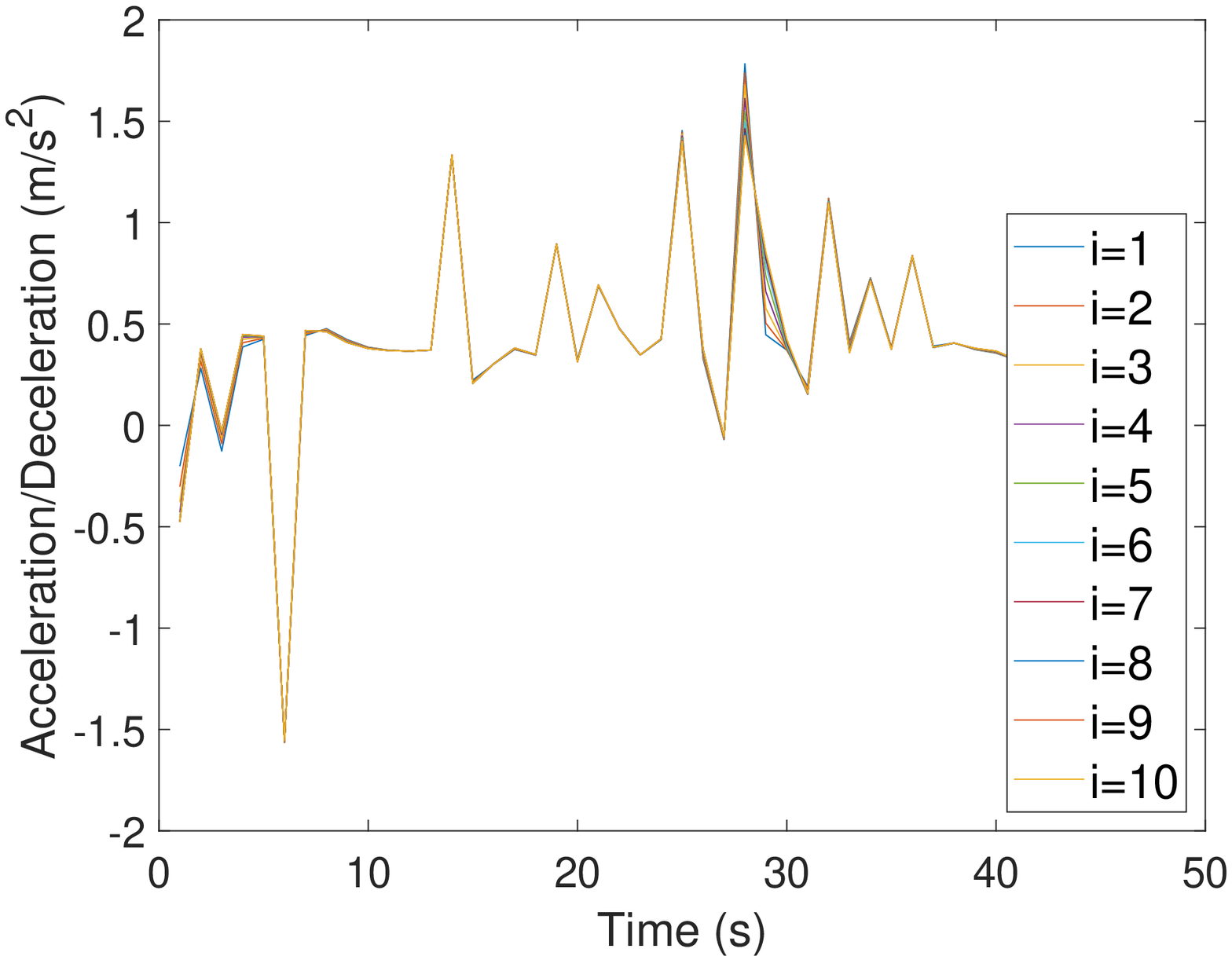}
}
\caption{Scenario 3 for the homogeneous large-size CAV platoon: platooning control with $p=1$ (left column) and $p=5$ (right column).}
\label{Fig:S3_large}
\end{figure}

%
\section{Conclusion} \label{sect:conclusion}

This paper develops a nonconvex, fully distributed  optimization based MPC scheme for CAV platooning control of a heterogeneous CAV platoon under the nonlinear vehicle dynamics. Different from the existing research on the linear vehicle dynamics, various new techniques are exploited to address several major challenges induced by the nonlinear vehicle dynamics, including distributed algorithm development for the coupled nonconvex MPC optimization problem, and stability analysis of time-varying nonlinear closed-loop dynamics. For the former, we apply locally coupled optimization and sequential convex programming for distributed algorithm development. For the latter, global implicit function theorems and Lyapunov theory for input-to-state stability, among many other techniques, are invoked for closed loop stability analysis.
Extensive numerical tests are conducted to illustrate the effectiveness of the proposed fully distributed schemes and CAV platooning control for homogeneous and heterogeneous CAV platoons in different scenarios.

\mycut{
Such schemes do not require centralized data processing or computation and are thus applicable to a wide range of vehicle communication networks.
%
%
New techniques are exploited to develop these schemes, including a new formulation of the MPC model, a decomposition method for a strongly convex quadratic objective function, formulating the underlying optimization problem as locally coupled optimization, and Douglas-Rachford method based distributed schemes. Control design and stability analysis of the closed loop dynamics is carried out for the new formulation of the MPC model. Numerical tests are conducted to illustrate the effectiveness of the proposed fully distributed schemes and CAV platooning control. Our future research will address nonlinear vehicle dynamics and robust issues under uncertainties, e.g., model mismatch, sensing errors, and communication delay.
}

%
\section*{Acknowledgements}

This research work is supported by the NSF grants  CMMI-1901994 and  CMMI-1902006.

{\small

\bibliographystyle{abbrv}
\bibliography{CAV_nonlin_bib01}
}

\end{document}